\documentclass{article}
\usepackage{tikz}
\usetikzlibrary{intersections, calc, arrows.meta}
\usepackage[utf8]{inputenc}
\usepackage{amsmath} 
\usepackage{amsmath}
\usepackage{amssymb} 
\usepackage{amsfonts} 
 \usepackage{amsthm}
 \usepackage{mathtools}
 \usepackage{latexsym}
 \usepackage{authblk}
 \usepackage[all]{xy}
\newtheorem{dfn}{Definition}[section]
 \newtheorem{them}[dfn]{Theorem}
 \newtheorem{lem}[dfn]{Lemma}
 \newtheorem{prp}[dfn]{Proposition}
\newtheorem{rem}[dfn]{Remark}
 \newtheorem{cla}[dfn]{Claim}
 \newtheorem{Que}[dfn]{Question}

 \usepackage{scalerel,stackengine}
\stackMath
\newcommand\reallywidecheck[1]{%
\savestack{\tmpbox}{\stretchto{%
  \scaleto{%
    \scalerel*[\widthof{\ensuremath{#1}}]{\kern-.6pt\bigwedge\kern-.6pt}%
    {\rule[-\textheight/2]{1ex}{\textheight}}
  }{\textheight}%
}{0.5ex}}%
\stackon[1pt]{#1}{\scalebox{-1}{\tmpbox}}%
}
\title{Existence of a positive hyperbolic orbit in the presence of an elliptic orbit in three-dimensional Reeb flows}
\author{Taisuke SHIBATA\thanks{Research Institute for Mathematical Sciences, Kyoto University, Kyoto 606-8502,
JAPAN. E-mail address: shibata@kurims.kyoto-u.ac.jp}}
\date{\today}

\begin{document}

\maketitle

\begin{abstract}
    Nondegenerate periodic orbits in three-dimensional Reeb flows can be classified into three types, positive hyperbolic, negative hyperbolic and elliptic. As a problem which involves refining the three-dimensional Weinstein conjecture, D. Cristofaro-Gardiner, M. Hutchings and D. Pomerleano proposed whether every nondegenerate closed contact three manifold has at least one positive hyperbolic orbit except for lens spaces. In the same paper, they showed the existence of at least one simple hyperbolic orbit under $b_{1}>0$ by  the isomorphism between ECH and Seiberg-Witten Floer (co)homology, especially, using the result that if $b_{1}>0$, the odd part  of ECH which detects the existence of a positive hyperbolic orbit does not vanish. But in the case of $b_{1}=0$,  such a way doesn't work.  In the present paper, we prove the existence of a positive hyperbolic orbit in the presence of at least one elliptic orbit except for some well-known trivial cases under $b_{1}=0$. The key points in this paper are the volume property with respect to ECH spectrums and the compactification of the moduli spaces of certain $J$-holomorphic curves counted by the $U$-map.
\end{abstract}
\tableofcontents
\section{Introduction}

\subsection{Background and the statement of the main result}

Let $(Y,\lambda)$ be a closed contact three manifold. $\lambda$ determines a vector field $X_{\lambda}$ which satisfies $i_{X_{\lambda}}\lambda=1$, $d\lambda(X_{\lambda},\,\,)=0$ and it is called the Reeb vector field on $(Y,\lambda)$.  A smooth map $\gamma : \mathbb{R}/T\mathbb{Z} \to Y$ is called a Reeb orbit with periodic $T$ if $\dot{\gamma} =  X_{\lambda}(\gamma)$ and simple if $\gamma$ is  a embedding map. In this paper, two Reeb orbits are considered equivalent if they differ by reparametrization. If its return map $d\phi^{T}|_{\mathrm{Ker}\lambda=\xi} :\xi_{\gamma(0)} \to \xi_{\gamma(0)}$ has no eigenvalue 1, we call it a non-degenerate Reeb orbit and we call a contact manifold $(Y,\lambda)$ non-degenerate if all Reeb orbits are non-degenerate.

The three-dimensional Weinstein conjecture was shown by C. H. Taubes by using  Seiberg Witten Floer (co)homology \cite{T1}. On the other hand, there is an usuful tool to research the properties of contact three manifolds, called Embedded contact homology(ECH), which was introduced by M. Hutchings in several papers (for example, it is briefly explained in \cite{H2}). 
In \cite{CH},  D. Cristofaro-Gardiner and M. Hucthings showed that every contact closed three manifold $(Y,\lambda)$ has at least two simple periodic orbits by using the volume property with respect to ECH spectrums.
It will be summarized as follows.
\begin{them}
A contact (possibly degenerate) closed three manifold $(Y,\lambda)$ has at least two simple periodic orbit.
\end{them}

Non-degenerate periodic orbits are classified into three types according to their eigenvalues of their return maps. A periodic orbit  is called a negative hyperbolic if $d\phi^{T}|_{\xi}$ has eigenvalues $h,h^{-1} < 0$, a positive hyperbolic if $d\phi^{T}|_{\xi}$ has eigenvalues $h,h^{-1} > 0$ and an elliptic if $d\phi^{T}|_{\xi}$ has eigenvalues $e^{\pm i2\pi\theta}$ for some $\theta \in \mathbb{R}\backslash \mathbb{Q}$.
As a refinement of the Weinstein conjecture, we can raise  questions; Does $(Y,\lambda)$ have infinity many Reeb orbits? or if the number of Reeb orbits is finite, how many are there?

In association with this question, the following holds.

\begin{them}[\cite{HT3}]
 Let $(Y,\lambda)$ be a closed contact non-degenerate three manifold. Assume that there exists exactly two simple Reeb orbit, then both of them are elliptic and $Y$ is a lens space.
\end{them}

\begin{them}[\cite{HCP}]\label{a}
Let $(Y,\lambda)$ be a non-degenerate contact three manifold. Let $\mathrm{Ker}\lambda=\xi$.
 Then

\item[1.] if $c_{1}(\xi)$ is torsion, there exists infinity many periodic orbits, or there exists exactly two elliptic simple periodic orbits and $Y$ is diffeomorphic to a lens space.
\item[2.] if $c_{1}(\xi)$ is not torsion, there exists at least four periodic orbit.

\end{them}

In the above theorem, the ways of the proof are completely different if  $c_{1}(\xi)$ is torsion or not.
If $c_{1}(\xi)$ is torsion, some indexes defined in ECH are well-controlled and this property plays an important role in the proof. But if $c_{1}(\xi)$ is not torsion, such a property does not hold. In \cite{HCP}, they proved this by splitting ECH into two parts $\mathrm{ECH}_{\mathrm{even}}$ and $\mathrm{ECH}_{\mathrm{odd}}$. In particular, if $b_{1}(Y)>0$, both elements called $U$-sequence does not vanish and hence by combining with the volume property of ECH spectrums, the proof was completed. 

On the other hand, the above $\mathrm{ECH}_{\mathrm{odd}}$ is the part which detects the existence of a positive hyperbolic orbit and thus they also obtained the following theorem.
\begin{them}[\cite{HCP}]\label{0}
If $b_{1}(Y)>0$, there exists at least one positive hyperbolic orbit.
\end{them}
In general, they asked the following question.

\begin{Que}[\cite{HCP}]\label{ques}
Let $Y$ be a closed connected three-manifold which is not $S^3$ or a lens space, and let $\lambda$ be a nondegenerate contact form on $Y$. Does $\lambda$ have a positive hyperbolic simple Reeb orbit?
\end{Que}
This is a kind of natural refining the Weinstein conjecture in the sense of looking for Reeb orbits of particular types.
On contrary to the case $b_{1}(Y)>0$, if $b_{1}(Y)=0$, then $\mathrm{ECH}_{\mathrm{odd}}$ may vanish. So to answer this question, we have to use something different way.

The next theorem is the main theorem in this paper.
\begin{them}\label{main}
Let $(Y,\lambda)$ be a nondegenerate contact three manifold with $b_{1}(Y)=0$. Suppose that $(Y,\lambda)$  has infinity many simple periodic orbits (that is, $(Y,\lambda)$ is not a lens space with exactly two simple Reeb orbits) and has at least one  elliptic orbit. Then, there exists at least one simple positive hyperbolic orbit.
\end{them}
We note that the existence of infinity many simple periodic orbits is a generic condition (see \cite{Ir}).

Properties of vector fields and maps with elliptic orbits have been classically studied for a long time related to classical mechanics and area preserving maps on surfaces and it has been known that such objects have many rich properties under some generic conditions (for example, see \cite{Ze}, \cite{Ne}, \cite{AMa} and their references).

In addition to this, the existence of an elliptic orbit also have been studied. For example, the conjecture that the Reeb flow of every convex hypersurface in $\mathbb{R}^{2n}$ carries an elliptic periodic orbit is in general open and has been researched under some additional hypothesises (for example, see \cite{AM1}, \cite{AM2}, \cite{DDE}, \cite{HrS}, \cite{Lo} and their references).  Here, it is worth to note that we can apply Theorem \ref{main} to such already known cases.

Under Theorem \ref{main}, to complete the answer of Question \ref{ques}, it is sufficient to consider the next question.

\begin{Que}
Let $(Y,\lambda)$ be a non-degenerate contact three manifold with $b_{1}(Y) =0$. If all periodic orbits are hyperbolic, does there exist at least one simple positive
hyperbolic orbit?
\end{Que}

For example, many nondegenerate contact hyperbolic three manifolds whose all periodic orbits are hyperbolic were constructed in \cite{FHa} as contact Anosov flows. In associated with this, the next question occurs.

\begin{Que}
What kind of three manifold admits a contact form whose all periodic orbits are hyperbolic?
\end{Que}

In the proof of this paper, we will exclude a case such that exactly one simple orbit is elliptic and all the others are negative hyperbolic orbits. So the next question also occurs. 

\begin{Que}
Are there any contact three manifolds such that exactly one simple orbit is elliptic and all the others are hyperbolic?
\end{Que}

Here, we introduce some results  which will be proved in the forthcoming papers \cite{S1} and \cite{S2} by using methods and results in this paper. 

\begin{enumerate}
    \item[\bf 1.] Let $L(p,q)$ $(p\neq \pm 1)$ be a lens space and $\lambda$ be a non-degenerate contact form on it with infinity many simple periodic orbits. If $p$ is odd, then there is a simple positive hyperbolic orbit (\cite{S1}).
    \item[\bf 2.] Let $(S^{3},\lambda)$ be a non-degenerate contact three sphere with infinity many simple orbits. If the Conley Zehnder index of a minimal periodic orbit with respect to the trivialization induced by a bounding disc  is more than or equal to three, then there is a simple positive hyperbolic orbit. In particular, generic compact strictly convex energy surface in $\mathbb{R}^4$ carries a simple positive hyperbolic orbit and more generally, every non-degenerate dynamically convex contact three sphere with infinity many simple orbits has a positive hyperbolic orbit (\cite{S2}).
\end{enumerate}
The above notion of dynamically convex was introduced in \cite{HWZ}.

The behaviors of ECH are completely different whether elliptic orbits exist or not. In particular,  the way in the present paper does not work if there is no elliptic orbit. But as results mentioned above, the author expects that the observations in this paper also furthermore help to study these kinds of question.

\subsection{Idea of the proof of the main theorem and the structure of this paper.}

First of all, in Section 2, we will review facts of Embedded contact homology that we will need in this paper.

To prove the main theorem, we divide the problem into two cases. One case is that $(Y,\lambda)$ has at least two simple elliptic orbits and the other is that $(Y,\lambda)$ has only one simple elliptic orbit. In Section 3, we will prove the former case by contradiction. Suppose that there is no simple positive hyperbolic orbit. Then the boundary operator $\partial$ used to define ECH always vanishes because of the property of ECH index (\ref{mod2}). Under the assumption, from $\partial=0$ we will introduce some notations and  use them to cause a contradiction with the volume property (Theorem \ref{volume}). 

In Section 4 and beyond, we will prove the main theorem under the case that $(Y,\lambda)$ has only one simple elliptic orbit. This is the most difficult part in this paper. The strategy of this part is inspired by the method in \cite{HCP} for proving the existence of infinity many Reeb orbits if $c_{1}(\mathrm{Ker}(\lambda)=\xi)$ is torsion (Theorem \ref{a}). Let us recall that way. Suppose that $c_{1}(\xi)$ is torsion and there are only finitely Reeb orbits. Then the increasing rate of the difference of  ECH index $I$ and $J_{0}$ index  are at most linear. By combining with volume property (Theorem \ref{volume}), we have either $J_{0}\leq 1$ holomorphic curve counted by the $U$-map  or sufficiently large consecutive $J_{0}=2$ holomorphic curves counted by the $U$-map. Moreover we can pick up such curves so that their energies are sufficiently small. If there is a $J_{0}\leq 1$ holomorphic curve, such a curve becomes a global section of the Reeb vector field and this implies there are infinity many Reeb orbits. If there are such sufficiently large consecutive $J_{0}=2$ holomorphic curves, then  at least one curve in them becomes a global section of the Reeb vector field and in the same as the other case, this implies there are infinity many Reeb orbits. As a result, we can conclude that $(Y,\lambda)$ has infinity many Reeb orbits. With this in mind, we now explain the strategy of the proof of the rest of the main theorem.

Suppose that there is no positive hyperbolic orbit. Note that $\partial=0$. At first, in the first half of Section 4, we will show Proposition \ref{asympdense} which asserts that the density of elements which does not have  good properties in a sense tend to 0 at the limit.  In this proof, we will also introduce some notations and technically use them combining with the volume property. In the latter half of this section, we will show Proposition \ref{main index 2} under Lemma \ref{mainlemma} by combining Proposition \ref{asympdense}. Proposition \ref{main index 2} asserts that there are  sufficiently large consecutive $J_{0}=2$ holomorphic curves counted by the $U$-map between two elements whose behaviors are  good in a sense especially whose energies are sufficiently small. Here the term “good” means that the $J$-holomorphic curves counted by the $U$-map are well-controlled for futher arguments. As the former part in this section, the volume property will also play an important role.

Next, in Section 5 we will prove Lemma \ref{mainlemma}. Lemma \ref{mainlemma} asserts that there is no $J_{0}\leq 1$ holomorphic curve counted by the $U$-map between two elements whose behaviors are  good. We will prove this by contradiction. Suppose that such a curve exists. To derive a contradiction, we will consider the compactification of the moduli spaces of the holomorphic curves counted by the $U$-map. Here we note that the $J$-holomorphic curves counted by the $U$-map through a fixed generic point $z\in Y$ and their topological types are controlled by $J_{0}$ index. By moving this $z$ and considering the properties of ECH index, we have some certain splitting curves and also approximate relations of actions in some orbits. This contradicts some properties of ECH index, especially the partition conditions of ends (Definition \ref{part}) and the properties of $S_{\pm\theta}$.

In Section 6 and Section 7, we will state and prove Proposition \ref{nagai}. Proposition \ref{nagai} asserts that such $J_{0}=2$ holomorphic curves obtained in Proposition \ref{main index 2} can be classified into six types. In particular, each type has some approximate  relations of actions in some orbits. In order to determine approximate relations, we will also observe the splitting behaviors of the $J$-holomorphic curves with small energy. More specifically, in this step,  we will list all possibilities of their splitting ways and solve dozens of simultaneous approximate equations to determine the approximate relations.

In Section 8, we will derive a contradiction from Proposition \ref{nagai} and Proposition \ref{main index 2}. More precisely, the approximate relations of actions in some orbits obtained in Proposition \ref{nagai} restrict the consecutiveness of $J_{0}=2$ holomorphic curves counted by the $U$-map. This  contradicts the result obtained in Proposition \ref{main index 2} that there are  sufficiently large consecutive $J_{0}=2$ holomorphic curves counted by the $U$-map. As a result, we will complete the proof of the main theorem.
\subsection*{Acknowledgement}
The author would like to thank his advisor Professor Kaoru Ono
for his discussion and checking the preliminary version of this paper in detail, and Suguru Ishikawa for a series of discussion. He also  thanks Masayuki Asaoka and Kei Irie for some comments.
This work was supported by JSPS KAKENHI Grant Number JP21J20300.

\section{Preliminaries}
\subsection{The definitions and properties of Embedded contact homology}
 
Let $(Y,\lambda)$ be a non-degenerate contact three manifold. For $\Gamma \in H_{1}(Y;\mathbb{Z})$, Embedded contact homology $\mathrm{ECH}(Y,\lambda,\Gamma)$ is defined. At first, we define the chain complex $(\mathrm{ECC}(Y,\lambda,\Gamma),\partial)$. In this paper, we consider ECH over $\mathbb{Z}/2\mathbb{Z}=\mathbb{F}$.

\begin{dfn} [{\cite[Definition 1.1]{H1}}]\label{qdef}
An orbit set $\alpha=\{(\alpha_{i},m_{i})\}$ is a finite pair of distinct simple periodic orbit $\alpha_{i}$ with positive integer $m_{i}$.
If $m_{i}=1$ whenever $\alpha_{i}$ is hyperboric orbit, then $\alpha=\{(\alpha_{i},m_{i})\}$ is called an admissible orbit set.
\end{dfn}
Set $[\alpha]=\sum m_{i}[\alpha_{i}] \in H_{1}(Y)$. For two orbit sets $\alpha=\{(\alpha_{i},m_{i})\}$ and $\beta=\{(\beta_{j},n_{j})\}$ with $[\alpha]=[\beta]$, we define  $H_{2}(Y,\alpha,\beta)$ to be the set of relative homology classes of
2-chains $Z$ in $Y$ with $\partial Z =\sum_{i}m_{i} \alpha_{i}-\sum_{j}m_{j}\beta_{j}$ . This is an affine space over $H_{2}(Y)$. From now on. we fix a trivialization of contact plane $\xi$ over each simple orbit and write it by $\tau$. 

\begin{dfn}[{\cite[{\S}8.2]{H1}}]
Let $\alpha_{1}$, $\beta_{1}$, $\alpha_{2}$ and $\beta_{2}$ be orbit sets with $[\alpha_{1}]=[\beta_{1}]$ and $[\alpha_{2}]=[\beta_{2}]$. For a trivialization $\tau$, we can define
\begin{equation}
    Q_{\tau}:H_{2}(Y;\alpha_{1},\beta_{1}) \times H_{2}(Y;\alpha_{2},\beta_{2}) \to \mathbb{Z}
\end{equation}
This is well-defined. Moreover if $Z_{1}\in H_{2}(Y;\alpha_{1},\beta_{1})$ and $Z_{2} \in H_{2}(Y;\alpha_{2},\beta_{2})$, then
\begin{equation}
    Q_{\tau}(Z_{1}+Z_{2},Z_{1}+Z_{2})=Q_{\tau}(Z_{1},Z_{1})+2Q_{\tau}(Z_{1},Z_{2})+Q_{\tau}(Z_{2},Z_{2}).
\end{equation}
See {\cite[{\S}8.2]{H1}} for more detail definitions.
\end{dfn}

\begin{dfn}[{\cite[Definition 1.5]{H1}}]
For $Z\in H_{2}(Y,\alpha,\beta)$, we define
\begin{equation}
    I(\alpha,\beta,Z):=c_{1}(\xi|_{Z},\tau)+Q_{\tau}(Z)+\sum_{i}\sum_{k=1}^{m_{i}}\mu_{\tau}(\alpha_{i}^{k})-\sum_{j}\sum_{k=1}^{n_{j}}\mu_{\tau}(\beta_{j}^{k}).
\end{equation}
We call $I(\alpha,\beta,Z)$ an ECH index. Here,  $\mu_{\tau}$ is the Conely Zhender index with respect to $\tau$ and $c_{1}(\xi|_{Z},\tau)$ is a reative Chern number  and $Q_{\tau}(Z)=Q_{\tau}(Z,Z)$. Moreover this is independent of $\tau$ (see  \cite{H1} for more details).
\end{dfn}
\begin{prp}[{\cite[Proposition 1.6]{H1}}]
 The ECH index $I$ has the following properties.
  \item[1.] For orbit sets $\alpha, \beta, \gamma$ with $[\alpha]=[\beta]=[\gamma]=\Gamma\in H_{1}(Y)$ and $Z\in H_{2}(Y,\alpha,\beta)$, $Z'\in H_{2}(Y,\beta,\gamma)$,
  \begin{equation}\label{adtiv}
  I(\alpha,\beta,Z)+I(\beta,\gamma,Z')=I(\alpha,\gamma,Z+Z').
  \end{equation}
  \item[2.] For $Z, Z'\in H_{2}(Y,\alpha,\beta)$,
  \begin{equation}\label{homimi}
      I(\alpha,\beta,Z)-I(\alpha,\beta,Z')=<c_{1}(\xi)+2\mathrm{PD}(\Gamma),Z-Z'>.
  \end{equation}
  \item[3.] If $\alpha$ and $\beta$ are admissible orbit sets,
  \begin{equation}\label{mod2}
      I(\alpha,\beta,Z)=\epsilon(\alpha)-\epsilon(\beta) \,\,\,\mathrm{mod}\,\,2.
  \end{equation}
  Here, $\epsilon(\alpha)$, $\epsilon(\beta)$ are the numbers of positive hyperbolic orbits in $\alpha$, $\beta$ respectively.

\end{prp}

For $\Gamma \in H_{1}(Y)$, we define $\mathrm{ECC}(Y,\lambda,\Gamma)$ as freely generated module over $\mathbb{Z}/2$ by admissible orbit sets $\alpha$ such that $[\alpha]=\Gamma$. That is
\begin{equation}
    \mathrm{ECC}(Y,\lambda,\Gamma):= \bigoplus_{\alpha:\mathrm{admissibe\,\,orbit\,\,set\,\,with\,\,}{[\alpha]=\Gamma}}\mathbb{Z}_{2}\langle \alpha \rangle.
\end{equation}
To define the differential $\partial:\mathrm{ECC}(Y,\lambda,\Gamma)\to \mathrm{ECC}(Y,\lambda,\Gamma) $, we pick a generic almost complex structure $J$  on $\mathbb{R}\times Y$ which satisfies

\begin{enumerate}
    \item $\mathbb{R}$-invariant
    \item $J(\frac{d}{ds})=X_{\lambda}$
    \item $J\xi=\xi$
\end{enumerate}

We consider $J$-holomorphic curves  $u:(\Sigma,j)\to (\mathbb{R}\times Y,J)$ where
the domain $(\Sigma, j)$ is a punctured compact Riemann surface. Here the domain $\Sigma$ is
not necessarily connected.  Let $\gamma$ be a (not necessarily simple) Reeb orbit.  If a puncture
of $u$ is asymptotic to $\mathbb{R}\times \gamma$ as $s\to \infty$, we call it a positive end of $u$ at $\gamma$ and if a puncture of $u$ is asymptotic to $\mathbb{R}\times \gamma$ as $s\to -\infty$, we call it a negative end of $u$ at $\gamma$ ( see \cite{H1} for more details ).

Let $u:(\Sigma,j)\to (\mathbb{R}\times Y,J)$ and $u':(\Sigma',j')\to (\mathbb{R}\times Y,J)$ be two $J$-holomorphic curves. If there is a biholomorphic map $\phi:(\Sigma,j)\to (\Sigma',j')$ with $u'\circ \phi= u$, we regard $u$ and $u'$ as equivalent.

 Let $\alpha=\{(\alpha_{i},m_{i})\}$ and $\beta=\{(\beta_{i},n_{i})\}$ be orbit sets. Let $\mathcal{M}^{J}(\alpha,\beta)$ denote the set of  $J$-holomorphic curves with positive ends
at covers of $\alpha_{i}$ with total covering multiplicity $m_{i}$, negative ends at covers of $\beta_{j}$
with total covering multiplicity $n_{j}$, and no other punctures. Moreover, in $\mathcal{M}^{J}(\alpha,\beta)$, we consider two
$J$-holomorphic curves   to be equivalent if they represent the same current in $\mathbb{R}\times Y$.
 
For $u \in \mathcal{M}^{J}(\alpha,\beta)$, we naturally have $[u]\in H_{2}(Y;\alpha,\beta)$ and set $I(u)=I(\alpha,\beta,[u])$. Moreover we define
\begin{equation}
     \mathcal{M}_{k}^{J}(\alpha,\beta):=\{\,u\in  \mathcal{M}^{J}(\alpha,\beta)\,|\,I(u)=k\,\,\}
\end{equation}

Under this notations, we define $\partial_{J}:\mathrm{ECC}(Y,\lambda,\Gamma)\to \mathrm{ECC}(Y,\lambda,\Gamma)$ as follows.

For admissible orbit set $\alpha$ with $[\alpha]=\Gamma$, we define

\begin{equation}
    \partial_{J} \langle \alpha \rangle=\sum_{\beta:\mathrm{admissible\,\,orbit\,\,set\,\,with\,\,}[\beta]=\Gamma} \# (\mathcal{M}_{1}^{J}(\alpha,\beta)/\mathbb{R})\cdot \langle \beta \rangle.
\end{equation}

Note that the above counting is well-defined and $\partial_{J} \circ \partial_{J}$. We can see the reason of the former in Proposition \ref{ind} and the later was proved in \cite{HT1} and \cite{HT2}. Moreover, the homology defined by $\partial_{J}$ does not depend on $J$ (see Theorem \ref{test}, or see \cite{T1}).

Before we get to the next subsection, recall (Fredholm) index.

For $u\in \mathcal{M}^{J}(\alpha,\beta)$, the  its (Fredholm) index is defined by
\begin{equation}
    \mathrm{ind}(u):=-\chi(u)+2c_{1}(\xi|_{[u]},\tau)+\sum_{k}\mu_{\tau}(\gamma_{k}^{+})-\sum_{l}\mu_{\tau}(\gamma_{l}^{-}).
\end{equation}
Here $\{\gamma_{k}^{+}\}$ is the set consisting of (not necessarilly simple) all positive ends of $u$ and $\{\gamma_{l}^{-}\}$ is that one of all negative ends.  Note that for generic $J$, if $u$ is connected and somewhere injective, then the moduli space of $J$-holomorphic
curves near $u$ is a manifold of dimension $\mathrm{ind}(u)$ (see \cite[Definition 1.3]{HT1}).

\subsection{$J$-holomorphic curves with small ECH index and partition conditions of multiplicities}
For $\theta\in \mathbb{R}\backslash \mathbb{Q}$, we define $S_{\theta}$ to be the set of positive integers $q$ such that $\frac{\lceil q\theta \rceil}{q}< \frac{\lceil q'\theta \rceil}{q'}$ for all $q'\in \{1,\,\,2,...,\,\,q-1\}$ and write $S_{\theta}=\{q_{0}=1,\,\,q_{1},\,\,q_{2},\,\,q_{3},...\}$ in increasing order. Also $S_{-\theta}=\{p_{0}=1,\,\,p_{1},\,\,p_{2},\,\,p_{3},...\}$.

\begin{prp}[{\cite[Proof of Lemma 3.3]{HT3}}, and {\cite[Proof of Remark 4.4]{H1}}]\label{s}
For $\theta \in \mathbb{R}\backslash \mathbb{Q}$,
    \item[1.] $q_{i+1}-q_{i}$ (resp. $p_{i+1}-p_{i}$) are nondecreasing with respect to $i$ and some elements of $S_{-\theta}$ (resp. $S_{\theta}$).
    \item[2.] $S_{\theta}\cap{S_{-\theta}}=\{1\}$,
    \item[3.] $q_{i+1}-q_{i}\to \infty$ (resp. $p_{i+1}-p_{i}\to \infty$) as $i\to \infty$.
\end{prp}

\begin{dfn}[{\cite[Definition 7.1]{HT1}}, or {\cite[{\S}4]{H1}}]
For non negative integer $M$, we inductively define the incoming partition $P_{\theta}^{\mathrm{in}}(M)$ as follows.

For $M=0$, $P_{\theta}^{\mathrm{in}}(0)=\emptyset$ and for $M>0$,
\begin{equation}
    P_{\theta}^{\mathrm{in}}(M):=P_{\theta}^{\mathrm{in}}(M-a)\cup{(a)}
\end{equation}

where $a:=\mathrm{max}(S_{\theta}\cap{\{1,\,\,2,...,\,\,M\}})$.
 Define outgoing partition
 \begin{equation}
      P_{\theta}^{\mathrm{out}}(M):= P_{-\theta}^{\mathrm{in}}(M).
 \end{equation}
\end{dfn}

\begin{dfn}[{\cite[Definition 7.11]{HT1}}, or cf.{\cite[{\S}4]{H1}}]\label{part}
For a simple Reeb orbit $\gamma$ and positive integer $M$, define two partitions of $M$, the incoming partition $P_{\gamma}^{\mathrm{in}}(M)$ and the outgoing partition $P_{\gamma}^{\mathrm{out}}(M)$ as follows.

    \item[1.] If $\gamma$ is positive hyperbolic, then
    \begin{equation}
        P_{\gamma}^{\mathrm{in}}(M)=P_{\gamma}^{\mathrm{out}}:=(1,...,\,1)
    \end{equation}
    \item[2.] If $\gamma$ is negative hyperbolic, then
    \begin{equation}
   P_{\gamma}^{\mathrm{in}}(M)=P_{\gamma}^{\mathrm{out}}(M):=
  \begin{cases}
    (2,...,\,2) & \mathrm{if}\,\,M\,\,\mathrm{is\,\,even}, \\
     (2,...,\,2,\,1)                & \mathrm{if}\,\,M\,\,\mathrm{is\,\,odd},
  \end{cases}
\end{equation}
\item[3.] If $\gamma$ is elliptic, then

\begin{equation}
    P_{\gamma}^{\mathrm{in}}(M):=P_{\theta}^{\mathrm{in}}(M),\,\,\,\,P_{\gamma}^{\mathrm{out}}(M):=P_{\theta}^{\mathrm{out}}(M).
\end{equation}
where $\theta$ is the rotation number of $\gamma$ up to $\mathbb{Z}$.

The standard ordering convention for $P_{\gamma}^{\mathrm{in}}(M)$ or $P_{\gamma}^{\mathrm{out}}(M)$ is to list the entries
in``nonincreasing'' order.

\end{dfn}

 Let $\alpha=\{(\alpha_{i},m_{i})\}$ and $\beta=\{(\beta_{i},n_{i})\}$. For  $u\in \mathcal{M}^{J}(\alpha,\beta)$, it can be uniquely
written as $u=u_{0}\cup{u_{1}}$ where $u_{0}$ are unions of all components which maps to $\mathbb{R}$-invariant cylinders in $u$ and $u_{1}$ is the rest of $u$.

For $u=u_{0}\cup{u_{1}}\in \mathcal{M}^{J}(\alpha,\beta)$, let $P_{\alpha_{i}}^{+}$ denote the set consisting of the multiplicities of the positive ends of $u_{1}$ at covers of $\alpha_{i}$. Define $P_{\beta_{j}}^{-}$ analogously for the negative end.

\begin{dfn}[{\cite[Definition 7.13]{HT1}}]\label{adm}
$u=u_{0}\cup{u_{1}}\in \mathcal{M}^{J}(\alpha,\beta)$ is admissible if
    \item[1.] $u_{1}$ is embedded and does not intersect $u_{0}$
    \item[2.] For each simple Reeb orbit $\alpha_{i}$ in $\alpha$ (resp. $\beta_{j}$ in $\beta$), under the standerd ordering convention, $P_{\alpha_{i}}^{+}$(resp. $P_{\beta_{j}}^{-}$) is an initial segment of $P_{\alpha_{i}}^{\mathrm{out}}(m_{i})$(resp. $P_{\beta_{j}}^{\mathrm{in}}(n_{j})$).
\end{dfn}

\begin{prp}[{\cite[Proposition 7.15]{HT1}}]\label{ind}
Suppose that $J$ is generic and $u=u_{0}\cup{u_{1}}\in \mathcal{M}^{J}(\alpha,\beta)$. Then
    \item[1.] $I(u)\geq 0$
    \item[2.] If $I(u)=0$, then $u_{1}=\emptyset$
    \item[3.] If $I(u)=1$, then $u$ is admissible and $\mathrm{ind}(u_{1})=1$.
    \item[4.] If $I(u)=2$ and $\alpha$ and $\beta$ are admissible, then u is admissible and $\mathrm{ind}(u_{1})=2$.
\end{prp}

\subsection{The gradings and the $U$-map}

By (\ref{homo}) and (\ref{adi}), we can see that  if $c_{1}(\xi)+2\mathrm{PD(\Gamma)}$ is divisible by $d$, admissible orbits sets have relative $\mathbb{Z}/d$-grading. So we can decompose $\mathrm{ECC}(Y,\lambda,\Gamma)$ as direct sum by $\mathbb{Z}/d$-grading.
\begin{equation}\label{direc}
    \mathrm{ECC}(Y,\lambda,\Gamma)= \bigoplus_{*:\,\,\mathbb{Z}/d\,\,\mathrm{grading}}\mathrm{ECC}_{*}(Y,\lambda,\Gamma).
\end{equation}
Note that if $c_{1}(\xi)+2\mathrm{PD}(\Gamma)$ is torsion, there exists the relative $\mathbb{Z}$-grading.
The same as (\ref{direc}), we can see that
 \begin{equation}
    \mathrm{ECH}(Y,\lambda,\Gamma):= \bigoplus_{*:\,\,\mathbb{Z}\,\,\mathrm{grading}}\mathrm{ECH}_{*}(Y,\lambda,\Gamma).
\end{equation}

Let $Y$ be connected.
Then there is degree$-2$ map $U$.
\begin{equation}\label{Umap}
    U:\mathrm{ECH}_{*}(Y,\lambda,\Gamma) \to \mathrm{ECH}_{*-2}(Y,\lambda,\Gamma).
\end{equation}

To define this, choose a base point $z\in Y$ which is not on the image of any Reeb orbit and let $J$ be generic.
Then define a map 

\begin{equation}
     U_{J,z}:\mathrm{ECC}_{*}(Y,\lambda,\Gamma) \to \mathrm{ECC}_{*-2}(Y,\lambda,\Gamma).
\end{equation}

by
\begin{equation}
    U_{J,z} \langle \alpha \rangle=\sum_{\beta:\mathrm{admissible\,\,orbit\,\,set\,\,with\,\,}[\beta]=\Gamma} \# \{\,u\in \mathcal{M}_{2}^{J}(\alpha,\beta)/\mathbb{R})\,|\,(0,z)\in u\,\}\cdot \langle \beta \rangle.
\end{equation}

The above map $U_{J,z}$ is a chain map, and we define the $U$ map
as the induced map on homology.  Under the assumption, this map is independent on $z$ (for a generic $J$). See \cite[{\S}2.5]{HT3} for more details. Moreover, in the same reason as $\partial$, $U_{J,z}$ does not depend on $J$ (see Theorem \ref{test}, and see \cite{T1}).

In this paper, we choose a suitable generic $J$ as necessary (Specifically, we choose a generic $J$ so that $U_{J,z}$ is well-defined for some countable sequences $z$ appearing in the future discussions).

The next isomorphism is important.

\begin{them}[\cite{T1}]\label{test}
For each $\Gamma\in H_{1}(Y)$, there is an isomorphism

\begin{equation}
\mathrm{ECH}_{*}(Y,\lambda,\Gamma) \cong \reallywidecheck{HM}_{*}(-Y,\mathfrak{s}(\xi)+2\mathrm{PD}(\Gamma))
\end{equation}
of relatively $\mathbb{Z}/d\mathbb{Z}$-graded abelian groups. Here $d$ is the divisibility of $\mathfrak{s}(\xi)+2\mathrm{PD}(\Gamma)$ in $H_{1}(Y)$ mod torsion and $\mathfrak{s}(\xi)$ is the
spin-c structure associated to the oriented 2–plane field  as in \cite{KM}.

Moerover, the above isomorphism interchanges
the map $U$ in (\ref{Umap}) with the map
\begin{equation}
   U_{\dag}: \reallywidecheck{HM}_{*}(-Y,\mathfrak{s}(\xi)+2\mathrm{PD}(\Gamma)) \longrightarrow \reallywidecheck{HM}_{*-2}(-Y,\mathfrak{s}(\xi)+2\mathrm{PD}(\Gamma))
\end{equation}
defined in \cite{KM}.

\end{them}

Here $\reallywidecheck{HM}_{*}(-Y,\mathfrak{s}(\xi)+2\mathrm{PD}(\Gamma))$ is a version of Seiberg-Witten Floer homology with
$\mathbb{Z}/2\mathbb{Z}$ coefficients defined by Kronheimer-Mrowka \cite{KM}.

\subsection{$J_{0}$ index and topological complexity of $J$-holomorphic curve}

In this subsection, we recall the $J_{0}$ index.

\begin{dfn}[{\cite[{\S}3.3]{HT3}}]
Let $\alpha=\{(\alpha_{i},m_{i})\}$ and $\beta=\{(\beta_{j},n_{j})\}$ be orbit sets with $[\alpha]=[\beta]$.
For $Z\in H_{2}(Y,\alpha,\beta)$, we define
\begin{equation}
    J_{0}(\alpha,\beta,Z):=-c_{1}(\xi|_{Z},\tau)+Q_{\tau}(Z)+\sum_{i}\sum_{k=1}^{m_{i}-1}\mu_{\tau}(\alpha_{i}^{k})-\sum_{j}\sum_{k=1}^{n_{j}-1}\mu_{\tau}(\beta_{j}^{k}).
\end{equation}
\end{dfn}

\begin{prp}[{\cite[{\S}3.3]{HT3}} {\cite[{\S}2.6]{CHR}}]
 The index $J_{0}$ has the following properties.

  \item[1.] For orbit sets $\alpha, \beta, \gamma$ with $[\alpha]=[\beta]=[\gamma]=\Gamma\in H_{1}(Y)$ and $Z\in H_{2}(Y,\alpha,\beta)$, $Z'\in H_{2}(Y,\beta,\gamma)$,
  \begin{equation}\label{adi}
  J_{0}(\alpha,\beta,Z)+J_{0}(\beta,\gamma,Z')=J_{0}(\alpha,\gamma,Z+Z').
  \end{equation}
  \item[2.] For $Z, Z'\in H_{2}(Y,\alpha,\beta)$,
  \begin{equation}\label{homo}
      J_{0}(\alpha,\beta,Z)-J_{0}(\alpha,\beta,Z')=<-c_{1}(\xi)+2\mathrm{PD}(\Gamma),Z-Z'>.
  \end{equation}

\end{prp}

\begin{dfn}[{\cite[just befor Proposition 5.8]{H3}}]
Let $u=u_{0}\cup{u_{1}}\in \mathcal{M}^{J}(\alpha,\beta)$. Suppose that $u_{1}$ is somewhere injective. Let $n_{i}^{+}$
be the number of positive ends of $u_{1}$ which are asymptotic to $\alpha_{i}$, plus 1 if $u_{0}$ includes the trivial cylinder $\mathbb{R}\times \alpha_{i}$ with some multiplicity. Likewise, let $n_{j}^{-}$ be the number of negative ends of $u_{1}$ which are asymptotic to $\beta_{j}$, plus 1 if $u_{0}$ includes the trivial cylinder $\mathbb{R}\times \beta_{j}$ with some multiplicity. 
\end{dfn}
Write $J_{0}(u)=J_{0}(\alpha,\beta,[u])$.
\begin{prp}[{\cite[Lemma 3.5]{HT3}} {\cite[Proposition 5.8]{H3}}]
Let $\alpha=\{(\alpha_{i},m_{i})\}$ and $\beta=\{(\beta_{j},n_{j})\}$ be admissible orbit sets, and let $u=u_{0}\cup{u_{1}}\in \mathcal{M}^{J}(\alpha,\beta)$. Then
\begin{equation}
    -\chi(u_{1})+\sum_{i}(n_{i}^{+}-1)+\sum_{j}(n_{j}^{-}-1)\leq J_{0}(u)
\end{equation}
If $u$ is counted by the ECH differential or the $U$-map, then the above equality holds. Note that $J_{0}(u)\geq -1$ in any case.

\end{prp}

\subsection{ECH spectrum and the volume property}

The action of an orbit set $\alpha=\{(\alpha_{i},m_{i})\}$ is defined by 
\begin{equation}
A(\alpha)=\sum m_{i}A(\alpha_{i})=\sum m_{i}\int_{\alpha_{i}}\lambda. 
\end{equation}

For any $L>0$,  $\mathrm{ECC}^{L}(Y,\lambda,\Gamma)$ denotes the subspace of  $\mathrm{ECC}(Y,\lambda,\Gamma)$ which is generated by admissible orbit sets whose actions are less than $L$. In the same way, $(\mathrm{ECC}^{L}(Y,\lambda,\Gamma),\partial)$ becomes a chain complex and the homology group $\mathrm{ECH}^{L}(Y,\lambda,\Gamma)$ is obtained. Here, we use the fact that if two admissible orbit sets $\alpha=\{(\alpha_{i},m_{i})\}$ and $\beta=\{(\beta_{i},n_{i})\}$ have $A(\alpha)\leq A(\beta)$, then the coefficient of $\beta$ in $\partial \alpha$  is $0$ because of the positivity of $J$ holomorphic curves over $d\lambda$ and the fact that $A(\alpha)-A(\beta)$ is equivalent to the integral value of $d\lambda$ over $J$-holomorphic punctured curves which is asymptotic to $\alpha$ at $+\infty$, $\beta$ at $-\infty$.

By construction, there exists a natural homomorphism $i_{L}:\mathrm{ECH}^{L}(Y,\lambda,\Gamma) \to \mathrm{ECH}(Y,\lambda,\Gamma)$.
\begin{dfn}[{\cite[Definition 4.1]{H2}}]
Let $Y$ be a closed oriented three manifold with a nondegenerate contact form $\lambda$ and $\Gamma \in H_{1}(Y,\mathbb{Z})$. If $0\neq \sigma \in \mathrm{ECH}(Y,\lambda,\Gamma)$, define
\begin{equation}\label{spect}
    c_{\sigma}(Y,\lambda)=\inf\{L>0 |\, \sigma \in \mathrm{Im}(i_{L}:\mathrm{ECH}^{L}(Y,\lambda,\Gamma) \to \mathrm{ECH}(Y,\lambda,\Gamma))\, \}
\end{equation}

\end{dfn}

Note that we can define $c_{\sigma}$ if $\lambda$ is degenerate but in this paper, this is not necessary.(for more details, see \cite{H2})

In the case that $c_{1}(\xi)+2\mathrm{PD}(\Gamma)$ is torsion, the above spectrum recover the volume of $\mathrm{Vol}(Y,\lambda)=\int_{Y}\lambda \wedge d\lambda$.

\begin{them}[{\cite[Theorem 1.3]{CHR}}]\label{volume}

Let $Y$ be a closed connected three-manifold with
a contact form $\lambda$, let $\Gamma \in H_{1}(Y)$ with $c_{1}(\xi)+2\mathrm{PD}(\gamma)$ torsion, and let $I$ be any
refinement of the relative $\mathbb{Z}$-grading on $\mathrm{ECH}(Y,\lambda,\Gamma)$ to an absolute $\mathbb{Z}$-grading.
Then for any sequence of nonzero homogeneous classes $\{\sigma_{k} \}_{1\geq k}$ in $\mathrm{ECH}(Y,\lambda,\Gamma)$
with $\lim_{k\to \infty}I(\sigma)=+\infty$, we have

\begin{equation}\label{vol}
    \lim_{k \to \infty} \frac{c_{\sigma_{k}}(Y,\lambda)^{2}}{I(\sigma_{k})}=\mathrm{Vol}(Y,\lambda).
\end{equation}
\end{them}

\section{The case that the number of simple elliptic orbits is at least two.}
Suppose that $b_{1}(Y)=0$. In this situation,  for any orbit sets $\alpha$ and $\beta$ with $[\alpha]=[\beta]$, $H_{2}(Y,\alpha,\beta)$ consists of only one component since $H_{2}(Y)=0$. So we may omit the homology component from the notation of ECH index $I$ and $J_{0}$, that is, they are just written by $I(\alpha,\beta)$ and $J_{0}(\alpha,\beta)$ respectively. 
Moreover, for any orbit sets $\alpha$ with $[\alpha]=0\in H_{1}(Y)$, we set $I(\alpha,[\emptyset]):=I(\alpha)$ and also $J_{0}(\alpha):=J_{0}(\alpha,[\emptyset])$. This $I(\alpha)$ defines an absolute $\mathbb{Z}$ grading in $\mathrm{ECH}(Y,\lambda,0)$. From now on, we suppose that $\mathrm{ECH}(Y,\lambda,0)$ is graded in this way.

The aim of this section is to prove the next proposition.

\begin{them}\label{el2}
Let $(Y,\lambda)$ be a connected non-degenerate closed contact three manifold with $b_{1}(Y)=0$. Assume that the number of simple elliptic orbits is at least two and the number of all simple orbit is infinity. then there exists at least one positive hyperbolic orbit.
\end{them}

We prove this by contradiction. At first, we show the following lemma.

\begin{lem}\label{isomero}
Let $(Y,\lambda)$ be a connected non-degenerate closed contact three manifold with $b_{1}(Y)=0$. Then for $\Gamma\in H_{1}(Y)$, there is some finite generated vector space $E_{\Gamma}$ such that
\begin{equation}\label{iso}
 \mathrm{ECH}(Y,\lambda,\Gamma)  \cong \mathbb{F}[U_{\dag}^{-1},U_{\dag}]/U_{\dag}\mathbb{F}[U_{\dag}]\bigoplus E_{\Gamma}.
\end{equation}
Morover, the above isomorphism interchanges
the map $U$ in (\ref{Umap}) with the action of the product by $U_{\dag}$ on $\mathbb{F}[U_{\dag}^{-1},U_{\dag}]/U_{\dag}\mathbb{F}[U_{\dag}]$.
\end{lem}

\begin{proof}[\bf Proof of Lemma \ref{isomero}]
There are three type homologies $\reallywidecheck{HM}_{*}(-Y,\mathfrak{s})$,  $\widehat{HM}_{*}(-Y,\mathfrak{s})$ and  $\overline{HM}_{*}(-Y,\mathfrak{s})$ in Seiberg-Witten Floer homologies and there exist an exact sequence.
\begin{equation}
... \longrightarrow \reallywidecheck{HM}_{*}(-Y,\mathfrak{s})\longrightarrow \widehat{HM}_{*}(-Y,\mathfrak{s}) \longrightarrow \overline{HM}_{*}(-Y,\mathfrak{s}) \longrightarrow \reallywidecheck{HM}_{*-1}(-Y,\mathfrak{s}) \longrightarrow ...
\end{equation}
As \cite[Proposition 35.3.1]{KM},  
\begin{equation}
\bigoplus_{*}\overline{HM}_{*}(-Y,\mathfrak{s}) \cong \mathbb{F}[U^{-1}_{\dag},U_{\dag}]
\end{equation}
By construction \cite[Definition 14.5.2, Subsection 22.1 and Subsection 22.3]{KM}, $\reallywidecheck{HM}_{*}(-Y,\mathfrak{s})$ vanishes if its grading is sufficiently low. Moreover the image of
\begin{equation}
    \bigoplus_{*}\reallywidecheck{HM}_{*}(-Y,\mathfrak{s})\longrightarrow \bigoplus_{*}\widehat{HM}_{*}(-Y,\mathfrak{s})
\end{equation}
is finite rank \cite[Proposition 22.2.3]{KM}. This finishes the proof of the lemma.
\end{proof}

\textbf{From now on, in this section we suppose that there is no positive hyperbolic orbit.}

For $M>0$ and $\Gamma\in H_{1}(Y)$, we set
\begin{equation}
    \Lambda(M,\Gamma):=\{\,\,\alpha\,\,|\,\,\alpha\,\,\mathrm{is\,\,admissible\,\,orbit\,\,set\,\,such\,\,that\,\,}[\alpha]=\Gamma \mathrm{\,\,and}\,\,A(\alpha)<M\,\,\}
\end{equation}

\begin{lem}\label{asympp}
For every  $\Gamma\in H_{1}(Y)$,
\begin{equation}\label{volp}
    \lim_{M \to \infty} \frac{M^{2}}{|\Lambda(M,\Gamma)|}=2\mathrm{Vol}(Y,\lambda).
\end{equation}
\end{lem}
\begin{proof}[\bf Proof of Lemma \ref{asympp}]
Fix  $\Gamma\in H_{1}(Y)$. By $\partial=0$ and (\ref{iso}), there are admissible orbit sets $\{\alpha_{i}^{\Gamma}\}_{0 \leq i }$ and $\{\beta_{j}^{\Gamma}\}_{0\leq j \leq m_{\Gamma}}$ with $[\alpha_{i}^{\Gamma}]=[\beta_{j}^{\Gamma}]=\Gamma$ satisfy

\begin{equation}\label{importantiso}
    \mathrm{ECH}(Y,\lambda,\Gamma)  \cong \bigoplus_{i=0}^{\infty}\mathbb{F}\langle \alpha_{i}^{\Gamma} \rangle \bigoplus \bigoplus_{j=1}^{m_{\Gamma}}\mathbb{F}\langle \beta _{j}^{\Gamma}\rangle
\end{equation}
and $U\langle \alpha_{i}^{\Gamma} \rangle=\langle \alpha_{i-1}^{\Gamma} \rangle$ for $i\geq 1$. Moreover since $\partial=0$, each admissible orbit set $\alpha$ with $[\alpha]=\Gamma$ is equal to either $\beta_{j}^{\Gamma}$ for some $1\leq j \leq m_{\Gamma}$ or $\alpha_{k}^{\Gamma}$ for some $0\leq k$.

Note that 
\begin{equation}\label{la}
    A(\alpha_{l}^{\Gamma} )-A(\alpha_{k}^{\Gamma} )>0\,\,\,\,\,\mathrm{if}\,\,\,\,l>k, 
\end{equation}
\begin{equation}\label{2}
    I(\alpha_{l}^{\Gamma})-I( \alpha_{k}^{\Gamma})=2(l-k).
\end{equation}

Since $A(\alpha_{k}^{\Gamma} )\to \infty$ as $k\to \infty$,  there is $k_{0}>0$ such that for every $k>k_{0}$  $A(\alpha_{k}^{\Gamma} )> A(\beta_{j}^{\Gamma})$ for $1\leq j \leq m_{\Gamma}$.
By (\ref{la}) and (\ref{2}),  we have
$I(\alpha_{k}^{\Gamma})=2k+I(\alpha_{0}^{\Gamma})$ and for sufficiently large $k$,  $|\Lambda(A(\alpha_{k}^{\Gamma}),\Gamma)|=k+m$. So for sufficiently large $k$, we have
\begin{equation}
    |2|\Lambda(A(\alpha_{k}^{\Gamma}),\Gamma)|-I(\alpha_{k}^{\Gamma})|\leq 2m+I(\alpha_{0}^{\Gamma})
\end{equation}

By the above we have,
\begin{equation}
    \lim_{k\to \infty}\frac{2|\Lambda(A(\alpha_{k}^{\Gamma}),\Gamma)|}{I( \alpha_{k}^{\Gamma})}=1.
\end{equation}

Since $\partial=0$ and the definition of the spectrum (\ref{spect}), $c_{\langle \alpha \rangle}(Y,\lambda)=A(\alpha)$.
So by (\ref{vol}),
\begin{equation}
    \lim_{k\to \infty}\frac{c_{\langle \alpha_{k} \rangle}(Y,\lambda)^{2}}{I(\alpha_{k}^{\Gamma})}=\lim_{k\to \infty}\frac{A(\alpha_{k}^{\Gamma})^{2}}{2|\Lambda(A(\alpha_{k}^{\Gamma}),\Gamma)|}=\mathrm{Vol}(Y,\lambda).
\end{equation}

Note that for large $M>0$, there is a large $k>0$ such that $A(\alpha_{k}^{\Gamma})\geq M \geq A(\alpha_{k-1}^{\Gamma})$.

Therefore we complete the proof of Lemma \ref{asympp}.
\end{proof}

\begin{proof}[\bf Proof of Theorem \ref{el2}]
We pick up two simple elliptic orbits $\gamma_{1}$, $\gamma_{2}$. Let $s_{1}$ and $s_{2}$ denote the orders of $[\gamma_{1}]$ and $[\gamma_{2}]$ in $H_{1}(Y)$ respectively.

Since $|H_{1}(Y)|<\infty$, we can choose infinity sequence of simple orbit $\{\delta_{i}\}_{i>0}$ satisfying

\begin{enumerate}
    \item their homology classes $[\delta_{i}]$ are in a same one.
    \item $A(\delta_{i})<A(\delta_{j})$ if $i<j$.
    \item Any $\delta_{i}$ is not equivalent to $\gamma_{1}$ and $\gamma_{2}$
\end{enumerate}

Let $r$ be the order of $[\delta_{i}]$ in $H_{1}(Y)$. Then we define a sequence of admissible orbit sets $\{\epsilon_{n}\}_{n>0}$ by $\epsilon_{n}:=\{(\delta_{r(n-1)+i},1)\}_{1\leq i \leq r}$.
By construction, $[\epsilon_{n}]=0$ and $A(\epsilon_{n})<A(\epsilon_{n+1})$.

For $t_{1},\,t_{2}\in \mathbb{Z}_{\geq 0}$, we set
\begin{equation}
    \epsilon_{(t_{1},t_{2},n)}:=(\gamma_{1},t_{1}s_{1})\cup{(\gamma_{2},t_{2}s_{2})}\cup{\epsilon_{n}}
\end{equation}
Note that $\epsilon_{(t_{1},t_{2},n)}$ is  admissible  with $[\epsilon_{(t_{1},t_{2},n)}]=0$ and morover if $(t_{1},t_{2},n)\neq(t_{1}',t_{2}',n')$ then $\epsilon_{(t_{1},t_{2},n)}\neq \epsilon_{(t_{1}',t_{2}',n')}$.

Let $T_{n}:=A(\epsilon_{n})$ and $R_{i}=s_{i}A(\delta_{i})$ for $i=1,2$, then 
\begin{equation}
A(\epsilon_{(t_{1},t_{2},n)})=t_{1}R_{1}+t_{2}R_{2}+T_{n}
\end{equation}
So for $n$,
\begin{equation}\label{eq}
     \{\,\,(t_{1},t_{2})\,\,|\,\,A(\epsilon_{(t_{1},t_{2},n)})<M\,\,\}=\{\,\,(t_{1},t_{2})\in \mathbb{Z}_{\geq 0}\times\mathbb{Z}_{\geq 0}\,\,|\,\,t_{1}R_{1}+t_{2}R_{2}<M-T_{n}\,\,\}
\end{equation}
In general, for any $S_{1},\,\,S_{2},\,\,T>0$, the number of
\begin{equation}
    \{\,\,(t_{1},t_{2})\in \mathbb{Z}_{\geq 0}\times\mathbb{Z}_{\geq 0}\,\,|\,\,t_{1}S_{1}+t_{2}S_{2}<T\,\,\}
\end{equation}
is $\frac{(T)^{2}}{2S_{1}S_{2}}+O(T)$ (for example, we can see the same argument in \cite{H2}).
So the right hand side of (\ref{eq}) is equal to $\frac{(M-T_{n})^{2}}{2R_{1}R_{2}}+O(M-T_{n})$.

For any $N\in \mathbb{Z}_{>0}$, we pick a sufficiently large $M>0$ satisfying $M>T_{N}$. Then
\begin{equation}
    \begin{split}
        |\Lambda(M,0)|>&|\{\,\,(t_{1},t_{2},n)\,\,|\,\,A(\epsilon_{(t_{1},t_{2},n)})<M\,\,,\,\,1\leq n\leq N\,\,\}|\\
        =&\sum_{k=1}^{N}|\{\,\,(t_{1},t_{2})\,\,|\,\,t_{1}R_{1}+t_{2}R_{2}<M-T_{k}\,\,\}|\\
        =&\sum_{k=1}^{N}\frac{(M-T_{k})^{2}}{2R_{1}R_{2}}+O(M-T_{k})\\
         =&\frac{(M)^{2}N}{2R_{1}R_{2}}+O(M)
    \end{split}
\end{equation}
So
\begin{equation}
    \lim_{M \to \infty} \frac{M^{2}}{|\Lambda(M,0)|}<\frac{2R_{1}R_{2}}{N}.
\end{equation}
Since we choose $N$ arbitrarily, by (\ref{volp}) we can see that $\mathrm{Vol}(Y,\lambda)=0$. This is a contradiction. We complete the proof of Theorem \ref{el2}
\end{proof}

\section{The case that the number of simple elliptic orbits is exactly one.}

We use the rest of this paper to prove the next theorem.

\begin{them}\label{mainmainmain}
Let $Y$ be a closed connected three manifold with $b_{1}(Y)=0$. Then,  $Y$ does not admit a non-degenerate contact form $\lambda$ such that  exactly one simple orbit is elliptic orbit and all the others are negative hyperbolic.
\end{them}
We prove this by contradiction.

\textbf{From now on, we assume that $(Y,\lambda)$ is non-degenerate contact three manifold such that exactly one simple orbit is elliptic and all the others are negative hyperbolic.}

Let $\gamma$ be the simple elliptic orbit. Moreover, let $A(\gamma)=R$ and $\theta$ be the rotation number with respect to some fixed trivialization $\tau$ over $\gamma$. This means $e^{\pm 2\pi \theta}$ are eigenvalues of $d\phi^{R}|_{\xi}$ and for every $k\in \mathbb{Z}$, $\mu_{\tau}(\gamma^{k})=2\lfloor k\theta \rfloor +1$.

\subsection{Density of orbit sets with some properties}

For admissible orbit set $\alpha$, Let $E(\alpha)$, $H(\alpha)$ be the multiplicity at $\gamma$ in $\alpha$ and the number of hyperbolic orbits in $\alpha$, respectively.

Recall that for $M\in \mathbb{R}$ and $\Gamma \in H_{1}(M)$, 
\begin{equation}
 \Lambda(M,\Gamma)=\{\,\alpha\,|\,\alpha\,\,\mathrm{is\,\,an\,\,admissible\,\,orbit\,\,set\,\,such\,\,that\,\,}[\alpha]=\Gamma \mathrm{\,\,and}\,\,A(\alpha)<M\,\}.
\end{equation}
In addition to this, we introduce some notations as follows.
\begin{equation}
    \Lambda_{(n,m)}(M,\Gamma):=\{\,\,\alpha \in \Lambda(M,\Gamma)\,|\,(E(\alpha),H(\alpha))=(n,m)\,\}
\end{equation}
\begin{equation}
    \Lambda_{(n,\infty)}(M,\Gamma):=\bigcup_{m=0}^{\infty}\Lambda_{(n,m)}(M,\Gamma)
\end{equation}
\begin{equation}
    \Lambda_{(\infty,m)}(M,\Gamma):=\bigcup_{n=0}^{\infty}\Lambda_{(n,m)}(M,\Gamma)
\end{equation}
\begin{equation}
     \Lambda(M):=\bigcup_{\Gamma\in H_{1}(M)}\Lambda(M,\Gamma)
\end{equation}
Note that if $M\leq 0$, the above sets become empty.
\begin{prp}\label{asympdense}
For every $\Gamma \in H_{1}(Y)$,
    \item[1.] For every positive integer $n$,
    \begin{equation}
        \lim_{M\to \infty}\frac{|\Lambda_{(n,\infty)}(M,\Gamma)|}{|\Lambda(M,\Gamma)|}=0
    \end{equation}
    \item[2.] For every positive integer $m$,
    \begin{equation}
        \lim_{M\to \infty}\frac{|\Lambda_{(\infty,m)}(M,\Gamma)|}{|\Lambda(M,\Gamma)|}=0
    \end{equation}
    \item [3.]
    \begin{equation}
        \lim_{M\to \infty}\frac{|\bigcup_{p_{i}\in S_{\theta} } \Lambda_{(p_{i},\infty)}(M,\Gamma)|}{|\Lambda(M,\Gamma)|}=0
    \end{equation}

\end{prp}

Before we try to prove Proposition \ref{asympdense}, we show the next almost trivial claim but this makes the proof of Proposition \ref{asympdense} easier.

\begin{cla}\label{easier}

For every $\Gamma\in H_{1}(Y)$, we have

\begin{equation}
    \lim_{M\to \infty}\frac{|\Lambda(M)|}{|\Lambda(M,\Gamma)|}=|H_{1}(Y)|
\end{equation}

\end{cla}

\begin{proof}[\bf Proof of Claim \ref{easier}]
By the definition, we have
\begin{equation}
    |\Lambda(M)|=\sum_{\Gamma\in H_{1}(Y)}|\Lambda(M,\Gamma)|
\end{equation}

And since Lemma \ref{asympp},

\begin{equation}
    \lim_{M\to \infty}\frac{|\Lambda(M)|}{M^{2}}= \lim_{M\to \infty}\sum_{\Gamma\in H_{1}(Y)}\frac{|\Lambda(M,\Gamma)|}{M^2}= \frac{|H_{1}(Y)|}{2\mathrm{Vol}(Y,\lambda)}.
\end{equation}

Hence
\begin{equation}
     \lim_{M\to \infty}\frac{|\Lambda(M)|}{|\Lambda(M,\Gamma)|}= \lim_{M\to \infty}\frac{|\Lambda(M)|}{M^{2}}\frac{M^{2}}{|\Lambda(M,\Gamma)|}=|H_{1}(Y)|.
\end{equation}
\end{proof}

\begin{proof}[\bf Proof of Proposition \ref{asympdense}]

Note that $|\Lambda_{(n,\infty)}(M,\Gamma)|=|\Lambda_{(n-k,\infty)}(M-kR,\Gamma-k[\gamma])|$. This is because the correspondence by adding $(\gamma,k)$ from $\Lambda_{(n-k,\infty)}(M-kR,\Gamma-k[\gamma])$ to $\Lambda_{(n,\infty)}(M,\Gamma)$ is bijective. Hence
\begin{equation}
    |\Lambda(M,\Gamma)|=\sum_{n=0}^{\infty}|\Lambda_{(n,\infty)}(M,\Gamma)|=\sum_{n=0}^{\infty}|\Lambda_{(0,\infty)}(M-nR,\Gamma-n[\gamma])|.
\end{equation}
Since $\lim_{M\to\infty}\frac{|\Lambda(M-R,\Gamma-[\gamma])|}{|\Lambda(M,\Gamma)|}=\lim_{M\to\infty}\frac{(M-R)^2}{M^2}=1$, we have
\begin{equation}
\begin{split}
    \lim_{M\to\infty}\frac{|\Lambda_{(0,\infty)}(M,\Gamma)|}{|\Lambda(M,\Gamma)|}
    =&1-\lim_{M\to \infty}\frac{\sum_{n=0}^{\infty}|\Lambda_{(0,\infty)}(M-(n+1)R,\Gamma-(n+1)[\gamma])|}{\sum_{n=0}^{\infty}|\Lambda_{(0,\infty)}(M-nR,\Gamma-n[\gamma])|}\\
    =&1-\lim_{M\to\infty}\frac{|\Lambda(M-R,\Gamma-[\gamma])|}{|\Lambda(M,\Gamma)|}\\
    =&0.
    \end{split}
\end{equation}
And hence for $n>0$,
\begin{equation}
\begin{split}
     \lim_{M\to\infty}\frac{|\Lambda_{(n,\infty)}(M,\Gamma)|}{|\Lambda(M,\Gamma)|}
    =& \lim_{M\to\infty}\frac{|\Lambda_{(0,\infty)}(M-nR,\Gamma-n[\gamma])|}{|\Lambda(M-nR,\Gamma-n[\gamma])|}\frac{|\Lambda(M-nR,\Gamma-n[\gamma])|}{|\Lambda(M,\Gamma)|}\\
    =&\lim_{M\to\infty}\frac{|\Lambda_{(0,\infty)}(M-nR,\Gamma-n[\gamma])|}{|\Lambda(M-nR,\Gamma-n[\gamma])|}\frac{(M-nR)^{2}}{M^{2}}=0.
\end{split}
\end{equation}
This completes the proof of the first statement.

To prove the second statement, we change the denominator in the statement to $|\Lambda(M)|$. By Claim \ref{easier}, it is sufficient to prove this version.
We define the map
\begin{equation}
    f:\Lambda_{(\infty,m)}(\frac{M}{2},\Gamma)\times \Lambda_{(\infty,m)}(\frac{M}{2},\Gamma) \to \Lambda(M)
\end{equation}
as follows.  
Let $\alpha$, $\beta \in I_{(\infty,m)}(\frac{M}{2})$. If $\alpha$ and $\beta$ have no common negative hyperbolic orbit, we define $f(\alpha,\beta)=\alpha\cup{\beta}$.  Otherwise, that is if $\alpha$ and $\beta$ have some common negative hyperbolic orbit, we define $f(\alpha,\beta)$ by changing all multiplicities of negative hyperbolic orbits in $\alpha\cup{\beta}$ to one. This definition is well-defined.

Let $\delta$ be an element in the image of $f$. Then the number of multiplicity at $\gamma$ in $\delta$ is at most $\frac{M}{R}$. So
by combinatorial arguments, we can find that for every $m$, there is $C_{m}$ such that $\frac{C_{m}M}{R}>|f^{-1}(\delta)|$ for any $\delta \in \Lambda(M)$. So,
\begin{equation}
    |\Lambda(M)|>\frac{|\Lambda_{(\infty,m)}(\frac{M}{2},\Gamma)|^2}{\frac{C_{m}M}{R}}
\end{equation}
thus
\begin{equation}
    \frac{\frac{C_{m}M}{R}}{|\Lambda_{(\infty,m)}(\frac{M}{2},\Gamma)|}>\frac{|\Lambda_{(\infty,m)}(\frac{M}{2},\Gamma)|}{|\Lambda(M)|}.
\end{equation}

Suppose that there are $\epsilon>0$ and $M_{k}\to \infty$ such that $\frac{|\Lambda_{(\infty,m)}(\frac{M_{k}}{2},\Gamma)|}{|\Lambda(M_{k})|}>\epsilon$ then,
\begin{equation}
\begin{split}
\lim_{k\to \infty}\frac{|\Lambda_{(\infty,m)}(\frac{M_{k}}{2},\Gamma)|}{|\Lambda(M_{k})|} \leq&
    \lim_{k\to \infty} \frac{C_{m}M_{k}}{R|\Lambda_{(\infty,m)}(\frac{M_{k}}{2},\Gamma)|}\\
    =& \lim_{k\to \infty} \frac{C_{m}M_{k}}{R|\Lambda(M_{k})|}\frac{|\Lambda(M_{k})|}{|\Lambda_{(\infty,m)}(\frac{M_{k}}{2})|}\\
    \leq&\lim_{k\to \infty} \frac{C_{m}M_{k}^{2}}{\epsilon R|\Lambda(M_{k})|}\cdot \frac{1}{M_{k}}\\
    =&\frac{2C_{m}\mathrm{Vol}(Y,\lambda)}{\epsilon R|H_{1}(Y)|}\lim_{k\to \infty}\frac{1}{M_{k}}=0.
    \end{split}
\end{equation}
This contradicts $\frac{|\Lambda_{(\infty,m)}(\frac{M_{k}}{2},\Gamma)|}{|\Lambda(M_{k})|}>\epsilon$. Hence $\lim_{M\to \infty}\frac{|\Lambda_{(\infty,m)}(\frac{M}{2},\Gamma)|}{|\Lambda(M)|}=0$ and so
\begin{equation}
\begin{split}
    \lim_{M\to \infty}\frac{|\Lambda_{(\infty,m)}(\frac{M}{2},\Gamma)|}{|\Lambda(\frac{M}{2})|}=&\lim_{M\to \infty}\frac{|\Lambda_{(\infty,m)}(\frac{M}{2})|}{|\Lambda(M)|}\frac{|\Lambda(M)|}{|\Lambda(\frac{M}{2})|}\\
    =&\lim_{M\to \infty}\frac{|\Lambda_{(\infty,m)}(\frac{M}{2},\Gamma)|}{|\Lambda(M)|}\frac{M^2}{(\frac{M}{2})^2}=0
\end{split}
\end{equation}
This completes the proof of the second statement.

Finally, we prove the third statement.
Note that for $p_{i}\in S_{\theta}$, the sequence of  $p_{i+1}-p_{i}$ is monotone increasing with respect to $i$ and diverges to infinity as $i\to \infty$ (Proposition \ref{s}). Let $s\in \mathbb{Z}_{>0}$ be the order of $[\gamma]$ in $H_{1}(Y)$. Then for every $N\in \mathbb{Z}_{>0}$, there is $l\in \mathbb{Z}_{>0}$ such that $p_{i+1}-p_{i}>sN$ for every $i>l$.
By the first statement,
\begin{equation}\label{to0}
     \lim_{M\to \infty}\frac{|\bigcup_{p_{i}\in S_{\theta},\,\,p_{i}<p_{l} } \Lambda_{(p_{i},\infty)}(M,\Gamma)|}{|\Lambda(M,\Gamma)|}=0.
\end{equation}
Since $|\Lambda_{(n,\infty)}(M,\Gamma)|\leq |\Lambda_{(n',\infty)}(M,\Gamma-(n-n')[\gamma])|$ for $n>n'$, if $p_{i}\geq p_{l}$,
\begin{equation}
    \sum_{n=p_{i}+1}^{p_{i+1}}|\Lambda_{(n,\infty)}(M,\Gamma)|\geq N|\Lambda_{(p_{i+1},\infty)}(M,\Gamma)|
\end{equation}
and then
\begin{equation}
    |\Lambda(M,\Gamma)|\geq N|\bigcup_{p_{i}\in S_{\theta},\,\,p_{l}\leq p_{i} } \Lambda_{(p_{i},\infty)}(M,\Gamma)|.
\end{equation}
By combining with (\ref{to0}), we have
\begin{equation}
    \frac{1}{N}\geq  \lim_{M\to \infty}\frac{|\bigcup_{p_{i}\in S_{\theta} } \Lambda_{(p_{i},\infty)}(M,\Gamma)|}{|\Lambda(M,\Gamma)|}.
\end{equation}
Since we can pick up $N$ arbitrary large, we complete the proof of the third statement.
\end{proof}
\subsection{Proof of Proposition \ref{main index 2} under Lemma \ref{mainlemma}}
Recall that since Lemma \ref{isomero} there is an isomorphism
\begin{equation}\label{isom}
     \mathrm{ECH}(Y,\lambda,0)  \cong \bigoplus_{k=0}^{\infty}\mathbb{F}\langle  \alpha_{k} \rangle \bigoplus \bigoplus_{j=1}^{m}\mathbb{F}\langle \beta_{j}\rangle
\end{equation}
with $U\langle  \alpha_{k} \rangle=\langle  \alpha_{k-1} \rangle$ for $k\geq 1$ (note that $U\langle  \alpha_{0} \rangle$ is not necessarily $0$).  Moreover all but finite admissible orbit sets are in $\{\alpha_{k}\}_{k\in \mathbb{Z}_{\geq0}}$. Note that $A(\alpha_{k})>A(\alpha_{l})$ if and only if $k>l$. Here we omit some notations from (\ref{importantiso}) and from now on, we do under this notations unless there is confusion.

The aim of this subsection is the proof of the next proposition under Lemmma \ref{mainlemma}.

\begin{prp}\label{main index 2}
For every $\epsilon>0$ and positive integer $l$, there is $k\in \mathbb{Z}_{>0}$ which satisfies the following condition.

The $l+1$ consecutive orbit sets $\alpha_{k},\,\,\alpha_{k+1},....\,\,\alpha_{k+l}$ satisfies for every $0\leq i \leq l$, 
\begin{enumerate}
    \item $J(\alpha_{k+i+1},\alpha_{k+i})=2$
    \item $A(\alpha_{k+i+1})-A(\alpha_{k+i})<\epsilon$ 
    \item $E(\alpha_{k+i})\notin S_{\theta}\cup{S_{-\theta}}$
    \item $E(\alpha_{k+i})>p_{1},\,\,q_{1}$
    \item $H(\alpha_{k+i})>4$
     \item In the notation of (\ref{importantiso}), for any $\Gamma \in H_{1}(Y)$ and $1\leq j \leq m_{\Gamma}$, $A(\alpha_{k+i})>A(\gamma)+A(\beta^{\Gamma}_{j})$ .
\end{enumerate}
\end{prp}

The next lemma plays an important role in the proof of Proposition \ref{main index 2}.

\begin{lem}\label{mainlemma}
For any sufficiently small $\epsilon>0$,  there is no positive integer $k$ which
satisfies
\begin{enumerate}
    \item $J(\alpha_{k+1},\alpha_{k})\leq 1$
    \item $A(\alpha_{k+1})-A(\alpha_{k})<\epsilon$ 
    \item $E(\alpha_{k+1}),\,\,E(\alpha_{k})>p_{1},\,\,q_{1}$
    \item $E(\alpha_{k+1}),\,\, E(\alpha_{k})\notin S_{\theta}\cup{ S_{-\theta}}$
    \item $H(\alpha_{k+1}),\,\,H(\alpha_{k})>4$
    \item  In the notation of (\ref{importantiso}), for any $\Gamma \in H_{1}(Y)$ and $1\leq j \leq m_{\Gamma}$, $A(\alpha_{k})>A(\gamma)+A(\beta^{\Gamma}_{j})$.
\end{enumerate}

Here we note that the number of $k$ which does not satisfy the  sixth condition is finite.
\end{lem}
Before proving Lemma \ref{mainlemma},  we will give the proof of Proposition \ref{main index 2} under Lemma \ref{mainlemma}. We will prove Lemma \ref{mainlemma} in the next section.

To prove Proposition \ref{main index 2}, we introduce some notations as follows.

For positive integer $k$ and $\epsilon>0$, we set
\begin{equation}
    \hat{I}_{(n,m)}(k):=\{\,\,\alpha_{k'}\,\,|\,\, k' \leq k, \,\, (E(\alpha_{k'}),H(\alpha_{k'}))=(n,m)\,\,\}
\end{equation}
\begin{equation}
    \hat{I}_{(n,\infty)}(k):=\bigcup_{m=0}^{\infty}\hat{I}_{(n,m)}(k)
\end{equation}
\begin{equation}
    \hat{I}_{(\infty,m)}(k):=\bigcup_{n=0}^{\infty}\hat{I}_{(n,m)}(k)
\end{equation}
\begin{equation}
    \hat{I}_{\geq\epsilon}(k):=\{\,\,\alpha_{k'}\,\,|\,\, k' \leq k, \,\, A(\alpha_{k'+1})-A(\alpha_{k'})\geq \epsilon\,\,\}
\end{equation}
\begin{equation}
    \hat{I}_{=2,<\epsilon}(k):=\{\,\,\alpha_{k'}\,\,|\,\, k' \leq k, \,\, J(\alpha_{k'+1},\alpha_{k'})=2,\,\,\,\alpha_{k'+1},\,\,\alpha_{k'}\,\,\mathrm{satisfy\,\,2,3,4,5,6\,\, in\,\,Lemma\,\, \ref{mainlemma}}\}
\end{equation}
\begin{equation}
    \hat{I}_{>2,<\epsilon}(k):=\{\,\,\alpha_{k'}\,\,|\,\, k' \leq k, \,\, J(\alpha_{k'+1},\alpha_{k'})>2,\,\,\,\alpha_{k'+1},\,\,\alpha_{k'}\,\,\mathrm{satisfy\,\,2,3,4,5,6\,\, in\,\,Lemma\,\, \ref{mainlemma}}\}
\end{equation}

\begin{proof}[\bf Proof of Proposition \ref{main index 2} under Lemma \ref{mainlemma}]

For large $n$, there is a $k_{n}\in \mathbb{Z}$ such that $\alpha_{k_{n}}=\{(\gamma,sn)\}$ where $s\in \mathbb{Z}_{>0}$ is the order of $[\gamma]$ in $H_{1}(Y)$. Here we use $[(\gamma,sn)]=0$ and (\ref{isom}).

\begin{cla}\label{asdg}
\begin{enumerate}
    \item For every positive integer $n$,
    $
    \lim_{k\to \infty}\frac{|\hat{I}_{(n,\infty)}(k)|}{k}=0
    $
    \item For every positive integer $m$,
    $
        \lim_{k\to \infty}\frac{|\hat{I}_{(\infty,m)}(k)|}{k}=0
    $
    \item 
   $
        \lim_{k\to \infty}\frac{|\bigcup_{n\in S_{\theta} } \hat{I}_{(n,\infty)}(k)|}{k}=0
    $
    \item $
    \lim_{n\to \infty}\frac{|\hat{I}_{\geq\epsilon}(k_{n})|}{k_{n}}=0
    $
\end{enumerate}
\end{cla}
\begin{proof}[\bf Proof of Claim \ref{asdg}]
The first three statement is just restatements of Proposition \ref{asympdense}. So we  have only to prove the forth statement. 

Since $I(\alpha_{k_{n}})=2k_{n}+I(\alpha_{0})$, $\lim_{n \to \infty}\frac{(snR)^2}{2k_{n}}=\mathrm{Vol}(Y,\lambda)$ and so $n<C\sqrt{k_{n}}$ for some $C>0$. So by the definition, for some $C>0$,

\begin{equation}
    \epsilon |\hat{I}_{\geq\epsilon}(k_{n})|<A(\alpha_{k_{n}})-A(\alpha_{0})=snR-A(\alpha_{0})<C\sqrt{k_{n}}
\end{equation}

Thus
\begin{equation}
   0= \lim_{n \to \infty}\frac{C\sqrt{k_{n}}}{\epsilon k_{n}}\geq  \lim_{n\to \infty}\frac{|\hat{I}_{\geq\epsilon}(k_{n})|}{k_{n}}.
\end{equation}

This finishes the proof of Claim \ref{asdg}.
\end{proof}

By definition,
\begin{equation}
    I(\alpha_{k_{n}})-J_{0}(\alpha_{k_{n}})=2nc_{s\gamma}+2\lfloor n\theta \rfloor +1
\end{equation}
where $c_{s\gamma}=c_{1}(\xi|_{*},\tau)$ with $\{*\}=H_{2}(Y;(\gamma,s),\emptyset)$ and $J_{0}(\alpha_{k_{n}})=J_{0}(\alpha_{k_{n}},\emptyset)$. So there is $C>0$ such that $|I(\alpha_{k_{n}})-J_{0}(\alpha_{k_{n}})|<Cn$ and then
\begin{equation}
    J_{0}(\alpha_{k_{n}})=\sum_{i=0}^{k_{n}-1}J_{0}(\alpha_{i+1},\alpha_{i})<2k_{n}+C\sqrt{k_{n}}.
\end{equation}
for some $C>0$. By considering $J_{0}(\alpha_{i+1},\alpha_{i})\geq -1$ and Lemma \ref{mainlemma}, for some $C>0$ we have

\begin{equation}
    \begin{split}
        2|\hat{I}_{=2,<\epsilon}(k_{n})|+3|\hat{I}_{>2,<\epsilon}(k_{n})|\leq & 2k_{n}+C\sqrt{k_{n}}+|\bigcup_{p_{i}\in S_{\theta}}\hat{I}_{(p_{i},\infty)}(k_{n}) |\\
        &+|\bigcup_{q_{i}\in S_{-\theta}}\hat{I}_{(q_{i},\infty)}(k_{n}) |
        +\sum_{i=0}^{\mathrm{max}(p_{1},q_{1})}|\hat{I}_{(i,\infty)}(k_{n})|\\
        &+\sum_{i=0}^{4}|\hat{I}_{(\infty,i)}(k_{n})|+|\hat{I}_{\geq\epsilon}(k_{n})|
    \end{split}
\end{equation}

Since the right hand side over $k_{n}$ converges to 2 as $n\to \infty$, we have
\begin{equation}\label{ikkaime}
    \lim_{n\to \infty}\frac{ 2|\hat{I}_{=2,<\epsilon}(k_{n})|+3|\hat{I}_{>2,<\epsilon}(k_{n})|}{k_{n}}\leq 2.
\end{equation}

On the other hand, we have
\begin{equation}
    \begin{split}
        k_{n}\geq |\hat{I}_{=2,<\epsilon}(k_{n})|+|\hat{I}_{>2,<\epsilon}(k_{n})|\geq & k_{n}-C\sqrt{k_{n}}-|\bigcup_{p_{i}\in S_{\theta}}\hat{I}_{(p_{i},\infty)}(k_{n}) |\\
        &-|\bigcup_{q_{i}\in S_{-\theta}}\hat{I}_{(q_{i},\infty)}(k_{n}) |
        -\sum_{i=0}^{\mathrm{max}(p_{1},q_{1})}|\hat{I}_{(i,\infty)}(k_{n})|\\
        &-\sum_{i=0}^{4}|\hat{I}_{(\infty,i)}(k_{n})|-|\hat{I}_{\geq\epsilon}(k_{n})|.
    \end{split}
\end{equation}

And so we have
\begin{equation}\label{nikaime}
    \lim_{n\to \infty}\frac{ |\hat{I}_{=2,<\epsilon}(k_{n})|+|\hat{I}_{>2,<\epsilon}(k_{n})|}{k_{n}}=1.
\end{equation}

From (\ref{ikkaime}) and (\ref{nikaime}), we have
\begin{equation}
    \lim_{n\to \infty}\frac{ |\hat{I}_{=2,<\epsilon}(k_{n})|}{k_{n}}=1.
\end{equation}
This implies that for every positive integer $l$, if $n$ is sufficiently large, we can pick up $l+1$ consecutive orbit sets $\alpha_{k}$, $\alpha_{k+1}$,...$\alpha_{k+l}$ which satisfy 2, 3, 4, 5 in Lemma \ref{mainlemma} and $J(\alpha_{k+i+1},\alpha_{k+i})=2$ for $0\leq i \leq l-1$. We complete the proof of Proposition \ref{main index 2}.
\end{proof}

\section{Proof of Lemma \ref{mainlemma}}
For our purpose, we choose $\epsilon>0$ such that $\epsilon < \frac{1}{10^{5}}\mathrm{min}\{A(\alpha)\,|\alpha\,\,\mathrm{is\,\,a\,\,Reeb\,\,orbit}\}$ and make it smaller as needed.

Since $U\langle \alpha_{k+1} \rangle= \langle \alpha_{k} \rangle$, there is at least one $J$-holomorphic curve whose ECH index is equal to 2, and we write $u=u_{0}\cup{u_{1}}\in \mathcal{M}^{J}(\alpha_{k+1},\alpha_{k})$. Note that $u_{1}$ is through a fixed generic point $z$.

To prove Lemma \ref{mainlemma}, we prepare some notations as follows. Let $g$, $k$ and $l$ be the genus of $u_{1}$, the number of punctures of $u_{1}$ and $\sum_{i}(n_{i}^{+}-1)+\sum_{j}(n_{j}^{-}-1)$ respectively. In this notation, $J_{0}(\alpha_{k+1},\alpha_{k})=-2+2g+k+l$. Note that $k$ is definitely positive and $u_{1}$ has at least one positive end because of the maximum principle. In the proof of Lemma \ref{mainlemma}, we have only to consider the cases $J_{0}=-1$, $0$, $1$.  To make the proof easier to understand, we make a list of their topological types as follows. 
\begin{itemize}
    \item[\bf Case] $J_{0}=-1$
    
    In this case,  $(g,k,l)=(0,1,0)$ may appear as  $J$-holomorphic curves counted by $U$-map.
    \item $(g,k,l)=(0,1,0)$
    
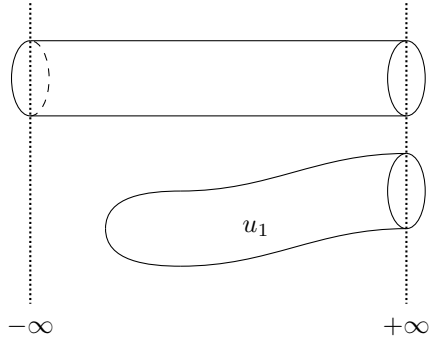
\begin{tikzpicture}
 
  \draw (7,-2.5)--(12,-2.5);
  \draw (7,-3.5)--(12,-3.5);
 \draw (7,-2.5) arc [start angle=90,end angle=270,x radius=0.25,y radius=0.5];
 \draw (7,-3.5)[dashed] arc [start angle=270,end angle=450,x radius=0.25,y radius=0.5];
 \draw (12,-2.5) arc [start angle=90,end angle=270,x radius=0.25,y radius=0.5];
 \draw (12,-3.5) arc [start angle=270,end angle=450,x radius=0.25,y radius=0.5];

 \draw (12,-4) arc [start angle=90,end angle=270,x radius=0.25,y radius=0.5];
 \draw (12,-5) arc [start angle=270,end angle=450,x radius=0.25,y radius=0.5];
 \draw (9,-4.5) to [out=0,in=180] (12,-4);
 \draw (9,-5.5) to [out=0,in=180] (12,-5);
 \draw (9,-4.5) to [out=180,in=90] (8,-5);
 \draw (8,-5) to [out=270,in=180] (9,-5.5);

  \draw[densely dotted,  thick] (7,-2)--(7,-6) ;
   \draw[densely dotted,  thick] (12,-2)--(12,-6) ;
 
\draw (7,-6.3) node{$-\infty$};  
\draw (12,-6.3) node{$+\infty$}; 

\draw (10,-5) node{$u_{1}$};

\end{tikzpicture}

    \item[\bf Case] $J_{0}=0$
    
    In this case,  $(g,k,l)=(0,1,1)$, $(0,2,0)$ may appear as $J$- holomorphic curves counted by the $U$-map.

    \item $(g,k,l)=(0,1,1)$
    
     This case has only one type as follows.

    
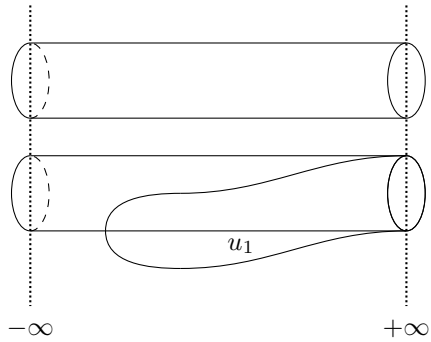
\begin{tikzpicture}
 
  \draw (7,-2.5)--(12,-2.5);
  \draw (7,-3.5)--(12,-3.5);
 \draw (7,-2.5) arc [start angle=90,end angle=270,x radius=0.25,y radius=0.5];
 \draw (7,-3.5)[dashed] arc [start angle=270,end angle=450,x radius=0.25,y radius=0.5];
 \draw (12,-2.5) arc [start angle=90,end angle=270,x radius=0.25,y radius=0.5];
 \draw (12,-3.5) arc [start angle=270,end angle=450,x radius=0.25,y radius=0.5];
 
  \draw (7,-4) arc [start angle=90,end angle=270,x radius=0.25,y radius=0.5];
 \draw (7,-5)[dashed] arc [start angle=270,end angle=450,x radius=0.25,y radius=0.5];
 \draw (12,-4) arc [start angle=90,end angle=270,x radius=0.25,y radius=0.5];
 \draw (12,-5) arc [start angle=270,end angle=450,x radius=0.25,y radius=0.5];
  \draw (7,-4)--(12,-4);
  \draw (7,-5)--(12,-5);

 \draw (12,-4) arc [start angle=90,end angle=270,x radius=0.25,y radius=0.5];
 \draw (12,-5) arc [start angle=270,end angle=450,x radius=0.25,y radius=0.5];
 \draw (9,-4.5) to [out=0,in=180] (12,-4);
 \draw (9,-5.5) to [out=0,in=180] (12,-5);
 \draw (9,-4.5) to [out=180,in=90] (8,-5);
 \draw (8,-5) to [out=270,in=180] (9,-5.5);

  \draw[densely dotted,  thick] (7,-2)--(7,-6) ;
   \draw[densely dotted,  thick] (12,-2)--(12,-6) ;
 
\draw (7,-6.3) node{$-\infty$};  
\draw (12,-6.3) node{$+\infty$}; 

\draw (9.8,-5.2) node{$u_{1}$};

\end{tikzpicture}

\item $(g,k,l)=(0,2,0)$

 This case has two types as follows.

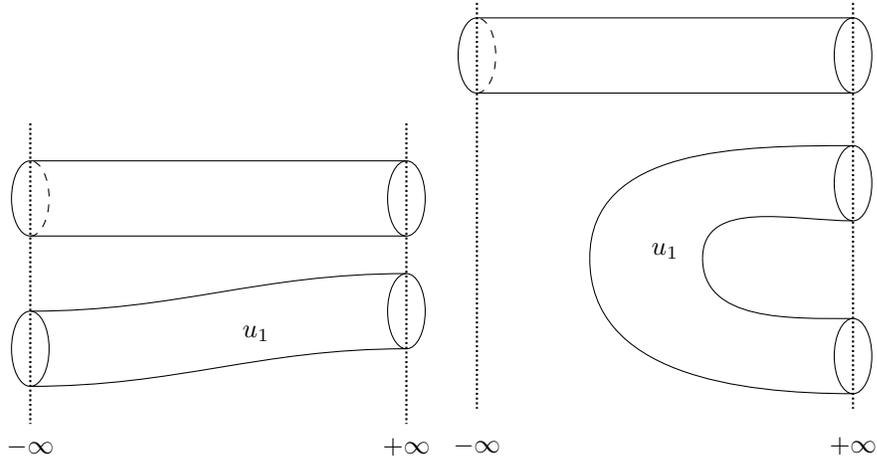
\begin{tikzpicture}
  \draw (7,-2.5)--(12,-2.5);
  \draw (7,-3.5)--(12,-3.5);
 \draw (7,-2.5) arc [start angle=90,end angle=270,x radius=0.25,y radius=0.5];
 \draw (7,-3.5)[dashed] arc [start angle=270,end angle=450,x radius=0.25,y radius=0.5];
 \draw (12,-2.5) arc [start angle=90,end angle=270,x radius=0.25,y radius=0.5];
 \draw (12,-3.5) arc [start angle=270,end angle=450,x radius=0.25,y radius=0.5];

 \draw (7,-4.5) arc [start angle=90,end angle=270,x radius=0.25,y radius=0.5];
 \draw (7,-5.5) arc [start angle=270,end angle=450,x radius=0.25,y radius=0.5];
 \draw (12,-4) arc [start angle=90,end angle=270,x radius=0.25,y radius=0.5];
 \draw (12,-5) arc [start angle=270,end angle=450,x radius=0.25,y radius=0.5];
 \draw (7,-4.5) to [out=0,in=180] (12,-4);
 \draw (7,-5.5) to [out=0,in=180] (12,-5);

  \draw[densely dotted,  thick] (7,-2)--(7,-6) ;
   \draw[densely dotted,  thick] (12,-2)--(12,-6) ;
 
\draw (7,-6.3) node{$-\infty$};  
\draw (12,-6.3) node{$+\infty$};

\draw (10,-4.8) node{$u_{1}$};

\end{tikzpicture}
\begin{tikzpicture}

\draw (19,-1.8) arc [start angle=90,end angle=270,x radius=0.25,y radius=0.5];
 \draw (19,-2.8) arc [start angle=270,end angle=450,x radius=0.25,y radius=0.5];
 \draw (19,-1.8)--(14,-1.8) ;
  \draw (19,-2.8)--(14,-2.8) ;
  \draw (14,-1.8) arc [start angle=90,end angle=270,x radius=0.25,y radius=0.5];
 \draw (14,-2.8)[dashed] arc [start angle=270,end angle=450,x radius=0.25,y radius=0.5];

 \draw (19,-3.5) arc [start angle=90,end angle=270,x radius=0.25,y radius=0.5];
 \draw (19,-4.5) arc [start angle=270,end angle=450,x radius=0.25,y radius=0.5];

 \draw (19,-5.8) arc [start angle=90,end angle=270,x radius=0.25,y radius=0.5];
 \draw (19,-6.8) arc [start angle=270,end angle=450,x radius=0.25,y radius=0.5];
 
 \draw (15.5,-5) to [out=90,in=180] (19,-3.5);
 \draw (15.5,-5) to [out=270,in=180] (19,-6.8);
 \draw (19,-5.8) to [out=180,in=270] (17,-5);
 \draw (17,-5) to [out=90,in=180] (19,-4.5);

 \draw[densely dotted,  thick] (14,-1.6)--(14,-7);
  \draw[densely dotted,  thick] (19,-1.6)--(19,-7) ;
  
  \draw (14,-7.5) node{$-\infty$};
  \draw (19,-7.5) node{$+\infty$};
  
  \draw (16.5,-4.9) node{$u_{1}$};

\end{tikzpicture}

    \item[\bf Case] $J_{0}=1$
    
    In this case,  $(g,k,l)=(0,3,0)$, $(0,2,1)$, $(1,1,0)$ may appear as $J$-holomorphic curves counted by the $U$-map. Note that we can see from the definitions of $g$, $k$, $l$ and geometric observation that the case  $(g,k,l)=(0,1,2)$ satisfies the equation $J_{0}=-2+2g+k+l=1$ but this can not appear as a $J$-holomorphic curve.

    \item $(g,k,l)=(0,3,0)$
    
     This case has three types as follows.

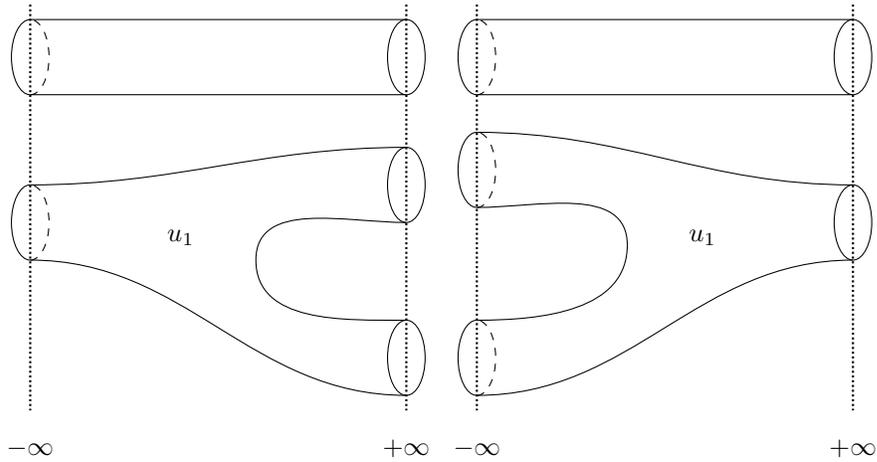
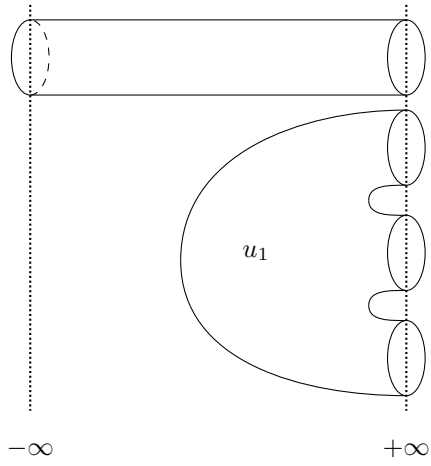
\begin{tikzpicture}

\draw (11,-1.8) arc [start angle=90,end angle=270,x radius=0.25,y radius=0.5];
 \draw (11,-2.8) arc [start angle=270,end angle=450,x radius=0.25,y radius=0.5];
 
 \draw (11,-1.8)--(6,-1.8) ;
  \draw (11,-2.8)--(6,-2.8) ;
  \draw (6,-1.8) arc [start angle=90,end angle=270,x radius=0.25,y radius=0.5];
 \draw (6,-2.8)[dashed] arc [start angle=270,end angle=450,x radius=0.25,y radius=0.5];

 \draw (6,-4) arc [start angle=90,end angle=270,x radius=0.25,y radius=0.5];
 \draw (6,-5)[dashed] arc [start angle=270,end angle=450,x radius=0.25,y radius=0.5];
 \draw (11,-3.5) arc [start angle=90,end angle=270,x radius=0.25,y radius=0.5];
 \draw (11,-4.5) arc [start angle=270,end angle=450,x radius=0.25,y radius=0.5];

 \draw (11,-5.8) arc [start angle=90,end angle=270,x radius=0.25,y radius=0.5];
 \draw (11,-6.8) arc [start angle=270,end angle=450,x radius=0.25,y radius=0.5];

 \draw (6,-4) to [out=0,in=180] (11,-3.5);
 \draw (6,-5) to [out=0,in=180] (11,-6.8);
 \draw (11,-5.8) to [out=180,in=270] (9,-5);
 \draw (9,-5) to [out=90,in=180] (11,-4.5);

 \draw[densely dotted,  thick] (6,-1.6)--(6,-7);
  \draw[densely dotted,  thick] (11,-1.6)--(11,-7) ;
  
  \draw (6,-7.5) node{$-\infty$};
  \draw (11,-7.5) node{$+\infty$};
  
  \draw (8,-4.7) node{$u_{1}$};
    
  \end{tikzpicture}
  \begin{tikzpicture}

\draw (18,-1.8) arc [start angle=90,end angle=270,x radius=0.25,y radius=0.5];
 \draw (18,-2.8) arc [start angle=270,end angle=450,x radius=0.25,y radius=0.5];
 
 \draw (18,-1.8)--(13,-1.8) ;
  \draw (18,-2.8)--(13,-2.8) ;
  \draw (13,-1.8) arc [start angle=90,end angle=270,x radius=0.25,y radius=0.5];
 \draw (13,-2.8)[dashed] arc [start angle=270,end angle=450,x radius=0.25,y radius=0.5];

 \draw (13,-3.3) arc [start angle=90,end angle=270,x radius=0.25,y radius=0.5];
 \draw (13,-4.3)[dashed] arc [start angle=270,end angle=450,x radius=0.25,y radius=0.5];

 \draw (18,-4) arc [start angle=90,end angle=270,x radius=0.25,y radius=0.5];
 \draw (18,-5) arc [start angle=270,end angle=450,x radius=0.25,y radius=0.5];

 \draw (18,-5) to [out=180,in=0] (13,-6.8);
 \draw (13,-5.8) to [out=0,in=270] (15,-4.8);
 \draw (15,-4.8) to [out=90,in=0] (13,-4.3);
 \draw (18,-4) to [out=180,in=0] (13,-3.3);

  \draw (13,-5.8) arc [start angle=90,end angle=270,x radius=0.25,y radius=0.5];
 \draw (13,-6.8)[dashed] arc [start angle=270,end angle=450,x radius=0.25,y radius=0.5];
 
 \draw[densely dotted,  thick] (13,-1.6)--(13,-7);
  \draw[densely dotted,  thick] (18,-1.6)--(18,-7) ;
  
  \draw (13,-7.5) node{$-\infty$};
  \draw (18,-7.5) node{$+\infty$};
  
  \draw (16,-4.7) node{$u_{1}$};

\end{tikzpicture}

\begin{tikzpicture}

\draw (11,-1.8) arc [start angle=90,end angle=270,x radius=0.25,y radius=0.5];
 \draw (11,-2.8) arc [start angle=270,end angle=450,x radius=0.25,y radius=0.5];
 
 \draw (11,-1.8)--(6,-1.8) ;
  \draw (11,-2.8)--(6,-2.8) ;
  \draw (6,-1.8) arc [start angle=90,end angle=270,x radius=0.25,y radius=0.5];
 \draw (6,-2.8)[dashed] arc [start angle=270,end angle=450,x radius=0.25,y radius=0.5];

 \draw (11,-3) arc [start angle=90,end angle=270,x radius=0.25,y radius=0.5];
 \draw (11,-4) arc [start angle=270,end angle=450,x radius=0.25,y radius=0.5];
 
 \draw (11,-4.4) arc [start angle=90,end angle=270,x radius=0.25,y radius=0.5];
 \draw (11,-5.4) arc [start angle=270,end angle=450,x radius=0.25,y radius=0.5];

 \draw (11,-5.8) arc [start angle=90,end angle=270,x radius=0.25,y radius=0.5];
 \draw (11,-6.8) arc [start angle=270,end angle=450,x radius=0.25,y radius=0.5];

 \draw (11,-6.8) to [out=180,in=270] (8,-5);
 \draw (8,-5) to [out=90,in=180] (11,-3);

 \draw[densely dotted,  thick] (6,-1.6)--(6,-7);
  \draw[densely dotted,  thick] (11,-1.6)--(11,-7) ;
  
  \draw (6,-7.5) node{$-\infty$};
  \draw (11,-7.5) node{$+\infty$};
  
  \draw (9,-4.9) node{$u_{1}$};
  
   \draw (11,-4) to [out=180,in=90] (10.5,-4.2);
   \draw (10.5,-4.2) to [out=270,in=180] (11,-4.4);

  \draw (11,-5.4) to [out=180,in=90] (10.5,-5.6);
    \draw (10.5,-5.6) to [out=270,in=180] (11,-5.8);

\end{tikzpicture}

\item $(g,k,l)=(0,2,1)$

 This case has three types as follows.

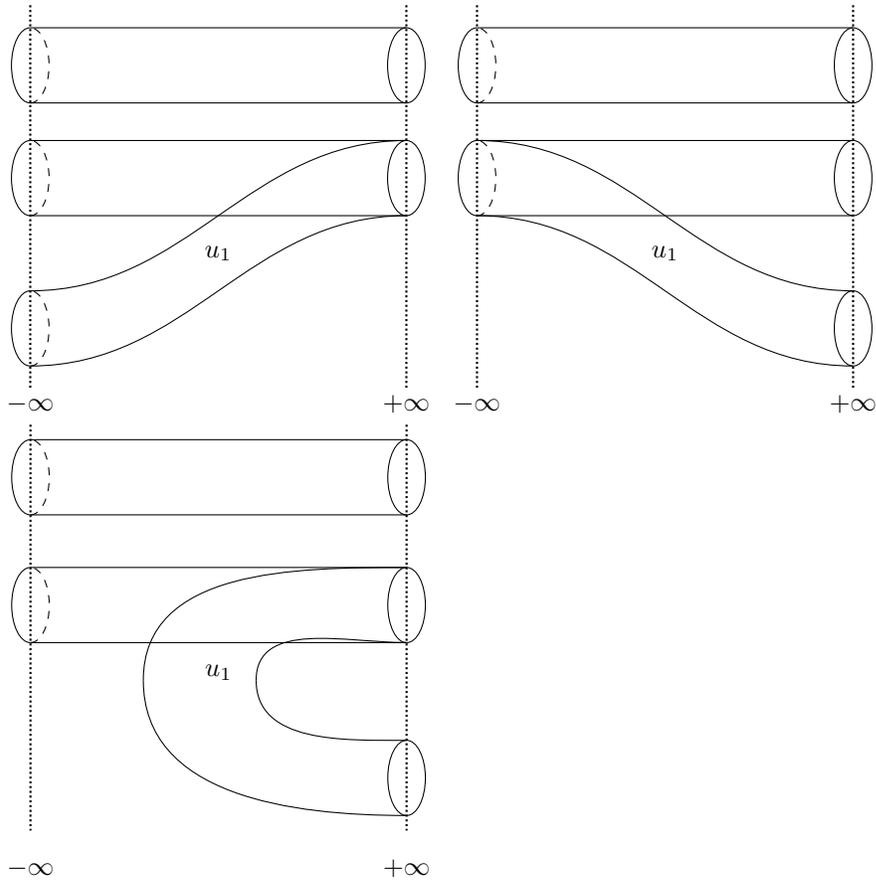
\begin{tikzpicture}
 
 \draw (2,-2.5)--(7,-2.5);
  \draw (2,-3.5)--(7,-3.5);
 \draw (2,-2.5) arc [start angle=90,end angle=270,x radius=0.25,y radius=0.5];
 \draw (2,-3.5)[dashed] arc [start angle=270,end angle=450,x radius=0.25,y radius=0.5];
 \draw (7,-2.5) arc [start angle=90,end angle=270,x radius=0.25,y radius=0.5];
 \draw (7,-3.5) arc [start angle=270,end angle=450,x radius=0.25,y radius=0.5];

\draw (2,-4)--(7,-4);
  \draw (2,-5)--(7,-5);
 \draw (2,-4) arc [start angle=90,end angle=270,x radius=0.25,y radius=0.5];
 \draw (2,-5)[dashed] arc [start angle=270,end angle=450,x radius=0.25,y radius=0.5];
 \draw (7,-4) arc [start angle=90,end angle=270,x radius=0.25,y radius=0.5];
 \draw (7,-5) arc [start angle=270,end angle=450,x radius=0.25,y radius=0.5];
 
  \draw (2,-6) arc [start angle=90,end angle=270,x radius=0.25,y radius=0.5];
 \draw (2,-7)[dashed] arc [start angle=270,end angle=450,x radius=0.25,y radius=0.5];
  \draw (2,-6) to [out=0,in=180] (7,-4);
 \draw (2,-7) to [out=0,in=180] (7,-5);
 
 \draw[densely dotted,  thick] (2,-2.2)--(2,-7.3);
  \draw[densely dotted,  thick] (7,-2.2)--(7,-7.3) ;

 \draw (4.5,-5.5) node{$u_{1}$};
 
 \draw (7,-7.5) node{$+\infty$};
  \draw (2,-7.5) node{$-\infty$};
 
\end{tikzpicture}
\begin{tikzpicture}
 
 \draw (2,-2.5)--(7,-2.5);
  \draw (2,-3.5)--(7,-3.5);
 \draw (2,-2.5) arc [start angle=90,end angle=270,x radius=0.25,y radius=0.5];
 \draw (2,-3.5)[dashed] arc [start angle=270,end angle=450,x radius=0.25,y radius=0.5];
 \draw (7,-2.5) arc [start angle=90,end angle=270,x radius=0.25,y radius=0.5];
 \draw (7,-3.5) arc [start angle=270,end angle=450,x radius=0.25,y radius=0.5];

\draw (2,-4)--(7,-4);
  \draw (2,-5)--(7,-5);
  \draw (2,-4) arc [start angle=90,end angle=270,x radius=0.25,y radius=0.5];
 \draw (2,-5)[dashed] arc [start angle=270,end angle=450,x radius=0.25,y radius=0.5];
 \draw (7,-4) arc [start angle=90,end angle=270,x radius=0.25,y radius=0.5];
 \draw (7,-5) arc [start angle=270,end angle=450,x radius=0.25,y radius=0.5];

  \draw (7,-6) arc [start angle=90,end angle=270,x radius=0.25,y radius=0.5];
 \draw (7,-7) arc [start angle=270,end angle=450,x radius=0.25,y radius=0.5];

  \draw (2,-4) to [out=0,in=180] (7,-6);
 \draw (2,-5) to [out=0,in=180] (7,-7);
 
 \draw[densely dotted,  thick] (2,-2.2)--(2,-7.3);
  \draw[densely dotted,  thick] (7,-2.2)--(7,-7.3) ;

 \draw (4.5,-5.5) node{$u_{1}$};
 
 \draw (7,-7.5) node{$+\infty$};
  \draw (2,-7.5) node{$-\infty$};
 
\end{tikzpicture}
\begin{tikzpicture}

\draw (19,-1.8) arc [start angle=90,end angle=270,x radius=0.25,y radius=0.5];
 \draw (19,-2.8) arc [start angle=270,end angle=450,x radius=0.25,y radius=0.5];
 \draw (19,-1.8)--(14,-1.8) ;
  \draw (19,-2.8)--(14,-2.8) ;
  \draw (14,-1.8) arc [start angle=90,end angle=270,x radius=0.25,y radius=0.5];
 \draw (14,-2.8)[dashed] arc [start angle=270,end angle=450,x radius=0.25,y radius=0.5];

 \draw (19,-3.5) arc [start angle=90,end angle=270,x radius=0.25,y radius=0.5];
 \draw (19,-4.5) arc [start angle=270,end angle=450,x radius=0.25,y radius=0.5];

 \draw (19,-5.8) arc [start angle=90,end angle=270,x radius=0.25,y radius=0.5];
 \draw (19,-6.8) arc [start angle=270,end angle=450,x radius=0.25,y radius=0.5];
 
 \draw (15.5,-5) to [out=90,in=180] (19,-3.5);
 \draw (15.5,-5) to [out=270,in=180] (19,-6.8);
 \draw (19,-5.8) to [out=180,in=270] (17,-5);
 \draw (17,-5) to [out=90,in=180] (19,-4.5);

 \draw[densely dotted,  thick] (14,-1.6)--(14,-7);
  \draw[densely dotted,  thick] (19,-1.6)--(19,-7) ;
  
  \draw (14,-7.5) node{$-\infty$};
  \draw (19,-7.5) node{$+\infty$};
  
  \draw (16.5,-4.9) node{$u_{1}$};
  
  \draw (14,-3.5) arc [start angle=90,end angle=270,x radius=0.25,y radius=0.5];
 \draw (14,-4.5)[dashed] arc [start angle=270,end angle=450,x radius=0.25,y radius=0.5];
  \draw (19,-3.5)--(14,-3.5) ;
  \draw (19,-4.5)--(14,-4.5) ;

\end{tikzpicture}

\item $(g,k,l)=(1,1,0)$

 This case has only one type as follows.

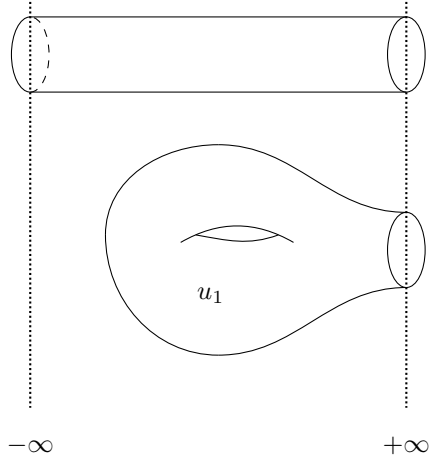
\begin{tikzpicture}

\draw (11,-1.8) arc [start angle=90,end angle=270,x radius=0.25,y radius=0.5];
 \draw (11,-2.8) arc [start angle=270,end angle=450,x radius=0.25,y radius=0.5];
 
 \draw (11,-1.8)--(6,-1.8) ;
  \draw (11,-2.8)--(6,-2.8) ;
  \draw (6,-1.8) arc [start angle=90,end angle=270,x radius=0.25,y radius=0.5];
 \draw (6,-2.8)[dashed] arc [start angle=270,end angle=450,x radius=0.25,y radius=0.5];

 \draw (11,-4.4) arc [start angle=90,end angle=270,x radius=0.25,y radius=0.5];
 \draw (11,-5.4) arc [start angle=270,end angle=450,x radius=0.25,y radius=0.5];

 \draw (11,-4.4) to [out=180,in=0] (8.5,-3.5);
  \draw (11,-5.4) to [out=180,in=0] (8.5,-6.3);
  
  \draw (8.5,-3.5) to [out=180,in=90] (7,-4.7);
\draw (7,-4.7) to [out=270,in=180] (8.5,-6.3);

 \draw[densely dotted,  thick] (6,-1.6)--(6,-7);
  \draw[densely dotted,  thick] (11,-1.6)--(11,-7) ;
  
  \draw (8,-4.8) to [out=30,in=150] (9.5,-4.8);
  \draw (8.2,-4.7) to [out=350,in=200] (9.3,-4.7);
  
  \draw (6,-7.5) node{$-\infty$};
  \draw (11,-7.5) node{$+\infty$};
  
  \draw (8.4,-5.5) node{$u_{1}$};

\end{tikzpicture}

\end{itemize}

\begin{proof}[\bf Proof of Lemma \ref{mainlemma}]
Since $J_{0}\geq -1$, we prove Lemma \ref{mainlemma} by dividing into three cases $J_{0}(\alpha_{k+1},\alpha_{k})=-1,\,\,0,\,\,1$.

\item[\bf Case1.] $J_{0}(\alpha_{k+1},\alpha_{k})=-1$

In this case, only $(g,k,l)=(0,1,0)$ satisfies this equation. The integral value of this $J$-holomorphic curve over $d\lambda$ is equal to $A(\alpha_{k+1})-A(\alpha_{k})$ by Stoke's theorem and so moreover equal to some action of a Reeb orbit. This contradicts  $A(\alpha_{k+1})-A(\alpha_{k})<\epsilon < \frac{1}{10^5}\mathrm{max}\{A(\alpha)\,\,|\alpha\,\,\mathrm{is\,\,a\,\,Reeb\,\,orbit}\,\,\}$.

\item[\bf Case2.] $J_{0}(\alpha_{k+1},\alpha_{k})=0$

In this case, $(g,k,l)=(0,2,0)$, $(0,1,1)$. For the same reason as Case1, we  have only to consider the case $(g,k,l)=(0,2,0)$ and $u_{1}$ has both positive and negative ends and their two orbits are different each other. Since $E(\alpha_{k+1}),\,\, E(\alpha_{k})\notin S_{\theta}\cup{S_{-\theta}}$, $l=0$ and the partition conditions of the ends of admissible curves (Definition \ref{part}, Definition \ref{adm} and Proposition \ref{ind}), $u_{1}$ has no end asymptotic to $\gamma$. Moreover since $E(\alpha_{k+1}),\,\, E(\alpha_{k})>p_{1},\,\,q_{1}>1$, $u_{0}$ contains some covering of  $\mathbb{R}\times \gamma$ and so we have $E(\alpha_{k+1})=E(\alpha_{k})$. Let $\delta_{1}$ and $\delta_{2}$ be the Reeb orbits where the positive and negative end of $u_{1}$ are  asymptotic respectively. We set $E(\alpha_{k+1})=E(\alpha_{k})=M$. Then we can describe $\alpha_{k+1}$, $\alpha_{k}$ as $\alpha_{k+1}=\hat{\alpha}\cup{(\gamma,M)\cup{(\delta_{1},1)}}$ and $\alpha_{k}=\hat{\alpha}\cup{(\gamma,M)\cup{(\delta_{2},1)}}$ respectively where $\hat{\alpha}$ is an admissible orbit set consisting of some negative hyperbolic orbits which do not contain $\delta_{1}$, $\delta_{2}$.

By the above argument, we can see that for any generic $z\in Y$, non trivial parts of  $J$-holomorhic curves counted by $U\langle \alpha_{k+1} \rangle=\langle \alpha_{k} \rangle$ through $z$ are in $\mathcal{M}^{J}(\delta_{1},\delta_{2})$ and of genus zero.

Now, we  consider the behaviors of such $J$-holomorphic curves as $z\to \gamma$. By compactness argument (in this situation, for example see \cite[$\S 9$]{H1}), its some subsequence have a limiting $J$-holomorphic curve $u_{1}^{\infty}$ up to $\mathbb{R}$-action which may be splitting into two floors.

Suppose that $u_{1}^{\infty}\in \mathcal{M}^{J}(\delta_{1},\delta_{2})$, then by construction, it intersects with $\mathbb{R}\times \gamma$ but this contradicts the admissibility of curves of ECH index 2. So we may assume that the limiting curve has two floors. By construction, both have ends asymptotic to $\gamma$ and same multiplicity. 

we set $u_{1}^{\infty}=(u_{-}^{\infty},u_{+}^{\infty})$ where $u_{\pm}^{\infty}$ are top and bottom curves up to $\mathbb{R}$-action respectively (see the below figure). The additivity and non negativeness of ECH index, we have $I(u_{-}^{\infty}\cup{u_{0}})=I(u_{+}^{\infty}\cup{u_{0}})=1$ and thus the multiplicity of positive or negative ends of $u_{\pm}^{\infty}$ are one since $S_{\theta}\cap{S_{-\theta}}=\{1\}$. Then  we have $u_{-}^{\infty}\cup{u_{0}}\in \mathcal{M}^{J}(\hat{\alpha}\cup{(\gamma,M+1)},\alpha_{k})$ and $u_{+}^{\infty}\cup{u_{0}}\in \mathcal{M}^{J}(\alpha_{k+1},\hat{\alpha}\cup{(\gamma,M+1)})$. 

By definition assumption, $\hat{\alpha}\cup{(\gamma,M+1)}$ is admissible and have no positive hyperbolic orbit. So by (\ref{mod2}), $I(u_{-}^{\infty}\cup{u_{0}})=I(u_{+}^{\infty}\cup{u_{0}})=0\,\,\mathrm{mod}\,2$. This is a contradiction.

Here, we introduce another way to derive a contradiction from $u_{-}^{\infty}\cup{u_{0}}\in \mathcal{M}^{J}(\hat{\alpha}\cup{(\gamma,M+1)},\alpha_{k})$, $u_{+}^{\infty}\cup{u_{0}}\in \mathcal{M}^{J}(\alpha_{k+1},\hat{\alpha}\cup{(\gamma,M+1)})$ and $I(u_{-}^{\infty}\cup{u_{0}})=I(u_{+}^{\infty}\cup{u_{0}})=1$. From the partition condition of admissible $J$-holomorphic curve at $\gamma$, we have $1=\mathrm{max}(S_{-\theta}\cap{\{1,\,\,2,...,\,\,M+1\}})=\mathrm{max}(S_{\theta}\cap{\{1,\,\,2,...,\,\,M+1\}})$. But  by the assumption $M=E(\alpha_{k+1})=E(\alpha_{k})>p_{1},\,\,q_{1}$, we have $\mathrm{max}(S_{-\theta}\cap{\{1,\,\,2,...,\,\,M+1\}})$, $\mathrm{max}(S_{\theta}\cap{\{1,\,\,2,...,\,\,M+1\}}) >p_{1},\,\,q_{1}>1$. This is a contradiction.

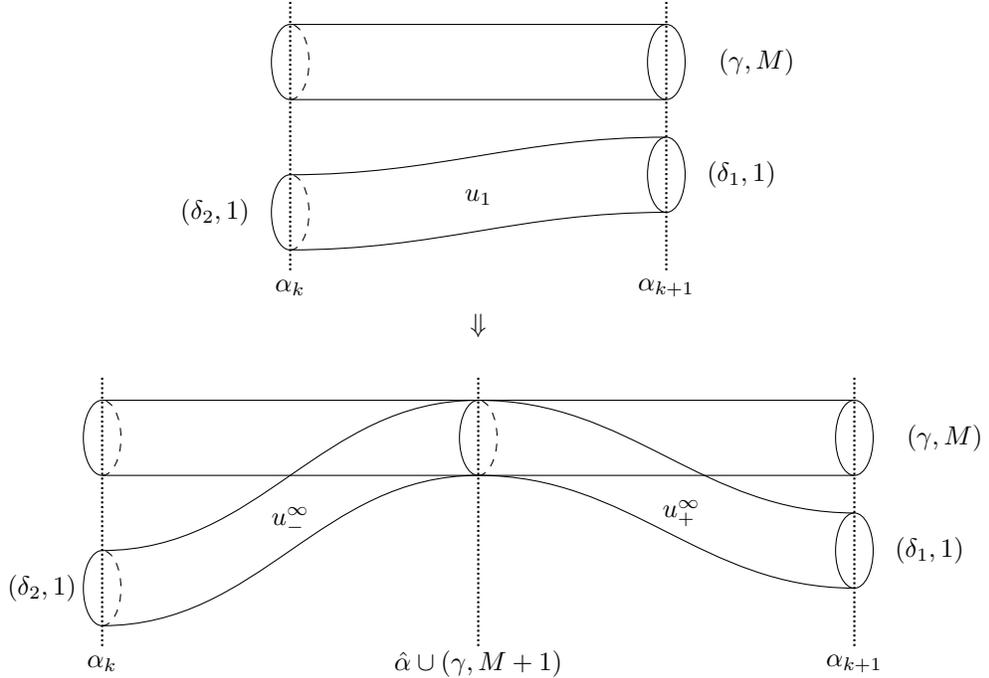
\begin{tikzpicture}\label{fig;1}
 
\draw (2,-4)--(7,-4);
  \draw (2,-5)--(7,-5);
 \draw (2,-4) arc [start angle=90,end angle=270,x radius=0.25,y radius=0.5];
 \draw (2,-5)[dashed] arc [start angle=270,end angle=450,x radius=0.25,y radius=0.5];

  \draw (2,-6) arc [start angle=90,end angle=270,x radius=0.25,y radius=0.5];
 \draw (2,-7)[dashed] arc [start angle=270,end angle=450,x radius=0.25,y radius=0.5];
  \draw (2,-6) to [out=0,in=180] (7,-4);
 \draw (2,-7) to [out=0,in=180] (7,-5);
 
 \draw[densely dotted,  thick] (2,-3.7)--(2,-7.3);
  \draw[densely dotted,  thick] (7,-3.7)--(7,-7.3) ;

  \draw (7,-1.3) node{$u_{1}$};
   \draw (4.5,-2.5) node{$\alpha_{k}$};
   \draw (9.5,-2.5) node{$\alpha_{k+1}$};
   \draw (10.7,0.5) node{$(\gamma,M)$};
   \draw (10.5,-1) node{$(\delta_{1},1)$};
   \draw (3.5,-1.5) node{$(\delta_{2},1)$};
   
   \draw (4.5,-5.6) node{$u_{-}^{\infty}$};
   \draw (7,-7.5) node{$\hat{\alpha}\cup{(\gamma,M+1)}$};
   \draw (9.7,-5.5) node{$u_{+}^{\infty}$};
    \draw (2,-7.5) node{$\alpha_{k}$};
   \draw (12,-7.5) node{$\alpha_{k+1}$};
   \draw (13.2,-4.5) node{$(\gamma,M)$};
   \draw (13,-6) node{$(\delta_{1},1)$};
   \draw (1.2,-6.5) node{$(\delta_{2},1)$};

\draw (4.5,1)--(9.5,1);
  \draw (4.5,0)--(9.5,0);
 \draw (4.5,1) arc [start angle=90,end angle=270,x radius=0.25,y radius=0.5];
 \draw (4.5,0)[dashed] arc [start angle=270,end angle=450,x radius=0.25,y radius=0.5];

  \draw (4.5,-1) arc [start angle=90,end angle=270,x radius=0.25,y radius=0.5];
 \draw (4.5,-2)[dashed] arc [start angle=270,end angle=450,x radius=0.25,y radius=0.5];

  \draw (4.5,-1) to [out=0,in=180] (9.5,-0.5);
 \draw (4.5,-2) to [out=0,in=180] (9.5,-1.5);
 
 \draw[densely dotted,  thick] (4.5,1.3)--(4.5,-2.3);
  \draw[densely dotted,  thick] (9.5,1.3)--(9.5,-2.3) ;

 \draw (9.5,1) arc [start angle=90,end angle=270,x radius=0.25,y radius=0.5];
 \draw (9.5,0) arc [start angle=270,end angle=450,x radius=0.25,y radius=0.5];
 
 \draw (9.5,-0.5) arc [start angle=90,end angle=270,x radius=0.25,y radius=0.5];
 \draw (9.5,-1.5) arc [start angle=270,end angle=450,x radius=0.25,y radius=0.5];

  \draw (7,-3)node{$\Downarrow$};

\draw (7,-4)--(12,-4);
  \draw (7,-5)--(12,-5);
 \draw (7,-4) arc [start angle=90,end angle=270,x radius=0.25,y radius=0.5];
 \draw (7,-5)[dashed] arc [start angle=270,end angle=450,x radius=0.25,y radius=0.5];
 \draw (12,-4) arc [start angle=90,end angle=270,x radius=0.25,y radius=0.5];
 \draw (12,-5) arc [start angle=270,end angle=450,x radius=0.25,y radius=0.5];
 
  \draw (12,-5.5) arc [start angle=90,end angle=270,x radius=0.25,y radius=0.5];
 \draw (12,-6.5) arc [start angle=270,end angle=450,x radius=0.25,y radius=0.5];

  \draw (7,-4) to [out=0,in=180] (12,-5.5);
 \draw (7,-5) to [out=0,in=180] (12,-6.5);
 
  \draw[densely dotted,  thick] (12,-3.7)--(12,-7.3) ;

\end{tikzpicture}

\item[\bf Case3.] $J_{0}(\alpha_{k+1},\alpha_{k})=1$

In this case, $(g,k,l)=(0,3,0)$, $(0,2,1)$, $(1,1,0)$. 

We can exclude the case $(g,k,l)=(1,1,0)$  in the same way as Case1. If $(g,k,l)=(0,3,0)$, we have $E(\alpha_{k+1})=E(\alpha_{k})$ just like the way explained in Case2. On the other hand, If $(g,k,l)=(0,2,1)$, the image of $u_{0}$ contains $\mathbb{R}\times \gamma$ and also one positive end or negative end of $u_{1}$ is asymptotic to $\gamma$ and thus  $E(\alpha_{k+1})\neq E(\alpha_{k})$. This implies that the pair $\alpha_{k+1}$, $\alpha_{k}$ which $(g,k,l)=(0,3,0)$ or $(g,k,l)=(0,2,1)$ types $J$-holomorphic curves by the $U$-map can occur are mutually exclusive.

\item[\bf 1.] If $(g,k,l)=(0,3,0)$.

Let $u_{0}\cup{u_{1}}\in \mathcal{M}^{J}(\alpha_{k+1},\alpha_{k})$ be a $J$-holomorphic curve counted the by $U$-map. Since $A(\alpha_{k+1})-A(\alpha_{k})<\epsilon$, $u_{1}$ has either two  positive ends and one negative end or one positive end and two negative ends. Without loss of generality, we may assume that $u_{1}$ has two positive ends asymptotic to $\delta_{1}$, $\delta_{2}$ and one negative end asymptotic to $\delta_{3}$. Note that $\delta_{1}$, $\delta_{2}$ and $\delta_{3}$ are mutually different because of the smallness of $A(\alpha_{k+1})-A(\alpha_{k})$.

In this notation. we have $u_{1}\in \mathcal{M}^{J}((\delta_{1},1)\cup{(\delta_{2},1)},(\delta_{3},1))$ for any non-trivial parts of $J$-holomorphic curve counted by $U\langle \alpha_{k+1} \rangle = \langle \alpha_{k} \rangle$ and also we write $\alpha_{k+1}=\hat{\alpha}\cup{(\delta_{2},1)\cup{(\delta_{1},1)\cup{(\gamma,M)}}}$, $\alpha_{k}=\hat{\alpha}\cup{(\delta_{3},1)\cup{(\gamma,M)}}$.

As $z\to \gamma$, we have three possibilities of splitting of $u_{1}$(see the below figure). In any case, this is a contradiction in the same reason as Case2.

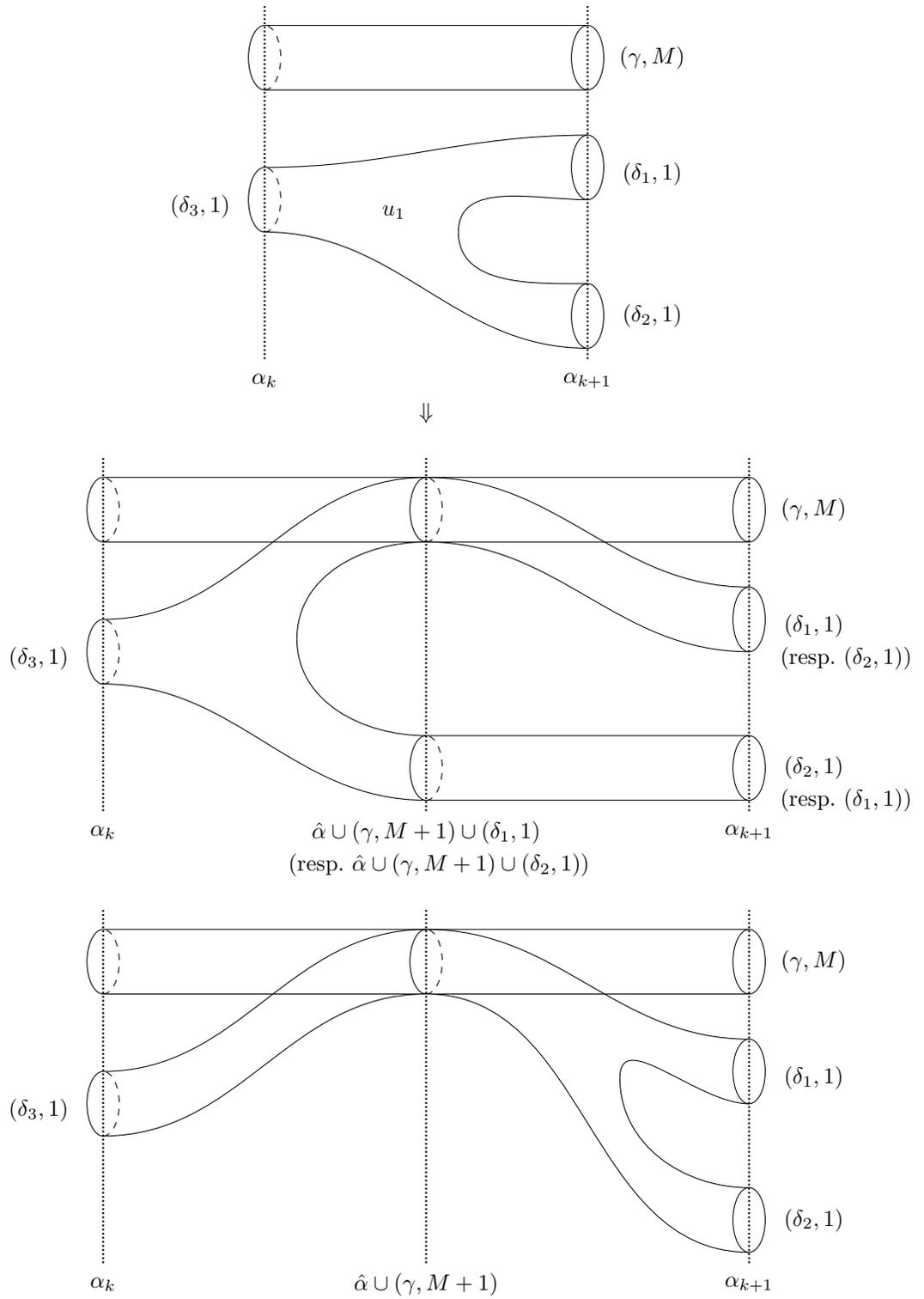
\begin{tikzpicture}

\draw (13.5,-1.0) arc [start angle=90,end angle=270,x radius=0.25,y radius=0.5];
 \draw (13.5,-2.0) arc [start angle=270,end angle=450,x radius=0.25,y radius=0.5];
 
 \draw (13.5,-1)--(8.5,-1) ;
  \draw (13.5,-2)--(8.5,-2) ;
  \draw (8.5,-1) arc [start angle=90,end angle=270,x radius=0.25,y radius=0.5];
 \draw (8.5,-2)[dashed] arc [start angle=270,end angle=450,x radius=0.25,y radius=0.5];
 
 \draw (13.5,-6.5) node{$\alpha_{k+1}$};
 \draw (8.5,-6.5) node{$\alpha_{k}$};
 \draw (14.5,-1.5) node{$(\gamma,M)$};
  \draw (14.5,-3.3) node{$(\delta_{1},1)$};
  \draw (14.5,-5.5) node{$(\delta_{2},1)$};
\draw (7.5,-3.8) node{$(\delta_{3},1)$};

\draw (16,-13.5) node{$\alpha_{k+1}$};
\draw (6,-13.5) node{$\alpha_{k}$};

\draw (11,-13.5) node{$\hat{\alpha}\cup{(\gamma,M+1)\cup{(\delta_{1},1)}}$};

\draw (11.2,-14) node{$(\mathrm{resp.}\,\,\hat{\alpha}\cup{(\gamma,M+1)\cup{(\delta_{2},1)}})$};

 \draw (17,-8.5) node{$(\gamma,M)$};
  \draw (17,-10.3) node{$(\delta_{1},1)$};
  \draw (17,-12.5) node{$(\delta_{2},1)$};
  
   \draw (17,-15.5) node{$(\gamma,M)$};
  \draw (17,-17.3) node{$(\delta_{1},1)$};
  \draw (17,-19.5) node{$(\delta_{2},1)$};
  
   \draw (17.5,-10.8) node{$(\mathrm{resp.}\,\,(\delta_{2},1))$};
  \draw (17.5,-13) node{$(\mathrm{resp.}\,\,(\delta_{1},1))$};
  
  \draw (16,-20.5) node{$\alpha_{k+1}$};
\draw (6,-20.5) node{$\alpha_{k}$};

\draw (11,-20.5) node{$\hat{\alpha}\cup{(\gamma,M+1)}$};

\draw (5,-10.8) node{$(\delta_{3},1)$};

\draw (5,-17.8) node{$(\delta_{3},1)$};

 \draw (8.5,-3.2) arc [start angle=90,end angle=270,x radius=0.25,y radius=0.5];
 \draw (8.5,-4.2)[dashed] arc [start angle=270,end angle=450,x radius=0.25,y radius=0.5];
 \draw (13.5,-2.7) arc [start angle=90,end angle=270,x radius=0.25,y radius=0.5];
 \draw (13.5,-3.7) arc [start angle=270,end angle=450,x radius=0.25,y radius=0.5];

 \draw (13.5,-5) arc [start angle=90,end angle=270,x radius=0.25,y radius=0.5];
 \draw (13.5,-6) arc [start angle=270,end angle=450,x radius=0.25,y radius=0.5];

 \draw (8.5,-3.2) to [out=0,in=180] (13.5,-2.7);
 \draw (8.5,-4.2) to [out=0,in=180] (13.5,-6);
 \draw (13.5,-5) to [out=180,in=270] (11.5,-4.2);
 \draw (11.5,-4.2) to [out=90,in=180] (13.5,-3.7);

 \draw[densely dotted,  thick] (8.5,-0.7)--(8.5,-6.2);
  \draw[densely dotted,  thick] (13.5,-0.7)--(13.5,-6.2) ;

  \draw (10.5,-3.9) node{$u_{1}$};

 \draw (11,-7)node{$\Downarrow$};

\draw (11,-8) arc [start angle=90,end angle=270,x radius=0.25,y radius=0.5];
 \draw (11,-9)[dashed] arc [start angle=270,end angle=450,x radius=0.25,y radius=0.5];
 
 \draw (11,-8)--(6,-8) ;
  \draw (11,-9)--(6,-9) ;
  \draw (6,-8) arc [start angle=90,end angle=270,x radius=0.25,y radius=0.5];
 \draw (6,-9)[dashed] arc [start angle=270,end angle=450,x radius=0.25,y radius=0.5];

 \draw (6,-10.2) arc [start angle=90,end angle=270,x radius=0.25,y radius=0.5];
 \draw (6,-11.2)[dashed] arc [start angle=270,end angle=450,x radius=0.25,y radius=0.5];

 \draw (11,-12) arc [start angle=90,end angle=270,x radius=0.25,y radius=0.5];
 \draw (11,-13)[dashed] arc [start angle=270,end angle=450,x radius=0.25,y radius=0.5];

 \draw (6,-10.2) to [out=0,in=180] (11,-8);
 \draw (6,-11.2) to [out=0,in=180] (11,-13);
 \draw (11,-9) to [out=180,in=90] (9,-10.5);
 \draw (9,-10.5) to [out=270,in=180] (11,-12);
 
 \draw (16,-9.7) arc [start angle=90,end angle=270,x radius=0.25,y radius=0.5];
 \draw (16,-10.7) arc [start angle=270,end angle=450,x radius=0.25,y radius=0.5];

  \draw (11,-8) to [out=0,in=180] (16,-9.7);
 \draw (11,-9) to [out=0,in=180] (16,-10.7);

 \draw[densely dotted,  thick] (6,-7.7)--(6,-13.2);
  \draw[densely dotted,  thick] (11,-7.7)--(11,-13.2) ;
  
   \draw[densely dotted,  thick] (16,-7.7)--(16,-13.2) ;

  \draw (16,-12) arc [start angle=90,end angle=270,x radius=0.25,y radius=0.5];
 \draw (16,-13) arc [start angle=270,end angle=450,x radius=0.25,y radius=0.5];
 
 \draw (16,-12)--(11,-12) ;
  \draw (16,-13)--(11,-13) ;

   \draw (16,-8) arc [start angle=90,end angle=270,x radius=0.25,y radius=0.5];
 \draw (16,-9) arc [start angle=270,end angle=450,x radius=0.25,y radius=0.5];
 
 \draw (16,-8)--(11,-8) ;
  \draw (16,-9)--(11,-9) ;

\draw (16,-15) arc [start angle=90,end angle=270,x radius=0.25,y radius=0.5];
 \draw (16,-16) arc [start angle=270,end angle=450,x radius=0.25,y radius=0.5];
 
 \draw (16,-15)--(6,-15) ;
  \draw (16,-16)--(6,-16) ;
  \draw (11,-15) arc [start angle=90,end angle=270,x radius=0.25,y radius=0.5];
 \draw (11,-16)[dashed] arc [start angle=270,end angle=450,x radius=0.25,y radius=0.5];
 
  \draw (6,-15) arc [start angle=90,end angle=270,x radius=0.25,y radius=0.5];
 \draw (6,-16)[dashed] arc [start angle=270,end angle=450,x radius=0.25,y radius=0.5];

 \draw (6,-17.2) arc [start angle=90,end angle=270,x radius=0.25,y radius=0.5];
 \draw (6,-18.2)[dashed] arc [start angle=270,end angle=450,x radius=0.25,y radius=0.5];
 \draw (16,-16.7) arc [start angle=90,end angle=270,x radius=0.25,y radius=0.5];
 \draw (16,-17.7) arc [start angle=270,end angle=450,x radius=0.25,y radius=0.5];
 
\draw (11,-15) to [out=180,in=0] (6,-17.2);
 \draw (11,-16) to [out=180,in=0] (6,-18.2);
 
 \draw (16,-19) arc [start angle=90,end angle=270,x radius=0.25,y radius=0.5];
 \draw (16,-20) arc [start angle=270,end angle=450,x radius=0.25,y radius=0.5];

 \draw (11,-15) to [out=0,in=180] (16,-16.7);
 \draw (11,-16) to [out=0,in=180] (16,-20);
 \draw (16,-19) to [out=180,in=270] (14,-17.3);
 \draw (14,-17.3) to [out=90,in=180] (16,-17.7);

  \draw[densely dotted,  thick] (6,-14.7)--(6,-20.2) ;
 \draw[densely dotted,  thick] (11,-14.7)--(11,-20.2);
  \draw[densely dotted,  thick] (16,-14.7)--(16,-20.2) ;

\end{tikzpicture}

\item[\bf 2.]If $(g,k,l)=(0,2,1)$.

From now on, we consider $(g,k,l)=(0,2,1)$. This case is more complicated but the way in this case also play an important role in Section 6 and beyond.

Since $A(\alpha_{k+1})-A(\alpha_{k})<\epsilon$, $u_{1}$ has both positive and negative end. Moreover by definition, $u_{0}$ contains some covering of $\mathbb{R}\times \gamma$ and either positive or negative end of $u_{1}$ asymptotic to $\gamma$. Because symmetry allows the same argument, here we consider only the case that the positive end of $u_{1}$ is asymptotic to $\gamma$  (see the below figure). Let $E(\alpha_{k+1})=M$. Then, By the admissibility of $u$, the multiplicity of positive end of $u_{1}$ at $\gamma$ is $p_{i}:=\mathrm{max}(S_{-\theta}\cap{\{1,\,\,2,...,\,\,M\}})$ and so $E(\alpha_{k})=M-p_{i}$. Let $\delta$ be the orbit where the negative end of $u_{1}$ is asymptotic . By the discussion so far, we can see that for any generic point $z\in Y$, any non trivial parts of $J$-holomorphic curves through $z$ counted by the $U$-map are in $\mathcal{M}^{J}((\gamma,p_{i}),(\delta,1))$.

We set $\alpha_{k+1}=\hat{\alpha}\cup{(\gamma,M)}$ and then $\alpha_{k}=\hat{\alpha}\cup{(\delta,1)\cup{(\gamma,M-p_{i})}}$ where $\hat{\alpha}$ only contains  negative hyperbolic orbits.

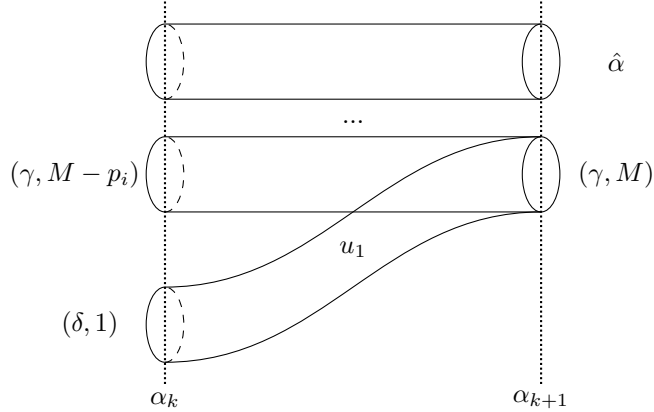
\begin{tikzpicture}
 
 \draw (2,-2.5)--(7,-2.5);
  \draw (2,-3.5)--(7,-3.5);
 \draw (2,-2.5) arc [start angle=90,end angle=270,x radius=0.25,y radius=0.5];
 \draw (2,-3.5)[dashed] arc [start angle=270,end angle=450,x radius=0.25,y radius=0.5];
 \draw (7,-2.5) arc [start angle=90,end angle=270,x radius=0.25,y radius=0.5];
 \draw (7,-3.5) arc [start angle=270,end angle=450,x radius=0.25,y radius=0.5];
 \draw (4.5,-3.8) node{...};
 
 \draw (8,-3) node{$\hat{\alpha}$};
 
\draw (2,-4)--(7,-4);
  \draw (2,-5)--(7,-5);
 \draw (2,-4) arc [start angle=90,end angle=270,x radius=0.25,y radius=0.5];
 \draw (2,-5)[dashed] arc [start angle=270,end angle=450,x radius=0.25,y radius=0.5];
 \draw (7,-4) arc [start angle=90,end angle=270,x radius=0.25,y radius=0.5];
 \draw (7,-5) arc [start angle=270,end angle=450,x radius=0.25,y radius=0.5];
 
  \draw (2,-6) arc [start angle=90,end angle=270,x radius=0.25,y radius=0.5];
 \draw (2,-7)[dashed] arc [start angle=270,end angle=450,x radius=0.25,y radius=0.5];
  \draw (2,-6) to [out=0,in=180] (7,-4);
 \draw (2,-7) to [out=0,in=180] (7,-5);
 
 \draw[densely dotted,  thick] (2,-2.2)--(2,-7.3);
  \draw[densely dotted,  thick] (7,-2.2)--(7,-7.3) ;

 \draw (8,-4.5) node{$(\gamma,M)$};
 
 \draw (0.8,-4.5) node{$(\gamma,M-p_{i})$};
 \draw (1,-6.5) node{$(\delta,1)$};
 \draw (4.5,-5.5) node{$u_{1}$};
 
 \draw (7,-7.5) node{$\alpha_{k+1}$};
  \draw (2,-7.5) node{$\alpha_{k}$};
 
\end{tikzpicture}

\begin{cla}\label{mi1} In the above notation,
\begin{equation}
I(\hat{\alpha}\cup{(\gamma,M-1)},\hat{\alpha}\cup{(\delta,1)}\cup{(\gamma,M-p_{i}-1)})=2
\end{equation}
\end{cla}
\begin{proof}[\bf Proof of Claim \ref{mi1}]

Suppose that this claim is false. Since $A(\alpha_{k+1})-A(\alpha_{k})=A(\hat{\alpha}\cup{(\gamma,M-1)})-A(\hat{\alpha}\cup{(\delta,1)}\cup{(\gamma,M-p_{i}-1)})>0$ and (\ref{isom}), we have
\begin{equation}
I(\hat{\alpha}\cup{(\gamma,M-1)},\hat{\alpha}\cup{(\delta,1)}\cup{(\gamma,M-p_{i}-1)})>2.
\end{equation}
By considering the sixth condition in Lemma \ref{mainlemma}, there is an admissible orbit set $\zeta$ with $[\zeta]=[\hat{\alpha}\cup{(\gamma,M-1)}]=[\hat{\alpha}\cup{(\delta,1)}\cup{(\gamma,M-p_{i}-1)}]$ such that $U\langle \hat{\alpha}\cup{(\gamma,M-1)} \rangle=\langle \zeta \rangle$ and $A(\hat{\alpha}\cup{(\gamma,M-1)})>A(\zeta)>A(\hat{\alpha}\cup{(\delta,1)}\cup{(\gamma,M-p_{i}-1)})$.  This implies that $A(\alpha_{k+1})>A(\zeta\cup{(\gamma,1)})>A(\alpha_{k})$.  This contradicts (\ref{importantiso}).

\end{proof}

\begin{rem}
In essence, Claim \ref{mi1} and Claim \ref{index2to4} come from only their topological conditions and the properties of ECH index (in particular the equation (\ref{espe}); For example see \cite[Proof of Proposition 7.1]{H1}). But to understand the proof only with the facts in this paper as much as possible, we prove them by using some special conditions.
\end{rem}

\begin{cla}\label{index2to4}
In the above notation, for any $p_{i}\leq N<p_{i+1}$ 
\begin{equation}
    I(\hat{\alpha}\cup{(\gamma,N)},\hat{\alpha}\cup{(\delta,1)\cup{(\gamma,N-p_{i})}})=2.
\end{equation}
Moreover
\begin{equation}
    I(\hat{\alpha}\cup{(\gamma,p_{i+1})},\hat{\alpha}\cup{(\delta,1)\cup{(\gamma,p_{i+1}-p_{i})}})=4.
\end{equation}
And for any $p_{i}<N\leq p_{i+1}$,
\begin{equation}
     J_{0}(\hat{\alpha}\cup{(\gamma,N)},\hat{\alpha}\cup{(\delta,1)\cup{(\gamma,N-p_{i})}})=1,
\end{equation}
\end{cla}

\begin{proof}[\bf Proof of Claim \ref{index2to4}]
Let $\{Z\}=H_{2}(Y;\hat{\alpha}\cup{(\gamma,p_{i}),\hat{\alpha}\cup{(\delta_{1},1)}})$ and $\{Z_{\gamma}\}=H_{2}(Y;\gamma,\gamma)$ respectively.
Then by the definition, we have
\begin{equation}
\begin{split}
    2=I(\alpha_{k+1},\alpha_{k})=&I(\hat{\alpha}\cup{(\gamma,M)},\hat{\alpha}\cup{(\delta,1)}\cup{(\gamma,M-p_{i})})\\
    =&c_{1}(\xi|_{Z},\tau)+Q_{\tau}(Z,Z)+2(M-p_{i})Q_{\tau}(Z,Z_{\gamma})\\
    &+\sum_{k=M-p_{i}+1}^{M}(2\lfloor k\theta \rfloor+1)-\mu_{\tau}(\delta).
    \end{split}
\end{equation}

Here, we  use the property $Q_{\tau}(Z_{1}+Z_{2},Z_{1}+Z_{2})=Q_{\tau}(Z_{1},Z_{1})+2Q_{\tau}(Z_{1},Z_{2})+Q_{\tau}(Z_{2},Z_{2})$ in Definition \ref{qdef} and $Q_{\tau}(Z_{\gamma})=0$ and  $Q_{\tau}(Z,mZ_{\gamma})=mQ_{\tau}(Z,Z_{\gamma})$. These properties easily follows from the definition of $Q_{\tau}$ (see \cite[Lemma 8.5]{H1}).

And also
\begin{equation}
\begin{split}
    2=I(\hat{\alpha}\cup{(\gamma,M-1)},&\hat{\alpha}\cup{(\delta,1)}\cup{(\gamma,M-p_{i}-1)})\\
    =&c_{1}(\xi|_{Z},\tau)+Q_{\tau}(Z,Z)+2(M-p_{i}-1)Q_{\tau}(Z,Z_{\gamma})\\
    &+\sum_{k=M-p_{i}}^{M-1}(2\lfloor k\theta \rfloor+1)-\mu_{\tau}(\delta).
\end{split}
\end{equation}

By taking the difference from the above equations, we have
\begin{equation}
    2Q_{\tau}(Z,Z_{\gamma})+2(\lfloor M\theta \rfloor - \lfloor (M-p_{i})\theta \rfloor)=0
\end{equation}

Note that for any $p_{i} \leq N<p_{i+1}$, $\lfloor N\theta \rfloor - \lfloor (N-p_{i})\theta \rfloor=\lfloor p_{i}\theta \rfloor$ and moreover $\lfloor p_{i+1}\theta \rfloor - \lfloor (p_{i+1}-p_{i})\theta \rfloor=\lfloor p_{i}\theta \rfloor+1$. These facts are  special cases of \cite[Lemma 4.5]{H1}. Hence
\begin{equation}\label{espe}
    2Q_{\tau}(Z,Z_{\gamma})+2\lfloor p_{i}\theta \rfloor =0
\end{equation}

On the other hand, in the same way, for any $p_{i}< N\leq p_{i+1}$ we have
\begin{equation}
\begin{split}
    I(\hat{\alpha}\cup{(\gamma,N)},&\hat{\alpha}\cup{(\delta,1)\cup{(\gamma,N-p_{i})}})\\
    &-I(\hat{\alpha}\cup{(\gamma,N-1)},\hat{\alpha}\cup{(\delta,1)\cup{(\gamma,N-p_{i}-1)}})
    \\
    =&2Q_{\tau}(Z,Z_{\gamma})+2(\lfloor N\theta \rfloor - \lfloor (N-p_{i})\theta \rfloor)\\
    =&2(\lfloor N\theta \rfloor - \lfloor (N-p_{i})\theta \rfloor-\lfloor p_{i}\theta \rfloor).
\end{split}
\end{equation}

This implies that for any $p_{i}\leq N<p_{i+1}$,  $I(\hat{\alpha}\cup{(\gamma,N)},\hat{\alpha}\cup{(\delta,1)\cup{(\gamma,N-p_{i})}})$ are equal to each other and hence 2, moreover we have $I(\hat{\alpha}\cup{(\gamma,p_{i+1})},\hat{\alpha}\cup{(\delta,1)\cup{(\gamma,p_{i+1}-p_{i})}})=4$.

In the same way, we have
\begin{equation}
\begin{split}
    J_{0}(\hat{\alpha}\cup{(\gamma,N)},&\hat{\alpha}\cup{(\delta,1)\cup{(\gamma,N-p_{i})}})\\
    &-J_{0}(\hat{\alpha}\cup{(\gamma,N-1)},\hat{\alpha}\cup{(\delta,1)\cup{(\gamma,N-p_{i}-1)}})
    \\
    =&2Q_{\tau}(Z,Z_{\gamma})+2(\lfloor (N-1)\theta \rfloor - \lfloor (N-p_{i}-1)\theta \rfloor)\\
    =&2(\lfloor(N-1)\theta \rfloor - \lfloor (N-p_{i}-1)\theta \rfloor-\lfloor p_{i}\theta \rfloor).
\end{split}
\end{equation}
This implies that  for any $p_{i}< N \leq p_{i+1}$,  $J_{0}(\hat{\alpha}\cup{(\gamma,N)},\hat{\alpha}\cup{(\delta,1)\cup{(\gamma,N-p_{i})}})=1$.

We complete the proof of Claim \ref{index2to4}.
\end{proof}

Since $ I(\hat{\alpha}\cup{(\gamma,p_{i+1})},\hat{\alpha}\cup{(\delta,1)\cup{(\gamma,p_{i+1}-p_{i})}})=4$, there is an unique admissible orbit set $\zeta$ such that  $ I(\hat{\alpha}\cup{(\gamma,p_{i+1})},\zeta)=2= I(\zeta,\hat{\alpha}\cup{(\delta,1)\cup{(\gamma,p_{i+1}-p_{i})}})$. Note that  $U\langle \hat{\alpha}\cup{(\gamma,p_{i+1})} \rangle=\langle \zeta \rangle$, $U\langle \zeta \rangle=\langle \hat{\alpha}\cup{(\delta,1)\cup{(\gamma,p_{i+1}-p_{i})}} \rangle$.

\begin{cla}\label{el0}
The above $\zeta$ satisfies $E(\zeta)=0$.

\end{cla}

\begin{proof}[\bf Proof of Claim \ref{el0}]

Suppose that $E(\zeta)>0$. Since $A(\hat{\alpha}\cup{(\gamma,p_{i+1})})>A(\zeta)>A(\hat{\alpha}\cup{(\delta,1)\cup{(\gamma,p_{i+1}-p_{i})}})$, we also have $A(\hat{\alpha}\cup{(\gamma,p_{i+1}-1)})>A(\zeta-(\gamma,1))>A(\hat{\alpha}\cup{(\delta,1)\cup{(\gamma,p_{i+1}-p_{i}-1)}})$. Since (\ref{isom}), this indicates $I(\hat{\alpha}\cup{(\gamma,p_{i+1}-1)},\hat{\alpha}\cup{(\delta,1)\cup{(\gamma,p_{i+1}-p_{i}-1)}})>2$. This contradicts Claim \ref{index2to4}. Therefore we complete the proof of Claim \ref{el0}.
\end{proof}

Since $J_{0}\geq -1$ and additivity of $J_{0}$, we have $(J_{0}(\hat{\alpha}\cup{(\gamma,p_{i+1})},\zeta),J_{0}(\zeta,\hat{\alpha}\cup{(\delta,1)\cup{(\gamma,p_{i+1}-p_{i})}}))$=$(2,-1)$, $(1,0)$, $(0,1)$ or $(-1,2)$.

If $(J_{0}(\hat{\alpha}\cup{(\gamma,p_{i+1})},\zeta),J_{0}(\zeta,\hat{\alpha}\cup{(\delta,1)\cup{(\gamma,p_{i+1}-p_{i})}}))$=$(2,-1)$ or $(-1,2)$, we can derive a contradiction in the same way as Case1 since $A(\hat{\alpha}\cup{(\gamma,p_{i+1})})-A(\zeta)<\epsilon$ and $A(\zeta)-A(\hat{\alpha}\cup{(\delta,1)\cup{(\gamma,p_{i+1}-p_{i})}})<\epsilon$. This is because $A(\hat{\alpha}\cup{(\gamma,p_{i+1})})-A(\hat{\alpha}\cup{(\delta,1)\cup{(\gamma,p_{i+1}-p_{i})}})$ $=A(\alpha_{k+1})-A(\alpha_{k})<\epsilon$. Thus we have only to consider the case $(J_{0}(\hat{\alpha}\cup{(\gamma,p_{i+1})},\zeta),J_{0}(\zeta,\hat{\alpha}\cup{(\delta,1)\cup{(\gamma,p_{i+1}-p_{i})}}))$=$(1,0)$ or $(0,1)$. Here we note that the assumption $H(\alpha_{k})$, $H(\alpha_{k+1})>4$ in Proposition \ref{main index 2} implies that $\hat{\alpha}$ contains more than four negative hyperbolic orbits. In these cases, we derive contradictions by using the splitting behaviors of $J$-holomorphic curves as follows.

\item[\bf $\rm(\hspace{.18em}i\hspace{.18em})$.] $(J_{0}(\hat{\alpha}\cup{(\gamma,p_{i+1})},\zeta),J_{0}(\zeta,\hat{\alpha}\cup{(\delta,1)\cup{(\gamma,p_{i+1}-p_{i})}}))$=$(0,1)$

From $A(\hat{\alpha}\cup{(\gamma,p_{i+1})})-A(\zeta)<\epsilon$ and $E(\zeta)=0$ and the partition condition of end, we have that there is a negative hyperbolic orbit $\delta'$ with $\delta'\notin \hat{\alpha}$ such that $\zeta=\hat{\alpha}\cup{(\delta',1)}$. Moreover the nontrivial parts of  any $J$-holomorphic curve counted by $U\langle \hat{\alpha}\cup{(\gamma,p_{i+1})} \rangle=\langle \zeta \rangle$ are of genus 0 and in $\mathcal{M}^{J}((\gamma,p_{i+1}),(\delta',1))$.

\ref{picture1}

\begin{tikzpicture}\label{picture1}
 
  \draw (7,-2.5)--(12,-2.5);
  \draw (7,-3.5)--(12,-3.5);
 \draw (7,-2.5) arc [start angle=90,end angle=270,x radius=0.25,y radius=0.5];
 \draw (7,-3.5) arc [start angle=270,end angle=450,x radius=0.25,y radius=0.5];
 \draw (12,-2.5) arc [start angle=90,end angle=270,x radius=0.25,y radius=0.5];
 \draw (12,-3.5) arc [start angle=270,end angle=450,x radius=0.25,y radius=0.5];

 \draw (7,-4.5) arc [start angle=90,end angle=270,x radius=0.25,y radius=0.5];
 \draw (7,-5.5) arc [start angle=270,end angle=450,x radius=0.25,y radius=0.5];
 \draw (12,-4) arc [start angle=90,end angle=270,x radius=0.25,y radius=0.5];
 \draw (12,-5) arc [start angle=270,end angle=450,x radius=0.25,y radius=0.5];
 \draw (7,-4.5) to [out=0,in=180] (12,-4);
 \draw (7,-5.5) to [out=0,in=180] (12,-5);

  \draw[densely dotted,  thick] (7,-2)--(7,-6) ;
   \draw[densely dotted,  thick] (12,-2)--(12,-6) ;
 
\draw (7,-6.5) node{$\zeta=\hat{\alpha}\cup{(\delta',1)}$};  
\draw (12,-6.5) node{$\hat{\alpha}\cup{(\gamma,p_{i+1})}$};

\draw (13.1,-4.5) node{$(\gamma,p_{i+1})$};

\draw (6,-5) node{$(\delta',1)$};

\draw (13.2,-3) node{$(\eta,1)\in \hat{\alpha}$};
 
\end{tikzpicture}

Let us consider the behaviors of such curves as $z\to \eta$. In the same way as Case2, the limiting curve of such sequence splits and each of them has end at $\eta$. Furthermore both ECH indexes are one. Its multiplicities are two because of the admissibility of curves of ECH index 1. This implies that $|2A(\eta)-p_{i+1}R|<\epsilon$.

Consider back to the $J$-homolorphic curves of $U\langle \alpha_{k+1} \rangle=\langle \alpha_{k} \rangle$. Its nontrivial parts are in $\mathcal{M}^{J}((\gamma,p_{i}),(\delta,{1}))$. By considering the behaviors of this curves as $z \to \eta$, we have $|2A(\eta)-p_{i}R|<\epsilon$. By combining with $|2A(\eta)-p_{i+1}R|<\epsilon$, we have $(p_{i+1}-p_{i})R<2\epsilon$ and so $p_{i+1}=p_{i}$. This is a contradiction.

\item[\bf $\rm(\hspace{.08em}ii\hspace{.08em})$.]If $(J_{0}(\hat{\alpha}\cup{(\gamma,p_{i+1})},\zeta),J_{0}(\zeta,\hat{\alpha}\cup{(\delta,1)\cup{(\gamma,p_{i+1}-p_{i})}}))$=$(1,0)$

Consider the $J$-holomorphic curves counted by $U\langle \zeta \rangle=\langle \hat{\alpha}\cup{(\delta,1)\cup{(\gamma,p_{i+1}-p_{i})}}) \rangle$. In the same way as the above argument, we can find that there is an hyperbolic orbit $\delta'$ such that $\zeta=\hat{\alpha}\cup{(\delta,1)}\cup{(\delta',1)}$ and that nontrivial parts of  any $J$-holomorphic curve counted by $U\langle \zeta \rangle=\langle \hat{\alpha}\cup{(\delta,1)\cup{(\gamma,p_{i+1}-p_{i})}}) \rangle$ are of genus 0 and in $\mathcal{M}^{J}((\delta',1),(\gamma,p_{i+1}-p_{i}))$. And also we have $|2A(\eta)-p_{i}R|<\epsilon$ and $|2A(\eta)-(p_{i+1}-p_{i})R|<\epsilon$.  Note that since $E(\alpha_{k+1}),\,\, E(\alpha_{k})>p_{1},\,\, q_{1}$, we have $p_{i}>1$.

The next claim is obvious but often used later on. So we state here.
\begin{cla}\label{fre}
Suppose that $q_{i}\in S_{\theta}$ (resp. $p_{i}\in S_{-\theta}$). If $q_{i}>1$ (resp. $p_{i}>1$), then  $ q_{i}\neq q_{i+1}-q_{i}$ (resp. $p_{i}\neq p_{i+1}-p_{i}$).
\end{cla}
\begin{proof}[\bf Proof of Claim \ref{fre}]
Suppose that $q_{i}\in S_{\theta}$, then by Proposition \ref{s}, $q_{i+1}-q_{i}\in S_{-\theta}$ and since $S_{\theta}\cap{S_{-\theta}}=\{1\}$, if $q_{i}>1$, we have $q_{i}\neq q_{i+1}-q_{i}$. We can do the same in the case $p_{i}$.
\end{proof}
From  $|2A(\eta)-p_{i}R|<\epsilon$ and $|2A(\eta)-(p_{i+1}-p_{i})R|<\epsilon$, we have $|p_{i}R-(p_{i+1}-p_{i})R|<2\epsilon$. This implies that $p_{i}=p_{i+1}-p_{i}$ but this contradicts Claim \ref{fre}.

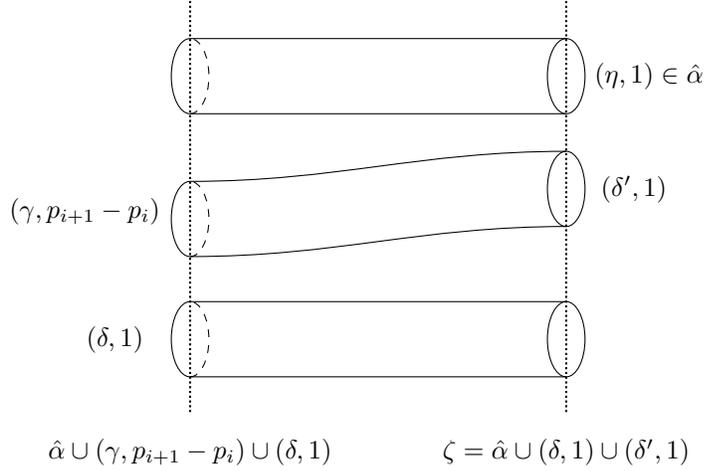
\begin{tikzpicture}

\draw (3.1,-3) node{$(\eta,1)\in \hat{\alpha}$};

 \draw (2,-2.5) arc [start angle=90,end angle=270,x radius=0.25,y radius=0.5];
 \draw (2,-3.5) arc [start angle=270,end angle=450,x radius=0.25,y radius=0.5];

 \draw (2,-4) arc [start angle=90,end angle=270,x radius=0.25,y radius=0.5];
 \draw (2,-5) arc [start angle=270,end angle=450,x radius=0.25,y radius=0.5];

  \draw (2,-6) arc [start angle=90,end angle=270,x radius=0.25,y radius=0.5];
 \draw (2,-7) arc [start angle=270,end angle=450,x radius=0.25,y radius=0.5];

  \draw (-3,-2.5)--(2,-2.5);
  \draw (-3,-3.5)--(2,-3.5);

 \draw (-3,-2.5) arc [start angle=90,end angle=270,x radius=0.25,y radius=0.5];
 \draw (-3,-3.5)[dashed] arc [start angle=270,end angle=450,x radius=0.25,y radius=0.5];

\draw (-3,-6) arc [start angle=90,end angle=270,x radius=0.25,y radius=0.5];
 \draw (-3,-7)[dashed] arc [start angle=270,end angle=450,x radius=0.25,y radius=0.5];
 
 \draw (-3,-4.4) arc [start angle=90,end angle=270,x radius=0.25,y radius=0.5];
 \draw (-3,-5.4)[dashed] arc [start angle=270,end angle=450,x radius=0.25,y radius=0.5];

 \draw (-3,-4.4) to [out=0,in=180] (2,-4);
  \draw (-3,-5.4) to [out=0,in=180] (2,-5);

 \draw (-3,-6)--(2,-6);
  \draw (-3,-7)--(2,-7);

 \draw[densely dotted,  thick] (-3,-2)--(-3,-7.5);
  \
   \draw[densely dotted,  thick] (2,-2)--(2,-7.5) ;
   
\draw (2,-8) node{$\zeta=\hat{\alpha}\cup{(\delta,1)\cup{(\delta',1)}}$};   
 
\draw (-3,-8) node{$\hat{\alpha}\cup{(\gamma,p_{i+1}-p_{i})}\cup{(\delta,1)}$}; 
\draw (-4,-6.5) node{$(\delta,1)$};

\draw (-4.4,-4.8) node{$(\gamma,p_{i+1}-p_{i})$};

\draw (2.9,-4.5) node{$(\delta',1)$};

\end{tikzpicture}

Combining the above arguments, we complete the proof of Lemma \ref{mainlemma}.
\end{proof}

\section{The properties of certain $J_{0}=2$ curves}
To derive a contradiction from Proposition \ref{main index 2}, at first we state Proposition \ref{nagai} and use Section 6 and Section 7 to prove Proposition \ref{nagai}.

\textbf{Notation}

For any $a,\,b\in \mathbb{R}$, we write $a\approx b$ if $|a-b|<\frac{1}{100}\mathrm{min}\{A(\alpha)\,|\alpha\,\,\mathrm{is\,\,a\,\,Reeb\,\,orbit}\}$. In this notation, for $n,\,m\in \mathbb{Z}$ and $\tau>\frac{1}{100}$, if $n\tau R\approx m \tau R$, then $n=m$.

\begin{prp}\label{nagai}
Let $\epsilon<\frac{1}{10^5}\mathrm{min}\{A(\alpha)\,\,|\alpha\,\,\mathrm{is\,\,a\,\,Reeb\,\,orbit}\,\,\}$ and $k$ sufficiently large. Suppose that two admissible orbit sets $\alpha_{k+1}$ and $\alpha_{k}$ with $U\langle \alpha_{k+1} \rangle =\langle \alpha_{k} \rangle$ satisfy the following conditions.
\begin{enumerate}
    \item $J(\alpha_{k+1},\alpha_{k})=2$
    \item $A(\alpha_{k+1})-A(\alpha_{k})<\epsilon$ 
    \item $E(\alpha_{k+1}),\,E(\alpha_{k})\notin S_{\theta}\cup{ S_{-\theta}}$
    \item $E(\alpha_{k+1}),\,E(\alpha_{k})>p_{1},\,\,q_{1}$
    \item $H(\alpha_{k+1}),\,H(\alpha_{k})>4$.
\end{enumerate}

Let $u=u_{0}\cup{u_{1}}\in \mathcal{M}^{J}(\alpha_{k+1},\alpha_{k})$ be any $J$-holomorphic curve counted by the $U$-map.

Then one of the following conditions holds. 
\begin{description}

\item[(a).] Let $E(\alpha_{k+1})=M$ and $p_{i}:=\mathrm{max}(S_{-\theta}\cap{\{1,\,\,2,...,\,\,M\}})$. Then there are  two negative hyperbolic orbits $\delta_{1}$, $\delta_{2}$ and an admissible orbit set $\hat{\alpha}$ consisting of negative hyperbolic orbits such that $\alpha_{k+1}=\hat{\alpha}\cup{(\gamma,M)\cup{(\delta_{1},1)}}$, $\alpha_{k}=\hat{\alpha}\cup{(\gamma,M-p_{i})\cup{(\delta_{2},1)}}$ and $u_{1}\in \mathcal{M}^{J}((\delta_{1},1)\cup{(\gamma,p_{i})},(\delta_{2},1))$. 

Moreover, $A(\delta_{1})\approx (p_{i+1}-p_{i})R$, $A(\delta_{2})\approx p_{i+1}R$ and for each $\eta \in \hat{\alpha}$, either $A(\eta)\approx \frac{1}{2}p_{i+1}R$ or $A(\eta)\approx \frac{1}{2}(p_{i+1}-p_{i})R$.

\item[(a').]Let $E(\alpha_{k})=M$ and $q_{i}:=\mathrm{max}(S_{\theta}\cap{\{1,\,\,2,...,\,\,M\}})$. Then there are  two  negative hyperbolic orbits $\delta_{1}$, $\delta_{2}$ and an admissible orbit set $\hat{\alpha}$ consisting of negative hyperbolic orbits such that $\alpha_{k+1}=\hat{\alpha}\cup{(\gamma,M-q_{i})\cup{(\delta_{1},1)}}$, $\alpha_{k}=\hat{\alpha}\cup{(\gamma,M)\cup{(\delta_{2},1)}}$ and $u_{1}\in \mathcal{M}^{J}((\delta_{1},1),(\delta_{2},1)\cup{(\gamma,q_{i})})$. 

Moreover, $A(\delta_{1})\approx q_{i+1}R$, $A(\delta_{2})\approx (q_{i+1}-q_{i})R$ and for each $\eta \in \hat{\alpha}$, either $A(\eta)\approx \frac{1}{2}q_{i+1}R$ or $A(\eta)\approx \frac{1}{2}(q_{i+1}-q_{i})R$.

\item[(b).]Let $E(\alpha_{k+1})=M$ and $p_{i}:=\mathrm{max}(S_{-\theta}\cap{\{1,\,\,2,...,\,\,M\}})$. Then $\frac{3}{2}p_{i}=p_{i+1}$ and there are two  negative hyperbolic orbits $\delta_{1}$, $\delta_{2}$ and  an admissible orbit set $\hat{\alpha}$ consisting of negative hyperbolic orbits such that $\alpha_{k+1}=\hat{\alpha}\cup{(\gamma,M)}$,  $\alpha_{k}=\hat{\alpha}\cup{(\gamma,M-p_{i})\cup{(\delta_{1},1)\cup{(\delta_{2},1)}}}$ and $u_{1}\in \mathcal{M}^{J}((\gamma,p_{i}),(\delta_{1},1)\cup{(\delta_{2},1)})$. 

Moreover, $A(\delta_{1})\approx \frac{1}{2}p_{i}R$, $A(\delta_{2})\approx \frac{1}{2}p_{i}R$ and for each $\eta \in \hat{\alpha}$, either $A(\eta)\approx \frac{1}{2}p_{i}R$ or $A(\eta)\approx \frac{1}{4}p_{i}R$.

\item[(b').]Let $E(\alpha_{k})=M$ and $q_{i}:=\mathrm{max}(S_{\theta}\cap{\{1,\,\,2,...,\,\,M\}})$. Then $\frac{3}{2}q_{i}=q_{i+1}$ and there are two  negative hyperbolic orbits $\delta_{1}$, $\delta_{2}$ and  an admissible orbit set $\hat{\alpha}$ consisting of negative hyperbolic orbits such that $\alpha_{k+1}=\hat{\alpha}\cup{(\gamma,M-q_{i})\cup{(\delta_{1},1)}\cup{(\delta_{2},1)}}$, $\alpha_{k}=\hat{\alpha}\cup{(\gamma,M)}$ and $u_{1}\in \mathcal{M}^{J}((\delta_{1},1)\cup{(\delta_{2},1)},(\gamma,q_{i}))$. 

Moreover, $A(\delta_{1})\approx \frac{1}{2}q_{i}R$, $A(\delta_{2})\approx \frac{1}{2}q_{i}R$ and for each $\eta \in \hat{\alpha}$, either $A(\eta)\approx \frac{1}{2}q_{i}R$ or $A(\eta)\approx \frac{1}{4}q_{i}R$.

\item[(c).]Let $E(\alpha_{k+1})=M$ and $p_{i}:=\mathrm{max}(S_{-\theta}\cap{\{1,\,\,2,...,\,\,M\}})$. Then $\frac{4}{3}p_{i}=p_{i+1}$ and there are  two  negative hyperbolic orbits $\delta_{1}$, $\delta_{2}$ and an admissible orbit set $\hat{\alpha}$ consisting of negative hyperbolic orbits such that $\alpha_{k+1}=\hat{\alpha}\cup{(\gamma,M)}$,  $\alpha_{k}=\hat{\alpha}\cup{(\gamma,M-p_{i})\cup{(\delta_{1},1)\cup{(\delta_{2},1)}}}$ and $u_{1}\in \mathcal{M}^{J}((\gamma,p_{i}),(\delta_{1},1)\cup{(\delta_{2},1)})$. 

Moreover, $A(\delta_{1})\approx \frac{2}{3}p_{i}R$, $A(\delta_{2})\approx \frac{1}{3}p_{i}R$ and for each $\eta \in \hat{\alpha}$, either $A(\eta)\approx \frac{1}{2}p_{i}R$ or $A(\eta)\approx \frac{1}{6}p_{i}R$.

\item[(c').]Let $E(\alpha_{k})=M$ and $q_{i}:=\mathrm{max}(S_{\theta}\cap{\{1,\,\,2,...,\,\,M\}})$. Then $\frac{4}{3}q_{i}=q_{i+1}$ and there are and two  negative hyperbolic orbits $\delta_{1}$, $\delta_{2}$ and an admissible orbit set $\hat{\alpha}$ consisting of negative hyperbolic orbits such that $\alpha_{k+1}=\hat{\alpha}\cup{(\gamma,M-q_{i})\cup{(\delta_{1},1)\cup{(\delta_{2},1)}}}$,  $\alpha_{k}=\hat{\alpha}\cup{(\gamma,M)}$ and $u_{1}\in \mathcal{M}^{J}((\delta_{1},1)\cup{(\delta_{2},1)},(\gamma,p
q_{i}))$. . 

Moreover, $A(\delta_{1})\approx \frac{2}{3}q_{i}R$, $A(\delta_{2})\approx \frac{1}{3}q_{i}R$ and for each $\eta \in \hat{\alpha}$, either $A(\eta)\approx \frac{1}{2}q_{i}R$ or $A(\eta)\approx \frac{1}{6}q_{i}R$.
    
\end{description}
Note that \textbf{(a)} and \textbf{(a')}, \textbf{(b)} and \textbf{(b')}, \textbf{(c)} and \textbf{(c')} are symmetrical
respectively and, \textbf{(a}), \textbf{(a')}, \textbf{(b)}, \textbf{(b')}, \textbf{(c)} and \textbf{(c')}  are mutually exclusive.
\end{prp}

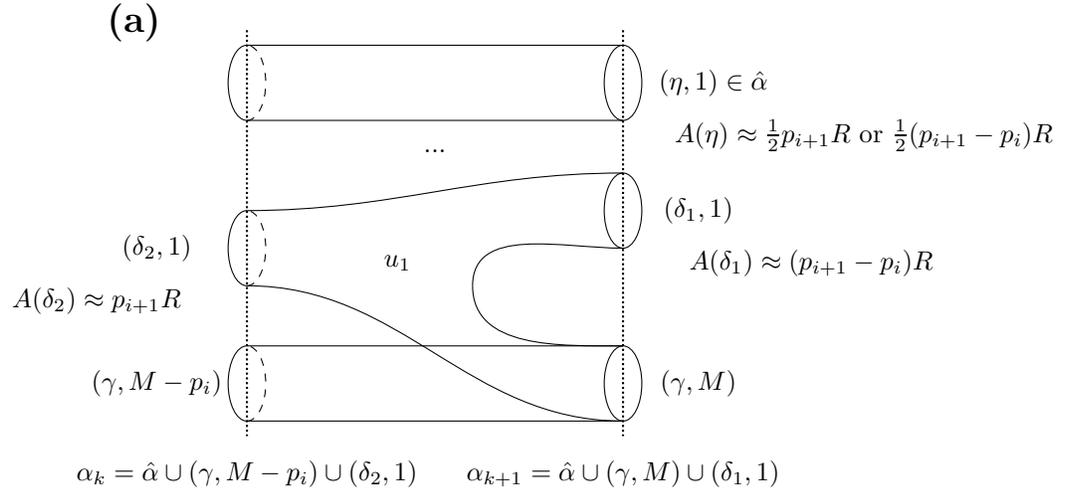
\begin{tikzpicture}

\draw (11,-1.8) arc [start angle=90,end angle=270,x radius=0.25,y radius=0.5];
 \draw (11,-2.8) arc [start angle=270,end angle=450,x radius=0.25,y radius=0.5];
 
 \draw (11,-1.8)--(6,-1.8) ;
  \draw (11,-2.8)--(6,-2.8) ;
  \draw (6,-1.8) arc [start angle=90,end angle=270,x radius=0.25,y radius=0.5];
 \draw (6,-2.8)[dashed] arc [start angle=270,end angle=450,x radius=0.25,y radius=0.5];
 
\draw (12.2,-2.3) node{$(\eta,1)\in \hat{\alpha}$};
\draw (14.2,-3) node{$A(\eta)\approx  \frac{1}{2}p_{i+1}R$ or $ \frac{1}{2}(p_{i+1}-p_{i})R$};

 \draw (8.5,-3.2) node{...};

 \draw (6,-4) arc [start angle=90,end angle=270,x radius=0.25,y radius=0.5];
 \draw (6,-5)[dashed] arc [start angle=270,end angle=450,x radius=0.25,y radius=0.5];
 \draw (11,-3.5) arc [start angle=90,end angle=270,x radius=0.25,y radius=0.5];
 \draw (11,-4.5) arc [start angle=270,end angle=450,x radius=0.25,y radius=0.5];
 
 \draw (12,-4) node{$(\delta_{1},1)$ } ;

 \draw (13.5,-4.7) node{ $A(\delta_{1})\approx (p_{i+1}-p_{i})R$ } ;
 
 \draw (11,-5.8) arc [start angle=90,end angle=270,x radius=0.25,y radius=0.5];
 \draw (11,-6.8) arc [start angle=270,end angle=450,x radius=0.25,y radius=0.5];
 
 \draw (12,-6.3) node{$(\gamma,M)$} ;
 
 \draw (4.8,-6.3) node{$(\gamma,M-p_{i})$} ;
 
 \draw (6,-4) to [out=0,in=180] (11,-3.5);
 \draw (6,-5) to [out=0,in=180] (11,-6.8);
 \draw (11,-5.8) to [out=180,in=270] (9,-5);
 \draw (9,-5) to [out=90,in=180] (11,-4.5);
 
 \draw (4.5,-1.5) node{\Large \textbf{(a)}} ;
 
 \draw (11,-5.8)--(6,-5.8) ;
  \draw (11,-6.8)--(6,-6.8) ;
  \draw (6,-5.8) arc [start angle=90,end angle=270,x radius=0.25,y radius=0.5];
 \draw (6,-6.8)[dashed] arc [start angle=270,end angle=450,x radius=0.25,y radius=0.5];
 \draw (4.8,-4.5) node{$(\delta_{2},1)$} ;
 \draw (4,-5.2) node{$A(\delta_{2})\approx p_{i+1}R$} ;
 
 \draw[densely dotted,  thick] (6,-1.6)--(6,-7);
  \draw[densely dotted,  thick] (11,-1.6)--(11,-7) ;
  
  \draw (6,-7.5) node{$\alpha_{k}=\hat{\alpha}\cup{(\gamma,M-p_{i})\cup{(\delta_{2},1)}}$};
  \draw (11,-7.5) node{$\alpha_{k+1}=\hat{\alpha}\cup{(\gamma,M)\cup{(\delta_{1},1)}}$};
  
  \draw (8,-4.7) node{$u_{1}$};

\end{tikzpicture}
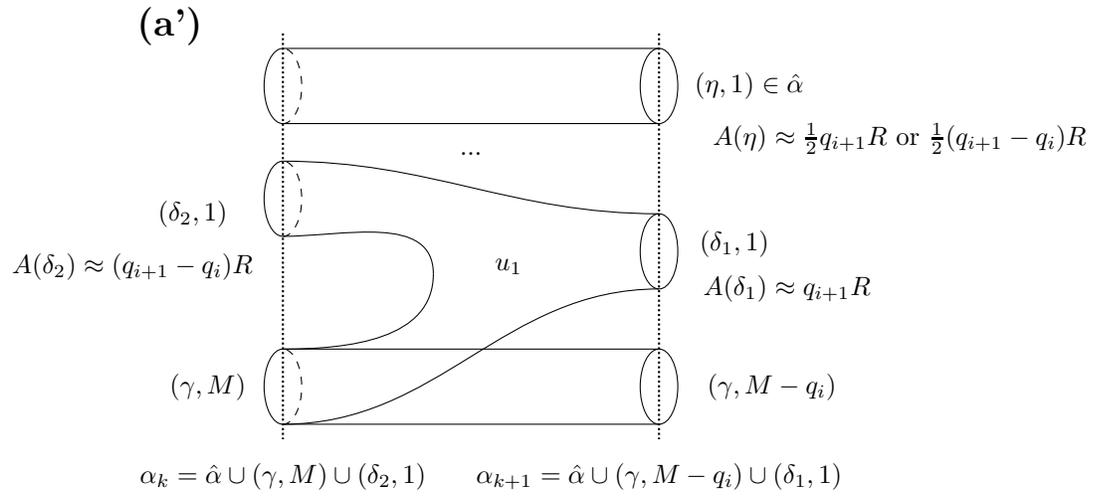
\vspace{5mm}

\begin{tikzpicture}

\draw (11,-1.8) arc [start angle=90,end angle=270,x radius=0.25,y radius=0.5];
 \draw (11,-2.8) arc [start angle=270,end angle=450,x radius=0.25,y radius=0.5];
 
 \draw (11,-1.8)--(6,-1.8) ;
  \draw (11,-2.8)--(6,-2.8) ;
  \draw (6,-1.8) arc [start angle=90,end angle=270,x radius=0.25,y radius=0.5];
 \draw (6,-2.8)[dashed] arc [start angle=270,end angle=450,x radius=0.25,y radius=0.5];
 
\draw (12.2,-2.3) node{$(\eta,1)\in \hat{\alpha}$};
\draw (14.2,-3) node{$A(\eta)\approx  \frac{1}{2}q_{i+1}R$ or $ \frac{1}{2}(q_{i+1}-q_{i})R$};

 \draw (8.5,-3.2) node{...};

 \draw (6,-3.3) arc [start angle=90,end angle=270,x radius=0.25,y radius=0.5];
 \draw (6,-4.3)[dashed] arc [start angle=270,end angle=450,x radius=0.25,y radius=0.5];

 \draw (11,-4) arc [start angle=90,end angle=270,x radius=0.25,y radius=0.5];
 \draw (11,-5) arc [start angle=270,end angle=450,x radius=0.25,y radius=0.5];
 
 \draw (12,-4.4) node{$(\delta_{1},1)$ } ;

 \draw (12.7,-5) node{ $A(\delta_{1})\approx q_{i+1}R$ } ;
 
 \draw (11,-5.8) arc [start angle=90,end angle=270,x radius=0.25,y radius=0.5];
 \draw (11,-6.8) arc [start angle=270,end angle=450,x radius=0.25,y radius=0.5];
 
 \draw (12.5,-6.3) node{$(\gamma,M-q_{i})$} ;
 
 \draw (5,-6.3) node{$(\gamma,M)$} ;

 \draw (11,-5) to [out=180,in=0] (6,-6.8);
 \draw (6,-5.8) to [out=0,in=270] (8,-4.8);
 \draw (8,-4.8) to [out=90,in=0] (6,-4.3);
 \draw (11,-4) to [out=180,in=0] (6,-3.3);

 \draw (11,-5.8)--(6,-5.8) ;
  \draw (11,-6.8)--(6,-6.8) ;
  \draw (6,-5.8) arc [start angle=90,end angle=270,x radius=0.25,y radius=0.5];
 \draw (6,-6.8)[dashed] arc [start angle=270,end angle=450,x radius=0.25,y radius=0.5];
 \draw (4.8,-4) node{$(\delta_{2},1)$} ;
 \draw (4,-4.7) node{$A(\delta_{2})\approx (q_{i+1}-q_{i})R$} ;
 
 \draw[densely dotted,  thick] (6,-1.6)--(6,-7);
  \draw[densely dotted,  thick] (11,-1.6)--(11,-7) ;
  
  \draw (6,-7.5) node{$\alpha_{k}=\hat{\alpha}\cup{(\gamma,M)\cup{(\delta_{2},1)}}$};
  \draw (11,-7.5) node{$\alpha_{k+1}=\hat{\alpha}\cup{(\gamma,M-q_{i})\cup{(\delta_{1},1)}}$};
  
  \draw (9,-4.7) node{$u_{1}$};

 \draw (4.5,-1.5) node{\Large \textbf{(a')}} ;
\end{tikzpicture}
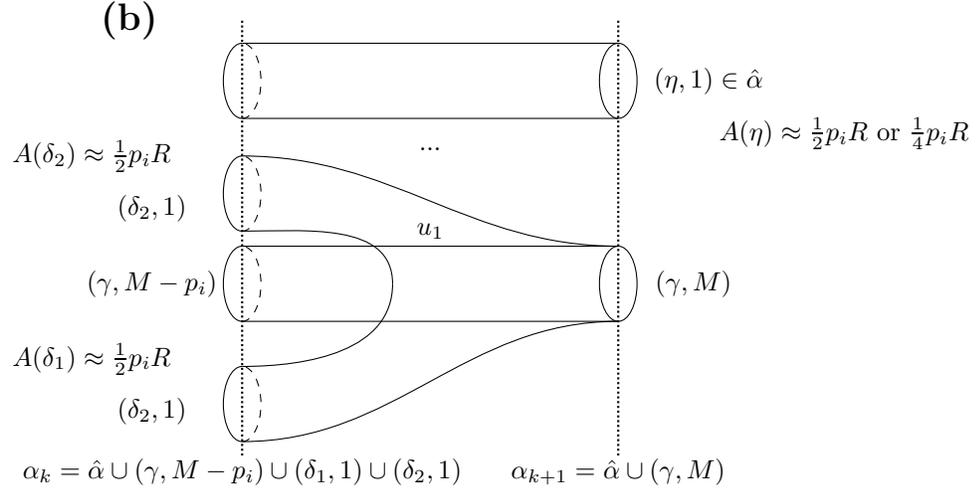
\vspace{5mm}

\begin{tikzpicture}
\draw (11,-1.8) arc [start angle=90,end angle=270,x radius=0.25,y radius=0.5];
 \draw (11,-2.8) arc [start angle=270,end angle=450,x radius=0.25,y radius=0.5];
 
 \draw (11,-1.8)--(6,-1.8) ;
  \draw (11,-2.8)--(6,-2.8) ;
  \draw (6,-1.8) arc [start angle=90,end angle=270,x radius=0.25,y radius=0.5];
 \draw (6,-2.8)[dashed] arc [start angle=270,end angle=450,x radius=0.25,y radius=0.5];
 
\draw (12.2,-2.3) node{$(\eta,1)\in \hat{\alpha}$};
\draw (14,-3) node{$A(\eta)\approx  \frac{1}{2}p_{i}R$ or $ \frac{1}{4}p_{i}R$};

 \draw (8.5,-3.2) node{...};

 \draw (6,-3.3) arc [start angle=90,end angle=270,x radius=0.25,y radius=0.5];
 \draw (6,-4.3)[dashed] arc [start angle=270,end angle=450,x radius=0.25,y radius=0.5];

 \draw (11,-4.5) arc [start angle=90,end angle=270,x radius=0.25,y radius=0.5];
 \draw (11,-5.5) arc [start angle=270,end angle=450,x radius=0.25,y radius=0.5];
  \draw (6,-4.5) arc [start angle=90,end angle=270,x radius=0.25,y radius=0.5];
 \draw (6,-5.5)[dashed] arc [start angle=270,end angle=450,x radius=0.25,y radius=0.5];
 
 \draw (6,-4.5)--(11,-4.5);
 
  \draw (6,-5.5)--(11,-5.5);
 
 \draw (12,-5) node{$(\gamma,M)$ } ;

 \draw (4.8,-5) node{$(\gamma,M-p_{i})$} ;
 
 \draw (11,-5.5) to [out=180,in=0] (6,-7.1);
 \draw (6,-6.1) to [out=0,in=270] (8,-5);
 \draw (8,-5) to [out=90,in=0] (6,-4.3);
 \draw (11,-4.5) to [out=180,in=0] (6,-3.3);

  \draw (6,-6.1) arc [start angle=90,end angle=270,x radius=0.25,y radius=0.5];
 \draw (6,-7.1)[dashed] arc [start angle=270,end angle=450,x radius=0.25,y radius=0.5];
 
 \draw (4.8,-4) node{$(\delta_{2},1)$} ;
 \draw (4,-3.3) node{$A(\delta_{2})\approx \frac{1}{2} p_{i}R$} ;
 
  \draw (4.8,-6.7) node{$(\delta_{2},1)$} ;
 \draw (4,-6) node{$A(\delta_{1})\approx \frac{1}{2} p_{i}R$} ;
 
 \draw[densely dotted,  thick] (6,-1.5)--(6,-7.3);
  \draw[densely dotted,  thick] (11,-1.5)--(11,-7.3) ;
  
  \draw (6,-7.5) node{$\alpha_{k}=\hat{\alpha}\cup{(\gamma,M-p_{i})\cup{(\delta_{1},1)\cup{(\delta_{2},1)}}}$};
  \draw (11,-7.5) node{$\alpha_{k+1}=\hat{\alpha}\cup{(\gamma,M)}$};
  
  \draw (8.5,-4.3) node{$u_{1}$};
 \draw (4.5,-1.5) node{\Large \textbf{(b)}} ;

\end{tikzpicture}

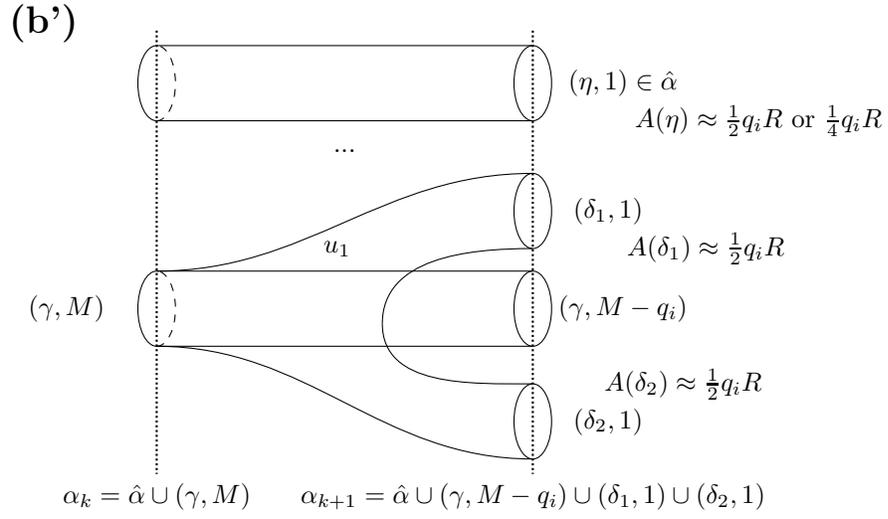
\vspace{5mm}

\begin{tikzpicture}
\draw (11,-1.8) arc [start angle=90,end angle=270,x radius=0.25,y radius=0.5];
 \draw (11,-2.8) arc [start angle=270,end angle=450,x radius=0.25,y radius=0.5];
 
 \draw (11,-1.8)--(6,-1.8) ;
  \draw (11,-2.8)--(6,-2.8) ;
  \draw (6,-1.8) arc [start angle=90,end angle=270,x radius=0.25,y radius=0.5];
 \draw (6,-2.8)[dashed] arc [start angle=270,end angle=450,x radius=0.25,y radius=0.5];
 
\draw (12.2,-2.3) node{$(\eta,1)\in \hat{\alpha}$};
\draw (14,-2.8) node{$A(\eta)\approx  \frac{1}{2}q_{i}R$ or $ \frac{1}{4}q_{i}R$};
 
  \draw (4.5,-1.5) node{\Large \textbf{(b')}} ;
 \draw (8.5,-3.2) node{...};

 \draw (6,-4.8) arc [start angle=90,end angle=270,x radius=0.25,y radius=0.5];
 \draw (6,-5.8)[dashed] arc [start angle=270,end angle=450,x radius=0.25,y radius=0.5];
 \draw (11,-3.5) arc [start angle=90,end angle=270,x radius=0.25,y radius=0.5];
 \draw (11,-4.5) arc [start angle=270,end angle=450,x radius=0.25,y radius=0.5];
 
 \draw (12,-4) node{$(\delta_{1},1)$ } ;

 \draw (13.3,-4.5) node{ $A(\delta_{1})\approx \frac{1}{2} q_{i}R$ } ;
 
 \draw (11,-6.3) arc [start angle=90,end angle=270,x radius=0.25,y radius=0.5];
 \draw (11,-7.3) arc [start angle=270,end angle=450,x radius=0.25,y radius=0.5];
 
 \draw (12.2,-5.3) node{$(\gamma,M-q_{i})$} ;
 
  \draw (12,-6.8) node{$(\delta_{2},1)$} ;
   \draw (13,-6.3) node{$A(\delta_{2})\approx \frac{1}{2}q_{i}R$} ;

 \draw (6,-4.8) to [out=0,in=180] (11,-3.5);
 \draw (6,-5.8) to [out=0,in=180] (11,-7.3);
 \draw (11,-6.3) to [out=180,in=270] (9,-5.5);
 \draw (9,-5.5) to [out=90,in=180] (11,-4.5);

 \draw (11,-4.8) arc [start angle=90,end angle=270,x radius=0.25,y radius=0.5];
 \draw (11,-5.8) arc [start angle=270,end angle=450,x radius=0.25,y radius=0.5];

\draw (6,-4.8)--(11,-4.8);
\draw (6,-5.8)--(11,-5.8);

 \draw (4.8,-5.3) node{$(\gamma,M)$} ;

 \draw[densely dotted,  thick] (6,-1.6)--(6,-7.5);
  \draw[densely dotted,  thick] (11,-1.6)--(11,-7.5) ;
  
  \draw (6,-7.8) node{$\alpha_{k}=\hat{\alpha}\cup{(\gamma,M)}$};
  \draw (11,-7.8) node{$\alpha_{k+1}=\hat{\alpha}\cup{(\gamma,M-q_{i})\cup{(\delta_{1},1)}\cup{(\delta_{2},1)}}$};
  
  \draw (8.4,-4.5) node{$u_{1}$};

\end{tikzpicture}

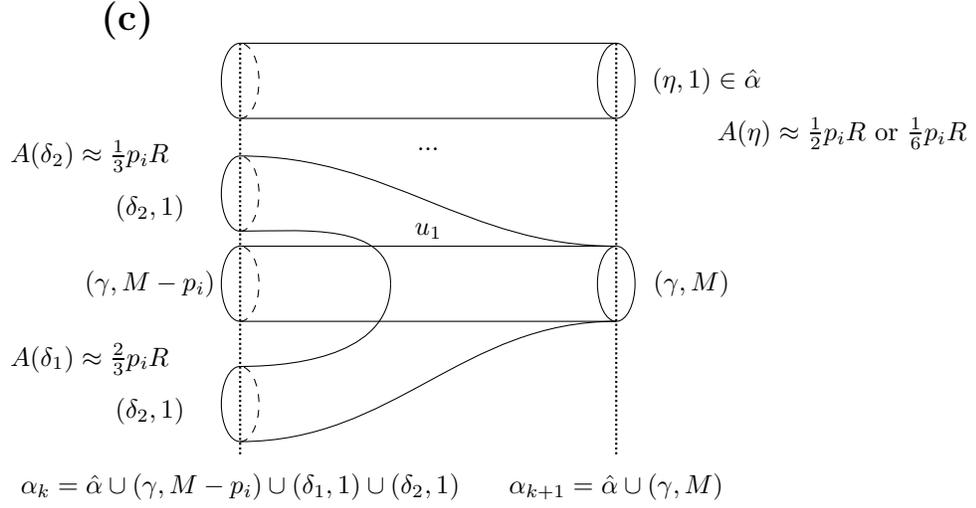
\vspace{5mm}

\begin{tikzpicture}
\draw (11,-1.8) arc [start angle=90,end angle=270,x radius=0.25,y radius=0.5];
 \draw (11,-2.8) arc [start angle=270,end angle=450,x radius=0.25,y radius=0.5];
  \draw (4.5,-1.5) node{\Large \textbf{(c)}} ;
 
 \draw (11,-1.8)--(6,-1.8) ;
  \draw (11,-2.8)--(6,-2.8) ;
  \draw (6,-1.8) arc [start angle=90,end angle=270,x radius=0.25,y radius=0.5];
 \draw (6,-2.8)[dashed] arc [start angle=270,end angle=450,x radius=0.25,y radius=0.5];
 
\draw (12.2,-2.3) node{$(\eta,1)\in \hat{\alpha}$};
\draw (14,-3) node{$A(\eta)\approx  \frac{1}{2}p_{i}R$ or $ \frac{1}{6}p_{i}R$};

 \draw (8.5,-3.2) node{...};

 \draw (6,-3.3) arc [start angle=90,end angle=270,x radius=0.25,y radius=0.5];
 \draw (6,-4.3)[dashed] arc [start angle=270,end angle=450,x radius=0.25,y radius=0.5];

 \draw (11,-4.5) arc [start angle=90,end angle=270,x radius=0.25,y radius=0.5];
 \draw (11,-5.5) arc [start angle=270,end angle=450,x radius=0.25,y radius=0.5];
  \draw (6,-4.5) arc [start angle=90,end angle=270,x radius=0.25,y radius=0.5];
 \draw (6,-5.5)[dashed] arc [start angle=270,end angle=450,x radius=0.25,y radius=0.5];
 
 \draw (6,-4.5)--(11,-4.5);
 
  \draw (6,-5.5)--(11,-5.5);
 
 \draw (12,-5) node{$(\gamma,M)$ } ;
 
 \draw (4.8,-5) node{$(\gamma,M-p_{i})$} ;
 
 \draw (11,-5.5) to [out=180,in=0] (6,-7.1);
 \draw (6,-6.1) to [out=0,in=270] (8,-5);
 \draw (8,-5) to [out=90,in=0] (6,-4.3);
 \draw (11,-4.5) to [out=180,in=0] (6,-3.3);

  \draw (6,-6.1) arc [start angle=90,end angle=270,x radius=0.25,y radius=0.5];
 \draw (6,-7.1)[dashed] arc [start angle=270,end angle=450,x radius=0.25,y radius=0.5];
 
 \draw (4.8,-4) node{$(\delta_{2},1)$} ;
 \draw (4,-3.3) node{$A(\delta_{2})\approx \frac{1}{3} p_{i}R$} ;
 
  \draw (4.8,-6.7) node{$(\delta_{2},1)$} ;
 \draw (4,-6) node{$A(\delta_{1})\approx \frac{2}{3} p_{i}R$} ;
 
 \draw[densely dotted,  thick] (6,-1.8)--(6,-7.3);
  \draw[densely dotted,  thick] (11,-1.8)--(11,-7.3) ;
  
  \draw (6,-7.7) node{$\alpha_{k}=\hat{\alpha}\cup{(\gamma,M-p_{i})\cup{(\delta_{1},1)\cup{(\delta_{2},1)}}}$};
  \draw (11,-7.7) node{$\alpha_{k+1}=\hat{\alpha}\cup{(\gamma,M)}$};
  
  \draw (8.5,-4.3) node{$u_{1}$};

\end{tikzpicture}

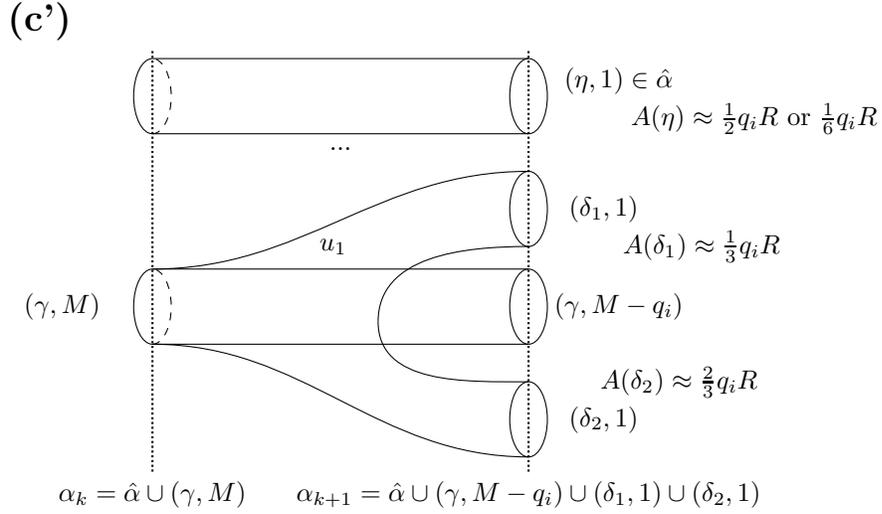
\vspace{5mm}

\begin{tikzpicture}
\draw (11,-2) arc [start angle=90,end angle=270,x radius=0.25,y radius=0.5];
 \draw (11,-3) arc [start angle=270,end angle=450,x radius=0.25,y radius=0.5];
 
 \draw (11,-2)--(6,-2) ;
  \draw (11,-3)--(6,-3) ;
  \draw (6,-2) arc [start angle=90,end angle=270,x radius=0.25,y radius=0.5];
 \draw (6,-3)[dashed] arc [start angle=270,end angle=450,x radius=0.25,y radius=0.5];
 
\draw (12.2,-2.3) node{$(\eta,1)\in \hat{\alpha}$};
\draw (14,-2.8) node{$A(\eta)\approx  \frac{1}{2}q_{i}R$ or $ \frac{1}{6}q_{i}R$};
  \draw (4.5,-1.5) node{\Large \textbf{(c')}} ;
 
 \draw (8.5,-3.2) node{...};

 \draw (6,-4.8) arc [start angle=90,end angle=270,x radius=0.25,y radius=0.5];
 \draw (6,-5.8)[dashed] arc [start angle=270,end angle=450,x radius=0.25,y radius=0.5];
 \draw (11,-3.5) arc [start angle=90,end angle=270,x radius=0.25,y radius=0.5];
 \draw (11,-4.5) arc [start angle=270,end angle=450,x radius=0.25,y radius=0.5];
 
 \draw (12,-4) node{$(\delta_{1},1)$ } ;

 \draw (13.3,-4.5) node{ $A(\delta_{1})\approx \frac{1}{3} q_{i}R$ } ;
 
 \draw (11,-6.3) arc [start angle=90,end angle=270,x radius=0.25,y radius=0.5];
 \draw (11,-7.3) arc [start angle=270,end angle=450,x radius=0.25,y radius=0.5];
 
 \draw (12.2,-5.3) node{$(\gamma,M-q_{i})$} ;
 
  \draw (12,-6.8) node{$(\delta_{2},1)$} ;
   \draw (13,-6.3) node{$A(\delta_{2})\approx \frac{2}{3}q_{i}R$} ;

 \draw (6,-4.8) to [out=0,in=180] (11,-3.5);
 \draw (6,-5.8) to [out=0,in=180] (11,-7.3);
 \draw (11,-6.3) to [out=180,in=270] (9,-5.5);
 \draw (9,-5.5) to [out=90,in=180] (11,-4.5);

 \draw (11,-4.8) arc [start angle=90,end angle=270,x radius=0.25,y radius=0.5];
 \draw (11,-5.8) arc [start angle=270,end angle=450,x radius=0.25,y radius=0.5];

\draw (6,-4.8)--(11,-4.8);
\draw (6,-5.8)--(11,-5.8);

 \draw (4.8,-5.3) node{$(\gamma,M)$} ;

 \draw[densely dotted,  thick] (6,-1.9)--(6,-7.5);
  \draw[densely dotted,  thick] (11,-1.9)--(11,-7.5) ;
  
  \draw (6,-7.8) node{$\alpha_{k}=\hat{\alpha}\cup{(\gamma,M)}$};
  \draw (11,-7.8) node{$\alpha_{k+1}=\hat{\alpha}\cup{(\gamma,M-q_{i})\cup{(\delta_{1},1)}\cup{(\delta_{2},1)}}$};
  
  \draw (8.4,-4.5) node{$u_{1}$};

\end{tikzpicture}

\vspace{1cm}

\subsection{Restriction of topological types of the $J$-holomorphic curves}
If $J_{0}(\alpha_{k+1},\alpha_{k})=-2+2g+k+l=2$, each topological type of $J$-holomorphic curve counted by $U\langle \alpha_{k+1} \rangle=\langle \alpha_{k} \rangle$  is  $(g,k,l)=(0,4,0)$, $(0,3,1)$, $(0,2,2)$, $(1,1,1)$ or $(1,2,0)$.
But in Proposition \ref{nagai},  only the type $(g,k,l)=(0,3,1)$ appears. So at first, we  exclude the others.

As is the same with the case $J_{0}\leq 1$, we make a list of topological types of $J_{0}=2$ as follows. 

\begin{itemize}

    \item $(g,k,l)=(0,4,0)$
    
     This case has four types as follows.

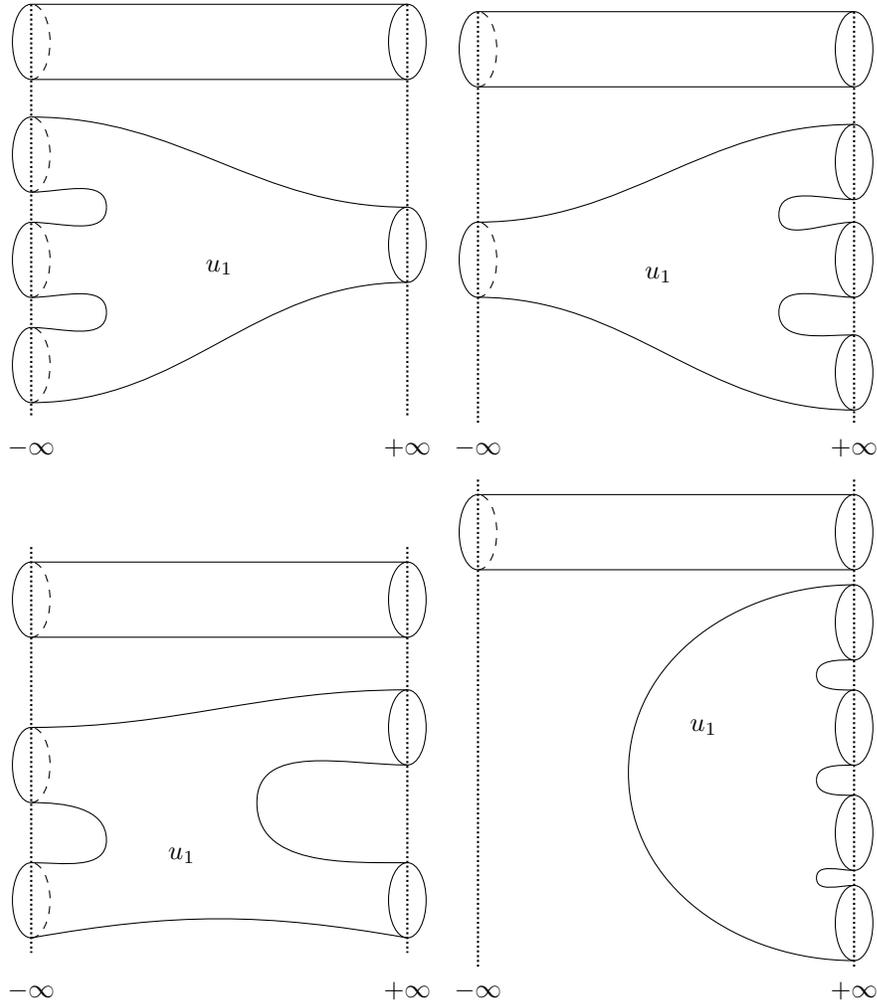
\begin{tikzpicture}
\draw (11,-1.8) arc [start angle=90,end angle=270,x radius=0.25,y radius=0.5];
 \draw (11,-2.8) arc [start angle=270,end angle=450,x radius=0.25,y radius=0.5];

 \draw (11,-1.8)--(6,-1.8) ;
  \draw (11,-2.8)--(6,-2.8) ;
  \draw (6,-1.8) arc [start angle=90,end angle=270,x radius=0.25,y radius=0.5];
 \draw (6,-2.8)[dashed] arc [start angle=270,end angle=450,x radius=0.25,y radius=0.5];

 \draw (6,-3.3) arc [start angle=90,end angle=270,x radius=0.25,y radius=0.5];
 \draw (6,-4.3)[dashed] arc [start angle=270,end angle=450,x radius=0.25,y radius=0.5];

 \draw (11,-4.5) arc [start angle=90,end angle=270,x radius=0.25,y radius=0.5];
 \draw (11,-5.5) arc [start angle=270,end angle=450,x radius=0.25,y radius=0.5];
  \draw (6,-4.7) arc [start angle=90,end angle=270,x radius=0.25,y radius=0.5];
 \draw (6,-5.7)[dashed] arc [start angle=270,end angle=450,x radius=0.25,y radius=0.5];

 \draw (11,-5.5) to [out=180,in=0] (6,-7.1);
 \draw (6,-6.1) to [out=0,in=270] (7,-5.9);
 \draw (7,-5.9) to [out=90,in=0] (6,-5.7);
 \draw (11,-4.5) to [out=180,in=0] (6,-3.3);
 
  \draw (6,-4.7) to [out=0,in=270] (7,-4.5);
 \draw (7,-4.5) to [out=90,in=0] (6,-4.3);

  \draw (6,-6.1) arc [start angle=90,end angle=270,x radius=0.25,y radius=0.5];
 \draw (6,-7.1)[dashed] arc [start angle=270,end angle=450,x radius=0.25,y radius=0.5];

 \draw[densely dotted,  thick] (6,-1.8)--(6,-7.3);
  \draw[densely dotted,  thick] (11,-1.8)--(11,-7.3) ;
  
  \draw (6,-7.7) node{$-\infty$};
  \draw (11,-7.7) node{$+\infty$};
  
  \draw (8.5,-5.3) node{$u_{1}$};

\end{tikzpicture}
\begin{tikzpicture}
\draw (11,-2) arc [start angle=90,end angle=270,x radius=0.25,y radius=0.5];
 \draw (11,-3) arc [start angle=270,end angle=450,x radius=0.25,y radius=0.5];
 
 \draw (11,-2)--(6,-2) ;
  \draw (11,-3)--(6,-3) ;
  \draw (6,-2) arc [start angle=90,end angle=270,x radius=0.25,y radius=0.5];
 \draw (6,-3)[dashed] arc [start angle=270,end angle=450,x radius=0.25,y radius=0.5];

 \draw (6,-4.8) arc [start angle=90,end angle=270,x radius=0.25,y radius=0.5];
 \draw (6,-5.8)[dashed] arc [start angle=270,end angle=450,x radius=0.25,y radius=0.5];
 \draw (11,-3.5) arc [start angle=90,end angle=270,x radius=0.25,y radius=0.5];
 \draw (11,-4.5) arc [start angle=270,end angle=450,x radius=0.25,y radius=0.5];

 \draw (11,-6.3) arc [start angle=90,end angle=270,x radius=0.25,y radius=0.5];
 \draw (11,-7.3) arc [start angle=270,end angle=450,x radius=0.25,y radius=0.5];

 \draw (6,-4.8) to [out=0,in=180] (11,-3.5);
 \draw (6,-5.8) to [out=0,in=180] (11,-7.3);

 \draw (11,-6.3) to [out=180,in=270] (10,-6);
 
 \draw (10,-6) to [out=90,in=180] (11,-5.8);
 
 \draw (11,-4.5) to [out=180,in=90] (10,-4.7);
 
 \draw (10,-4.7) to [out=270,in=180] (11,-4.8);

 \draw (11,-4.8) arc [start angle=90,end angle=270,x radius=0.25,y radius=0.5];
 \draw (11,-5.8) arc [start angle=270,end angle=450,x radius=0.25,y radius=0.5];

 \draw[densely dotted,  thick] (6,-2)--(6,-7.5);
  \draw[densely dotted,  thick] (11,-2)--(11,-7.5) ;

  \draw (6,-7.8) node{$-\infty$};
  \draw (11,-7.8) node{$+\infty$};
  
  \draw (8.4,-5.5) node{$u_{1}$};

\end{tikzpicture}

    \begin{tikzpicture}

\draw (11,-1.8) arc [start angle=90,end angle=270,x radius=0.25,y radius=0.5];
 \draw (11,-2.8) arc [start angle=270,end angle=450,x radius=0.25,y radius=0.5];
 
 \draw (11,-1.8)--(6,-1.8) ;
  \draw (11,-2.8)--(6,-2.8) ;
  \draw (6,-1.8) arc [start angle=90,end angle=270,x radius=0.25,y radius=0.5];
 \draw (6,-2.8)[dashed] arc [start angle=270,end angle=450,x radius=0.25,y radius=0.5];

 \draw (6,-4) arc [start angle=90,end angle=270,x radius=0.25,y radius=0.5];
 \draw (6,-5)[dashed] arc [start angle=270,end angle=450,x radius=0.25,y radius=0.5];
 \draw (11,-3.5) arc [start angle=90,end angle=270,x radius=0.25,y radius=0.5];
 \draw (11,-4.5) arc [start angle=270,end angle=450,x radius=0.25,y radius=0.5];

 \draw (11,-5.8) arc [start angle=90,end angle=270,x radius=0.25,y radius=0.5];
 \draw (11,-6.8) arc [start angle=270,end angle=450,x radius=0.25,y radius=0.5];
 
 \draw (6,-4) to [out=0,in=180] (11,-3.5);
 \draw (6,-6.8) to [out=10,in=170] (11,-6.8);
 \draw (11,-5.8) to [out=180,in=270] (9,-5);
 \draw (9,-5) to [out=90,in=180] (11,-4.5);

  \draw (6,-5.8) arc [start angle=90,end angle=270,x radius=0.25,y radius=0.5];
 \draw (6,-6.8)[dashed] arc [start angle=270,end angle=450,x radius=0.25,y radius=0.5];

 \draw[densely dotted,  thick] (6,-1.6)--(6,-7);
  \draw[densely dotted,  thick] (11,-1.6)--(11,-7) ;
  
  \draw (6,-7.5) node{$-\infty$};
  \draw (11,-7.5) node{$+\infty$};
  
  \draw (8,-5.7) node{$u_{1}$};

 \draw (6,-5.8) to [out=0,in=270] (7,-5.5);
 \draw (7,-5.5) to [out=90,in=0] (6,-5);

\end{tikzpicture}
\begin{tikzpicture}

\draw (11,-1.8) arc [start angle=90,end angle=270,x radius=0.25,y radius=0.5];
 \draw (11,-2.8) arc [start angle=270,end angle=450,x radius=0.25,y radius=0.5];
 
 \draw (11,-1.8)--(6,-1.8) ;
  \draw (11,-2.8)--(6,-2.8) ;
  \draw (6,-1.8) arc [start angle=90,end angle=270,x radius=0.25,y radius=0.5];
 \draw (6,-2.8)[dashed] arc [start angle=270,end angle=450,x radius=0.25,y radius=0.5];

 \draw (11,-3) arc [start angle=90,end angle=270,x radius=0.25,y radius=0.5];
 \draw (11,-4) arc [start angle=270,end angle=450,x radius=0.25,y radius=0.5];
 
 \draw (11,-4.4) arc [start angle=90,end angle=270,x radius=0.25,y radius=0.5];
 \draw (11,-5.4) arc [start angle=270,end angle=450,x radius=0.25,y radius=0.5];

 \draw (11,-5.8) arc [start angle=90,end angle=270,x radius=0.25,y radius=0.5];
 \draw (11,-6.8) arc [start angle=270,end angle=450,x radius=0.25,y radius=0.5];
 
 \draw (11,-7) arc [start angle=90,end angle=270,x radius=0.25,y radius=0.5];
 \draw (11,-8) arc [start angle=270,end angle=450,x radius=0.25,y radius=0.5];

 \draw (11,-8) to [out=180,in=270] (8,-5.5);
 \draw (8,-5.5) to [out=90,in=180] (11,-3);
 \draw (11,-4) to [out=180,in=90] (10.5,-4.2);
   \draw (10.5,-4.2) to [out=270,in=180] (11,-4.4);

  \draw (11,-5.4) to [out=180,in=90] (10.5,-5.6);
    \draw (10.5,-5.6) to [out=270,in=180] (11,-5.8);

  \draw (11,-7) to [out=180,in=270] (10.5,-6.9);
  \draw (10.5,-6.9) to [out=90,in=180] (11,-6.8);

 \draw[densely dotted,  thick] (6,-1.6)--(6,-8.1);
  \draw[densely dotted,  thick] (11,-1.6)--(11,-8.1) ;
  
  \draw (6,-8.4) node{$-\infty$};
  \draw (11,-8.4) node{$+\infty$};
  
  \draw (9,-4.9) node{$u_{1}$};

\end{tikzpicture}

    \item  $(g,k,l)=(0,3,1)$
    
     This case has eight types as follows.
    
    
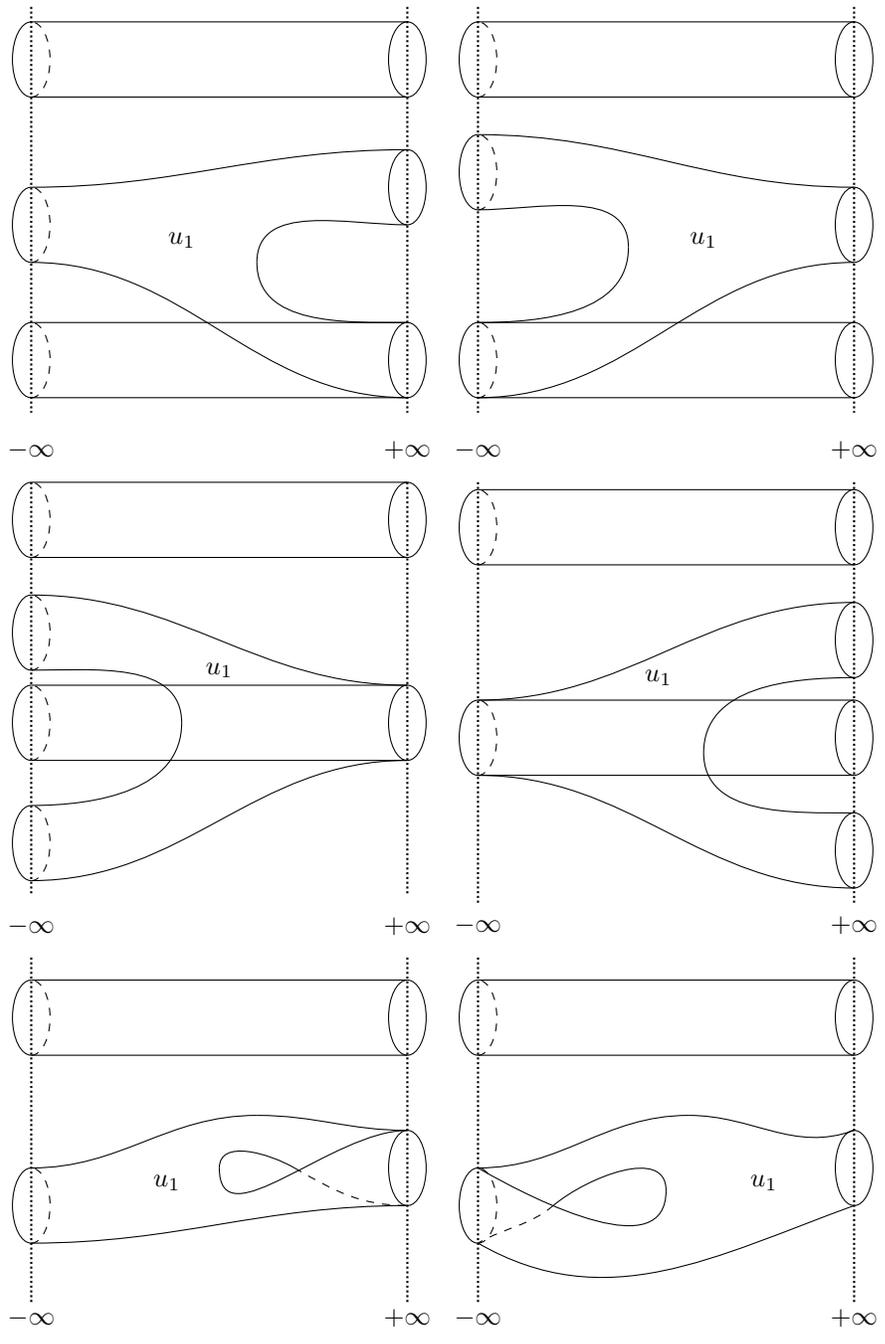
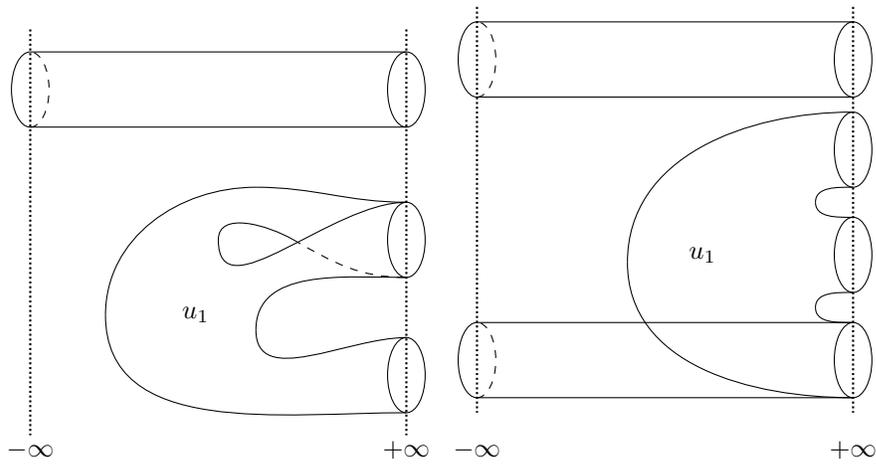
\begin{tikzpicture}

\draw (11,-1.8) arc [start angle=90,end angle=270,x radius=0.25,y radius=0.5];
 \draw (11,-2.8) arc [start angle=270,end angle=450,x radius=0.25,y radius=0.5];
 
 \draw (11,-1.8)--(6,-1.8) ;
  \draw (11,-2.8)--(6,-2.8) ;
  \draw (6,-1.8) arc [start angle=90,end angle=270,x radius=0.25,y radius=0.5];
 \draw (6,-2.8)[dashed] arc [start angle=270,end angle=450,x radius=0.25,y radius=0.5];

 \draw (6,-4) arc [start angle=90,end angle=270,x radius=0.25,y radius=0.5];
 \draw (6,-5)[dashed] arc [start angle=270,end angle=450,x radius=0.25,y radius=0.5];
 \draw (11,-3.5) arc [start angle=90,end angle=270,x radius=0.25,y radius=0.5];
 \draw (11,-4.5) arc [start angle=270,end angle=450,x radius=0.25,y radius=0.5];

 \draw (11,-5.8) arc [start angle=90,end angle=270,x radius=0.25,y radius=0.5];
 \draw (11,-6.8) arc [start angle=270,end angle=450,x radius=0.25,y radius=0.5];
 
 \draw (6,-4) to [out=0,in=180] (11,-3.5);
 \draw (6,-5) to [out=0,in=180] (11,-6.8);
 \draw (11,-5.8) to [out=180,in=270] (9,-5);
 \draw (9,-5) to [out=90,in=180] (11,-4.5);
 
 \draw (11,-5.8)--(6,-5.8) ;
  \draw (11,-6.8)--(6,-6.8) ;
  \draw (6,-5.8) arc [start angle=90,end angle=270,x radius=0.25,y radius=0.5];
 \draw (6,-6.8)[dashed] arc [start angle=270,end angle=450,x radius=0.25,y radius=0.5];

 \draw[densely dotted,  thick] (6,-1.6)--(6,-7);
  \draw[densely dotted,  thick] (11,-1.6)--(11,-7) ;
  
  \draw (6,-7.5) node{$-\infty$};
  \draw (11,-7.5) node{$+\infty$};
  
  \draw (8,-4.7) node{$u_{1}$};

\end{tikzpicture}
\begin{tikzpicture}

\draw (11,-1.8) arc [start angle=90,end angle=270,x radius=0.25,y radius=0.5];
 \draw (11,-2.8) arc [start angle=270,end angle=450,x radius=0.25,y radius=0.5];
 
 \draw (11,-1.8)--(6,-1.8) ;
  \draw (11,-2.8)--(6,-2.8) ;
  \draw (6,-1.8) arc [start angle=90,end angle=270,x radius=0.25,y radius=0.5];
 \draw (6,-2.8)[dashed] arc [start angle=270,end angle=450,x radius=0.25,y radius=0.5];

 \draw (6,-3.3) arc [start angle=90,end angle=270,x radius=0.25,y radius=0.5];
 \draw (6,-4.3)[dashed] arc [start angle=270,end angle=450,x radius=0.25,y radius=0.5];

 \draw (11,-4) arc [start angle=90,end angle=270,x radius=0.25,y radius=0.5];
 \draw (11,-5) arc [start angle=270,end angle=450,x radius=0.25,y radius=0.5];

 \draw (11,-5.8) arc [start angle=90,end angle=270,x radius=0.25,y radius=0.5];
 \draw (11,-6.8) arc [start angle=270,end angle=450,x radius=0.25,y radius=0.5];

 \draw (11,-5) to [out=180,in=0] (6,-6.8);
 \draw (6,-5.8) to [out=0,in=270] (8,-4.8);
 \draw (8,-4.8) to [out=90,in=0] (6,-4.3);
 \draw (11,-4) to [out=180,in=0] (6,-3.3);

 \draw (11,-5.8)--(6,-5.8) ;
  \draw (11,-6.8)--(6,-6.8) ;
  \draw (6,-5.8) arc [start angle=90,end angle=270,x radius=0.25,y radius=0.5];
 \draw (6,-6.8)[dashed] arc [start angle=270,end angle=450,x radius=0.25,y radius=0.5];

 \draw[densely dotted,  thick] (6,-1.6)--(6,-7);
  \draw[densely dotted,  thick] (11,-1.6)--(11,-7) ;
  
  \draw (6,-7.5) node{$-\infty$};
  \draw (11,-7.5) node{$+\infty$};
  
  \draw (9,-4.7) node{$u_{1}$};

\end{tikzpicture}
    
    \begin{tikzpicture}
\draw (11,-1.8) arc [start angle=90,end angle=270,x radius=0.25,y radius=0.5];
 \draw (11,-2.8) arc [start angle=270,end angle=450,x radius=0.25,y radius=0.5];

 \draw (11,-1.8)--(6,-1.8) ;
  \draw (11,-2.8)--(6,-2.8) ;
  \draw (6,-1.8) arc [start angle=90,end angle=270,x radius=0.25,y radius=0.5];
 \draw (6,-2.8)[dashed] arc [start angle=270,end angle=450,x radius=0.25,y radius=0.5];

 \draw (6,-3.3) arc [start angle=90,end angle=270,x radius=0.25,y radius=0.5];
 \draw (6,-4.3)[dashed] arc [start angle=270,end angle=450,x radius=0.25,y radius=0.5];

 \draw (11,-4.5) arc [start angle=90,end angle=270,x radius=0.25,y radius=0.5];
 \draw (11,-5.5) arc [start angle=270,end angle=450,x radius=0.25,y radius=0.5];
  \draw (6,-4.5) arc [start angle=90,end angle=270,x radius=0.25,y radius=0.5];
 \draw (6,-5.5)[dashed] arc [start angle=270,end angle=450,x radius=0.25,y radius=0.5];
 
 \draw (6,-4.5)--(11,-4.5);
 
  \draw (6,-5.5)--(11,-5.5);

 \draw (11,-5.5) to [out=180,in=0] (6,-7.1);
 \draw (6,-6.1) to [out=0,in=270] (8,-5);
 \draw (8,-5) to [out=90,in=0] (6,-4.3);
 \draw (11,-4.5) to [out=180,in=0] (6,-3.3);

  \draw (6,-6.1) arc [start angle=90,end angle=270,x radius=0.25,y radius=0.5];
 \draw (6,-7.1)[dashed] arc [start angle=270,end angle=450,x radius=0.25,y radius=0.5];

 \draw[densely dotted,  thick] (6,-1.8)--(6,-7.3);
  \draw[densely dotted,  thick] (11,-1.8)--(11,-7.3) ;
  
  \draw (6,-7.7) node{$-\infty$};
  \draw (11,-7.7) node{$+\infty$};
  
  \draw (8.5,-4.3) node{$u_{1}$};

\end{tikzpicture}
\begin{tikzpicture}
\draw (11,-2) arc [start angle=90,end angle=270,x radius=0.25,y radius=0.5];
 \draw (11,-3) arc [start angle=270,end angle=450,x radius=0.25,y radius=0.5];
 
 \draw (11,-2)--(6,-2) ;
  \draw (11,-3)--(6,-3) ;
  \draw (6,-2) arc [start angle=90,end angle=270,x radius=0.25,y radius=0.5];
 \draw (6,-3)[dashed] arc [start angle=270,end angle=450,x radius=0.25,y radius=0.5];

 \draw (6,-4.8) arc [start angle=90,end angle=270,x radius=0.25,y radius=0.5];
 \draw (6,-5.8)[dashed] arc [start angle=270,end angle=450,x radius=0.25,y radius=0.5];
 \draw (11,-3.5) arc [start angle=90,end angle=270,x radius=0.25,y radius=0.5];
 \draw (11,-4.5) arc [start angle=270,end angle=450,x radius=0.25,y radius=0.5];

 \draw (11,-6.3) arc [start angle=90,end angle=270,x radius=0.25,y radius=0.5];
 \draw (11,-7.3) arc [start angle=270,end angle=450,x radius=0.25,y radius=0.5];

 \draw (6,-4.8) to [out=0,in=180] (11,-3.5);
 \draw (6,-5.8) to [out=0,in=180] (11,-7.3);
 \draw (11,-6.3) to [out=180,in=270] (9,-5.5);
 \draw (9,-5.5) to [out=90,in=180] (11,-4.5);

 \draw (11,-4.8) arc [start angle=90,end angle=270,x radius=0.25,y radius=0.5];
 \draw (11,-5.8) arc [start angle=270,end angle=450,x radius=0.25,y radius=0.5];

\draw (6,-4.8)--(11,-4.8);
\draw (6,-5.8)--(11,-5.8);

 \draw[densely dotted,  thick] (6,-1.9)--(6,-7.5);
  \draw[densely dotted,  thick] (11,-1.9)--(11,-7.5) ;
  
  \draw (6,-7.8) node{$-\infty$};
  \draw (11,-7.8) node{$+\infty$};
  
  \draw (8.4,-4.5) node{$u_{1}$};

\end{tikzpicture}

\begin{tikzpicture}
 
 \draw (2,-3)--(7,-3);
  \draw (2,-4)--(7,-4);
 \draw (2,-3) arc [start angle=90,end angle=270,x radius=0.25,y radius=0.5];
 \draw (2,-4)[dashed] arc [start angle=270,end angle=450,x radius=0.25,y radius=0.5];
 \draw (7,-3) arc [start angle=90,end angle=270,x radius=0.25,y radius=0.5];
 \draw (7,-4) arc [start angle=270,end angle=450,x radius=0.25,y radius=0.5];

 \draw (2,-5.5) arc [start angle=90,end angle=270,x radius=0.25,y radius=0.5];
 \draw (2,-6.5)[dashed] arc [start angle=270,end angle=450,x radius=0.25,y radius=0.5];
 \draw (7,-5) arc [start angle=90,end angle=270,x radius=0.25,y radius=0.5];
 \draw (7,-6) arc [start angle=270,end angle=450,x radius=0.25,y radius=0.5];

  \draw (4.5,-5.5) to [out=270,in=180] (7,-5);
 \draw (2,-6.5) to [out=0,in=180] (7,-6);
 
  \draw (5,-4.8) to [out=0,in=180] (7,-5);
   \draw (2,-5.5) to [out=0,in=180] (5,-4.8);
 \draw (4.5,-5.5) to [out=90,in=150] (5.5,-5.5);
  \draw [dashed] (5.5,-5.5) to [out=330,in=180] (7,-6);
 
 \draw[densely dotted,  thick] (2,-2.7)--(2,-7.3);
  \draw[densely dotted,  thick] (7,-2.7)--(7,-7.3) ;

 \draw (3.8,-5.7) node{$u_{1}$};
 
 \draw (7,-7.5) node{$+\infty$};
  \draw (2,-7.5) node{$-\infty$};
 
\end{tikzpicture}
\begin{tikzpicture}
 
 \draw (2,-3)--(7,-3);
  \draw (2,-4)--(7,-4);
 \draw (2,-3) arc [start angle=90,end angle=270,x radius=0.25,y radius=0.5];
 \draw (2,-4)[dashed] arc [start angle=270,end angle=450,x radius=0.25,y radius=0.5];
 \draw (7,-3) arc [start angle=90,end angle=270,x radius=0.25,y radius=0.5];
 \draw (7,-4) arc [start angle=270,end angle=450,x radius=0.25,y radius=0.5];

 \draw (2,-5.5) arc [start angle=90,end angle=270,x radius=0.25,y radius=0.5];
 \draw (2,-6.5)[dashed] arc [start angle=270,end angle=450,x radius=0.25,y radius=0.5];
 \draw (7,-5) arc [start angle=90,end angle=270,x radius=0.25,y radius=0.5];
 \draw (7,-6) arc [start angle=270,end angle=450,x radius=0.25,y radius=0.5];

 \draw (2,-6.5) to [out=330,in=200] (7,-6);
  \draw (4.8,-4.8) to [out=0,in=200] (7,-5);
   \draw (2,-5.5) to [out=0,in=180] (4.8,-4.8);
    \draw (2,-5.5) to [out=330,in=270] (4.5,-5.8);
       \draw (4.5,-5.8) to [out=90,in=40] (3,-6);
       \draw[dashed] (3,-6) to [out=220,in=30] (2,-6.5);

 \draw[densely dotted,  thick] (2,-2.7)--(2,-7.3);
  \draw[densely dotted,  thick] (7,-2.7)--(7,-7.3) ;

 \draw (5.8,-5.7) node{$u_{1}$};
 
 \draw (7,-7.5) node{$+\infty$};
  \draw (2,-7.5) node{$-\infty$};
 
\end{tikzpicture}

\begin{tikzpicture}
 
 \draw (2,-3)--(7,-3);
  \draw (2,-4)--(7,-4);
 \draw (2,-3) arc [start angle=90,end angle=270,x radius=0.25,y radius=0.5];
 \draw (2,-4)[dashed] arc [start angle=270,end angle=450,x radius=0.25,y radius=0.5];
 \draw (7,-3) arc [start angle=90,end angle=270,x radius=0.25,y radius=0.5];
 \draw (7,-4) arc [start angle=270,end angle=450,x radius=0.25,y radius=0.5];

 \draw (7,-6.8) arc [start angle=90,end angle=270,x radius=0.25,y radius=0.5];
 \draw (7,-7.8) arc [start angle=270,end angle=450,x radius=0.25,y radius=0.5];
 \draw (7,-5) arc [start angle=90,end angle=270,x radius=0.25,y radius=0.5];
 \draw (7,-6) arc [start angle=270,end angle=450,x radius=0.25,y radius=0.5];

  \draw (4.5,-5.5) to [out=270,in=180] (7,-5);
 \draw (5,-6.7) to [out=90,in=180] (7,-6);
 \draw (5,-6.7) to [out=270,in=180] (7,-6.8);
  \draw (3,-6.5) to [out=270,in=180] (7,-7.8);
 
  \draw (5,-4.8) to [out=0,in=180] (7,-5);
   \draw (3,-6.5) to [out=90,in=180] (5,-4.8);
 \draw (4.5,-5.5) to [out=90,in=150] (5.5,-5.5);
  \draw [dashed] (5.5,-5.5) to [out=330,in=180] (7,-6);
 
 \draw[densely dotted,  thick] (2,-2.7)--(2,-8.1);
  \draw[densely dotted,  thick] (7,-2.7)--(7,-8.1) ;

 \draw (4.2,-6.5) node{$u_{1}$};
 
 \draw (7,-8.3) node{$+\infty$};
  \draw (2,-8.3) node{$-\infty$};
 
\end{tikzpicture}
\begin{tikzpicture}

\draw (11,-1.8) arc [start angle=90,end angle=270,x radius=0.25,y radius=0.5];
 \draw (11,-2.8) arc [start angle=270,end angle=450,x radius=0.25,y radius=0.5];
 
 \draw (11,-1.8)--(6,-1.8) ;
  \draw (11,-2.8)--(6,-2.8) ;
  \draw (6,-1.8) arc [start angle=90,end angle=270,x radius=0.25,y radius=0.5];
 \draw (6,-2.8)[dashed] arc [start angle=270,end angle=450,x radius=0.25,y radius=0.5];

 \draw (11,-5.8)--(6,-5.8) ;
  \draw (11,-6.8)--(6,-6.8) ;
  \draw (6,-5.8) arc [start angle=90,end angle=270,x radius=0.25,y radius=0.5];
 \draw (6,-6.8)[dashed] arc [start angle=270,end angle=450,x radius=0.25,y radius=0.5];

 \draw (11,-3) arc [start angle=90,end angle=270,x radius=0.25,y radius=0.5];
 \draw (11,-4) arc [start angle=270,end angle=450,x radius=0.25,y radius=0.5];
 
 \draw (11,-4.4) arc [start angle=90,end angle=270,x radius=0.25,y radius=0.5];
 \draw (11,-5.4) arc [start angle=270,end angle=450,x radius=0.25,y radius=0.5];

 \draw (11,-5.8) arc [start angle=90,end angle=270,x radius=0.25,y radius=0.5];
 \draw (11,-6.8) arc [start angle=270,end angle=450,x radius=0.25,y radius=0.5];

 \draw (11,-6.8) to [out=180,in=270] (8,-5);
 \draw (8,-5) to [out=90,in=180] (11,-3);

 \draw[densely dotted,  thick] (6,-1.6)--(6,-7);
  \draw[densely dotted,  thick] (11,-1.6)--(11,-7) ;
  
  \draw (6,-7.5) node{$-\infty$};
  \draw (11,-7.5) node{$+\infty$};
  
  \draw (9,-4.9) node{$u_{1}$};
  
   \draw (11,-4) to [out=180,in=90] (10.5,-4.2);
   \draw (10.5,-4.2) to [out=270,in=180] (11,-4.4);

  \draw (11,-5.4) to [out=180,in=90] (10.5,-5.6);
    \draw (10.5,-5.6) to [out=270,in=180] (11,-5.8);

\end{tikzpicture}
    
    \item  $(g,k,l)=(0,2,2)$
    
     This case has four types as follows. Note that under the assumption that there is only one simple elliptic orbit, the first and the fourth cases can not occur.
    
    
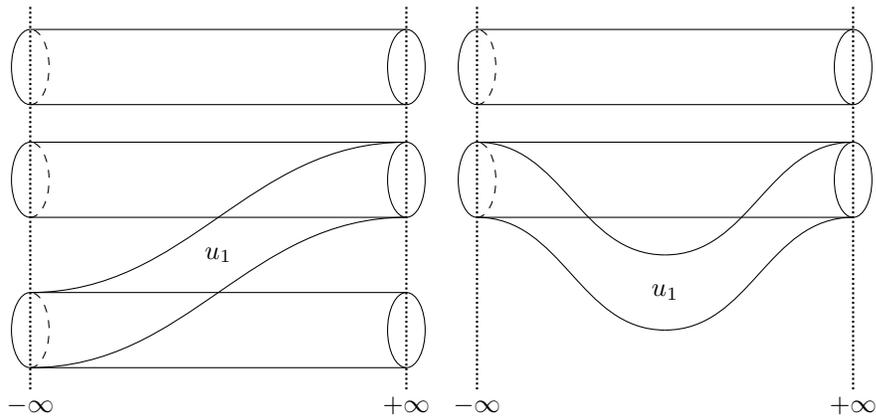
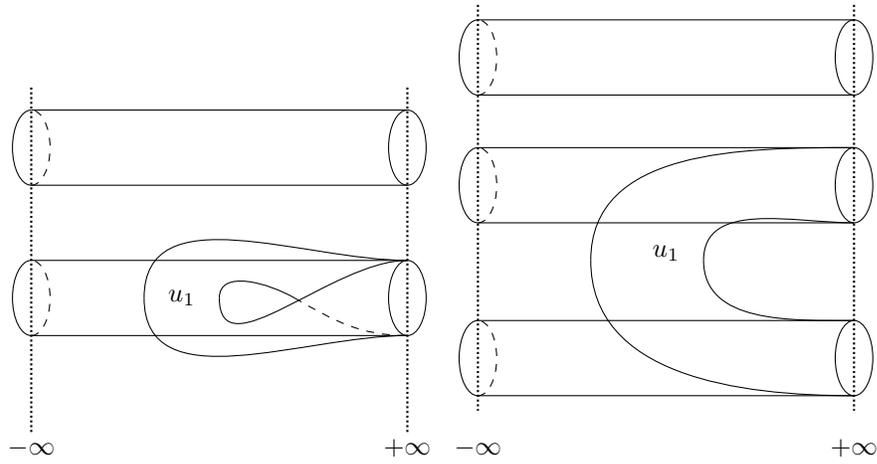
\begin{tikzpicture}
 
 \draw (2,-2.5)--(7,-2.5);
  \draw (2,-3.5)--(7,-3.5);
 \draw (2,-2.5) arc [start angle=90,end angle=270,x radius=0.25,y radius=0.5];
 \draw (2,-3.5)[dashed] arc [start angle=270,end angle=450,x radius=0.25,y radius=0.5];
 \draw (7,-2.5) arc [start angle=90,end angle=270,x radius=0.25,y radius=0.5];
 \draw (7,-3.5) arc [start angle=270,end angle=450,x radius=0.25,y radius=0.5];

\draw (2,-4)--(7,-4);
  \draw (2,-5)--(7,-5);
  \draw (2,-7)--(7,-7);
  \draw (2,-6)--(7,-6);
 \draw (2,-4) arc [start angle=90,end angle=270,x radius=0.25,y radius=0.5];
 \draw (2,-5)[dashed] arc [start angle=270,end angle=450,x radius=0.25,y radius=0.5];
 \draw (7,-4) arc [start angle=90,end angle=270,x radius=0.25,y radius=0.5];
 \draw (7,-5) arc [start angle=270,end angle=450,x radius=0.25,y radius=0.5];
 
  \draw (2,-6) arc [start angle=90,end angle=270,x radius=0.25,y radius=0.5];
 \draw (2,-7)[dashed] arc [start angle=270,end angle=450,x radius=0.25,y radius=0.5];
  \draw (7,-6) arc [start angle=90,end angle=270,x radius=0.25,y radius=0.5];
 \draw (7,-7) arc [start angle=270,end angle=450,x radius=0.25,y radius=0.5];
  \draw (2,-6) to [out=0,in=180] (7,-4);
 \draw (2,-7) to [out=0,in=180] (7,-5);
 
 \draw[densely dotted,  thick] (2,-2.2)--(2,-7.3);
  \draw[densely dotted,  thick] (7,-2.2)--(7,-7.3) ;

 \draw (4.5,-5.5) node{$u_{1}$};
 
 \draw (7,-7.5) node{$+\infty$};
  \draw (2,-7.5) node{$-\infty$};
 
\end{tikzpicture}
\begin{tikzpicture}
 
 \draw (2,-2.5)--(7,-2.5);
  \draw (2,-3.5)--(7,-3.5);
 \draw (2,-2.5) arc [start angle=90,end angle=270,x radius=0.25,y radius=0.5];
 \draw (2,-3.5)[dashed] arc [start angle=270,end angle=450,x radius=0.25,y radius=0.5];
 \draw (7,-2.5) arc [start angle=90,end angle=270,x radius=0.25,y radius=0.5];
 \draw (7,-3.5) arc [start angle=270,end angle=450,x radius=0.25,y radius=0.5];

\draw (2,-4)--(7,-4);
  \draw (2,-5)--(7,-5);
 \draw (2,-4) arc [start angle=90,end angle=270,x radius=0.25,y radius=0.5];
 \draw (2,-5)[dashed] arc [start angle=270,end angle=450,x radius=0.25,y radius=0.5];
 \draw (7,-4) arc [start angle=90,end angle=270,x radius=0.25,y radius=0.5];
 \draw (7,-5) arc [start angle=270,end angle=450,x radius=0.25,y radius=0.5];

  \draw (4.5,-5.5) to [out=0,in=180] (7,-4);
 \draw (4.5,-6.5) to [out=0,in=180] (7,-5);
  \draw (4.5,-5.5) to [out=180,in=0] (2,-4);
 \draw (4.5,-6.5) to [out=180,in=0] (2,-5);
 
 \draw[densely dotted,  thick] (2,-2.2)--(2,-7.3);
  \draw[densely dotted,  thick] (7,-2.2)--(7,-7.3) ;

 \draw (4.5,-6) node{$u_{1}$};
 
 \draw (7,-7.5) node{$+\infty$};
  \draw (2,-7.5) node{$-\infty$};
 
\end{tikzpicture}
    
\begin{tikzpicture}
 
 \draw (2,-3)--(7,-3);
  \draw (2,-4)--(7,-4);
 \draw (2,-3) arc [start angle=90,end angle=270,x radius=0.25,y radius=0.5];
 \draw (2,-4)[dashed] arc [start angle=270,end angle=450,x radius=0.25,y radius=0.5];
 \draw (7,-3) arc [start angle=90,end angle=270,x radius=0.25,y radius=0.5];
 \draw (7,-4) arc [start angle=270,end angle=450,x radius=0.25,y radius=0.5];

\draw (2,-5)--(7,-5);
  \draw (2,-6)--(7,-6);
 \draw (2,-5) arc [start angle=90,end angle=270,x radius=0.25,y radius=0.5];
 \draw (2,-6)[dashed] arc [start angle=270,end angle=450,x radius=0.25,y radius=0.5];
 \draw (7,-5) arc [start angle=90,end angle=270,x radius=0.25,y radius=0.5];
 \draw (7,-6) arc [start angle=270,end angle=450,x radius=0.25,y radius=0.5];

  \draw (4.5,-5.5) to [out=270,in=180] (7,-5);
 \draw (3.5,-5.5) to [out=270,in=180] (7,-6);
 
  \draw (3.5,-5.5) to [out=90,in=180] (7,-5);
 \draw (4.5,-5.5) to [out=90,in=150] (5.5,-5.5);
  \draw [dashed] (5.5,-5.5) to [out=330,in=180] (7,-6);
 
 \draw[densely dotted,  thick] (2,-2.7)--(2,-7.3);
  \draw[densely dotted,  thick] (7,-2.7)--(7,-7.3) ;

 \draw (4,-5.5) node{$u_{1}$};
 
 \draw (7,-7.5) node{$+\infty$};
  \draw (2,-7.5) node{$-\infty$};
 
\end{tikzpicture}
\begin{tikzpicture}

\draw (19,-1.8) arc [start angle=90,end angle=270,x radius=0.25,y radius=0.5];
 \draw (19,-2.8) arc [start angle=270,end angle=450,x radius=0.25,y radius=0.5];
 \draw (19,-1.8)--(14,-1.8) ;
  \draw (19,-2.8)--(14,-2.8) ;
  \draw (14,-1.8) arc [start angle=90,end angle=270,x radius=0.25,y radius=0.5];
 \draw (14,-2.8)[dashed] arc [start angle=270,end angle=450,x radius=0.25,y radius=0.5];

 \draw (19,-3.5) arc [start angle=90,end angle=270,x radius=0.25,y radius=0.5];
 \draw (19,-4.5) arc [start angle=270,end angle=450,x radius=0.25,y radius=0.5];

 \draw (19,-5.8) arc [start angle=90,end angle=270,x radius=0.25,y radius=0.5];
 \draw (19,-6.8) arc [start angle=270,end angle=450,x radius=0.25,y radius=0.5];
 
 \draw (15.5,-5) to [out=90,in=180] (19,-3.5);
 \draw (15.5,-5) to [out=270,in=180] (19,-6.8);
 \draw (19,-5.8) to [out=180,in=270] (17,-5);
 \draw (17,-5) to [out=90,in=180] (19,-4.5);

 \draw[densely dotted,  thick] (14,-1.6)--(14,-7);
  \draw[densely dotted,  thick] (19,-1.6)--(19,-7) ;
  
  \draw (14,-7.5) node{$-\infty$};
  \draw (19,-7.5) node{$+\infty$};
  
  \draw (16.5,-4.9) node{$u_{1}$};
  
   \draw (19,-3.5)--(14,-3.5) ;
  \draw (19,-4.5)--(14,-4.5) ;
  \draw (14,-3.5) arc [start angle=90,end angle=270,x radius=0.25,y radius=0.5];
 \draw (14,-4.5)[dashed] arc [start angle=270,end angle=450,x radius=0.25,y radius=0.5];
  \draw (19,-5.8)--(14,-5.8) ;
  \draw (19,-6.8)--(14,-6.8) ;
  \draw (14,-5.8) arc [start angle=90,end angle=270,x radius=0.25,y radius=0.5];
 \draw (14,-6.8)[dashed] arc [start angle=270,end angle=450,x radius=0.25,y radius=0.5];

\end{tikzpicture}
    
    \item  $(g,k,l)=(1,1,1)$
    
     This case has only one type as follows.
    
    
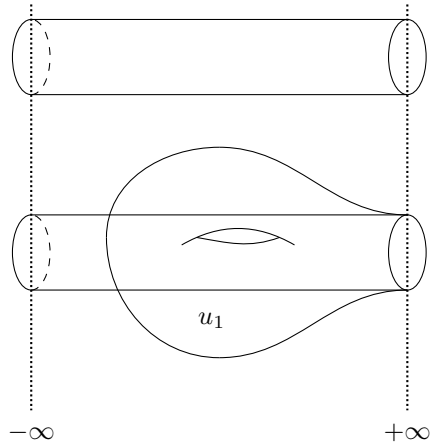
\begin{tikzpicture}

\draw (11,-1.8) arc [start angle=90,end angle=270,x radius=0.25,y radius=0.5];
 \draw (11,-2.8) arc [start angle=270,end angle=450,x radius=0.25,y radius=0.5];
 
 \draw (11,-1.8)--(6,-1.8) ;
  \draw (11,-2.8)--(6,-2.8) ;
  \draw (6,-1.8) arc [start angle=90,end angle=270,x radius=0.25,y radius=0.5];
 \draw (6,-2.8)[dashed] arc [start angle=270,end angle=450,x radius=0.25,y radius=0.5];
 
 \draw (11,-4.4) arc [start angle=90,end angle=270,x radius=0.25,y radius=0.5];
 \draw (11,-5.4) arc [start angle=270,end angle=450,x radius=0.25,y radius=0.5];

 \draw (11,-4.4) to [out=180,in=0] (8.5,-3.5);
  \draw (11,-5.4) to [out=180,in=0] (8.5,-6.3);
  
  \draw (8.5,-3.5) to [out=180,in=90] (7,-4.7);
\draw (7,-4.7) to [out=270,in=180] (8.5,-6.3);

\draw (11,-4.4)--(6,-4.4) ;
  \draw (11,-5.4)--(6,-5.4) ;
  
   \draw (6,-4.4) arc [start angle=90,end angle=270,x radius=0.25,y radius=0.5];
 \draw (6,-5.4)[dashed] arc [start angle=270,end angle=450,x radius=0.25,y radius=0.5];

 \draw[densely dotted,  thick] (6,-1.6)--(6,-7);
  \draw[densely dotted,  thick] (11,-1.6)--(11,-7) ;
  
  \draw (8,-4.8) to [out=30,in=150] (9.5,-4.8);
  \draw (8.2,-4.7) to [out=350,in=200] (9.3,-4.7);
  
  \draw (6,-7.3) node{$-\infty$};
  \draw (11,-7.3) node{$+\infty$};
  
  \draw (8.4,-5.8) node{$u_{1}$};

\end{tikzpicture}
    
    \item  $(g,k,l)=(1,2,0)$

 This case has two types as follows.

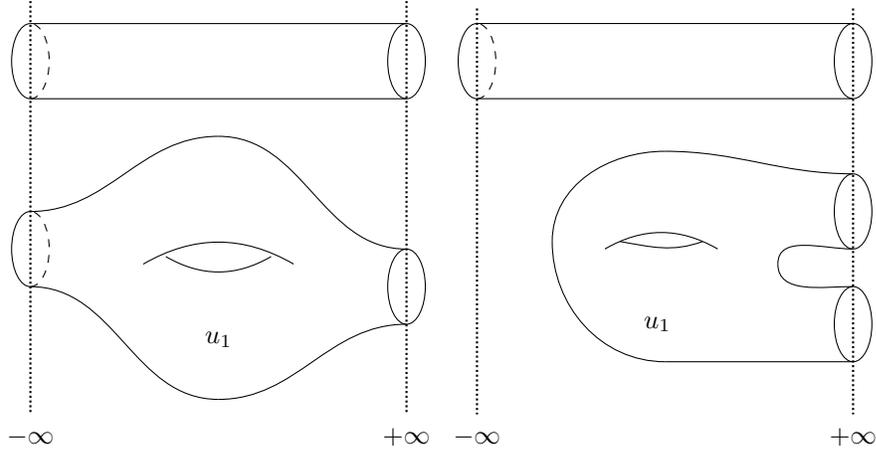
\begin{tikzpicture}

 \draw (13.5,6.2)--(8.5,6.2) ;
  \draw (13.5,5.2)--(8.5,5.2) ;
  \draw (8.5,6.2) arc [start angle=90,end angle=270,x radius=0.25,y radius=0.5];
 \draw (8.5,5.2)[dashed] arc [start angle=270,end angle=450,x radius=0.25,y radius=0.5];
 
 \draw (13.5,6.2) arc [start angle=90,end angle=270,x radius=0.25,y radius=0.5];
 \draw (13.5,5.2) arc [start angle=270,end angle=450,x radius=0.25,y radius=0.5];

 \draw (8.5,3.7) arc [start angle=90,end angle=270,x radius=0.25,y radius=0.5];
 \draw (8.5,2.7)[dashed] arc [start angle=270,end angle=450,x radius=0.25,y radius=0.5];

 \draw (13.5,3.2) arc [start angle=90,end angle=270,x radius=0.25,y radius=0.5];
 \draw (13.5,2.2) arc [start angle=270,end angle=450,x radius=0.25,y radius=0.5];
 
  \draw (8.5,3.7) to [out=0,in=180] (11,4.7);
  \draw (11,4.7) to [out=0,in=180] (13.5,3.2);
 \draw (8.5,2.7) to [out=0,in=180] (11,1.2);
 \draw (11,1.2) to [out=0,in=180] (13.5,2.2);
\draw (10,3) to [out=30,in=150] (12,3);
  \draw (10.3,3.1) to [out=-30,in=210] (11.7,3.1);

 \draw[densely dotted,  thick] (8.5,6.5)--(8.5,1);
 
 \draw[densely dotted,  thick] (13.5,6.5)--(13.5,1);

  \draw (8.5,0.7) node{$-\infty$};
  
    \draw (13.5,0.7) node{$+\infty$};

  \draw (11,2) node{$u_{1}$};

\end{tikzpicture}
    \begin{tikzpicture}

\draw (11,-1.8) arc [start angle=90,end angle=270,x radius=0.25,y radius=0.5];
 \draw (11,-2.8) arc [start angle=270,end angle=450,x radius=0.25,y radius=0.5];
 
 \draw (11,-1.8)--(6,-1.8) ;
  \draw (11,-2.8)--(6,-2.8) ;
  \draw (6,-1.8) arc [start angle=90,end angle=270,x radius=0.25,y radius=0.5];
 \draw (6,-2.8)[dashed] arc [start angle=270,end angle=450,x radius=0.25,y radius=0.5];
 
 \draw (11,-3.8) arc [start angle=90,end angle=270,x radius=0.25,y radius=0.5];
 \draw (11,-4.8) arc [start angle=270,end angle=450,x radius=0.25,y radius=0.5];
 \draw (11,-5.3) arc [start angle=90,end angle=270,x radius=0.25,y radius=0.5];
 \draw (11,-6.3) arc [start angle=270,end angle=450,x radius=0.25,y radius=0.5];
 \draw (11,-3.8) to [out=180,in=0] (8.5,-3.5);
  \draw (11,-6.3) to [out=180,in=0] (8.5,-6.3);
  
  \draw (8.5,-3.5) to [out=180,in=90] (7,-4.7);
\draw (7,-4.7) to [out=270,in=180] (8.5,-6.3);

 \draw (11,-4.8) to [out=180,in=90] (10,-5);
\draw (10,-5) to [out=270,in=180] (11,-5.3);

 \draw[densely dotted,  thick] (6,-1.6)--(6,-7);
  \draw[densely dotted,  thick] (11,-1.6)--(11,-7) ;
  
  \draw (7.7,-4.8) to [out=30,in=150] (9.2,-4.8);
  \draw (7.9,-4.7) to [out=350,in=200] (9,-4.7);
  
  \draw (6,-7.3) node{$-\infty$};
  \draw (11,-7.3) node{$+\infty$};
  
  \draw (8.4,-5.8) node{$u_{1}$};

\end{tikzpicture}

\end{itemize}

\begin{lem}\label{toptyp}

Suppose that $\alpha_{k+1}$ and $\alpha_{k}$ satisfy the assumptions 1, 2, 3, 4 and 5 in Proposition \ref{nagai}. 
Let $u=u_{0}\cup{u_{1}}\in \mathcal{M}^{J}(\alpha_{k+1},\alpha_{k})$ be
any $J$-holomorphic curve counted by $U\langle \alpha_{k+1} \rangle=\langle \alpha_{k} \rangle$. Then $u$ is $(g,k,l)=(0,3,1)$.

\end{lem}

\begin{proof}[\bf Proof of Lemma \ref{toptyp}]

In the cases $(g,k,l)=(0,2,2)$, $(1,1,1)$, we can see from the topological types that $A(\alpha_{k+1})-A(\alpha_{k})$ have to be larger than some action of orbit. But this contradicts $A(\alpha_{k+1})-A(\alpha_{k})<\epsilon$.

From now on, we consider  $(g,k,l)=(0,4,0)$, $(1,2,0)$. As a matter of fact,  we can easily exclude these cases in almost the same way as Lemma \ref{mainlemma}. But to make sure, we explain how to do in detail.

Since $l=0$ and $E(\alpha_{k+1})$, $E(\alpha_{k})\notin S_{-\theta}\cup{S_{\theta}}$, $u_{1}$ has no end asymptotic to $\gamma$ and since $E(\alpha_{k+1})$, $E(\alpha_{k})>p_{1}$, $q_{1}>0$, $u_{0}$ has some covering of $\mathbb{R}\times \gamma$. Let $z_{i}\to \gamma$. Then, we obtain a sequence of $J$-holomorphic curves $u_{1}^{i}$ which are through $z_{i}$ and either $(g,k)=(0,4)$ or $(1,2)$. Note that their topological types and orbit where their ends are asymptotic may change in the sequence.

At first, suppose that the sequence contains infinity many $J$-holomorphic curves whose topological types are $(g,k)=(0,4)$. By the compactness argument, there is an $J$- holomoprhic curve $u_{1}^{\infty}$ which may be splitting into some floors.  By its topological type and properties of ECH and Fredholm indexes, we can find that the number of floors are at most two.  If $u_{1}^{\infty}$ does not split, $\mathbb{R}\times \gamma \cap{u_{1}^{\infty}}\neq \emptyset$ and so $u_{0}\cap{u_{1}^{\infty}}\neq \emptyset$. This contradict $I(u_{0}\cup{u_{1}^{\infty}})=2$ and its admissibility. So we may assume that $u_{1}^{\infty}$ has two floors and write $u_{1}^{\infty}=(u_{-}^{\infty},u_{+}^{\infty})$ up to $\mathbb{
R}$-action.  By the additivity of ECH index and Fredholm index, we have $I(u_{0}\cup{u_{-}^{\infty}})=I(u_{0}\cup{u_{+}^{\infty}})=1$ and each non trivial part of $u_{\pm}^{\infty}$ are connected. Moreover, by the assumption of their topological type, one of the  non trivial parts of $u_{\pm}^{\infty}$ is of genus 0 with one end of it asymptotic to $\gamma$. This indicates that the middle orbit set, that is, the orbit set consisting of orbits where positive ends of $u_{0}\cup{u_{-}^{\infty}}$ are asymptotic is admissible. This contradicts (\ref{mod2}) and  $I(u_{0}\cup{u_{-}^{\infty}})=I(u_{0}\cup{u_{+}^{\infty}})=1$.

Next, suppose that the sequence contains infinity many $J$-holomorphic curves whose topological types are $(g,k)=(1,2)$.  In the same way, we have a splitting  curve.  Let $u_{\pm}^{\infty}$ be the curves in top and bottom floors of the splitting curve respectively.   Then  $I(u_{0}\cup{u_{-}^{\infty}})=I(u_{0}\cup{u_{+}^{\infty}})=1$. 
By geometric observation, $u_{+}^{\infty}$ has one or two negative ends and moreover at least one of them is elliptic. Also $u_{-}^{\infty}$ has one or two positive ends and at least one of them is elliptic. If all negative ends of $u_{-}^{\infty}$ are elliptic, this contradicts (\ref{mod2}). Also if all positive ends of $u_{+}^{\infty}$ are elliptic, this contradicts (\ref{mod2}). 
This means that the number of negative ends of $u_{+}^{\infty}$ and that one of positive ends of $u_{-}^{\infty}$ are both two and only one of them is asymptotic to the elliptic orbit respectively. Moreover, their multiplicities are the same (see the below figure).

Let $E(\alpha_{k+1})=E(\alpha_{k})=M$. Then the total multiplicity of the middle orbit set is $M+1$ and by the partition condition, we have $1=\mathrm{max}(S_{-\theta}\cap{\{1,\,\,2,...,\,\,M+1\}})=\mathrm{max}(S_{\theta}\cap{\{1,\,\,2,...,\,\,M+1\}})$.  But this is  a contradiction because of the assumption $M>q_{1},\,\,p_{1}$.

Combining the above argument, we complete  the proof of Lemma \ref{toptyp}.

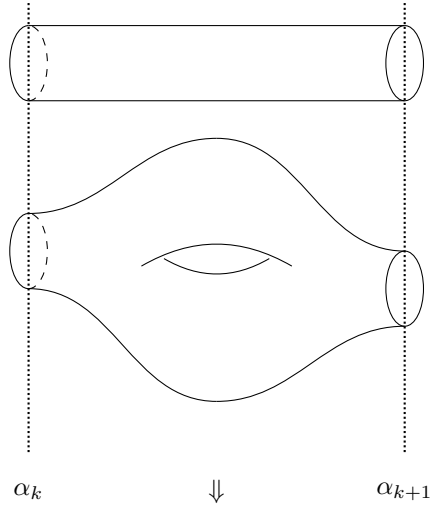
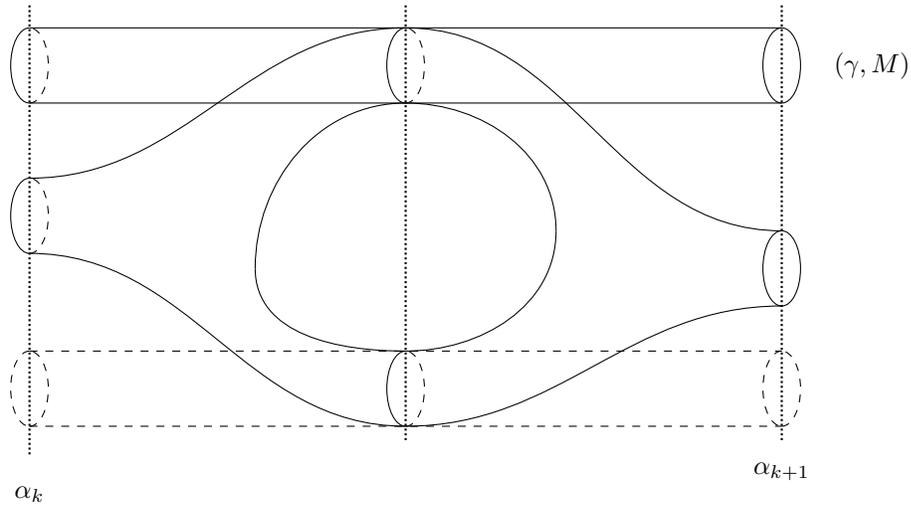
\begin{tikzpicture}
\draw (16,-1.8) arc [start angle=90,end angle=270,x radius=0.25,y radius=0.5];
 \draw (16,-2.8) arc [start angle=270,end angle=450,x radius=0.25,y radius=0.5];
 
 \draw (16,-1.8)--(11,-1.8) ;
  \draw (16,-2.8)--(11,-2.8) ;
  \draw (11,-1.8) arc [start angle=90,end angle=270,x radius=0.25,y radius=0.5];
 \draw (11,-2.8)[dashed] arc [start angle=270,end angle=450,x radius=0.25,y radius=0.5];
 
\draw (17.2,-2.3) node{$(\gamma,M)$};

 \draw (16,-4.5) arc [start angle=90,end angle=270,x radius=0.25,y radius=0.5];
 \draw (16,-5.5) arc [start angle=270,end angle=450,x radius=0.25,y radius=0.5];

 \draw (16,-5.5) to [out=180,in=0] (11,-7.1);
 \draw (11,-6.1) to [out=0,in=270] (13,-4.5);
 \draw (13,-4.5) to [out=90,in=0] (11,-2.8);
 \draw (16,-4.5) to [out=180,in=0] (11,-1.8);
 
  \draw (11,-6.1) arc [start angle=90,end angle=270,x radius=0.25,y radius=0.5];
 \draw (11,-7.1)[dashed] arc [start angle=270,end angle=450,x radius=0.25,y radius=0.5];
 
 \draw[densely dotted,  thick] (11,-1.5)--(11,-7.3);
  \draw[densely dotted,  thick] (16,-1.5)--(16,-7.3) ;
  
  \draw (16,-7.7) node{$\alpha_{k+1}$};

 \draw (11,0)node{$\Downarrow$};

 \draw (11,-1.8)--(6,-1.8) ;
  \draw (11,-2.8)--(6,-2.8) ;
  \draw (6,-1.8) arc [start angle=90,end angle=270,x radius=0.25,y radius=0.5];
 \draw (6,-2.8)[dashed] arc [start angle=270,end angle=450,x radius=0.25,y radius=0.5];

 \draw (6,-3.8) arc [start angle=90,end angle=270,x radius=0.25,y radius=0.5];
 \draw (6,-4.8)[dashed] arc [start angle=270,end angle=450,x radius=0.25,y radius=0.5];

 \draw (6,-3.8) to [out=0,in=180] (11,-1.8);
 \draw (6,-4.8) to [out=0,in=180] (11,-7.1);
 \draw (11,-6.1) to [out=180,in=270] (9,-5);
 \draw (9,-5) to [out=90,in=180] (11,-2.8);

  \draw (6,-6.1)[dashed] arc [start angle=90,end angle=270,x radius=0.25,y radius=0.5];
 \draw (6,-7.1)[dashed] arc [start angle=270,end angle=450,x radius=0.25,y radius=0.5];

  \draw (16,-6.1)[dashed] arc [start angle=90,end angle=270,x radius=0.25,y radius=0.5];
 \draw (16,-7.1)[dashed] arc [start angle=270,end angle=450,x radius=0.25,y radius=0.5];
 
 \draw [dashed] (6,-6.1)--(16,-6.1);
  \draw [dashed] (6,-7.1)--(16,-7.1) ;

 \draw[densely dotted,  thick] (6,-1.5)--(6,-7.5);

  \draw (6,-8) node{$\alpha_{k}$};

 \draw (13.5,6.2)--(8.5,6.2) ;
  \draw (13.5,5.2)--(8.5,5.2) ;
  \draw (8.5,6.2) arc [start angle=90,end angle=270,x radius=0.25,y radius=0.5];
 \draw (8.5,5.2)[dashed] arc [start angle=270,end angle=450,x radius=0.25,y radius=0.5];
 
 \draw (13.5,6.2) arc [start angle=90,end angle=270,x radius=0.25,y radius=0.5];
 \draw (13.5,5.2) arc [start angle=270,end angle=450,x radius=0.25,y radius=0.5];

 \draw (8.5,3.7) arc [start angle=90,end angle=270,x radius=0.25,y radius=0.5];
 \draw (8.5,2.7)[dashed] arc [start angle=270,end angle=450,x radius=0.25,y radius=0.5];

 \draw (13.5,3.2) arc [start angle=90,end angle=270,x radius=0.25,y radius=0.5];
 \draw (13.5,2.2) arc [start angle=270,end angle=450,x radius=0.25,y radius=0.5];
 
  \draw (8.5,3.7) to [out=0,in=180] (11,4.7);
  \draw (11,4.7) to [out=0,in=180] (13.5,3.2);
 \draw (8.5,2.7) to [out=0,in=180] (11,1.2);
 \draw (11,1.2) to [out=0,in=180] (13.5,2.2);
\draw (10,3) to [out=30,in=150] (12,3);
  \draw (10.3,3.1) to [out=-30,in=210] (11.7,3.1);

 \draw[densely dotted,  thick] (8.5,6.5)--(8.5,0.5);
 
 \draw[densely dotted,  thick] (13.5,6.5)--(13.5,0.5);

  \draw (8.5,0) node{$\alpha_{k}$};
  
    \draw (13.5,0) node{$\alpha_{k+1}$};

\end{tikzpicture}

\end{proof}

More precisely, we have the next lemma.

\begin{lem}\label{top}

Suppose that $\alpha_{k+1}$ and $\alpha_{k}$ satisfy the assumptions 1, 2, 3, 4 and 5 in Proposition \ref{nagai}. 
Let $u=u_{0}\cup{u_{1}}\in \mathcal{M}^{J}(\alpha_{k+1},\alpha_{k})$ be
any $J$-holomorphic curve counted by $U\langle \alpha_{k+1} \rangle=\langle \alpha_{k} \rangle$.  Then $\alpha_{k+1}$, $\alpha_{k}$ and $u$ hold  one of the following conditions.

\item[\bf (A).] Let $E(\alpha_{k+1})=M$ and $p_{i}:=\mathrm{max}(S_{-\theta}\cap{\{1,\,\,2,...,\,\,M\}})$. Then there are two  negative hyperbolic orbits $\delta_{1}$, $\delta_{2}$ and an admissible orbit set $\hat{\alpha}$ consisting of negative hyperbolic orbits   such that $\alpha_{k+1}=\hat{\alpha}\cup{(\gamma,M)\cup{(\delta_{1},1)}}$, $\alpha_{k}=\hat{\alpha}\cup{(\gamma,M-p_{i})\cup{(\delta_{2},1)}}$ and $u_{1}\in \mathcal{M}^{J}((\delta_{1},1)\cup{(\gamma,p_{i})},(\delta_{2},1))$. 

\item[\bf (A').]Let $E(\alpha_{k})=M$ and $q_{i}:=\mathrm{max}(S_{\theta}\cap{\{1,\,\,2,...,\,\,M\}})$. Then there are two  negative hyperbolic orbits $\delta_{1}$, $\delta_{2}$ and an admissible orbit set $\hat{\alpha}$ consisting of negative hyperbolic orbits  such that $\alpha_{k+1}=\hat{\alpha}\cup{(\gamma,M-q_{i})\cup{(\delta_{1},1)}}$, $\alpha_{k}=\hat{\alpha}\cup{(\gamma,M)\cup{(\delta_{2},1)}}$ and $u_{1}\in \mathcal{M}^{J}((\delta_{1},1),(\delta_{2},1)\cup{(\gamma,q_{i})})$. 

\item[\bf (B).]Let $E(\alpha_{k+1})=M$ and $p_{i}:=\mathrm{max}(S_{-\theta}\cap{\{1,\,\,2,...,\,\,M\}})$. There are two  negative hyperbolic orbits $\delta_{1}$, $\delta_{2}$ and an admissible orbit set $\hat{\alpha}$ consisting of negative hyperbolic orbits  such that $\alpha_{k+1}=\hat{\alpha}\cup{(\gamma,M)}$,  $\alpha_{k}=\hat{\alpha}\cup{(\gamma,M-p_{i})\cup{(\delta_{1},1)\cup{(\delta_{2},1)}}}$ and $u_{1}\in \mathcal{M}^{J}((\gamma,p_{i}),(\delta_{1},1)\cup{(\delta_{2},1)})$. 

\item[\bf (B').]Let $E(\alpha_{k})=M$ and $q_{i}:=\mathrm{max}(S_{\theta}\cap{\{1,\,\,2,...,\,\,M\}})$. There are two  negative hyperbolic orbits $\delta_{1}$, $\delta_{2}$ and an admissible orbit set $\hat{\alpha}$ consisting of negative hyperbolic orbits such that $\alpha_{k+1}=\hat{\alpha}\cup{(\gamma,M-q_{i})\cup{(\delta_{1},1)}\cup{(\delta_{2},1)}}$, $\alpha_{k}=\hat{\alpha}\cup{(\gamma,M)}$ and $u_{1}\in \mathcal{M}^{J}((\delta_{1},1)\cup{(\delta_{2},1)},(\gamma,q_{i}))$. 

\end{lem}

\begin{proof}[\bf Proof of Lemma \ref{top}]

Since $A(\alpha_{k+1})-A(\alpha_{k})<\epsilon$, $u_{1}$ has at least one negative end. Moreover, at least one end of $u_{1}$ have to be asymptotic to some negative hyperbolic orbit because the fact causes a contradiction that if all ends are asymptotic to $\gamma$,  the value $A(\alpha_{k+1})-A(\alpha_{k})$ have to be lager than or equal to $A(\gamma)$. From the assumptions $E(\alpha_{k+1})$, $E(\alpha_{k})>p_{1}$, $q_{1}>0$ and the partition conditions of the ends, we have Lemma \ref{top}.
\end{proof}

\subsection{Restriction of $J_{0}$ combinations}

In the previous subsection, we decided the topological type of $J$-holomorphic curves counted by $U\langle \alpha_{k+1} \rangle =\langle \alpha_{k} \rangle$. To prove Proposition \ref{nagai}, we have to decide the approximate relations in the actions of the orbits in $\alpha_{k+1}$ and $\alpha_{k}$.

From now on, because symmetry allows the same argument, , we only consider the cases (A) and (B).

The next claim is almost the same as Claim \ref{index2to4}.
\begin{cla}\label{index22}
In the case of (A) in Lemma \ref{top},
\item[\bf (A).]
For any $p_{i}\leq N<p_{i+1}$ ,
\begin{equation}
    I(\hat{\alpha}\cup{(\gamma,N)\cup{(\delta_{1},1)}},\hat{\alpha}\cup{(\gamma,N-p_{i})}\cup{(\delta_{2},1)})=2.
\end{equation}
Moreover
\begin{equation}
    I(\hat{\alpha}\cup{(\gamma,p_{i+1})}\cup{(\delta_{1},1)},\hat{\alpha}\cup{(\gamma,p_{i+1}-p_{i})}\cup{(\delta_{2},1)})=4.
\end{equation}
And for any $p_{i}<N\leq p_{i+1}$,
\begin{equation}
     J_{0}(\hat{\alpha}\cup{(\gamma,N)}\cup{(\delta_{1},1)},\hat{\alpha}\cup{(\gamma,N-p_{i})}\cup{(\delta_{2},1)})=1.
\end{equation}

In the case of (B) in Lemma \ref{top},
\item[\bf (B).]
For any $p_{i}\leq N<p_{i+1}$ ,
\begin{equation}
    I(\hat{\alpha}\cup{(\gamma,N)},\hat{\alpha}\cup{(\gamma,N-p_{i})}\cup{(\delta_{1},1)}\cup{(\delta_{2},1)})=2.
\end{equation}
Moreover
\begin{equation}
    I(\hat{\alpha}\cup{(\gamma,p_{i+1})},\hat{\alpha}\cup{(\gamma,p_{i+1}-p_{i})}\cup{(\delta_{1},1)}\cup{(\delta_{2},1)})=4.
\end{equation}
And for any $p_{i}<N\leq p_{i+1}$,
\begin{equation}
     J_{0}(\hat{\alpha}\cup{(\gamma,N)},\hat{\alpha}\cup{(\gamma,N-p_{i})}\cup{(\delta_{1},1)}\cup{(\delta_{2},1)})=1.
\end{equation}

\end{cla}

\begin{proof}[\bf Proof of Claim \ref{index22}]

We can prove this in the same way as Claim \ref{mi1} and Claim \ref{index2to4}.
\end{proof}

\begin{cla}\label{basis}
In the case of (A) in Lemma \ref{top},
\item[\bf (A).]
There is an admissible orbit set $\zeta$ with $U\langle \hat{\alpha}\cup{(\delta_{1},1)}\cup{(\gamma,p_{i+1})} \rangle=\langle \zeta \rangle$ and $U\langle \zeta \rangle= \langle \hat{\alpha}\cup{(\gamma,p_{i+1}-p_{i})}\cup{(\delta_{2},1)} \rangle$. Moreover, $E(\zeta)=0$.

In the case of (B) in Lemma \ref{top},
\item[\bf (B).]
There is an admissible orbit set $\zeta$ with $U\langle \hat{\alpha}\cup{(\gamma,p_{i+1})} \rangle=\langle \zeta \rangle$ and $U\langle \zeta \rangle= \langle \hat{\alpha}\cup{(\gamma,p_{i+1}-p_{i})}\cup{{(\delta_{1},1)}}\cup{(\delta_{2},1)} \rangle$. Moreover, $E(\zeta)=0$.
\end{cla}

\begin{proof}[\bf Proof of Claim \ref{basis}]

In the same way as Claim \ref{el0} and just before that.
\end{proof}

Since $J_{0}\geq-1$,  there are five possibilities,
$(J_{0}(\hat{\alpha}\cup{(\delta_{1},1)}\cup{(\gamma,p_{i+1})},\zeta),J_{0}(\zeta,\hat{\alpha}\cup{(\gamma,p_{i+1}-p_{i})}\cup{(\delta_{2},1)}))(\mathrm{resp.}\,\,(J_{0}(\hat{\alpha}\cup{(\gamma,p_{i+1})},\zeta),J_{0}(\zeta,\hat{\alpha}\cup{(\gamma,p_{i+1}-p_{i})}\cup{(\delta_{1},1)}\cup{(\delta_{2},1)}))) =(3,-1)$, $(2,0)$, $(1,1)$, $(0,2)$ or $(-1,3)$. But except for $(1,1)$, the behaviors of $J$-holomorphic curves counted by the $U$-map cause contradictions. 

Now, we use the rest of Section
6 to prove the next lemma.

\begin{lem}\label{exc}
Under the notation in Claim \ref{basis}, $J$-holomorphic curves counted by the $U$-map cause  contradictions except for $(J_{0}(\hat{\alpha}\cup{(\delta_{1},1)}\cup{(\gamma,p_{i+1})},\zeta),J_{0}(\zeta,\hat{\alpha}\cup{(\gamma,p_{i+1}-p_{i})}\cup{(\delta_{2},1)}))=(J_{0}(\hat{\alpha}\cup{(\gamma,p_{i+1})},\zeta),J_{0}(\zeta,\hat{\alpha}\cup{(\gamma,p_{i+1}-p_{i})}\cup{(\delta_{1},1)}\cup{(\delta_{2},1)}))=(1,1)$.

\end{lem}

\begin{proof}[\bf Proof of Lemma \ref{exc}]

At first, we can easily exclude the cases $(J_{0}(\hat{\alpha}\cup{(\delta_{1},1)}\cup{(\gamma,p_{i+1})},\zeta),J_{0}(\zeta,\hat{\alpha}\cup{(\gamma,p_{i+1}-p_{i})}\cup{(\delta_{2},1)}))=(3,-1)$, $(-1,3)$ because of their smallness of the difference of their actions.

Next,  we will consider the cases (A), (B) respectively.

\item[\bf Case] (A).

At first, we consider the splitting behaviors of $J$-holomorphic curve counted by $U\langle \alpha_{k+1} \rangle=\langle \alpha_{k} \rangle$ as $z \to \eta$ for some fixed $\eta\in \hat{\alpha}$. Then there are three possibilities of splitting of holomorphic curves and also we have three possibilities of approximate  relations as follows.
\item[$(\mathfrak{a}_{1})$.] $|A(\delta_{1})-2A(\eta)|<\epsilon$
\item[$(\mathfrak{a}_{2})$.] $|A(\delta_{2})-2A(\eta)|<\epsilon$
\item[$(\mathfrak{a}_{3})$.] $|p_{i}R-2A(\eta)|<\epsilon$.

Moreover, we always have
\item[$(\clubsuit)$]$A(\alpha_{k+1})-A(\alpha_{k})=|A(\delta_{1})+p_{i}R-A(\delta_{2})|<\epsilon$.

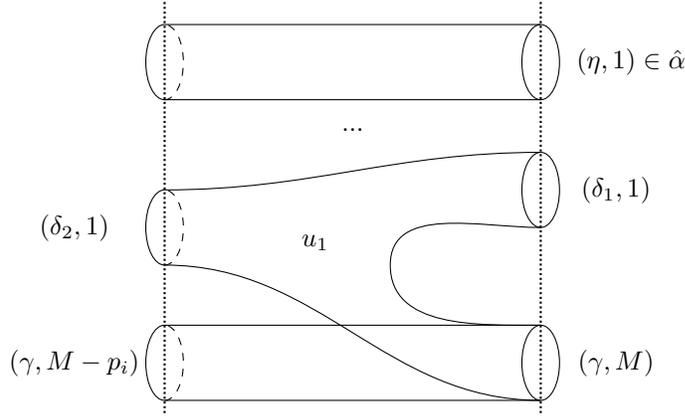
\begin{tikzpicture}

\draw (11,-1.8) arc [start angle=90,end angle=270,x radius=0.25,y radius=0.5];
 \draw (11,-2.8) arc [start angle=270,end angle=450,x radius=0.25,y radius=0.5];
 
 \draw (11,-1.8)--(6,-1.8) ;
  \draw (11,-2.8)--(6,-2.8) ;
  \draw (6,-1.8) arc [start angle=90,end angle=270,x radius=0.25,y radius=0.5];
 \draw (6,-2.8)[dashed] arc [start angle=270,end angle=450,x radius=0.25,y radius=0.5];
 
\draw (12.2,-2.3) node{$(\eta,1)\in \hat{\alpha}$};
 
 \draw (8.5,-3.2) node{...};

 \draw (6,-4) arc [start angle=90,end angle=270,x radius=0.25,y radius=0.5];
 \draw (6,-5)[dashed] arc [start angle=270,end angle=450,x radius=0.25,y radius=0.5];
 \draw (11,-3.5) arc [start angle=90,end angle=270,x radius=0.25,y radius=0.5];
 \draw (11,-4.5) arc [start angle=270,end angle=450,x radius=0.25,y radius=0.5];
 
 \draw (12,-4) node{$(\delta_{1},1)$ } ;
 
 \draw (11,-5.8) arc [start angle=90,end angle=270,x radius=0.25,y radius=0.5];
 \draw (11,-6.8) arc [start angle=270,end angle=450,x radius=0.25,y radius=0.5];
 
 \draw (12,-6.3) node{$(\gamma,M)$} ;
 
 \draw (4.8,-6.3) node{$(\gamma,M-p_{i})$} ;
 
 \draw (6,-4) to [out=0,in=180] (11,-3.5);
 \draw (6,-5) to [out=0,in=180] (11,-6.8);
 \draw (11,-5.8) to [out=180,in=270] (9,-5);
 \draw (9,-5) to [out=90,in=180] (11,-4.5);

 \draw (11,-5.8)--(6,-5.8) ;
  \draw (11,-6.8)--(6,-6.8) ;
  \draw (6,-5.8) arc [start angle=90,end angle=270,x radius=0.25,y radius=0.5];
 \draw (6,-6.8)[dashed] arc [start angle=270,end angle=450,x radius=0.25,y radius=0.5];
 \draw (4.8,-4.5) node{$(\delta_{2},1)$} ;

 \draw[densely dotted,  thick] (6,-1.5)--(6,-7);
  \draw[densely dotted,  thick] (11,-1.5)--(11,-7) ;
  
  \draw (6,-7.5) node{$\alpha_{k}=\hat{\alpha}\cup{(\gamma,M-p_{i})\cup{(\delta_{2},1)}}$};
  \draw (11,-7.5) node{$\alpha_{k+1}=\hat{\alpha}\cup{(\gamma,M)\cup{(\delta_{1},1)}}$};
  
  \draw (8,-4.7) node{$u_{1}$};

\end{tikzpicture}

\item [$\rm(\hspace{.18em}i\hspace{.18em})$.] If $(J_{0}(\hat{\alpha}\cup{(\delta_{1},1)}\cup{(\gamma,p_{i+1})},\zeta),J_{0}(\zeta,\hat{\alpha}\cup{(\gamma,p_{i+1}-p_{i})}\cup{(\delta_{2},1)}))=(2,0)$

Note that genus of each $J$-holomorphic curve counted by $U\langle \zeta \rangle=\langle \hat{\alpha}\cup{(\gamma,p_{i+1}-p_{i})}\cup{(\delta_{2},1)}) \rangle$ are 0. Moreover, each curve has both negative end covering at $\gamma$ with multiplicity $p_{i+1}-p_{i}$  and positive end covering at negative hyperbolic orbit $(\delta',1)$ which is not equivalent to $\delta_{2}$ because  $A(\zeta)-A(\hat{\alpha}\cup{(\gamma,p_{i+1}-p_{i})}\cup{(\delta_{2},1)})<\epsilon$ and so in $\mathcal{M}^{J}((\delta',1),(\gamma,p_{i+1}-p_{i}))$.

 Then from the splitting behavior as $z \to \delta_{2}$, we have $|2A(\delta_{2})-(p_{i+1}-p_{i})R|<\epsilon$ and also as $z\to \eta\in  \hat{\alpha}$, we have $|2A(\eta)-(p_{i+1}-p_{i})R|<\epsilon$. These two relations indicate that $|A(\eta)-A(\delta_{2})|<\epsilon$.  Since $|A(\delta_{1})+p_{i}R-A(\delta_{2})|<\epsilon$, we have $A(\delta_{2})>p_{i}R$, $ A(\delta_{1})$ and hence $A(\eta)>p_{i}R$, $ A(\delta_{1})$.  These relations contradict $(\mathfrak{a}_{1})$, $(\mathfrak{a}_{2})$ and $(\mathfrak{a}_{3})$ in any case. Therefore, this case can not occur.

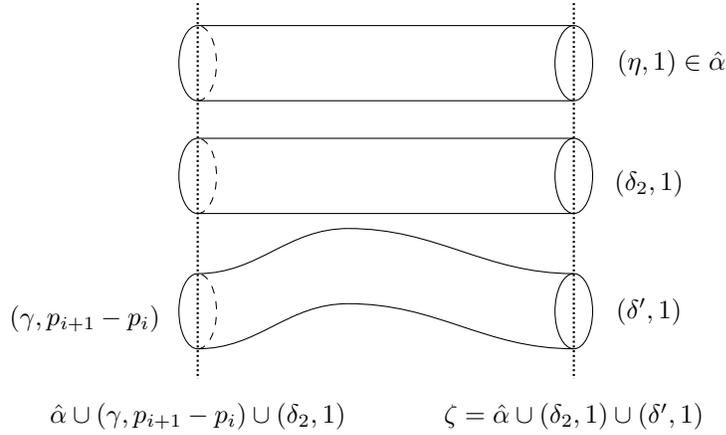
\begin{tikzpicture}

  \draw (6,-2.8) arc [start angle=90,end angle=270,x radius=0.25,y radius=0.5];
 \draw (6,-3.8) arc [start angle=270,end angle=450,x radius=0.25,y radius=0.5];

 \draw (6,-4.3) arc [start angle=90,end angle=270,x radius=0.25,y radius=0.5];
 \draw (6,-5.3) arc [start angle=270,end angle=450,x radius=0.25,y radius=0.5];
 
  \draw (6,-6.1) arc [start angle=90,end angle=270,x radius=0.25,y radius=0.5];
 \draw (6,-7.1) arc [start angle=270,end angle=450,x radius=0.25,y radius=0.5];

 \draw[densely dotted,  thick] (6,-2.5)--(6,-7.5);

 \draw (1,-2.8)--(6,-2.8) ;
  \draw (1,-3.8)--(6,-3.8) ;
  
  \draw (1,-4.3)--(6,-4.3) ;
  \draw (1,-5.3)--(6,-5.3) ;
  
  \draw (1,-2.8) arc [start angle=90,end angle=270,x radius=0.25,y radius=0.5];
 \draw (1,-3.8)[dashed] arc [start angle=270,end angle=450,x radius=0.25,y radius=0.5];

 \draw (1,-4.3) arc [start angle=90,end angle=270,x radius=0.25,y radius=0.5];
 \draw (1,-5.3)[dashed] arc [start angle=270,end angle=450,x radius=0.25,y radius=0.5];
 
  \draw (1,-6.1) arc [start angle=90,end angle=270,x radius=0.25,y radius=0.5];
 \draw (1,-7.1)[dashed] arc [start angle=270,end angle=450,x radius=0.25,y radius=0.5];

\draw (1,-6.1) to [out=0,in=180] (3,-5.5);
 \draw (1,-7.1) to [out=0,in=180] (3,-6.5);
 \draw (3,-5.5) to [out=0,in=180] (6,-6.1);
 \draw (3,-6.5) to [out=0,in=180] (6,-7.1);

\draw[densely dotted,  thick] (1,-2.5)--(1,-7.5) ;
  
\draw (1,-8) node{$\hat{\alpha}\cup{(\gamma,p_{i+1}-p_{i})}\cup{(\delta_{2},1)}$};
\draw (6,-8) node{$\zeta=\hat{\alpha}\cup{(\delta_{2},1)}\cup{(\delta',1)}$};
  
\draw (-0.5,-6.7) node{$(\gamma,p_{i+1}-p_{i})$};

\draw (7,-4.9) node{$(\delta_{2},1)$};

\draw (7.3,-3.3) node{$(\eta,1)\in \hat{\alpha}$};

\draw (7,-6.6) node{$(\delta',1)$};

\end{tikzpicture}

\item[$\rm(\hspace{.08em}ii\hspace{.08em})$.] If $(J_{0}(\hat{\alpha}\cup{(\delta_{1},1)}\cup{(\gamma,p_{i+1})},\zeta),J_{0}(\zeta,\hat{\alpha}\cup{(\gamma,p_{i+1}-p_{i})}\cup{(\delta_{2},1)}))=(0,2)$

In the same way as above, $J$-holomorphic curves counted by $U\langle\hat{\alpha}\cup{(\delta_{1},1)}\cup{(\gamma,p_{i+1})}  \rangle=\langle \zeta \rangle$ are of genus 0 and have both negative end covering at $\gamma$ with multiplicity $p_{i+1}$  and positive end covering one negative hyperbolic orbit $(\delta',1)$ which is not equivalent to $\delta_{1}$ and so in $\mathcal{M}^{J}((\gamma,p_{i+1}),(\delta',1))$.

 Then from the splitting behavior as $z \to \delta_{1}$, we have $|2A(\delta_{1})-p_{i+1}R|<\epsilon$ and also as $z\to \eta$, we have $|2A(\eta)-p_{i+1}R|<\epsilon$.  Since $(\clubsuit)$.$|A(\delta_{1})+p_{i}R-A(\delta_{2})|<\epsilon$, we have $|(\frac{1}{2}p_{i+1}+p_{i})R-A(\delta_{2})|<\frac{3}{2}\epsilon$.
  These relations contradict $(\mathfrak{a}_{1})$, $(\mathfrak{a}_{2})$ and $(\mathfrak{a}_{3})$ in any case. Here, we use Claim \ref{fre} implicitly.

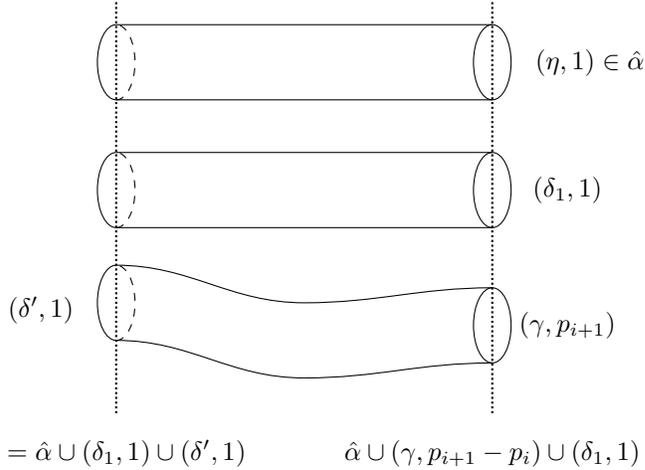
\begin{tikzpicture}

\draw (11,-1.8) arc [start angle=90,end angle=270,x radius=0.25,y radius=0.5];
 \draw (11,-2.8)[dashed] arc [start angle=270,end angle=450,x radius=0.25,y radius=0.5];

 \draw (11,-3.5) arc [start angle=90,end angle=270,x radius=0.25,y radius=0.5];
 \draw (11,-4.5)[dashed] arc [start angle=270,end angle=450,x radius=0.25,y radius=0.5];

 \draw (11,-5) arc [start angle=90,end angle=270,x radius=0.25,y radius=0.5];
 \draw (11,-6)[dashed] arc [start angle=270,end angle=450,x radius=0.25,y radius=0.5];

  \draw (11,-1.8)--(16,-1.8) ;
  \draw (11,-2.8)--(16,-2.8) ;
  
   \draw (11,-3.5)--(16,-3.5) ;
  \draw (11,-4.5)--(16,-4.5) ;
 
  \draw (16,-3.5) arc [start angle=90,end angle=270,x radius=0.25,y radius=0.5];
 \draw (16,-4.5) arc [start angle=270,end angle=450,x radius=0.25,y radius=0.5];

 \draw (16,-5.3) arc [start angle=90,end angle=270,x radius=0.25,y radius=0.5];
 \draw (16,-6.3) arc [start angle=270,end angle=450,x radius=0.25,y radius=0.5];
 
\draw (16,-1.8) arc [start angle=90,end angle=270,x radius=0.25,y radius=0.5];
 \draw (16,-2.8) arc [start angle=270,end angle=450,x radius=0.25,y radius=0.5];

  \draw[densely dotted,  thick] (11,-1.5)--(11,-7) ;
  
  \draw[densely dotted,  thick] (16,-1.5)--(16,-7) ;
  
  \draw (16,-7.5) node{$\hat{\alpha}\cup{(\gamma,p_{i+1}-p_{i})\cup{(\delta_{1},1)}}$};
  
   \draw (11,-7.5) node{$\zeta=\hat{\alpha}\cup{(\delta_{1},1)}\cup{(\delta',1)}$};

 \draw (11,-5) to [out=0,in=180] (13.5,-5.5);
 \draw (11,-6) to [out=0,in=180] (13.5,-6.5);
 \draw (13.5,-5.5) to [out=0,in=180] (16,-5.3);
 \draw (13.5,-6.5) to [out=0,in=180] (16,-6.3);

  \draw (17,-4) node{$(\delta_{1},1)$};
  
  \draw (17,-5.8) node{$(\gamma,p_{i+1})$};
  
   \draw (10,-5.6) node{$(\delta',1)$};
  
  \draw (17.3,-2.3) node{$(\eta,1)\in \hat{\alpha}$};

\end{tikzpicture}

By the arguments so far, we can see  that only the case $(J_{0}(\hat{\alpha}\cup{(\delta_{1},1)}\cup{(\gamma,p_{i+1})},\zeta),J_{0}(\zeta,\hat{\alpha}\cup{(\gamma,p_{i+1}-p_{i})}\cup{(\delta_{2},1)}))=(1,1)$ may occur.

\item[\bf Case (B).]

At first, we consider the splitting behaviors of $J$-holomorphic curves counted by $U\langle \alpha_{k+1} \rangle=\langle \alpha_{k} \rangle$ as $z \to \eta$ for some $\eta\in \hat{\alpha}$. Then there are three possibilities of splitting of holomorphic curve and also we have three possibilities of approximate relations as follows.
\item[($\mathfrak{b}_{1})$.] $|A(\delta_{1})-2A(\eta)|<\epsilon$
\item[$(\mathfrak{b}_{2})$.] $|A(\delta_{2})-2A(\eta)|<\epsilon$
\item[$(\mathfrak{b}_{3})$.] $|p_{i}R-2A(\eta)|<\epsilon$.

Moreover, we always have
\item[$(\spadesuit)$.]$A(\alpha_{k+1})-A(\alpha_{k})=|p_{i}R-(A(\delta_{1})+A(\delta_{2}))|<\epsilon$.

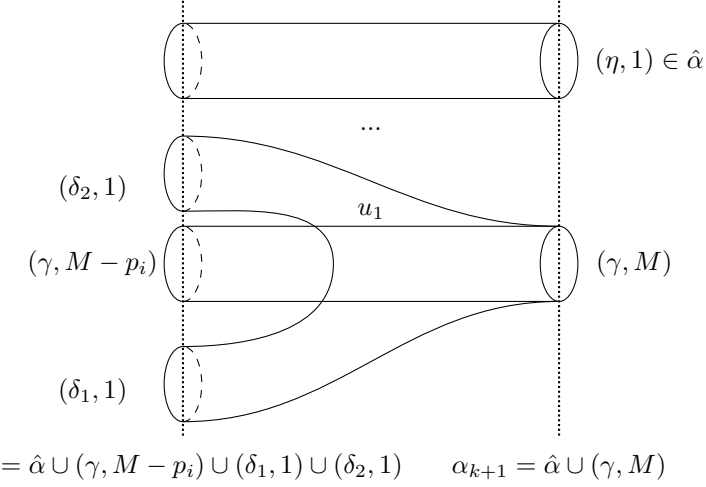
\begin{tikzpicture}
\draw (11,-1.8) arc [start angle=90,end angle=270,x radius=0.25,y radius=0.5];
 \draw (11,-2.8) arc [start angle=270,end angle=450,x radius=0.25,y radius=0.5];
 
 \draw (11,-1.8)--(6,-1.8) ;
  \draw (11,-2.8)--(6,-2.8) ;
  \draw (6,-1.8) arc [start angle=90,end angle=270,x radius=0.25,y radius=0.5];
 \draw (6,-2.8)[dashed] arc [start angle=270,end angle=450,x radius=0.25,y radius=0.5];
 
\draw (12.2,-2.3) node{$(\eta,1)\in \hat{\alpha}$};

 \draw (8.5,-3.2) node{...};

 \draw (6,-3.3) arc [start angle=90,end angle=270,x radius=0.25,y radius=0.5];
 \draw (6,-4.3)[dashed] arc [start angle=270,end angle=450,x radius=0.25,y radius=0.5];

 \draw (11,-4.5) arc [start angle=90,end angle=270,x radius=0.25,y radius=0.5];
 \draw (11,-5.5) arc [start angle=270,end angle=450,x radius=0.25,y radius=0.5];
  \draw (6,-4.5) arc [start angle=90,end angle=270,x radius=0.25,y radius=0.5];
 \draw (6,-5.5)[dashed] arc [start angle=270,end angle=450,x radius=0.25,y radius=0.5];
 
 \draw (6,-4.5)--(11,-4.5);
 
  \draw (6,-5.5)--(11,-5.5);
 
 \draw (12,-5) node{$(\gamma,M)$ } ;
 
 \draw (4.8,-5) node{$(\gamma,M-p_{i})$} ;
 
 \draw (11,-5.5) to [out=180,in=0] (6,-7.1);
 \draw (6,-6.1) to [out=0,in=270] (8,-5);
 \draw (8,-5) to [out=90,in=0] (6,-4.3);
 \draw (11,-4.5) to [out=180,in=0] (6,-3.3);

  \draw (6,-6.1) arc [start angle=90,end angle=270,x radius=0.25,y radius=0.5];
 \draw (6,-7.1)[dashed] arc [start angle=270,end angle=450,x radius=0.25,y radius=0.5];
 
 \draw (4.8,-4) node{$(\delta_{2},1)$} ;

  \draw (4.8,-6.7) node{$(\delta_{1},1)$} ;

 \draw[densely dotted,  thick] (6,-1.5)--(6,-7.3);
  \draw[densely dotted,  thick] (11,-1.5)--(11,-7.3) ;
  
  \draw (6,-7.7) node{$\alpha_{k}=\hat{\alpha}\cup{(\gamma,M-p_{i})\cup{(\delta_{1},1)\cup{(\delta_{2},1)}}}$};
  \draw (11,-7.7) node{$\alpha_{k+1}=\hat{\alpha}\cup{(\gamma,M)}$};
  
  \draw (8.5,-4.3) node{$u_{1}$};

\end{tikzpicture}

\item[$\rm(\hspace{.18em}i\hspace{.18em})$.] If $(J_{0}(\hat{\alpha}\cup{(\gamma,p_{i+1})},\zeta),J_{0}(\zeta,\hat{\alpha}\cup{(\gamma,p_{i+1}-p_{i})}\cup{(\delta_{1},1)}\cup{(\delta_{2},1)}))=(0,2)$

In the same as before, $J$-holomorphic curves counted by $U\langle\hat{\alpha}\cup{(\gamma,p_{i+1})}  \rangle=\langle \zeta \rangle$ are of genus 0 and in $\mathcal{M}^{J}((\gamma,p_{i+1}),(\delta',1))$ for some negative hyperbolic orbit $\delta'$ with $\zeta = \hat{\alpha}\cup{(\delta',1)}$.

Then from the splitting behavior as $z \to \eta \in \hat{\alpha}$, we have $|2A(\eta)-p_{i+1}R|<\epsilon$. Since ($\spadesuit$). $|p_{i}R-(A(\delta_{1})+A(\delta_{2}))|<\epsilon$, we have $p_{i}R>A(\delta_{1})$, $A(\delta_{2})$. So we can easily see that  the relations contradict $(\mathfrak{b}_{1})$, $(\mathfrak{b}_{2})$ and $(\mathfrak{b}_{3})$ in any case.  Therefore this case can not occur.

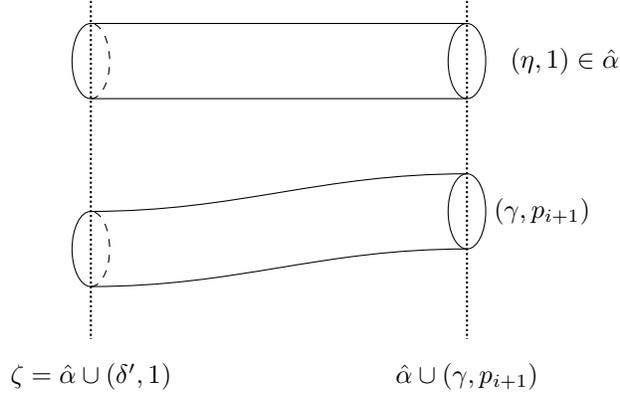
\begin{tikzpicture}

\draw (11,-1.3) arc [start angle=90,end angle=270,x radius=0.25,y radius=0.5];
 \draw (11,-2.3)[dashed] arc [start angle=270,end angle=450,x radius=0.25,y radius=0.5];
 
 \draw (11,-1.3)--(16,-1.3) ;
  \draw (11,-2.3)--(16,-2.3) ;
  
  \draw (16,-1.3) arc [start angle=90,end angle=270,x radius=0.25,y radius=0.5];
 \draw (16,-2.3) arc [start angle=270,end angle=450,x radius=0.25,y radius=0.5];

   \draw (16,-3.3) arc [start angle=90,end angle=270,x radius=0.25,y radius=0.5];
\draw (16,-4.3) arc [start angle=270,end angle=450,x radius=0.25,y radius=0.5];

 \draw (11,-3.8) arc [start angle=90,end angle=270,x radius=0.25,y radius=0.5];
 \draw (11,-4.8)[dashed] arc [start angle=270,end angle=450,x radius=0.25,y radius=0.5];

 \draw (11,-3.8) to [out=0,in=180] (16,-3.3);
 \draw (11,-4.8) to [out=0,in=180] (16,-4.3);

  \draw[densely dotted,  thick] (11,-1)--(11,-5.5) ;
   \draw[densely dotted,  thick] (16,-1)--(16,-5.5) ;

  \draw (11,-6) node{$\zeta=\hat{\alpha}\cup{(\delta',1)}$};
  
  \draw (16,-6) node{$\hat{\alpha}\cup{(\gamma,p_{i+1})}$};

    \draw (17,-3.8) node{$(\gamma,p_{i+1})$};
    
    \draw (17.3,-1.8) node{$(\eta,1)\in \hat{\alpha} $};

\end{tikzpicture}

\item[$\rm(\hspace{.08em}ii\hspace{.08em})$.] If $(J_{0}(\hat{\alpha}\cup{(\gamma,p_{i+1})},\zeta),J_{0}(\zeta,\hat{\alpha}\cup{(\gamma,p_{i+1}-p_{i})}\cup{(\delta_{1},1)}\cup{(\delta_{2},1)}))=(2,0)$

In the same as before, $J$-holomorphic curves counted by $U\langle \zeta  \rangle=\langle \hat{\alpha}\cup{(\gamma,p_{i+1}-p_{i})}\cup{(\delta_{1},1)}\cup{(\delta_{2},1)}) \rangle$ are of genus 0 and in $\mathcal{M}^{J}((\delta',1),(\gamma,p_{i+1}-p_{i}))$ for some negative hyperbolic orbit $\delta'$ with $\zeta = \hat{\alpha}\cup{(\delta',1)}\cup{(\delta_{1},1)}\cup{(\delta_{2},1)}$.

Then from the splitting behavior as $z \to \delta_{1}$, $\delta_{2}$, we have $|2A(\delta_{1})-(p_{i+1}-p_{i})R|<\epsilon$ and $|2A(\delta_{2})-(p_{i+1}-p_{i})R|<\epsilon$.  From these relations, we also have $|(A(\delta_{1})+A(\delta_{2}))-(p_{i+1}-p_{i})R|<\epsilon$. By  combining with ($\spadesuit$).$|p_{i}R-(A(\delta_{1})+A(\delta_{2}))|<\epsilon$, we have $|p_{i}R-(p_{i+1}-p_{i})R|<\epsilon$ and so $p_{i}=p_{i+1}-p_{i}$. This contradicts Claim \ref{fre}.

\begin{tikzpicture}

  \draw (1,-1.3) arc [start angle=90,end angle=270,x radius=0.25,y radius=0.5];
 \draw (1,-2.3)[dashed] arc [start angle=270,end angle=450,x radius=0.25,y radius=0.5];
 
 \draw (1,-1.3)--(6,-1.3) ;
  \draw (1,-2.3)--(6,-2.3) ;

  \draw (6,-1.3) arc [start angle=90,end angle=270,x radius=0.25,y radius=0.5];
 \draw (6,-2.3) arc [start angle=270,end angle=450,x radius=0.25,y radius=0.5];

   \draw (6,-2.6) arc [start angle=90,end angle=270,x radius=0.25,y radius=0.5];
 \draw (6,-3.6) arc [start angle=270,end angle=450,x radius=0.25,y radius=0.5];

 \draw (6,-4.5) arc [start angle=90,end angle=270,x radius=0.25,y radius=0.5];
 \draw (6,-5.5) arc [start angle=270,end angle=450,x radius=0.25,y radius=0.5];
 
  \draw (6,-6.1) arc [start angle=90,end angle=270,x radius=0.25,y radius=0.5];
 \draw (6,-7.1) arc [start angle=270,end angle=450,x radius=0.25,y radius=0.5];
 
 \draw (1,-2.6)--(6,-2.6) ;
  \draw (1,-3.6)--(6,-3.6) ;

   \draw (1,-6.1)--(6,-6.1) ;
  \draw (1,-7.1)--(6,-7.1) ;
  
  \draw (1,-4.5) to [out=0,in=180] (3,-4.1);
 \draw (3,-4.1) to [out=0,in=180] (6,-4.5);
 
 \draw (1,-5.5) to [out=0,in=180] (3,-5.1);
 \draw (3,-5.1) to [out=0,in=180] (6,-5.5);

  \draw (1,-2.6) arc [start angle=90,end angle=270,x radius=0.25,y radius=0.5];
 \draw (1,-3.6)[dashed] arc [start angle=270,end angle=450,x radius=0.25,y radius=0.5];

 \draw (1,-4.5) arc [start angle=90,end angle=270,x radius=0.25,y radius=0.5];
 \draw (1,-5.5)[dashed] arc [start angle=270,end angle=450,x radius=0.25,y radius=0.5];
 
  \draw (1,-6.1) arc [start angle=90,end angle=270,x radius=0.25,y radius=0.5];
 \draw (1,-7.1) [dashed] arc [start angle=270,end angle=450,x radius=0.25,y radius=0.5];

 \draw[densely dotted,  thick] (6,-1)--(6,-7.5);

  \draw (0.7,-8) node{$\hat{\alpha}\cup{(\gamma,p_{i+1}-p_{i})}\cup{(\delta_{1},1)}\cup{(\delta_{2},1)}$};
  \draw (6.3,-8) node{$\zeta =\hat{\alpha}\cup{(\delta_{1},1)}\cup{(\delta_{2},1)}\cup{(\delta',1)}$};
 
    \draw (0,-3.2) node{$(\delta_{1},1)$};
     \draw (0,-6.7) node{$(\delta_{2},1)$};
      \draw (-0.5,-5) node{$(\gamma,p_{i+1}-p_{i})$};
      
      \draw (7,-5) node{$(\delta',1)$};
      
        \draw (7.2,-1.8) node{$(\eta,1) \in \hat{\alpha}$};

   \draw[densely dotted,  thick] (1,-1)--(1,-7.5) ;

  Summarizing the above arguments, we complete the proof of Lemma \ref{exc}.

\end{tikzpicture}

By the arguments so far, we can see  that only the case $(J_{0}(\hat{\alpha}\cup{(\gamma,p_{i+1})},\zeta),J_{0}(\zeta,\hat{\alpha}\cup{(\gamma,p_{i+1}-p_{i})}\cup{(\delta_{1},1)}\cup{(\delta_{2},1)}))=(1,1)$ may occur.
\end{proof}

\section{Calculations of the approximate values of the actions of the orbits}

In this section, we compute the approximate values of the  actions of the orbits and complete the proof of Proposition \ref{nagai} under the result obtained so far.

In the same way as before, the splitting behaviors of  $J$-holomorphic curves counted by the $U$-map play an important rule.

In this section, we consider $(A)$ and  $(B)$ in Lemma \ref{top} respectively.

\subsection{Type (A)}

Since $E(\zeta)=0$, $p_{i+1}\in S_{-\theta}$ and $p_{i+1}-p_{i}\in S_{\theta}$, the topological types of $J$-holomoprhic curves counted by $U\langle\hat{\alpha}\cup{(\delta_{1},1)}\cup{(\gamma,p_{i+1})}  \rangle=\langle \zeta \rangle$ and  $U\langle \zeta \rangle=\langle \hat{\alpha}\cup{(\gamma,p_{i+1}-p_{i})}\cup{(\delta_{2},1)}) \rangle$ are both $(g,k,l)=(0,3,0)$. 

By considering which ends of $J$-holomorphic curves counted by $U\langle\hat{\alpha}\cup{(\delta_{1},1)}\cup{(\gamma,p_{i+1})}  \rangle=\langle \zeta \rangle$ correspond to ones of  $J$-holomorphic curves counted by  $U\langle \zeta \rangle=\langle \hat{\alpha}\cup{(\gamma,p_{i+1}-p_{i})}\cup{(\delta_{2},1)}) \rangle$ respectively, we can see that there are three pairs of $\zeta$, $J$-holomorphic curves counted by $U\langle\hat{\alpha}\cup{(\delta_{1},1)}\cup{(\gamma,p_{i+1})}  \rangle=\langle \zeta \rangle$ and   $U\langle \zeta \rangle=\langle \hat{\alpha}\cup{(\gamma,p_{i+1}-p_{i})}\cup{(\delta_{2},1)}) \rangle$  as follows.

\begin{enumerate}
    \item [\bf Type $(A_{1})$]
    
    There is a negative hyperbolic orbit $\delta'$ with $\zeta=\hat{\alpha}\cup{
    (\delta',1)}$. Moreover any $J$-holomorphic curves counted by $U\langle\hat{\alpha}\cup{(\delta_{1},1)}\cup{(\gamma,p_{i+1})}  \rangle=\langle \zeta \rangle$ and  $U\langle \zeta \rangle=\langle \hat{\alpha}\cup{(\gamma,p_{i+1}-p_{i})}\cup{(\delta_{2},1)}) \rangle$ are in $\mathcal{M}^{J}((\delta_{1},1)\cup{(\gamma,p_{i+1})},(\delta',1))$ and $\mathcal{M}^{J}((\delta',1),(\gamma,p_{i+1}-p_{i})\cup{(\delta_{2},1)})$ respectively.
    
    \item [\bf Type $(A_{2})$]
    
    There is a negative hyperbolic orbit $\delta'$ with $\zeta=\hat{\alpha}\cup{
    (\delta',1)}\cup{(\delta_{1},1)}\cup{(\delta_{2},1)}$. Moreover any $J$-holomorphic curves counted by $U\langle\hat{\alpha}\cup{(\delta_{1},1)}\cup{(\gamma,p_{i+1})}  \rangle=\langle \zeta \rangle$ and  $U\langle \zeta \rangle=\langle \hat{\alpha}\cup{(\gamma,p_{i+1}-p_{i})}\cup{(\delta_{2},1)}) \rangle$ are in $\mathcal{M}^{J}((\gamma,p_{i+1}),(\delta',1)\cup{(\delta_{2},1)})$ and $\mathcal{M}^{J}((\delta',1)\cup{(\delta_{1})},(\gamma,p_{i+1}-p_{i}))$ respectively.
    
    \item[\bf Type $(A_{3})$] 
    
       There is a negative hyperbolic orbit $\delta'$ with $\delta'\in \hat{\alpha}$ such that  any $J$-holomorphic curves counted by $U\langle\hat{\alpha}\cup{(\delta_{1},1)}\cup{(\gamma,p_{i+1})}  \rangle=\langle \zeta \rangle$ and  $U\langle \zeta \rangle=\langle \hat{\alpha}\cup{(\gamma,p_{i+1}-p_{i})}\cup{(\delta_{2},1)}) \rangle$  are in $\mathcal{M}^{J}((\gamma,p_{i+1})\cup{(\delta',1)},(\delta_{2},1))$ and in $\mathcal{M}^{J}((\delta_{1}),(\gamma,p_{i+1}-p_{i})\cup{(\delta',1)})$ respectively.

\end{enumerate}

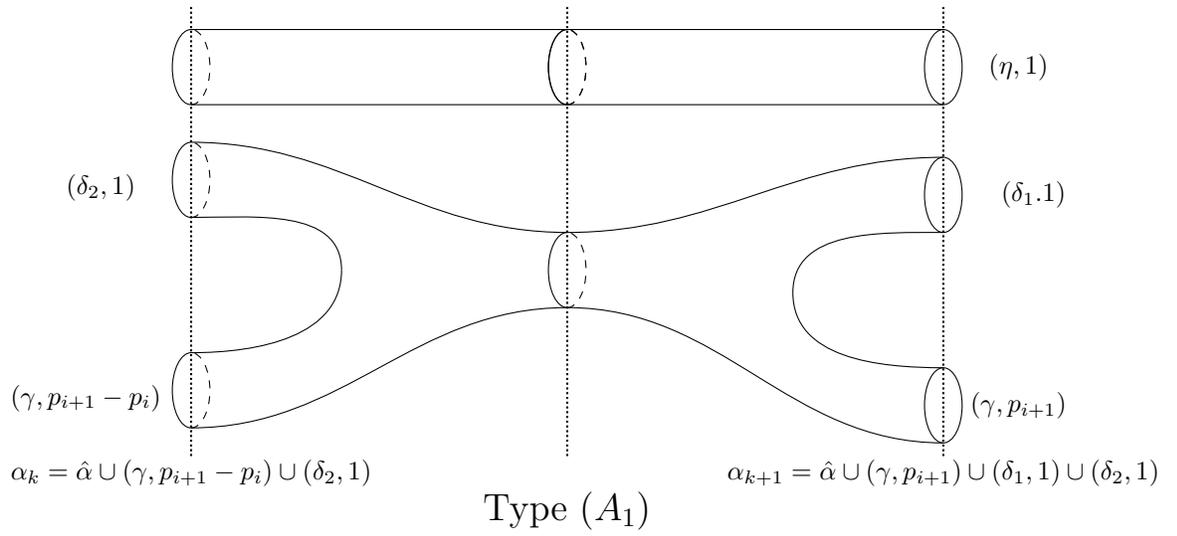
\begin{tikzpicture}
\draw (11,-1.8) arc [start angle=90,end angle=270,x radius=0.25,y radius=0.5];
 \draw (11,-2.8)[dashed] arc [start angle=270,end angle=450,x radius=0.25,y radius=0.5];
 
 \draw (11,-1.8)--(6,-1.8) ;
  \draw (11,-2.8)--(6,-2.8) ;
  \draw (6,-1.8) arc [start angle=90,end angle=270,x radius=0.25,y radius=0.5];
 \draw (6,-2.8)[dashed] arc [start angle=270,end angle=450,x radius=0.25,y radius=0.5];

 \draw (6,-3.3) arc [start angle=90,end angle=270,x radius=0.25,y radius=0.5];
 \draw (6,-4.3)[dashed] arc [start angle=270,end angle=450,x radius=0.25,y radius=0.5];

 \draw (11,-4.5) arc [start angle=90,end angle=270,x radius=0.25,y radius=0.5];
 \draw (11,-5.5)[dashed] arc [start angle=270,end angle=450,x radius=0.25,y radius=0.5];

 \draw (11,-5.5) to [out=180,in=0] (6,-7.1);
 \draw (6,-6.1) to [out=0,in=270] (8,-5);
 \draw (8,-5) to [out=90,in=0] (6,-4.3);
 \draw (11,-4.5) to [out=180,in=0] (6,-3.3);

  \draw (6,-6.1) arc [start angle=90,end angle=270,x radius=0.25,y radius=0.5];
 \draw (6,-7.1)[dashed] arc [start angle=270,end angle=450,x radius=0.25,y radius=0.5];
 
 \draw (4.8,-3.9) node{$(\delta_{2},1)$} ;

  \draw (4.6,-6.7) node{$(\gamma,p_{i+1}-p_{i})$} ;

 \draw[densely dotted,  thick] (6,-1.5)--(6,-7.5);

  \draw (6,-7.7) node{$\alpha_{k}=\hat{\alpha}\cup{(\gamma,p_{i+1}-p_{i})\cup{(\delta_{2},1)}}$};

\draw (16,-1.8) arc [start angle=90,end angle=270,x radius=0.25,y radius=0.5];
 \draw (16,-2.8) arc [start angle=270,end angle=450,x radius=0.25,y radius=0.5];
 
 \draw (16,-1.8)--(11,-1.8) ;
  \draw (16,-2.8)--(11,-2.8) ;
  \draw (11,-1.8) arc [start angle=90,end angle=270,x radius=0.25,y radius=0.5];
 \draw (11,-2.8)[dashed] arc [start angle=270,end angle=450,x radius=0.25,y radius=0.5];
 
\draw (17,-2.3) node{$(\eta,1)$};

 \draw (16,-3.5) arc [start angle=90,end angle=270,x radius=0.25,y radius=0.5];
 \draw (16,-4.5) arc [start angle=270,end angle=450,x radius=0.25,y radius=0.5];

 \draw (16,-6.3) arc [start angle=90,end angle=270,x radius=0.25,y radius=0.5];
 \draw (16,-7.3) arc [start angle=270,end angle=450,x radius=0.25,y radius=0.5];
 
 \draw (17.2,-4) node{$(\delta_{1}.1)$} ;
 
  \draw (17,-6.8) node{$(\gamma,p_{i+1})$} ;

 \draw (11,-4.5) to [out=0,in=180] (16,-3.5);
 \draw (11,-5.5) to [out=0,in=180] (16,-7.3);
 \draw (16,-6.3) to [out=180,in=270] (14,-5.3);
 \draw (14,-5.3) to [out=90,in=180] (16,-4.5);

 \draw[densely dotted,  thick] (11,-1.5)--(11,-7.5);
  \draw[densely dotted,  thick] (16,-1.5)--(16,-7.5) ;
  
  \draw (11,-8.2) node{\Large Type $(A_{1})$} ;
  
  \draw (16,-7.7) node{$\alpha_{k+1}=\hat{\alpha}\cup{(\gamma,p_{i+1})\cup{(\delta_{1},1)}\cup{(\delta_{2},1)}}$};

\end{tikzpicture}

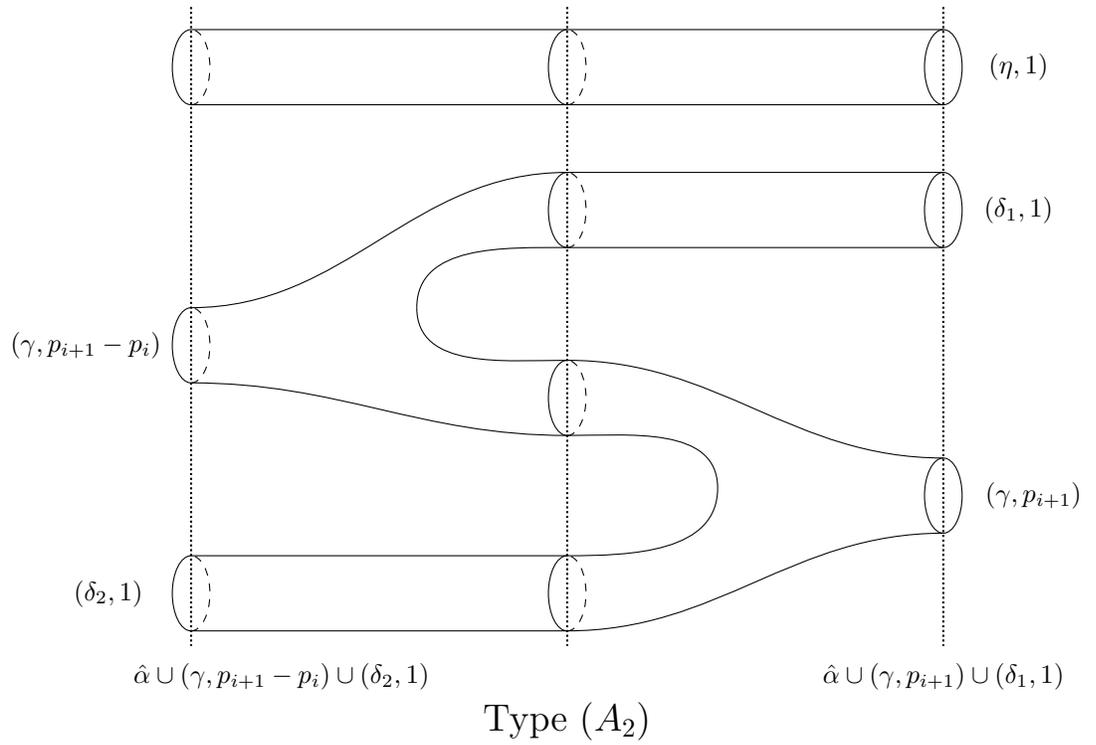
\vspace{5mm}

\begin{tikzpicture}

 \draw (16,-0.3)--(6,-0.3) ;
  \draw (16,-1.3)--(6,-1.3) ;
  \draw (6,-0.3) arc [start angle=90,end angle=270,x radius=0.25,y radius=0.5];
 \draw (6,-1.3)[dashed] arc [start angle=270,end angle=450,x radius=0.25,y radius=0.5];
 
\draw (16,-0.3) arc [start angle=90,end angle=270,x radius=0.25,y radius=0.5];
 \draw (16,-1.3) arc [start angle=270,end angle=450,x radius=0.25,y radius=0.5];
 
  \draw (11,-0.3) arc [start angle=90,end angle=270,x radius=0.25,y radius=0.5];
 \draw (11,-1.3)[dashed] arc [start angle=270,end angle=450,x radius=0.25,y radius=0.5];

 \draw (6,-4) arc [start angle=90,end angle=270,x radius=0.25,y radius=0.5];
 \draw (6,-5)[dashed] arc [start angle=270,end angle=450,x radius=0.25,y radius=0.5];

 \draw (11,-4.7) arc [start angle=90,end angle=270,x radius=0.25,y radius=0.5];
 \draw (11,-5.7)[dashed] arc [start angle=270,end angle=450,x radius=0.25,y radius=0.5];

\draw (11,-2.2) arc [start angle=90,end angle=270,x radius=0.25,y radius=0.5];
 \draw (11,-3.2)[dashed] arc [start angle=270,end angle=450,x radius=0.25,y radius=0.5];
 
 \draw (16,-2.2)--(11,-2.2) ;
  \draw (16,-3.2)--(11,-3.2) ;

\draw (16,-2.2) arc [start angle=90,end angle=270,x radius=0.25,y radius=0.5];
 \draw (16,-3.2) arc [start angle=270,end angle=450,x radius=0.25,y radius=0.5];

 \draw (11,-5.7) to [out=180,in=0] (6,-5);
  \draw (11,-2.2) to [out=180,in=0] (6,-4);
   \draw (11,-4.7) to [out=180,in=270] (9,-4);
  \draw (11,-3.2) to [out=180,in=90] (9,-4);

 \draw (4.6,-4.5) node{$(\gamma,p_{i+1}-p_{i})$} ;

 \draw (17,-2.7) node{$(\delta_{1},1)$} ;
  \draw (4.9,-7.8) node{$(\delta_{2},1)$} ;

\draw (17.2,-6.5) node{$(\gamma,p_{i+1})$} ;

\draw (17,-0.8) node{$(\eta,1)$} ;
 
 \draw[densely dotted,  thick] (6,0)--(6,-8.5);
  \draw[densely dotted,  thick] (11,0)--(11,-8.5) ;
   \draw[densely dotted,  thick] (16,0)--(16,-8.5);
  
  \draw (7.2,-8.9) node{$\hat{\alpha}\cup{(\gamma,p_{i+1}-p_{i})\cup{(\delta_{2},1)}}$};
  \draw (16,-8.9) node{$\hat{\alpha}\cup{(\gamma,p_{i+1})\cup{(\delta_{1},1)}}$};
  
   \draw (16,-6) arc [start angle=90,end angle=270,x radius=0.25,y radius=0.5];
 \draw (16,-7) arc [start angle=270,end angle=450,x radius=0.25,y radius=0.5];
 
  \draw (11,-4.7) to [out=0,in=180] (16,-6);
 \draw (11,-5.7) to [out=0,in=90] (13,-6.4);
 \draw (13,-6.4) to [out=270,in=0] (11,-7.3);
 \draw (16,-7) to [out=180,in=0] (11,-8.3);
 
  \draw (11,-7.3) arc [start angle=90,end angle=270,x radius=0.25,y radius=0.5];
 \draw (11,-8.3)[dashed] arc [start angle=270,end angle=450,x radius=0.25,y radius=0.5];
 
 \draw (6,-7.3) arc [start angle=90,end angle=270,x radius=0.25,y radius=0.5];
 \draw (6,-8.3)[dashed] arc [start angle=270,end angle=450,x radius=0.25,y radius=0.5];
 
  \draw (11,-7.3)--(6,-7.3) ;
  \draw (11,-8.3)--(6,-8.3) ;
 
 \draw (11,-9.5) node{\Large Type $(A_{2})$} ;
 
\end{tikzpicture}

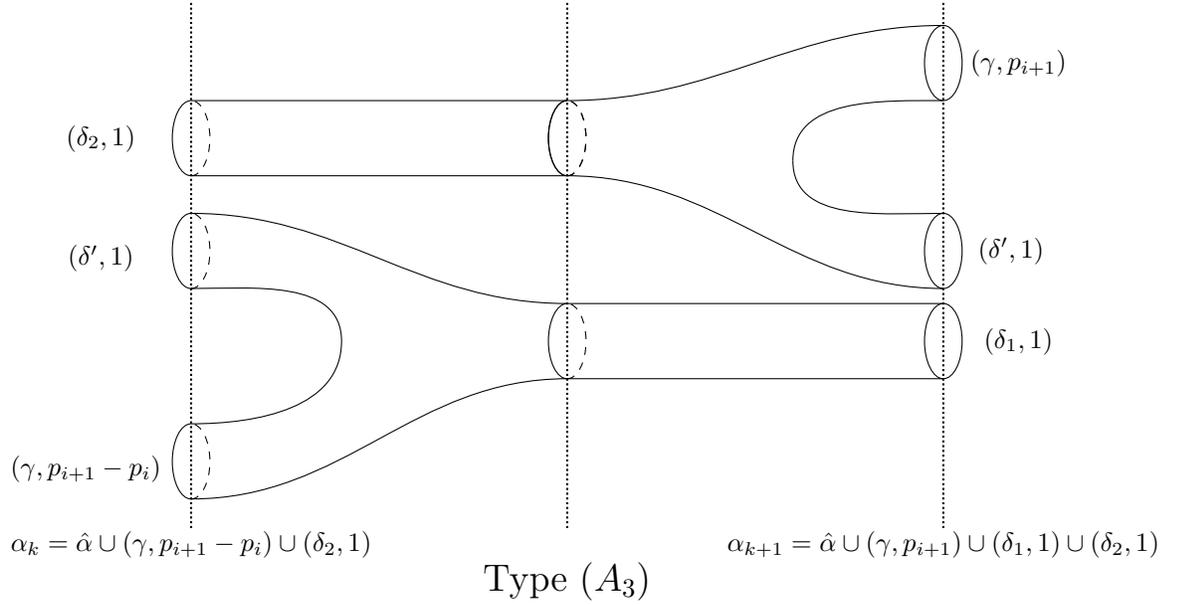
\vspace{5mm}

\begin{tikzpicture}
\draw (11,-1.8) arc [start angle=90,end angle=270,x radius=0.25,y radius=0.5];
 \draw (11,-2.8)[dashed] arc [start angle=270,end angle=450,x radius=0.25,y radius=0.5];
 
 \draw (11,-1.8)--(6,-1.8) ;
  \draw (11,-2.8)--(6,-2.8) ;
  \draw (6,-1.8) arc [start angle=90,end angle=270,x radius=0.25,y radius=0.5];
 \draw (6,-2.8)[dashed] arc [start angle=270,end angle=450,x radius=0.25,y radius=0.5];

 \draw (6,-3.3) arc [start angle=90,end angle=270,x radius=0.25,y radius=0.5];
 \draw (6,-4.3)[dashed] arc [start angle=270,end angle=450,x radius=0.25,y radius=0.5];

 \draw (11,-4.5) arc [start angle=90,end angle=270,x radius=0.25,y radius=0.5];
 \draw (11,-5.5)[dashed] arc [start angle=270,end angle=450,x radius=0.25,y radius=0.5];

 \draw (11,-5.5) to [out=180,in=0] (6,-7.1);
 \draw (6,-6.1) to [out=0,in=270] (8,-5);
 \draw (8,-5) to [out=90,in=0] (6,-4.3);
 \draw (11,-4.5) to [out=180,in=0] (6,-3.3);

  \draw (6,-6.1) arc [start angle=90,end angle=270,x radius=0.25,y radius=0.5];
 \draw (6,-7.1)[dashed] arc [start angle=270,end angle=450,x radius=0.25,y radius=0.5];
 
 \draw (4.8,-3.9) node{$(\delta',1)$} ;

  \draw (4.6,-6.7) node{$(\gamma,p_{i+1}-p_{i})$} ;
  \draw (4.8,-2.3) node{$(\delta_{2},1)$} ;

 \draw[densely dotted,  thick] (6,-0.5)--(6,-7.5);

  \draw (6,-7.7) node{$\alpha_{k}=\hat{\alpha}\cup{(\gamma,p_{i+1}-p_{i})\cup{(\delta_{2},1)}}$};

\draw (16,-4.5) arc [start angle=90,end angle=270,x radius=0.25,y radius=0.5];
 \draw (16,-5.5) arc [start angle=270,end angle=450,x radius=0.25,y radius=0.5];
 
 \draw (16,-4.5)--(11,-4.5) ;
  \draw (16,-5.5)--(11,-5.5) ;
  \draw (11,-1.8) arc [start angle=90,end angle=270,x radius=0.25,y radius=0.5];
 \draw (11,-2.8)[dashed] arc [start angle=270,end angle=450,x radius=0.25,y radius=0.5];
 
\draw (17,-1.3) node{$(\gamma,p_{i+1})$};

 \draw (16,-0.8) arc [start angle=90,end angle=270,x radius=0.25,y radius=0.5];
 \draw (16,-1.8) arc [start angle=270,end angle=450,x radius=0.25,y radius=0.5];

 \draw (16,-3.3) arc [start angle=90,end angle=270,x radius=0.25,y radius=0.5];
 \draw (16,-4.3) arc [start angle=270,end angle=450,x radius=0.25,y radius=0.5];
 
 \draw (16.9,-3.8) node{$(\delta',1)$} ;
 
  \draw (17,-5) node{$(\delta_{1},1)$} ;

 \draw (11,-1.8) to [out=0,in=180] (16,-0.8);
 \draw (11,-2.8) to [out=0,in=180] (16,-4.3);
 \draw (16,-3.3) to [out=180,in=270] (14,-2.6);
 \draw (14,-2.6) to [out=90,in=180] (16,-1.8);

 \draw[densely dotted,  thick] (11,-0.5)--(11,-7.5);
  \draw[densely dotted,  thick] (16,-0.5)--(16,-7.5) ;
  
  \draw (16,-7.7) node{$\alpha_{k+1}=\hat{\alpha}\cup{(\gamma,p_{i+1})\cup{(\delta_{1},1)}\cup{(\delta_{2},1)}}$};
  
\draw (11,-8.2) node{\Large Type $(A_{3})$} ;

\end{tikzpicture}

To prove Proposition \ref{nagai}, at first, we calculate the approximate actions of $\delta_{1}$, $\delta_{2}$ and $\eta \in \hat{\alpha}$ from $(A_{1})$ by using the splitting behavior of $J$-holomorphic curves. Next, we will show that any splitting behaviors of $(A_{2})$ and $(A_{3})$ cause contradictions.

\subsubsection{Type $(A_{1})$}

\begin{lem}\label{typa1}
Suppose that the pair of $\zeta$, $J$-holomorphic curves counted by $U\langle\hat{\alpha}\cup{(\delta_{1},1)}\cup{(\gamma,p_{i+1})}  \rangle=\langle \zeta \rangle$ and   $U\langle \zeta \rangle=\langle \hat{\alpha}\cup{(\gamma,p_{i+1}-p_{i})}\cup{(\delta_{2},1)}) \rangle$ is Type $(A_{1})$. Then, we have  $A(\delta_{1})\approx (p_{i+1}-p_{i})R$, $A(\delta_{2})\approx p_{i+1}R$ and for each $\eta \in \hat{\alpha}$, either $A(\eta)\approx \frac{1}{2}p_{i+1}R$ or $A(\eta)\approx \frac{1}{2}(p_{i+1}-p_{i})R$.

\end{lem}
\begin{proof}[\bf Proof of Lemma \ref{typa1}]

We will consider the behavior of $J$-holomorphic curves counted by $U$-map as $z \to \eta$. Then each $J$-holomorphic curve counted by $U\langle\hat{\alpha}\cup{(\delta_{1},1)}\cup{(\gamma,p_{i+1})}  \rangle=\langle \zeta \rangle$ and  $U\langle \zeta \rangle=\langle \hat{\alpha}\cup{(\gamma,p_{i+1}-p_{i})}\cup{(\delta_{2},1)}) \rangle$ have three possibilities of splitting and we have three estimates respectively as follows.

From  $U\langle \zeta \rangle=\langle \hat{\alpha}\cup{(\gamma,p_{i+1}-p_{i})}\cup{(\delta_{2},1)}) \rangle$, we have
\begin{enumerate}
\item[($\mathfrak{c}_{1}$).] $|(p_{i+1}-p_{i})R-2A(\eta)|<\epsilon$
        \item [($\mathfrak{c}_{2}$).] $|A(\delta_{2})-2A(\eta)|<\epsilon$
        \item[($\mathfrak{c}_{3}$).]$|(p_{i+1}-p_{i})R+A(\delta_{2})-2A(\eta)|<\epsilon$
  \end{enumerate}      
From  $U\langle\hat{\alpha}\cup{(\delta_{1},1)}\cup{(\gamma,p_{i+1})}  \rangle=\langle \zeta \rangle$, we have
\begin{enumerate}
  \item[($\mathfrak{d}_{1}$).] $|A(\delta_{1})-2A(\eta)|<\epsilon$
        \item [($\mathfrak{d}_{2}$).] $|p_{i+1}R-2A(\eta)|<\epsilon$
        \item[($\mathfrak{d}_{3}$).]$|p_{i+1}R+A(\delta_{1})-2A(\eta)|<\epsilon$
\end{enumerate}      
Here recall that  we  have from $U\langle \alpha_{k+1} \rangle=\langle \alpha_{k} \rangle$,

\begin{enumerate}
     \item[($\mathfrak{a}_{1}$).] $|p_{i}R-2A(\eta)|<\epsilon$
        \item [($\mathfrak{a}_{2}$).] $|A(\delta_{1})-2A(\eta)|<\epsilon$
        \item[($\mathfrak{a}_{3}$).]$|A(\delta_{2})-2A(\eta)|<\epsilon$
\end{enumerate}

and moreover we always have

$(\clubsuit)$.$|A(\delta_{1})+p_{i}R-A(\delta_{2})|<\epsilon$.

At first, we prove the next lemma.

\begin{lem}\label{restA1}

Any above pair (($\mathfrak{c}_{i_{1}}$),($\mathfrak{d}_{i_{2}}$),($\mathfrak{a}_{i_{3}}$)) causes a contradiction except for the following two cases.

$(i_{1},i_{2},i_{3})=(1,1,2),(2,2,3)$.

\end{lem}

In the next claim, we exclude pairs (($\mathfrak{c}_{i_{1}}$),($\mathfrak{d}_{i_{2}}$)) such that we  can derive contradictions by only them.

\begin{cla}\label{firstA1}
Any pair (($\mathfrak{c}_{i_{1}}$),($\mathfrak{d}_{i_{2}}$)) causes a contradiction except for the following cases

$(i_{1},i_{2})=(1,1),(2,2),(3,3)$.
\end{cla}

\begin{proof}[\bf Proof of Claim \ref{firstA1}]

We have to derive a contradiction from $(i_{1},i_{2})=(1,2)$, $(1,3)$, $(2,1)$, $(2,3)$, $(3,1)$, $(3,2)$ respectively.

\item Case $(i_{1},i_{2})=(1,2)$

By cancelling the terms $p_{i+1}R$ and $2A(\eta)$ from ($\mathfrak{c}_{1}$).$|(p_{i+1}-p_{i})R-2A(\eta)|<\epsilon$ and  ($\mathfrak{d}_{2}$).$|p_{i+1}R-2A(\eta)|<\epsilon$, we have $|p_{i}R|<2\epsilon$. This is a contradiction since $\epsilon$ is sufficiently small.

\item Case $(i_{1},i_{2})=(1,3)$

By cancelling the terms $p_{i+1}R$ and $2A(\eta)$ from
($\mathfrak{c}_{1}$).$|(p_{i+1}-p_{i})R-2A(\eta)|<\epsilon$ and ($\mathfrak{d}_{3}$).$|p_{i+1}R+A(\delta_{1})-2A(\eta)|<\epsilon$, we have $|A(\delta_{1})+p_{i}R|<2\epsilon$. This is a contradiction.

\item Case $(i_{1},i_{2})=(2,1)$

From
($\mathfrak{c}_{2}$).$|A(\delta_{2})-2A(\eta)|<\epsilon$, ($\mathfrak{d}_{1}$).$|A(\delta_{1})-2A(\eta)|<\epsilon$ and  ($\clubsuit$).$|p_{i}R+A(\delta_{1})-A(\delta_{2})|<\epsilon$, we have $|p_{i}R|<3\epsilon$. This is a contradiction.

\item Case $(i_{1},i_{2})=(2,3)$

In the same way, from
($\mathfrak{c}_{2}$).$|A(\delta_{2})-2A(\eta)|<\epsilon$, ($\mathfrak{d}_{3}$).$|p_{i+1}R+A(\delta_{1})-2A(\eta)|<\epsilon$ and ($\clubsuit$).$|p_{i}R+A(\delta_{1})-A(\delta_{2})|<\epsilon$, we have $|(p_{i+1}-p_{i})R|<3\epsilon$. This is a contradiction.

\item Case $(i_{1},i_{2})=(3,1)$

By cancelling the terms $A(\delta_{1})$, $A(\delta_{2})$ and $2A(\eta)$ from
($\mathfrak{c}_{3}$).$|(p_{i+1}-p_{i})R+A(\delta_{2})-2A(\eta)|<\epsilon$, ($\mathfrak{d}_{1}$).$|A(\delta_{1})-2A(\eta)|<\epsilon$ and ($\clubsuit$).$|p_{i}R+A(\delta_{1})-A(\delta_{2})|<\epsilon$, we have $|p_{i+1}R|<3\epsilon$. This is a contradiction.

\item Case $(i_{1},i_{2})=(3,2)$

By cancelling the terms $p_{i+1}R$ and $2A(\eta)$ from
($\mathfrak{c}_{3}$).$|p_{i+1}R+A(\delta_{1})-2A(\eta)|<\epsilon$ and ($\mathfrak{d}_{2}$).$|p_{i+1}R-2A(\eta)|<\epsilon$, we have$|A(\delta_{1})|<2\epsilon$. This is a contradiction.

By the above arguments, we complete the proof of Claim \ref{firstA1}.
\end{proof}

\begin{proof}[\bf Proof of Lemma \ref{restA1}]

By Claim \ref{firstA1}, we can see that the rest cases which we have to exclude are
$(i_{1},i_{2},i_{3})=(1,1,1)$, $(1,1,3)$, $(2,2,1)$, $(2,2,2)$, $(3,3,1)$, $(3,3,2)$, $(3,3,3)$.

\item Case $(i_{1},i_{2},i_{3})=(1,1,1)$

By cancelling the term $2A(\eta)$ from
($\mathfrak{c}_{1}$).$|(p_{i+1}-p_{i})R-2A(\eta)|<\epsilon$ and ($\mathfrak{a}_{1}$).$|p_{i}R-2A(\eta)|<\epsilon$, we have $|(p_{i+1}-2p_{i})R|<\epsilon$ and thus $p_{i}=p_{i+1}-p_{i}$. This contradicts Claim \ref{fre}.

\item Case $(i_{1},i_{2},i_{3})=(1,1,3)$, (resp. $(2,2,2)$)

By cancelling the term $2A(\eta)$ from
($\mathfrak{d}_{1}$).(resp. ($\mathfrak{a}_{2}$).)$|A(\delta_{1})-2A(\eta)|<\epsilon$ and ($\mathfrak{a}_{3}$).(resp. ($\mathfrak{c}_{2}$).)$|A(\delta_{2})-2A(\eta)|<\epsilon$, we have $|A(\delta_{1})-A(\delta_{2})|<2\epsilon$. This  contradicts ($\clubsuit$).$|p_{i}R+A(\delta_{1})-A(\delta_{2})|<\epsilon$.

\item Case $(i_{1},i_{2},i_{3})=(2,2,1)$

By cancelling the terms $2A(\eta)$, $A(\delta_{2})$ and $p_{i}R$ from
($\mathfrak{c}_{2}$).$|A(\delta_{2})-2A(\eta)|<\epsilon$,  ($\mathfrak{a}_{1}$).$|p_{i}R-2A(\eta)|<\epsilon$ and ($\clubsuit$).$|p_{i}R+A(\delta_{1})-A(\delta_{2})|<\epsilon$, we have $|A(\delta_{1})|<3\epsilon$. This is a contradiction.

\item Case $(i_{1},i_{2},i_{3})=(3,3,1)$

By cancelling the terms $2A(\eta)$ from ($\mathfrak{d}_{3}$).$|p_{i+1}R+A(\delta_{1})-2A(\eta)|<\epsilon$ and ($\mathfrak{a}_{1}$).$|p_{i}R-2A(\eta)|<\epsilon$, we have $|(p_{i+1}-p_{i})R+A(\delta_{1})|<2\epsilon$. This is a contradiction.

\item Case $(i_{1},i_{2},i_{3})=(3,3,2)$

($\mathfrak{d}_{3}$).$|p_{i+1}R+A(\delta_{1})-2A(\eta)|<\epsilon$ and ($\mathfrak{a}_{2}$).$|A(\delta_{1})-2A(\eta)|<\epsilon$ indicate $|p_{i+1}R|<2\epsilon$ and so $p_{i+1}=0$. This is a contradiction.

\item Case $(i_{1},i_{2},i_{3})=(3,3,3)$

In the same way,
($\mathfrak{a}_{3}$).$|A(\delta_{2})-2A(\eta)|<\epsilon$ and ($\mathfrak{c}_{3}$).$|(p_{i+1}-p_{i})R+A(\delta_{2})-2A(\eta)|<\epsilon$ indicate  $p_{i+1}=p_{i}$. This is a contradiction.

Combining the above argument, we complete the proof of Lemma \ref{restA1}.

\end{proof}

By the discussions so far, the rest cases are $(i_{1},i_{2},i_{3})=(1,1,2)$ and $(2,2,3)$.

\begin{enumerate}

\item[ Case] $(i_{1},i_{2},i_{3})=(1,1,2)$.

From ($\mathfrak{c}_{1}$).$|(p_{i+1}-p_{i})R-2A(\eta)|<\epsilon$, ($\mathfrak{d}_{1}$).$|A(\delta_{1})-2A(\eta)|<\epsilon$ and ($\clubsuit$).$|p_{i}R+A(\delta_{1})-A(\delta_{2})|<\epsilon$, we have

\begin{equation}
    A(\delta_{1})\approx (p_{i+1}-p_{i})R,\,\,A(\delta_{2})\approx p_{i+1}R,\,\,A(\eta)\approx \frac{1}{2} (p_{i+1}-p_{i})R.
\end{equation}

\item [Case] $(i_{1},i_{2},i_{3})=(2,2,3)$.

From
($\mathfrak{c}_{2}$).$|A(\delta_{2})-2A(\eta)|<\epsilon$, ($\mathfrak{d}_{2}$).$|p_{i+1}R-2A(\eta)|<\epsilon$ and ($\clubsuit$).$|p_{i}R+A(\delta_{1})-A(\delta_{2})|<\epsilon$, we have

\begin{equation}
     A(\delta_{1})\approx (p_{i+1}-p_{i})R,\,\,A(\delta_{2})\approx p_{i+1}R,\,\,A(\eta)\approx \frac{1}{2}p_{i+1}R.
\end{equation}

\end{enumerate}

Combining  the arguments, we complete the proof of Lemma \ref{typa1}.
\end{proof}

\subsubsection{Type $(A_{2})$}

\begin{lem}\label{typa2}
Suppose that the pair of $\zeta$, $J$-holomorphic curves counted by $U\langle\hat{\alpha}\cup{(\delta_{1},1)}\cup{(\gamma,p_{i+1})}  \rangle=\langle \zeta \rangle$ and   $U\langle \zeta \rangle=\langle \hat{\alpha}\cup{(\gamma,p_{i+1}-p_{i})}\cup{(\delta_{2},1)}) \rangle$ is Type $(A_{2})$. Then, 
 the $J$-holomorphic curves counted by the $U$-map cause a contradiction.
\end{lem}

\begin{proof}[\bf Proof of Lemma \ref{typa2}]

Consider the behaviors of $J$-holomorphic curves counted by $U$-map as $z\to \eta$. Then we obtain three possibilities from $U\langle\hat{\alpha}\cup{(\delta_{1},1)}\cup{(\gamma,p_{i+1})}  \rangle=\langle \zeta \rangle$.
\begin{enumerate}
       \item[($\mathfrak{e}_{1}$).] $|p_{i+1}R-2A(\eta)|<\epsilon$
        \item [($\mathfrak{e}_{2}$).] $|A(\delta_{2})-2A(\eta)|<\epsilon$
        \item[($\mathfrak{e}_{3}$).]$|p_{i+1}R-(2A(\eta)+A(\delta_{2}))|<\epsilon$
        
        By the same way, we obtain three possibilities from $U\langle \zeta \rangle=\langle \hat{\alpha}\cup{(\gamma,p_{i+1}-p_{i})}\cup{(\delta_{2},1)}) \rangle$.
        
         \item[($\mathfrak{f}_{1}$).] $|(p_{i+1}-p_{i})R-2A(\eta)|<\epsilon$
        \item [($\mathfrak{f}_{2}$).] $|A(\delta_{1})-2A(\eta)|<\epsilon$
        \item[($\mathfrak{f}_{3}$).]$|A(\delta_{1})+2A(\eta)-(p_{i+1}-p_{i})R|<\epsilon$
        
        Also as $z\to \delta_{2}$. then by the splitting behaviors of the curves counted by  $U\langle \zeta \rangle=\langle \hat{\alpha}\cup{(\gamma,p_{i+1}-p_{i})}\cup{(\delta_{2},1)}) \rangle$, we have
        
     \item[($\mathfrak{g}_{1}$).] $|(p_{i+1}-p_{i})R-2A(\delta_{2})|<\epsilon$
        \item [($\mathfrak{g}_{2}$).] $|A(\delta_{1})-2A(\delta_{2})|<\epsilon$
        \item[($\mathfrak{g}_{3}$).]$|A(\delta_{1})+2A(\delta_{2})-(p_{i+1}-p_{i})R|<\epsilon$

        Also as $z\to \delta_{1}$. then From $U\langle\hat{\alpha}\cup{(\delta_{1},1)}\cup{(\gamma,p_{i+1})}  \rangle=\langle \zeta \rangle$, we have
        
         \item[($\mathfrak{h}_{1}$).] $|p_{i+1}R-2A(\delta_{1})|<\epsilon$
        \item [($\mathfrak{h}_{2}$).]$|A(\delta_{2})-2A(\delta_{1})|<\epsilon$
        \item[($\mathfrak{h}_{3}$).]$|A(\delta_{2})+2A(\delta_{1})-p_{i+1}R|<\epsilon$

        Recall that we always have ($\clubsuit$).$|A(\delta_{1})+p_{i}R-A(\delta_{2})|<\epsilon$, since $A(\alpha_{k+1})-A(\alpha_{k})<\epsilon$.

        We exclude the pairs (($\mathfrak{e}_{i_{1}}$),($\mathfrak{f}_{i_{2}}$),($\mathfrak{g}_{i_{3}}$),($\mathfrak{h}_{i_{4}}$)) in the same way as Type $(A_{1})$.

        \end{enumerate}
        
        At first, we prove the next lemma.
        
        \begin{lem}\label{afo}
        Any above pairs (($\mathfrak{e}_{i_{1}}$),($\mathfrak{f}_{i_{2}}$),($\mathfrak{g}_{i_{3}}$),($\mathfrak{h}_{i_{4}}$)) causes a contradiction except for the following cases.
        
        $(i_{1},i_{2},i_{3},i_{4})$=$(3,3,1,2)$, $(3,3,1,3)$, $(3,3,3,2)$.

        \end{lem}
        
To prove Lemma \ref{afo}, first of all, we will exclude pairs (($\mathfrak{e}_{i_{1}}$),($\mathfrak{f}_{i_{2}}$)) such that we can derive  contradictions by only them.

  \begin{cla}\label{abu}
The following each pair (($\mathfrak{e}_{i_{1}}$),($\mathfrak{f}_{i_{2}}$)) causes a contradiction.

$(i_{1},i_{2})$=$(1,1)$, $(3,1)$, $(2,2)$, $(1,3)$.
\end{cla}

\begin{proof}[\bf Proof of Claim \ref{abu}]

\item Case $(i_{1},i_{2})$=$(1,1)$

This is trivial.

\item Case $(i_{1},i_{2})$=$(3,1)$

By cancelling the term $2A(\eta)$ from ($\mathfrak{e}_{3}$).$|p_{i+1}R-(2A(\eta)+A(\delta_{2}))|<\epsilon$ and ($\mathfrak{f}_{1}$).$|(p_{i+1}-p_{i})R-2A(\eta)|<\epsilon$, we have $|p_{i}R-A(\delta_{2})|<3\epsilon$. But this contradicts ($\clubsuit$).$|A(\delta_{1})+p_{i}R-A(\delta_{2})|<\epsilon$.

\item Case $(i_{1},i_{2})$=$(2,2)$

By cancelling the term $2A(\eta)$ from ($\mathfrak{e}_{2}$).$|A(\delta_{2})-2A(\eta)|<\epsilon$ and ($\mathfrak{f}_{2}$).$|A(\delta_{1})-2A(\eta)|<\epsilon$, we have $|A(\delta_{1})-A(\delta_{2})|<2\epsilon$. This  obviously contradict ($\clubsuit$).$|A(\delta_{1})+p_{i}R-A(\delta_{2})|<\epsilon$.

\item Case $(i_{1},i_{2})$=$(1,3)$

($\mathfrak{e}_{1}$).$|p_{i+1}R-2A(\eta)|<\epsilon$ and ($\mathfrak{f}_{3}$).$|A(\delta_{1})+2A(\eta)-(p_{i+1}-p_{i})R|<\epsilon$ imply that $|A(\delta_{1})+p_{i}R|<2\epsilon$. This is a contradiction.
\end{proof}

\begin{cla}\label{adddv}
        Any pairs (($\mathfrak{e}_{i_{1}}$),($\mathfrak{f}_{i_{2}}$),($\mathfrak{g}_{i_{3}}$)) causes a contradiction except for the following cases.
        
        $(i_{1},i_{2},i_{3})$=$(3,3,1)$, $(3,3,3)$.
        
\end{cla}

\begin{proof}[\bf Proof of Claim \ref{adddv}]

At first, note that  ($\mathfrak{g}_{2}$).$|A(\delta_{1})-2A(\delta_{2})|<\epsilon$ obviously contradicts ($\clubsuit$).$|A(\delta_{1})+p_{i}R-A(\delta_{2})|<\epsilon$. So we have only to consider the cases $i_{3}=1$, $2$.

Hence by Claim \ref{abu}, we can find that it is sufficient to exclude the cases $(i_{1},i_{2},i_{3})=(2,1,1)$, $(2,1,3)$, $(1,2,1)$, $(1,2,3)$, $(3,2,1)$, $(3,2,3)$, $(2,3,1)$, $(2,3,3)$.

\item Case $(i_{1},i_{2},i_{3})$=$(2,1,1)$

By cancelling the term $2A(\eta)$ from ($\mathfrak{e}_{2}$).$|A(\delta_{2})-2A(\eta)|<\epsilon$ and ($\mathfrak{f}_{1}$).$|(p_{i+1}-p_{i})R-2A(\eta)|<\epsilon$, we have $|A(\delta_{2})-(p_{i+1}-p_{i})R|<2\epsilon$. This contradicts  ($\mathfrak{g}_{1}$).$|(p_{i+1}-p_{i})R-2A(\delta_{2})|<\epsilon$.

\item Case $(i_{1},i_{2},i_{3})$=$(2,1,3)$

    In the same way as above, by cancelling the term $2A(\eta)$ from ($\mathfrak{e}_{2}$).$|A(\delta_{2})-2A(\eta)|<\epsilon$ and ($\mathfrak{f}_{1}$).$|(p_{i+1}-p_{i})R-2A(\eta)|<\epsilon$, we have $|A(\delta_{2})-(p_{i+1}-p_{i})R|<2\epsilon$. This contradicts  ($\mathfrak{g}_{3}$).$|A(\delta_{1})+2A(\delta_{2})-(p_{i+1}-p_{i})R|<\epsilon$.

\item Case $(i_{1},i_{2},i_{3})$=$(1,2,1)$

By cancelling the terms $2A(\eta)$ and $A(\delta_{1})$ from
($\mathfrak{e}_{1}$).$|p_{i+1}R-2A(\eta)|<\epsilon$,  ($\mathfrak{f}_{2}$).$|A(\delta_{1})-2A(\eta)|<\epsilon$ and ($\clubsuit$).$|A(\delta_{1})+p_{i}R-A(\delta_{2})|<\epsilon$, we have $|A(\delta_{2})-(p_{i+1}+p_{i})R|<3\epsilon$. This contradicts  ($\mathfrak{g}_{1}$).$|(p_{i+1}-p_{i})R-2A(\delta_{2})|<\epsilon$.

\item Case $(i_{1},i_{2},i_{3})$=$(1,2,3)$

In the same way as above, by cancelling the terms $2A(\eta)$ and $A(\delta_{1})$ from
($\mathfrak{e}_{1}$).$|p_{i+1}R-2A(\eta)|<\epsilon$,  ($\mathfrak{f}_{2}$).$|A(\delta_{1})-2A(\eta)|<\epsilon$ and ($\clubsuit$).$|A(\delta_{1})+p_{i}R-A(\delta_{2})|<\epsilon$, we have $|A(\delta_{2})-(p_{i+1}+p_{i})R|<3\epsilon$. This contradicts  ($\mathfrak{g}_{3}$).$|A(\delta_{1})+2A(\delta_{2})-(p_{i+1}-p_{i})R|<\epsilon$.

\item Case $(i_{1},i_{2},i_{3})$=$(3,2,1)$

By cancelling the terms $2A(\eta)$ and $A(\delta_{1})$ from ($\mathfrak{e}_{3}$).$|p_{i+1}R-(2A(\eta)+A(\delta_{2}))|<\epsilon$, ($\mathfrak{f}$).$|A(\delta_{i})-2A(\eta)|<\epsilon$ and ($\clubsuit$).$|A(\delta_{1})+p_{i}R-A(\delta_{2})|<\epsilon$, we have $|2A(\delta_{2})-(p_{i+1}+p_{i})R|<3\epsilon$. This contradicts ($\mathfrak{g}_{1}$).$|(p_{i+1}-p_{i})R-2A(\delta_{2})|<\epsilon$.

\item Case $(i_{1},i_{2},i_{3})$=$(3,2,3)$

In the same way as above, by cancelling the terms $2A(\eta)$ and $A(\delta_{1})$ from ($\mathfrak{e}_{3}$).$|p_{i+1}R-(2A(\eta)+A(\delta_{2}))|<\epsilon$, ($\mathfrak{f}$).$|A(\delta_{i})-2A(\eta)|<\epsilon$ and ($\clubsuit$).$|A(\delta_{1})+p_{i}R-A(\delta_{2})|<\epsilon$, we have $|2A(\delta_{2})-(p_{i+1}+p_{i})R|<3\epsilon$. This contradicts   ($\mathfrak{g}_{3}$).$|A(\delta_{1})+2A(\delta_{2})-(p_{i+1}-p_{i})R|<\epsilon$.

\item Case $(i_{1},i_{2},i_{3})$=$(2,3,1)$

 ($\mathfrak{e}_{2}$).$|A(\delta_{2})-2A(\eta)|<\epsilon$, ($\mathfrak{f}_{3}$).$|A(\delta_{1})+2A(\eta)-(p_{i+1}-p_{i})R|<\epsilon$ and ($\mathfrak{g}_{1}$).$|(p_{i+1}-p_{i})R-2A(\delta_{2})|<\epsilon$ imply that $|A(\delta_{1})-A(\delta_{2})|<3\epsilon$. This contradicts ($\clubsuit$).$|A(\delta_{1})+p_{i}R-A(\delta_{2})|<\epsilon$.

\item Case $(i_{1},i_{2},i_{3})$=$(2,3,3)$

($\mathfrak{e}_{2}$).$|A(\delta_{2})-2A(\eta)|<\epsilon$ and ($\mathfrak{f}_{3}$).$|A(\delta_{1})+2A(\eta)-(p_{i+1}-p_{i})R|<\epsilon$ imply $|A(\delta_{1})+A(\delta_{2})-(p_{i+1}-p_{i})R|<2\epsilon$. This obviously contradicts  ($\mathfrak{g}_{3}$).$|A(\delta_{1})+2A(\delta_{2})-(p_{i+1}-p_{i})R|<\epsilon$.

\end{proof}

\begin{proof}[\bf Proof of Lemma \ref{afo}]

By the above arguments, we can find that the remaining cases are  $(i_{1},i_{2},i_{3},i_{4})=(3,3,1,1)$, $(3,3,1,2)$, $(3,3,1,3)$, $(3,3,3,1)$, $(3,3,3,2)$, $(3,3,3,3)$.

To prove the lemma, we will exclude the cases $(i_{1},i_{2},i_{3},i_{4})=(3,3,1,1)$, $(3,3,1,3)$ $(3,3,3,1)$, $(3,3,3,3)$ as follows.

\item Case $(i_{1},i_{2},i_{3},i_{4})=(3,3,1,1)$

By cancelling the term $p_{i+1}R$ from ($\mathfrak{g}_{1}$).$|(p_{i+1}-p_{i})R-2A(\delta_{2})|<\epsilon$ and ($\mathfrak{h}_{1}$).$|p_{i+1}R-2A(\delta_{1})|<\epsilon$, we have $|A(\delta_{1})-\frac{1}{2}p_{i}R-A(\delta_{2})|<\epsilon$. This contradicts ($\clubsuit$).$|A(\delta_{1})+p_{i}R-A(\delta_{2})|<\epsilon$

\item Case $(i_{1},i_{2},i_{3},i_{4})=(3,3,3,1)$

By cancelling the term $A(\delta_{1})$ from ($\mathfrak{g}_{3}$).$|A(\delta_{1})+2A(\delta_{2})-(p_{i+1}-p_{i})R|<\epsilon$ and ($\mathfrak{h}_{1}$). $|p_{i+1}R-2A(\delta_{1})|<\epsilon$, we have $|4A(\delta_{2})-(p_{i+1}-2p_{i})R|<3\epsilon$. Also from ($\mathfrak{h}_{1}$). $|p_{i+1}R-2A(\delta_{1})|<\epsilon$ and ($\clubsuit$).$|A(\delta_{1})+p_{i}R-A(\delta_{2})|<\epsilon$, we have $|2A(\delta_{2})-(p_{i+1}+2p_{i})R|<3\epsilon$.

$|4A(\delta_{2})-(p_{i+1}-2p_{i})R|<3\epsilon$ and $|2A(\delta_{2})-(p_{i+1}+2p_{i})R|<3\epsilon$ imply $|(p_{i+1}+6p_{i})R|<9\epsilon$. This is a contradiction.

\item Case $(i_{1},i_{2},i_{3},i_{4})=(3,3,3,3)$

By cancelling the term $p_{i+1}R$ from ($\mathfrak{g}_{3}$).$|A(\delta_{1})+2A(\delta_{2})-(p_{i+1}-p_{i})R|<\epsilon$ and ($\mathfrak{h}_{3}$).$|A(\delta_{2})+2A(\delta_{1})-p_{i+1}R|<\epsilon$, we have $|A(\delta_{2})-A(\delta_{1})+p_{i}R|<2\epsilon$. This contradicts ($\clubsuit$).$|A(\delta_{1})+p_{i}R-A(\delta_{2})|<\epsilon$.

Combining the above consequences, we finish the proof of Lemma \ref{afo}.
\end{proof}

By Lemma \ref{afo}, we still have the following pairs.

 $(i_{1},i_{2},i_{3},i_{4})$=$(3,3,1,2)$, $(3,3,1,3)$, $(3,3,3,2)$.
 
 From these pairs. we can decide the approximate relations as follows.
 
 \item[Case]  $(i_{1},i_{2},i_{3},i_{4})$=$(3,3,1,2)$
 
From  ($\mathfrak{e}_{3}$).$|p_{i+1}R-(2A(\eta)+A(\delta_{2}))|<\epsilon$, ($\mathfrak{f}_{3}$).$|A(\delta_{1})+2A(\eta)-(p_{i+1}-p_{i})R|<\epsilon$, ($\mathfrak{g}_{1}$).$|(p_{i+1}-p_{i})R-2A(\delta_{2})|<\epsilon$,  ($\mathfrak{h}_{2}$).$|A(\delta_{2})-2A(\delta_{1})|<\epsilon$ and ($\clubsuit$).$|A(\delta_{1})+p_{i}R-A(\delta_{2})|<\epsilon$, we have

\begin{equation}\label{i}
    A(\delta_{1})\approx p_{i}R,\,\,A(\delta_{2})\approx2p_{i}R,\,\,A(\eta)\approx\frac{3}{2}p_{i}R
\end{equation}

 \item [ Case]  $(i_{1},i_{2},i_{3},i_{4})$=$(3,3,1,3)$
 
 From  ($\mathfrak{e}_{3}$).$|p_{i+1}R-(2A(\eta)+A(\delta_{2}))|<\epsilon$, ($\mathfrak{f}_{3}$).$|A(\delta_{1})+2A(\eta)-(p_{i+1}-p_{i})R|<\epsilon$, ($\mathfrak{g}_{1}$).$|(p_{i+1}-p_{i})R-2A(\delta_{2})|<\epsilon$,  ($\mathfrak{h}_{3}$).$|A(\delta_{2})+2A(\delta_{1})-p_{i+1}R|<\epsilon$ and ($\clubsuit$).$|A(\delta_{1})+p_{i}R-A(\delta_{2})|<\epsilon$, we have

\begin{equation}\label{ii}
    A(\delta_{1})\approx 2p_{i}R,\,\,A(\delta_{2})\approx 3p_{i}R,\,\,A(\eta)\approx 2p_{i}R
\end{equation}.

 \item [\bf Case]  $(i_{1},i_{2},i_{3},i_{4})$=$(3,3,3,2)$
 
 From  ($\mathfrak{e}_{3}$).$|p_{i+1}R-(2A(\eta)+A(\delta_{2}))|<\epsilon$, ($\mathfrak{f}_{3}$).$|A(\delta_{1})+2A(\eta)-(p_{i+1}-p_{i})R|<\epsilon$, ($\mathfrak{g}_{3}$).$|A(\delta_{1})+2A(\delta_{2})-(p_{i+1}-p_{i})R|<\epsilon$,  ($\mathfrak{h}_{2}$).$|A(\delta_{2})-2A(\delta_{1})|<\epsilon$ and ($\clubsuit$).$|A(\delta_{1})+p_{i}R-A(\delta_{2})|<\epsilon$, we have

\begin{equation}\label{iii}
    A(\delta_{1})\approx p_{i}R,\,\,A(\delta_{2})\approx2p_{i}R,\,\,A(\eta)\approx2p_{i}R
\end{equation}

Recall that from the splitting behaviors of $J$-holomorphic curve counted by $U\langle \alpha_{k+1} \rangle=\langle \alpha_{k} \rangle$ as $z \to \eta$, we have three possibilities of splitting of holomorphic curve and also three possibilities of approximate relations as follows.
  \item[($\mathfrak{a}_{1}$).] $|p_{i}R-2A(\eta)|<\epsilon$
        \item [($\mathfrak{a}_{2}$).] $|A(\delta_{1})-2A(\eta)|<\epsilon$
        \item[($\mathfrak{a}_{3}$).]$|A(\delta_{2})-2A(\eta)|<\epsilon$

But it is easy to check that (\ref{i}), (\ref{ii}), (\ref{iii}) can not hold any these relation. We complete the proof of Lemma \ref{typa2}.
\end{proof}

\subsubsection{Type $(A_{3})$}

\begin{lem}\label{typa3}
Suppose that the pair of $\zeta$, $J$-holomorphic curves counted by $U\langle\hat{\alpha}\cup{(\delta_{1},1)}\cup{(\gamma,p_{i+1})}  \rangle=\langle \zeta \rangle$ and   $U\langle \zeta \rangle=\langle \hat{\alpha}\cup{(\gamma,p_{i+1}-p_{i})}\cup{(\delta_{2},1)}) \rangle$ is Type $(A_{3})$. Then, 
 the $J$-holomorphic curves counted by the $U$-map cause a contradiction.
\end{lem}

\begin{proof}[\bf Proof of Lemma \ref{typa3}]

Consider the behaviors of $J$-holomorphic curves counted by $U$-map as $z\to \delta_{2}$.

Then we obtain three possibilities of the approximate relations in the actions from the splitting behaviors of the curves counted by  $U\langle \zeta \rangle=\langle \hat{\alpha}\cup{(\gamma,p_{i+1}-p_{i})}\cup{(\delta_{2},1)}) \rangle$.

\begin{enumerate}
\item[($\mathfrak{i}_{1}$).] $|A(\delta_{1})-(p_{i+1}-p_{i})R-2A(\delta_{2})|<\epsilon$
\item[($\mathfrak{i}_{2}$).] $|(p_{i+1}-p_{i})R-2A(\delta_{2})|<\epsilon$
\item[($\mathfrak{i}_{3}$).] $|A(\delta_{1})-2A(\delta_{2})|<\epsilon$.

\end{enumerate}

But Every ($\mathfrak{i}_{1}$), ($\mathfrak{i}_{2}$) and ($\mathfrak{i}_{3}$) causes a contradiction as follows

\begin{enumerate}
    \item Case ($\mathfrak{i}_{1}$)
    
     Recall ($\clubsuit$).$|A(\delta_{1})+p_{i}R-A(\delta_{2})|<\epsilon$. This obviously contradicts $|A(\delta_{1})-(p_{i+1}-p_{i})R-2A(\delta_{2})|<\epsilon$. In fact, by cancelling the term $A(\delta_{1})$ from them, we have $|p_{i+1}R+A(\delta_{2})|<2$. This is a contradiction.
     
    \item Case ($\mathfrak{i}_{2}$)
    
    Since any $J$-holomorphic curves counted by $U\langle\hat{\alpha}\cup{(\delta_{1},1)}\cup{(\gamma,p_{i+1})}  \rangle=\langle \zeta \rangle$ $\delta'\in \hat{\alpha}$  are in $\mathcal{M}^{J}((\gamma,p_{i+1})\cup{(\delta',1)},(\delta_{2},1))$, we have $|p_{i+1}R+A(\delta')-A(\delta_{2})|<\epsilon$.
    
    By cancelling the term $A(\delta_{2})$ from the above inequality and ($\mathfrak{i}_{2}$).$|(p_{i+1}-p_{i})R-2A(\delta_{2})|<\epsilon$, we have $|\frac{1}{2}(p_{i+1}+p_{i})R+A(\delta')|<\frac{3}{2}\epsilon$. This is a contradiction.
    
    \item Case ($\mathfrak{i}_{3}$)
    
   Recall ($\clubsuit$).$|A(\delta_{1})+p_{i}R-A(\delta_{2})|<\epsilon$. This obviously contradicts ($\mathfrak{i}_{3}$).$|A(\delta_{1})-2A(\delta_{2})|<\epsilon$.

\end{enumerate}

Combining the above arguments, we can see that Type $(A_{3})$ can not occur.
\end{proof}

\subsection{Type (B)}

In the same way as Type (A), the topological types of $J$-holomoprhic curves counted by $U\langle\hat{\alpha}\cup{(\gamma,p_{i+1})}  \rangle=\langle \zeta \rangle$ and  $U\langle \zeta \rangle=\langle \hat{\alpha}\cup{(\gamma,p_{i+1}-p_{i})}\cup{(\delta_{1},1)}\cup{(\delta_{2},1)}) \rangle$ are both $(g,k,l)=(0,3,0)$. 

And also we can see that there is a negative hyperbolic orbit $\delta'$ such that  $\zeta=\hat{\alpha}\cup{
    (\delta',1)}\cup{(\delta_{2},1)}$ and, any $J$-holomorphic curves counted by $U\langle\hat{\alpha}\cup{(\gamma,p_{i+1})}  \rangle=\langle \zeta \rangle$ and  $U\langle \zeta \rangle=\langle \hat{\alpha}\cup{(\gamma,p_{i+1}-p_{i})}\cup{(\delta_{1},1)}\cup{(\delta_{2},1)}) \rangle$ are in $\mathcal{M}^{J}((\gamma,p_{i+1}),(\delta',1)\cup{(\delta_{2},1)})$ and $\mathcal{M}^{J}((\delta',1),(\gamma,p_{i+1}-p_{i})\cup{(\delta_{1},1)})$ respectively.

\begin{tikzpicture}
\draw (11,-1.8) arc [start angle=90,end angle=270,x radius=0.25,y radius=0.5];
 \draw (11,-2.8)[dashed] arc [start angle=270,end angle=450,x radius=0.25,y radius=0.5];
 
 \draw (16,-1.8)--(6,-1.8) ;
  \draw (16,-2.8)--(6,-2.8) ;
  \draw (6,-1.8) arc [start angle=90,end angle=270,x radius=0.25,y radius=0.5];
 \draw (6,-2.8)[dashed] arc [start angle=270,end angle=450,x radius=0.25,y radius=0.5];
 
\draw (16,-1.8) arc [start angle=90,end angle=270,x radius=0.25,y radius=0.5];
 \draw (16,-2.8) arc [start angle=270,end angle=450,x radius=0.25,y radius=0.5];

 \draw (6,-3.3) arc [start angle=90,end angle=270,x radius=0.25,y radius=0.5];
 \draw (6,-4.3)[dashed] arc [start angle=270,end angle=450,x radius=0.25,y radius=0.5];

 \draw (11,-4) arc [start angle=90,end angle=270,x radius=0.25,y radius=0.5];
 \draw (11,-5)[dashed] arc [start angle=270,end angle=450,x radius=0.25,y radius=0.5];
 
 \draw (11,-5) to [out=180,in=0] (6,-6.5);
 \draw (6,-5.5) to [out=0,in=270] (8,-4.7);
 \draw (8,-4.7) to [out=90,in=0] (6,-4.3);
 \draw (11,-4) to [out=180,in=0] (6,-3.3);

  \draw (6,-5.5) arc [start angle=90,end angle=270,x radius=0.25,y radius=0.5];
 \draw (6,-6.5)[dashed] arc [start angle=270,end angle=450,x radius=0.25,y radius=0.5];
 
 \draw (4.6,-6) node{$(\gamma,p_{i+1}-p_{i})$} ;

 \draw (4.9,-3.9) node{$(\delta_{1},1)$} ;
  \draw (4.9,-7.8) node{$(\delta_{2},1)$} ;

\draw (17.2,-5.2) node{$(\gamma,p_{i+1})$} ;

\draw (17.1,-2.3) node{$(\eta,1) \in \hat{\alpha}$} ;
 
 \draw[densely dotted,  thick] (6,-1.5)--(6,-8.5);
  \draw[densely dotted,  thick] (11,-1.5)--(11,-8.5) ;
   \draw[densely dotted,  thick] (16,-1.5)--(16,-8.5);
  
  \draw (7.2,-8.9) node{$\hat{\alpha}\cup{(\gamma,p_{i+1}-p_{i})\cup{(\delta_{1},1)\cup{(\delta_{2},1)}}}$};
  \draw (16,-8.9) node{$\hat{\alpha}\cup{(\gamma,p_{i+1})}$};
  \draw (12,-8.9) node{$\hat{\alpha}\cup{(\delta',1)}\cup{(\delta_{2},1)}$};
  \draw (11.5,-5.5) node{$(\delta',1)$};
  
   \draw (16,-4.7) arc [start angle=90,end angle=270,x radius=0.25,y radius=0.5];
 \draw (16,-5.7) arc [start angle=270,end angle=450,x radius=0.25,y radius=0.5];
 
  \draw (11,-4) to [out=0,in=180] (16,-4.7);
 \draw (11,-5) to [out=0,in=90] (13,-6.1);
 \draw (13,-6.1) to [out=270,in=0] (11,-7.3);
 \draw (16,-5.7) to [out=180,in=0] (11,-8.3);
 
  \draw (11,-7.3) arc [start angle=90,end angle=270,x radius=0.25,y radius=0.5];
 \draw (11,-8.3)[dashed] arc [start angle=270,end angle=450,x radius=0.25,y radius=0.5];
 
 \draw (6,-7.3) arc [start angle=90,end angle=270,x radius=0.25,y radius=0.5];
 \draw (6,-8.3)[dashed] arc [start angle=270,end angle=450,x radius=0.25,y radius=0.5];
 
  \draw (11,-7.3)--(6,-7.3) ;
  \draw (11,-8.3)--(6,-8.3) ;

  \draw (11,-4) arc [start angle=90,end angle=270,x radius=0.25,y radius=0.5];
 \draw (11,-5)[dashed] arc [start angle=270,end angle=450,x radius=0.25,y radius=0.5];

\draw (11,-9.8) node{\Large Type $(B)$} ;
\end{tikzpicture}

\begin{lem}\label{typb}
Suppose that the pair of $\zeta$, $J$-holomorphic curves counted by $U\langle\hat{\alpha}\cup{(\gamma,p_{i+1})}  \rangle=\langle \zeta \rangle$ and   $U\langle \zeta \rangle=\langle \hat{\alpha}\cup{(\gamma,p_{i+1}-p_{i})}\cup{(\delta_{1},1)}\cup{(\delta_{2},1)}) \rangle$ are above. Then, one of the following holds.

\item[($\Delta_{1}$).] $\frac{3}{2}p_{i}=p_{i+1}$. Moreover
$A(\delta_{1})\approx \frac{1}{2}p_{i}R$, $A(\delta_{2})\approx \frac{1}{2}p_{i}R$ and for each $\eta \in \hat{\alpha}$, either $A(\eta)\approx \frac{1}{2}p_{i}R$ or $A(\eta)\approx \frac{1}{4}p_{i}R$.

\item[($\Delta_{2}$).]$\frac{4}{3}p_{i}=p_{i+1}$. Moreover, $A(\delta_{1})\approx \frac{2}{3}p_{i}R$, $A(\delta_{2})\approx \frac{1}{3}p_{i}R$ and for each $\eta \in \hat{\alpha}$, either $A(\eta)\approx \frac{1}{2}p_{i}R$ or $A(\eta)\approx \frac{1}{6}p_{i}R$.
\end{lem}

\begin{proof}[\bf Proof of Proposition \ref{nagai}]

Since ($\Delta_{1}$) and ($\Delta_{2}$) are correspond to ($b$) and ($c$) respectively, combine with the result of Type (A), we complete the proof of Proposition \ref{nagai}.
\end{proof}

\begin{proof}[\bf Proof of Lemma \ref{typb}]

Let $z\to \eta$. Then we obtain three  possibilities from the splitting behaviors of the curves counted by $U\langle\hat{\alpha}\cup{(\gamma,p_{i+1})}  \rangle=\langle \zeta \rangle$.

       \item[($\mathfrak{j}_{1}$).] $|A(\delta_{2})-2A(\eta)|<\epsilon$
        \item [($\mathfrak{j}_{2}$).] $|p_{i+1}R-2A(\eta)|<\epsilon$
        \item[($\mathfrak{j}_{3}$).]$|p_{i+1}R-(2A(\eta)+A(\delta_{2}))|<\epsilon$
        
        By the same way, we have three possibilities of the approximate actions in the orbits from the splitting behaviors of the curves counted by $U\langle \zeta \rangle=\langle \hat{\alpha}\cup{(\gamma,p_{i+1}-p_{i})}\cup{(\delta_{1},1)}\cup{(\delta_{2},1)}) \rangle$.
        
         \item[($\mathfrak{k}_{1}$).] $|A(\delta_{1})-2A(\eta)|<\epsilon$
        \item [($\mathfrak{k}_{2}$).] $|(p_{i+1}-p_{i})R-2A(\eta)|<\epsilon$
        \item[($\mathfrak{k}_{3}$).]$|(p_{i+1}-p_{i})R+A(\delta_{1})-2A(\eta)|<\epsilon$
        
        Also let $z\to \delta_{2}$. then by the splitting behaviors of the curves counted by $U\langle \zeta \rangle=\langle \hat{\alpha}\cup{(\gamma,p_{i+1}-p_{i})}\cup{(\delta_{1},1)}\cup{(\delta_{2},1)}) \rangle$, we have
        
     \item[($\mathfrak{l}_{1}$).] $|A(\delta_{1})-2A(\delta_{2})|<\epsilon$
        \item [($\mathfrak{l}_{2}$).] $|(p_{i+1}-p_{i})R-2A(\delta_{2})|<\epsilon$
        \item[($\mathfrak{l}_{3}$).]$|(p_{i+1}-p_{i})R+A(\delta_{1})-2A(\delta_{2})|<\epsilon$

        Recall that we always have ($\spadesuit$).$|A(\delta_{1})+A(\delta_{2})-p_{i}R|<\epsilon$, since $A(\alpha_{k+1})-A(\alpha_{k})<\epsilon$.
         
        As is the same with so far, Some pair (($\mathfrak{j}_{i_{1}}$),($\mathfrak{k}_{i_{2}}$),($\mathfrak{l}_{i_{3}}$)) cause contradictions as follows.

\begin{lem}\label{typB}
Any above pairs (($\mathfrak{j}_{i_{1}}$),($\mathfrak{k}_{i_{2}}$),($\mathfrak{l}_{i_{3}}$)) causes a contradiction except for the following cases.

$(i_{1},i_{2},i_{3})=(1,1,3),(1,2,1),(1,2,3),(3,3,1),(3,3,3)$.

\end{lem}

\begin{proof}[\bf Proof of Lemma \ref{typB}]

At first, we will exclude pairs (($\mathfrak{j}_{i_{1}}$),($\mathfrak{k}_{i_{2}}$)) such that we can derive a contradiction by only them.        

\begin{cla}\label{clB}
Each of  the following pairs (($\mathfrak{j}_{i_{1}}$),($\mathfrak{k}_{i_{2}}$)) causes a contradiction.

$(i_{1},i_{2})$=$(2,1)$, $(2,2)$, $(2,3)$, $(3,1)$, $(3,2)$.

\end{cla}

\begin{proof}[\bf Proof of Claim \ref{clB}]

\item Case $(i_{1},i_{2})$=$(2,1)$

By cancelling the term $2A(\eta)$ from
($\mathfrak{j}_{2}$).$|p_{i+1}R-2A(\eta)|<\epsilon$ and ($\mathfrak{k}_{1}$).$|A(\delta_{1})-2A(\eta)|<\epsilon$, we have $|p_{i+1}R-A(\delta_{1})|<2\epsilon$. This contradicts ($\spadesuit$).$|p_{i}R-(A(\delta_{1})+A(\delta_{2}))|<\epsilon$.

\item Case $(i_{1},i_{2})$=$(2,2)$

By cancelling the term $2A(\eta)$ from
($\mathfrak{j}_{2}$).$|p_{i+1}R-2A(\eta)|<\epsilon$ and ($\mathfrak{k}_{2}$).$|(p_{i+1}-p_{i})R-2A(\eta)|<\epsilon$, we have $|p_{i}R|<2\epsilon$. This is a contradiction.

\item Case $(i_{1},i_{2})$=$(2,3)$

By cancelling the term $2A(\eta)$ from
($\mathfrak{j}_{2}$).$|p_{i+1}R-2A(\eta)|<\epsilon$ and ($\mathfrak{k}_{3}$).$|(p_{i+1}-p_{i})R+A(\delta_{1})-2A(\eta)|<\epsilon$, we have $|p_{i}R-A(\delta_{1})|<2\epsilon$. This contradicts ($\spadesuit$).$|p_{i}R-(A(\delta_{1})+A(\delta_{2}))|<\epsilon$.

\item Case $(i_{1},i_{2})$=$(3,1)$

By cancelling the term $2A(\eta)$ from
($\mathfrak{j}_{3}$).$|p_{i+1}R-(2A(\eta)+A(\delta_{2}))|<\epsilon$ and ($\mathfrak{k}_{1}$).$|A(\delta_{1})-2A(\eta)|<\epsilon$, we have $|p_{i+1}R-(A(\delta_{1})+A(\delta_{2}))|<2\epsilon$. But this contradicts ($\spadesuit$).$|p_{i}R-(A(\delta_{1})+A(\delta_{2}))|<\epsilon$.

\item Case $(i_{1},i_{2})$=$(3,2)$

By cancelling the term $2A(\eta)$ from
($\mathfrak{j}_{3}$).$|p_{i+1}R-(2A(\eta)+A(\delta_{2}))|<\epsilon$ and ($\mathfrak{k}_{2}$).$|(p_{i+1}-p_{i})R-2A(\eta)|<\epsilon$, we have $|p_{i}R-A(\delta_{2})|<2\epsilon$. But this contradicts ($\spadesuit$).$|p_{i}R-(A(\delta_{1})+A(\delta_{2}))|<\epsilon$.
\end{proof}

 By the arguments so far, we have only to show that
each of the following pairs (($\mathfrak{j}_{i_{1}}$),($\mathfrak{k}_{i_{2}}$),($\mathfrak{l}_{i_{3}}$)) causes a contradiction.

$(i_{1},i_{2},i_{3})$=$(1,1,1)$, $(1,1,2)$, $(1,3,1)$, $(1,3,2)$, $(1,3,3)$, $(1,2,2)$, $(3,3,2)$.

\item Case $(i_{1},i_{2},i_{3})$=$(1,1,1)$

Obviously, ($\mathfrak{j}_{1}$).$|A(\delta_{2})-2A(\eta)|<\epsilon$ and ($\mathfrak{k}_{1}$).$|A(\delta_{1})-2A(\eta)|<\epsilon$ contradict
 ($\mathfrak{l}_{1}$).$|A(\delta_{1})-2A(\delta_{2})|<\epsilon$

\item Case $(i_{1},i_{2},i_{3})$=$(1,1,2)$

By cancelling the terms $A(\delta_{1})$, $A(\delta_{2})$ and $2A(\eta)$ from ($\mathfrak{j}_{1}$).$|A(\delta_{2})-2A(\eta)|<\epsilon$, ($\mathfrak{k}_{1}$).$|A(\delta_{1})-2A(\eta)|<\epsilon$ ($\mathfrak{l}_{2}$).$|(p_{i+1}-p_{i})R-2A(\delta_{2})|<\epsilon$, and ($\spadesuit$).$|p_{i}R-(A(\delta_{1})+A(\delta_{2}))|<\epsilon$, we have $|p_{i}R-(p_{i+1}-p_{i})R|<4\epsilon$. Hence $p_{i}=p_{i+1}-p_{i}$. This is a contradiction.

\item Case $(i_{1},i_{2},i_{3})$=$(1,3,1)$

By cancelling the terms $A(\delta_{1})$ and $2A(\eta)$ from ($\mathfrak{j}_{1}$).$|A(\delta_{2})-2A(\eta)|<\epsilon$, ($\mathfrak{k}_{3}$).$|(p_{i+1}-p_{i})R+A(\delta_{1})-2A(\eta)|<\epsilon$ and ($\spadesuit$).$|p_{i}R-(A(\delta_{1})+A(\delta_{2}))|<\epsilon$, we have $|p_{i+1}R-2A(\delta_{2})|<3\epsilon$ and hence $|(p_{i}-\frac{1}{2}p_{i+1})R-A(\delta_{1})|<\frac{5}{2}\epsilon$.

By combining with ($\mathfrak{l}_{1}$).$|A(\delta_{1})-2A(\delta_{2})|<\epsilon$, we have $|(\frac{3}{2}p_{i+1}-p_{i})R|<\frac{13}{2}\epsilon$ This is a contradiction.

\item Case $(i_{1},i_{2},i_{3})$=$(1,3,2)$

In the same way as above, from ($\mathfrak{j}_{1}$).$|A(\delta_{2})-2A(\eta)|<\epsilon$, ($\mathfrak{k}_{3}$).$|(p_{i+1}-p_{i})R+A(\delta_{1})-2A(\eta)|<\epsilon$, and ($\spadesuit$).$|p_{i}R-(A(\delta_{1})+A(\delta_{2}))|<\epsilon$, we have $|p_{i+1}R-2A(\delta_{2})|<3\epsilon$. This obviously contradicts ($\mathfrak{l}_{2}$).$|(p_{i+1}-p_{i})R-2A(\delta_{2})|<\epsilon$.

\item Case $(i_{1},i_{2},i_{3})$=$(1,3,3)$

In the same way as above, from ($\mathfrak{j}_{1}$).$|A(\delta_{2})-2A(\eta)|<\epsilon$, ($\mathfrak{k}_{3}$).$|(p_{i+1}-p_{i})R+A(\delta_{1})-2A(\eta)|<\epsilon$, and ($\spadesuit$).$|p_{i}R-(A(\delta_{1})+A(\delta_{2}))|<\epsilon$, we have $|p_{i+1}R-2A(\delta_{2})|<3\epsilon$ and $|(p_{i}-\frac{1}{2}p_{i+1})R-A(\delta_{1})|<\frac{5}{2}\epsilon$.

By combining with ($\mathfrak{l}_{3}$).$|(p_{i+1}-p_{i})R+A(\delta_{1})-2A(\delta_{2})|<\epsilon$, we have $|\frac{1}{2}p_{i+1}R|<\frac{13}{2}\epsilon$.

\item Case $(i_{1},i_{2},i_{3})$=$(1,2,2)$

By cancelling the term $2A(\eta)$ from ($\mathfrak{j}_{1}$).$|A(\delta_{2})-2A(\eta)|<\epsilon$ and  ($\mathfrak{j}_{2}$).$|p_{i+1}R-2A(\eta)|<\epsilon$, we have $|p_{i+1}R-A(\delta_{2})|<2\epsilon$. This obviously contradicts ($\mathfrak{l}_{2}$).$|(p_{i+1}-p_{i})R-2A(\delta_{2})|<\epsilon$.

\item Case $(i_{1},i_{2},i_{3})$=$(3,3,2)$

From ($\mathfrak{j}_{3}$).$|p_{i+1}R-(2A(\eta)+A(\delta_{2}))|<\epsilon$, ($\mathfrak{l}_{2}$).$|(p_{i+1}-p_{i})R-2A(\delta_{2})|<\epsilon$ and ($\spadesuit$).$|p_{i}R-A(\delta_{1})-A(\delta_{2})|<\epsilon$, we have 
\begin{equation}\label{atla}
(\frac{3}{2}p_{i}-\frac{1}{2}p_{i+1})R\approx A(\delta_{1}),\,\,\frac{1}{2}(p_{i+1}-p_{i})R\approx A(\delta_{2}),\,\,\frac{1}{4}(p_{i+1}+p_{i})R\approx A(\eta).
\end{equation}

Recall that when we consider the splitting behaviors of $J$-holomorphic curve counted by $U\langle \alpha_{k+1} \rangle=\langle \alpha_{k} \rangle$ as $z \to \eta$ for some $\eta\in \hat{\alpha}$. Then there are three possibilities of splitting of holomorphic curve and also we have three possibilities of the approximate relation in the actions as follows.
\item[($\mathfrak{b}_{1})$.] $|A(\delta_{1})-2A(\eta)|<\epsilon$
\item[$(\mathfrak{b}_{2})$.] $|A(\delta_{2})-2A(\eta)|<\epsilon$
\item[$(\mathfrak{b}_{3})$.] $|p_{i}R-2A(\eta)|<\epsilon$.

But obviously, (\ref{atla}) contradicts the above relations in any case.

Combining the arguments, we complete the proof of Lemma \ref{typB}.
\end{proof}

To complete the proof of Lemma \ref{typb}, it is sufficient to compute the approximate relations from the rest cases $(i_{1},i_{2},i_{3})=(1,1,3)$, $(1,2,1)$, $(1,2,3)$, $(3,3,1)$, $(3,3,3)$.

\item Case $(i_{1},i_{2},i_{3})=(1,1,3)$

From $(\mathfrak{j}_{1})$.$|A(\delta_{2})-2A(\eta)|<\epsilon$, $(\mathfrak{k}_{1})$.$|A(\delta_{1})-2A(\eta)|<\epsilon$, 
$(\mathfrak{l}_{3})$.$|(p_{i+1}-p_{i})R+A(\delta_{1})-2A(\delta_{2})|<\epsilon$ and ($\spadesuit$).$|A(\delta_{1})+A(\delta_{2})-p_{i}R|<\epsilon$, we have
\begin{equation}
    A(\delta_{1})\approx \frac{1}{2}p_{i}R,\,\, A(\delta_{2}) \approx \frac{1}{2}p_{i}R,\,\,p_{i+1}= \frac{3}{2}p_{i},\,\,A(\eta)\approx \frac{1}{4}p_{i}.
\end{equation}

\item Case $(i_{1},i_{2},i_{3})=(1,2,1)$

From $(\mathfrak{j}_{1})$.$|A(\delta_{2})-2A(\eta)|<\epsilon$, $(\mathfrak{k}_{2})$.$|(p_{i+1}-p_{i})R-2A(\eta)|<\epsilon$, $(\mathfrak{l}_{1})$.$|A(\delta_{1})-2A(\delta_{2})|<\epsilon$ and ($\spadesuit$).$|A(\delta_{1})+A(\delta_{2})-p_{i}R|<\epsilon$. we have
\begin{equation}
    A(\delta_{1})\approx \frac{2}{3}p_{i}R,\,\,A(\delta_{2})\approx \frac{1}{3}p_{i}R,\,\,p_{i+1}=\frac{4}{3}p_{i},\,\,A(\eta)\approx \frac{1}{6}p_{i}R.
\end{equation}

\item Case $(i_{1},i_{2},i_{3})=(1,2,3)$

From $(\mathfrak{j}_{1})$.$|A(\delta_{2})-2A(\eta)|<\epsilon$, $(\mathfrak{k}_{2})$.$|(p_{i+1}-p_{i})R-2A(\eta)|<\epsilon$, $(\mathfrak{l}_{3})$.$|(p_{i+1}-p_{i})R+A(\delta_{1})-2A(\delta_{2})|<\epsilon$  and ($\spadesuit$).$|A(\delta_{1})+A(\delta_{2})-p_{i}R|<\epsilon$, we have
\begin{equation}
     A(\delta_{1})\approx\frac{1}{2}p_{i}R,\,\, A(\delta_{2}) \approx \frac{1}{2}p_{i}R,\,\,p_{i+1}= \frac{3}{2}p_{i},\,\,A(\eta)\approx \frac{1}{4}p_{i}.
\end{equation}

\item Case $(i_{1},i_{2},i_{3})=(3,3,1)$

Recall that by considering the splitting behavior of $J$-holomorphic curves counted by $U$-map from $\langle \alpha_{k+1} \rangle$ to $\langle \alpha_{k} \rangle$ as $z\to \eta$. Then we have three possibilties,

\item[$(\mathfrak{b}_{1})$.] $|A(\delta_{1})-2A(\eta)|<\epsilon$
\item[$(\mathfrak{b}_{2})$.] $|A(\delta_{2})-2A(\eta)|<\epsilon$
\item[$(\mathfrak{b}_{3})$.] $|p_{i}R-2A(\eta)|<\epsilon$

The first inequality contradicts $(\mathfrak{k}_{3})$.$|(p_{i+1}-p_{i})R+A(\delta_{1})-2A(\eta)|<\epsilon$ and also the second one does $(\mathfrak{k}_{3})$.$|(p_{i+1}-p_{i})R+A(\delta_{1})-2A(\eta)|<\epsilon$ and $|A(\delta_{1})-2A(\delta_{2})|<\epsilon$. So only third one may be possible.

By combining with $(\mathfrak{j}_{3})$.$|p_{i+1}R-(2A(\eta)+A(\delta_{2}))|<\epsilon$, $(\mathfrak{k}_{3})$.$|(p_{i+1}-p_{i})R+A(\delta_{1})-2A(\eta)|<\epsilon$, $(\mathfrak{l}_{1})$.$|A(\delta_{1})-2A(\delta_{2})|<\epsilon$ and ($\spadesuit$).$|A(\delta_{1})+A(\delta_{2})-p_{i}R|<\epsilon$, we have
\begin{equation}\label{}
    A(\delta_{1})\approx \frac{2}{3}p_{i}R,\,\,A(\delta_{2})\approx \frac{1}{3}p_{i}R,\,\,p_{i+1}=\frac{4}{3}p_{i},\,\,A(\eta)\approx \frac{1}{2}p_{i}R.
\end{equation}

\item Case $(i_{1},i_{2},i_{3})=(3,3,3)$

From $(\mathfrak{j}_{3})$.$|p_{i+1}R-(2A(\eta)+A(\delta_{2}))|<\epsilon$, $(\mathfrak{k}_{3})$.$|(p_{i+1}-p_{i})R+A(\delta_{1})-2A(\eta)|<\epsilon$, $(\mathfrak{l}_{3})$.$|(p_{i+1}-p_{i})R+A(\delta_{1})-2A(\delta_{2})|<\epsilon$ and ($\spadesuit$).$|A(\delta_{1})+A(\delta_{2})-p_{i}R|<\epsilon$ we have
\begin{equation}
     A(\delta_{1})\approx\frac{1}{2}p_{i}R,\,\, A(\delta_{2}) \approx \frac{1}{2}p_{i}R,\,\,p_{i+1}= \frac{3}{2}p_{i},\,\,A(\eta)\approx \frac{1}{2}p_{i}.
\end{equation}

Combining the arguments, we complete the proof of Lemma \ref{typb}.
\end{proof}

\section{Proof of Theorem \ref{mainmainmain}}
Suppose that $\alpha_{k}$, $\alpha_{k+1}$ satisfy the assumptions 1, 2, 3, 4 and 5 in Proposition \ref{nagai}, then  the pair $(\alpha_{k},\alpha_{k+1})$ is one of the types of \textbf{(a)}, \textbf{(a')}, \textbf{(b)}, \textbf{(b')}, \textbf{(c)}, \textbf{(c')}. Moreover, we can see that  any actions of negative hyperbolic orbits in $\alpha_{k+1}$, $\alpha_{k}$ are in $(\frac{1}{12}R\mathbb{Z})_{\epsilon'}=\{\,x \in \mathbb{R}\,|\,\mathrm{dist}(x,\frac{1}{12}R\mathbb{Z})<\epsilon'\,\,\}$ where $\epsilon'=\frac{1}{100}\mathrm{max}\{A(\alpha)\,|\alpha\,\,\mathrm{is\,\,a\,\,Reeb\,\,orbit}\,\}$. Therefore, we can define a map
\begin{equation}
    f:\{\,\eta\in \alpha_{k}\,|\,\eta\,\,\mathrm{is\,\,negative\,\,hyperbolic}\,\}\to  \frac{1}{12}R\mathbb{Z}
\end{equation}
Here, $f(\eta)\in \frac{1}{12}R\mathbb{Z}$ is an image by the natural projection of $A(\eta)$ from $(\frac{1}{12}R\mathbb{Z})_{\epsilon'}$ to $\frac{1}{12}R\mathbb{Z}$. Note that $\frac{1}{12}R\mathbb{Z}$ is discrete.

Let $\alpha_{k+1}$, $\alpha_{k}$ and $\alpha_{k-1}$ be such 3 consecutive admissible orbit sets. then each type of the pairs $(\alpha_{k},\alpha_{k+1})$ and $(\alpha_{k-1},\alpha_{k})$ decides the action relations in $\alpha_{k}$ respectively. But  sometimes, they may be incompatible. In particular, we can show that such fact induces  contradictions with Proposition \ref{main index 2}.

At first, we  show the next lemma.

\begin{lem}\label{las}
Suppose that $\alpha_{k-1}$, $\alpha_{k}$ and $\alpha_{k+1}$ satisfy the assumptions 1, 2, 3, 4 and 5 in Proposition \ref{nagai}. Then the following pairs of types of $(\alpha_{k},\alpha_{k+1})$ and $(\alpha_{k-1},\alpha_{k})$ classified in Proposition \ref{nagai} cause contradictions.

    (type of ($\alpha_{k-1},\alpha_{k}$), type of  ($\alpha_{k},\alpha_{k+1}$))\\
    =(\textbf{a'}, \textbf{a}), (\textbf{b'}, \textbf{a'}), (\textbf{b'}, \textbf{a}), (\textbf{b'}, \textbf{b}), (\textbf{b'}, \textbf{b'}), (\textbf{b'}, \textbf{c}), (\textbf{b'}, \textbf{c'}), (\textbf{c'}, \textbf{a}), (\textbf{c'}, \textbf{a'}),
    (\textbf{c'}, \textbf{b'}), (\textbf{c'}, \textbf{c}), (\textbf{c'}, \textbf{c'}), (\textbf{a'}, \textbf{b}), (\textbf{b}, \textbf{b}), (\textbf{c'}, \textbf{b}), (\textbf{c}, \textbf{b}), (\textbf{a}, \textbf{c}),
    (\textbf{b}, \textbf{c}), (\textbf{c}, \textbf{c}), (\textbf{a}, \textbf{a}), (\textbf{a'}, \textbf{a'}), (\textbf{a}, \textbf{b}), (\textbf{a'}, \textbf{c}), (\textbf{a}, \textbf{a'}).

\end{lem}

\begin{proof}[\bf Proof of Lemma \ref{las}]

We derive a contradiction respectively.

\item [\bf Case] (\textbf{b'}, \textbf{a})(or by symmetry, (\textbf{a'}, \textbf{b}))

Since the type of\,\,($\alpha_{k-1},\alpha_{k}$) is (\textbf{b'}),  the number of the inverse image of the largest value of $f:\{\,\eta\in \alpha_{k}\,|\,\eta\,\,\mathrm{is\,\,negative\,\,hyperbolic}\,\}\to  \frac{1}{12}R\mathbb{Z}$ is at least two. But the assumption  that the type of\,\,($\alpha_{k},\alpha_{k+1}$) is (\textbf{a}) indicates that  the number of the inverse image of the largest value of $f:\{\,\eta\in \alpha_{k}\,|\,\eta\,\,\mathrm{is\,\,negative\,\,hyperbolic}\,\}\to  \frac{1}{12}R\mathbb{Z}$ is one. This is a contradiction.

\item [\bf Case] (\textbf{b'}, \textbf{a'})(or by symmetry, (\textbf{a}, \textbf{b}))

Let $E(\alpha_{k-1})=M$, then $E(\alpha_{k})=M-q_{i}$ where $q_{i}=\mathrm{max}(S_{\theta}\cap{\{1,\,\,2,...,\,\,M\}})$. Let $q_{i'}=\mathrm{max}(S_{\theta}\cap{\{1,\,\,2,...,\,\,M-q_{i}\}})$. 

Since the type of\,\,($\alpha_{k},\alpha_{k+1}$) is (\textbf{a'}),  the image of $f$ on $\alpha_{k}$ is in $\{\,\frac{1}{2}q_{i'+1}R,\,\,\frac{1}{2}(q_{i'+1}-q_{i'})R,\,\,(q_{i'+1}-q_{i'})R\,\}$.

Since the type of\,\,($\alpha_{k-1},\alpha_{k}$) is (\textbf{a'}), we have $q_{i+1}=\frac{3}{2}q_{i}$. 
Hence
\begin{equation}
  q_{i'} \leq M-q_{i}\leq q_{i+1}-q_{i}= \frac{1}{2}q_{i}<q_{i}.
\end{equation}
Therefore $q_{i'+1}\leq q_{i}$.

If $q_{i'+1}<q_{i}$, the actions of each orbit in $\hat{\alpha}$ are less than $\frac{1}{2} q_{i}R$. The number of the positive ends of (\textbf{b'}) whose image of $f$ are largest is at least two but that one of negative ends of (\textbf{a'}) whose image of $f$ are largest is only  $\delta_{2}$. This is a contradiction.

If $q_{i'+1}=q_{i}$, by considering the correspondence of the approximate action values and Claim \ref{fre}, we can see $\frac{1}{4}q_{i}R = (q_{i'+1}-q_{i'})R$ and so $\frac{1}{4}q_{i}=q_{i'+1}-q_{i'}=q_{i}-q_{i'}$ hence $\frac{3}{4}q_{i}=q_{i'}$. This contradict  $q_{i'}\leq \frac{1}{2}q_{i}$.

\item [\bf Case] (\textbf{b'}, \textbf{b})

Let $E(\alpha_{k-1})=M$ $E(\alpha_{k+1})=M'$, then $E(\alpha_{k})=M-q_{i}=M'-p_{i'}$ where $q_{i}=\mathrm{max}(S_{\theta}\cap{\{1,\,\,2,...,\,\,M\}})$ and $p_{i'}=\mathrm{max}(S_{-\theta}\cap{\{1,\,\,2,...,\,\,M'\}})$.

Since the type of\,\,($\alpha_{k-1},\alpha_{k}$) is (\textbf{b'}), the largest value of $f$ on $\alpha_{k}$ is $\frac{1}{2}q_{i}R$ and Since the type of\,\,($\alpha_{k},\alpha_{k+1}$) is (\textbf{b}), the largest value of $f$ on $\alpha_{k}$ is $\frac{1}{2} p_{i'}R$. Therefore we have $q_{i}=p_{i'}$. This contradicts $S_{\theta}\cap{S_{-\theta}}=\{1\}$.

\item [\bf Case] (\textbf{b'}, \textbf{b'}) (or by symmetry, (\textbf{b}, \textbf{b}))

Let $E(\alpha_{k-1})=M$, then $E(\alpha_{k})=M-q_{i}<\frac{1}{2}q_{i}$. This implies that the largest approximate value decided by the positive ends of former (\textbf{b'}) is larger than that one decided by the negative ends of the later (\textbf{b'}) and so this is a contradiction.

\item [\bf Case] (\textbf{b'}, \textbf{c}) (or by symmetry, (\textbf{c'}, \textbf{b}))

Since the type of\,\,($\alpha_{k-1},\alpha_{k}$) is (\textbf{b'}),  the number of the  image of  $f$ on $\alpha_{k}$ is at most 2. On the other hand, since the type of\,\,($\alpha_{k},\alpha_{k+1}$) is (\textbf{c}),  the number of the  image of  $f$ on $\alpha_{k}$ is at least three. This is a contradiction.

\item [\bf Case] (\textbf{b'}, \textbf{c'}) (or by symmetry, (\textbf{c}, \textbf{b}))

Let $E(\alpha_{k-1})=M$. Then $E(\alpha_{k})=M-q_{i}<\frac{1}{2}q_{i}$. This implies that the largest approximate  value decided by the action of the positive ends of  (\textbf{b'}) is larger than that one decided by the action of the negative ends of  (\textbf{c'}) and so this is a contradiction.

\item [\bf Case] (\textbf{c'}, \textbf{a}) (or by symmetry, (\textbf{a'}, \textbf{c}))

Let $E(\alpha_{k-1})=M$ $E(\alpha_{k+1})=M'$, then $E(\alpha_{k})=M-q_{i}=M'-p_{i'}$ where $q_{i}=\mathrm{max}(S_{\theta}\cap{\{1,\,\,2,...,\,\,M\}})$ and $p_{i'}=\mathrm{max}(S_{-\theta}\cap{\{1,\,\,2,...,\,\,M'\}})$.

Since the type of\,\,($\alpha_{k-1},\alpha_{k}$) is (\textbf{c'}), the largest value of $f$ on $\alpha_{k}$ is $\frac{2}{3}q_{i}R$ and Since the type of\,\,($\alpha_{k},\alpha_{k+1}$) is (\textbf{a}), the largest value of $f$ on $\alpha_{k}$ is $p_{i'+1}R$. Therefore we have $\frac{2}{3}q_{i}=p_{i'+1}$.
By considering the correspondence of the actions of  the part of trivial cylinders, we have two possibilities, $\frac{1}{2}q_{i}=\frac{1}{2}(p_{i'+1}-p_{i'})$ or $\frac{1}{6}q_{i}=\frac{1}{2}(p_{i'+1}-p_{i'})$.

If $\frac{1}{2}q_{i}=\frac{1}{2}(p_{i'+1}-p_{i'})$, this contradicts  $\frac{2}{3}q_{i}=p_{i'+1}$.
If $\frac{1}{6}q_{i}=\frac{1}{2}(p_{i'+1}-p_{i'})$,  we have $\frac{1}{2}p_{i'+1}=(p_{i'+1}-p_{i'})$ since $\frac{2}{3}q_{i}=p_{i'+1}$. This contradicts Claim \ref{fre}.

Combining the above arguments, we can see that this case causes a contradiction.

\item [\bf Case] (\textbf{c'}, \textbf{a'}) (or by symmetry, (\textbf{a}, \textbf{c}))

Let $E(\alpha_{k-1})=M$. Then $E(\alpha_{k})=M-q_{i}$ $E(\alpha_{k+1})=M-q_{i}-q_{i'}$ where $q_{i}=\mathrm{max}(S_{\theta}\cap{\{1,\,\,2,...,\,\,M\}})$ and $q_{i'}=\mathrm{max}(S_{\theta}\cap{\{1,\,\,2,...,\,\,M-q_{i}\}})$.

Moreover, since $\frac{4}{3}q_{i+1}=q_{i}$, 
\begin{equation}\label{addd}
    q_{i'}<E(\alpha_{k})=M-q_{i}<\frac{1}{3}q_{i}<q_{i}
\end{equation}
and so 
\begin{equation}\label{asda}
    q_{i'+1}\leq q_{i}. 
\end{equation}

Since the type of\,\,($\alpha_{k-1},\alpha_{k}$) is (\textbf{c'}), the largest value of $f$ on $\alpha_{k}$ is $\frac{2}{3}q_{i}R$. On the other hand, Since the type of\,\,($\alpha_{k-1},\alpha_{k}$) is (\textbf{a'}), the largest value of $f$ on $\alpha_{k}$ is $(q_{i'+1}-q_{i'})R$ or $\frac{1}{2}q_{i'+1}$. This indicates that $\frac{2}{3}q_{i}=q_{i'+1}-q_{i'}$ or $\frac{2}{3}q_{i}=\frac{1}{2}q_{i'+1}$ but the later case contradicts  (\ref{asda}). So it is sufficient to consider the former case.

Suppose that $\frac{2}{3}q_{i}=q_{i'+1}-q_{i'}$.
By considering the correspondence of the actions of the part of  trivial cylinder, we have two possibilities, $\frac{1}{2}q_{i}=\frac{1}{2}q_{i'+1}$ or $\frac{1}{6}q_{i}=\frac{1}{2}q_{i'+1}$.

If $\frac{1}{2}q_{i}=\frac{1}{2}q_{i'+1}$, we can see $\frac{1}{3}q_{i}=q_{i'}$ by $\frac{2}{3}q_{i}=q_{i'+1}-q_{i'}$. But this contradict (\ref{addd}) $q_{i'}<\frac{1}{3}q_{i}$.

If $\frac{1}{6}q_{i}=\frac{1}{2}q_{i'+1}$, we can see $\frac{1}{3}q_{i}+q_{i'}=0$ by $\frac{2}{3}q_{i}=q_{i'+1}-q_{i'}$. This is a contradiction.

Combining the above argument, we can see that this case cause a contradiction.

\item [\bf Case] (\textbf{c'}, \textbf{b'}) (or by symmetry, (\textbf{b}, \textbf{c}))

Since the type of\,\,($\alpha_{k-1},\alpha_{k}$) is (\textbf{c'}),  the number of the  image of  $f$ on $\alpha_{k}$ is at least three. On the other hand, Since the type of\,\,($\alpha_{k},\alpha_{k+1}$) is (\textbf{b'}),  the number of the  image of  $f$ on $\alpha_{k}$ is at most two. This is a contradiction.

\item [\bf Case] (\textbf{c'}, \textbf{c})

Let $E(\alpha_{k-1})=M$ and $E(\alpha_{k+1})=M'$. Let $q_{i}=\mathrm{max}(S_{\theta}\cap{\{1,\,\,2,...,\,\,M\}})$ and $p_{i'}=\mathrm{max}(S_{-\theta}\cap{\{1,\,\,2,...,\,\,M'\}})$.

Since the type of\,\,($\alpha_{k-1},\alpha_{k}$) is (\textbf{c'}), the largest value of $f$ on $\alpha_{k}$ is $\frac{2}{3}q_{i}R$. On the other hand, Since the type of\,\,($\alpha_{k},\alpha_{k+1}$) is (\textbf{c}), the largest value of $f$ on $\alpha_{k}$ is $\frac{2}{3}p_{i'}R$. Therefore $q_{i}=p_{i'}$. This  contradicts $S_{\theta}\cap{S_{-\theta}}=\{1\}$.

\item [\bf Case] (\textbf{c'}, \textbf{c'}) (or by symmetry, (\textbf{c}, \textbf{c}))

Since the type of\,\,($\alpha_{k-1},\alpha_{k}$) is (\textbf{c'}),  the number of the image of  $f$ on $\alpha_{k}$ on $\alpha_{k}$ is at least three. On the other hand, Since the type of\,\,($\alpha_{k},\alpha_{k+1}$) is (\textbf{c'}),  the number of the  image of  $f$ on $\alpha_{k}$ is at most two. This is a contradiction.

\item [\bf Case] (\textbf{a}, \textbf{a}) (or by symmetry, (\textbf{a'}, \textbf{a'}))

Let $E(\alpha_{k+1})=M$. Then $E(\alpha_{k})=M-p_{i}$ where $p_{i}=\mathrm{max}(S_{\theta}\cap{\{1,\,\,2,...,\,\,M\}})$. Set $p_{i'}=\mathrm{max}(S_{\theta}\cap{\{1,\,\,2,...,\,\,M-p_{i}\}})$.
By construction, $p_{i}\leq p_{i'}$ and so $p_{i+1}\leq p_{i'+1}$.

Since the type of\,\,($\alpha_{k},\alpha_{k+1}$) is (\textbf{a}),  the largest value of $f$ on $\alpha_{k}$ is $p_{i+1}R$. On the other hand, since the type of\,\,($\alpha_{k-1},\alpha_{k}$) is (\textbf{a}),  the image of $f$ on $\alpha_{k}$ is in $\{\,\frac{1}{2}p_{i'+1}R,\,\,\frac{1}{2}(p_{i'+1}-p_{i'})R,\,\,(p_{i'+1}-p_{i'})R\,\}$.  This is obviously a contradiction.

\item [\bf Case] (\textbf{a}, \textbf{a'}) 

Let $E(\alpha_{k})=M$, $p_{i}=\mathrm{max}(S_{\theta}\cap{\{1,\,\,2,...,\,\,M\}})$, $q_{i'}=\mathrm{max}(S_{-\theta}\cap{\{1,\,\,2,...,\,\,M\}})$.

By considering the correspondence of the actions of the part of  trivial cylinder and Claim \ref{fre}, we have two possibilities, $\frac{1}{2}p_{i+1}=\frac{1}{2}(q_{i'+1}-q_{i'})$ or $\frac{1}{2}q_{i'+1}=\frac{1}{2}(p_{i+1}-p_{i})$. Here we use $S_{\theta}\cap{S_{-\theta}}=\{1\}$ implicitly.

Since the type of\,\,($\alpha_{k-1},\alpha_{k}$) is (\textbf{a}),  the image of $f$ on $\alpha_{k}$ is in $\{\,\frac{1}{2}p_{i+1}R,\,\,\frac{1}{2}(p_{i+1}-p_{i})R,\,\,(p_{i+1}-p_{i})R\,\}$. 

Suppose that $\frac{1}{2}p_{i+1}=\frac{1}{2}(q_{i'+1}-q_{i'})$.  Since the type of\,\,($\alpha_{k},\alpha_{k+1}$) is (\textbf{a'}),  the image of $f$ on $\alpha_{k}$ contains $(q_{i'+1}-q_{i'})R$. But since $\frac{1}{2}p_{i+1}=\frac{1}{2}(q_{i'+1}-q_{i'})$, then $(q_{i'+1}-q_{i'})R$ is larger than any $\{\,\frac{1}{2}p_{i+1}R,\,\,\frac{1}{2}(p_{i+1}-p_{i})R,\,\,(p_{i+1}-p_{i})R\,\}$. This is a contradiction.

In the case of $\frac{1}{2}q_{i'+1}=\frac{1}{2}(p_{i+1}-p_{i})$, we can also derive a contradiction in the same way.

\item [\bf Case] (\textbf{a'}, \textbf{a})

Since the type of\,\,($\alpha_{k-1},\alpha_{k}$) is (\textbf{a'}),
 the largest value in the image of $f$ on $\alpha_{k}$  is $ q_{i+1}R$ and also since the type of\,\,($\alpha_{k},\alpha_{k+1}$) is (\textbf{a}), we have that the largest value in the image of $f$ on $\alpha_{k}$  is $ p_{i+1}R$. This indicates  $p_{i+1}=q_{i+1}$ but this contradicts  $S_{\theta}\cap{S_{-\theta}}=\{1\}$.

 Combining the discussions so far, we complete the proof of Lemma.
\end{proof}

\begin{proof}[\bf Proof of Theorem \ref{mainmainmain}]
By Lemma \ref{las}, we have the rest possibilities of  pairs, (\textbf{b}, \textbf{a}), (\textbf{a}, \textbf{b'}), (\textbf{b}, \textbf{b'}), (\textbf{c}, \textbf{b'}), (\textbf{a'}, \textbf{b'}), (\textbf{c}, \textbf{a}), (\textbf{a}, \textbf{c'}), (\textbf{b}, \textbf{c'}), (\textbf{c}, \textbf{c'}), (\textbf{a'}, \textbf{c'}), (\textbf{b}, \textbf{a'}), (\textbf{c}, \textbf{a'}).

It is easy to check that we can not connect the above pairs more than two. This contradicts Proposition \ref{main index 2} and we complete the proof of Theorem \ref{mainmainmain}.
\end{proof}

\end{document}